\documentclass{amsart}
\usepackage[numeric]{amsrefs}
\usepackage[utf8]{inputenc}
\usepackage{tikz-cd}
\usepackage{quiver}
\usepackage{amsmath}
\usepackage{mathtools} 
\usepackage{amsfonts}
\usepackage{amssymb} 
\usepackage{amsthm}
\usepackage{bbm}
\usepackage{mathrsfs}
\usepackage{stmaryrd}
\usepackage{comment}
\usepackage[hidelinks]{hyperref}
\newcommand{\VAN}[3]{#2}
\mathtoolsset{showonlyrefs}
\newtheorem{Def}[subsubsection]{Definition}
\newtheorem{Thm}[subsubsection]{Theorem}
\newtheorem{Lem}[subsubsection]{Lemma}
\newtheorem{Prop}[subsubsection]{Proposition}
\newtheorem{Cor}[subsubsection]{Corollary}
\newtheorem{Conj}[subsubsection]{Conjecture}
\newtheorem{Question}{Question}
\newtheorem*{Thm*}{Theorem}

\newcounter{alphalabels}

\newtheorem{mainThm}[alphalabels]{Theorem}

\theoremstyle{remark}
\newtheorem{Eg}[subsubsection]{Example}
\newtheorem{Rem}[subsubsection]{Remark}
\newtheorem{Claim}[subsubsection]{Claim}

\numberwithin{equation}{subsection}

\newcommand{\Hom}{\underline{\operatorname{Hom}}}
\newcommand{\Aut}{\mathbf{Aut}}
\newcommand{\spec}{\operatorname{Spec}}
\newcommand{\spf}{\operatorname{Spf}}
\newcommand{\spa}{\operatorname{Spa}}
\newcommand{\spd}{\operatorname{Spd}}
\newcommand{\fp}{\mathbb{F}_{p}}
\newcommand{\fpbar}{\overline{\mathbb{F}}_{p}}
\newcommand{\ovmu}{\overline{\mu}}
\newcommand{\zp}{\mathbb{Z}_{p}}

\newcommand{\zlocp}{\mathbb{Z}_{(p)}}
\newcommand{\zpbr}{\breve{\mathbb{Z}}_{p}}

\newcommand{\zpbreve}{\breve{\mathbb{Z}}_{p}}
\newcommand{\qpbreve}{\breve{\mathbb{Q}}_{p}}
\newcommand{\shg}{\operatorname{Sh}_{K}}
\newcommand{\shgv}{\operatorname{Sh}_{U}\gvx}
\newcommand{\shgb}{\operatorname{Sh}_{K}^{[b]}}

\newcommand{\cb}{\operatorname{Sh}_{K}^{\llbracket b \rrbracket }}

\newcommand{\igcs}{\operatorname{Ig}_{\mathrm{CS}}}

\newcommand{\ovfp}{\overline{\mathbb{F}}_{p}}
\newcommand{\qpbr}{\breve{\mathbb{Q}}_{p}}
\newcommand{\qpbar}{\overline{\mathbb{Q}}_p}
\newcommand{\qp}{\mathbb{Q}_{p}}
\newcommand{\grgmu}{\operatorname{Gr}_{G,[\mu^{-1}]}}
\newcommand{\afp}{\mathbb{A}_{f}^{p}}

\newcommand{\gal}{\operatorname{Gal}}
\newcommand{\gafp}{\mathsf{G}(\afp)}

\newcommand{\gx}{{(\mathsf{G}, \mathsf{X})}}
\newcommand{\hy}{{(\mathsf{H}, \mathsf{Y})}}
\newcommand{\gvx}{(\mathsf{G}_{V},\mathsf{H}_{V})}
\newcommand{\gv}{\mathsf{G}_{V}}
\newcommand{\g}{\mathsf{G}}
\newcommand{\scrs}{\mathscr{S}}
\newcommand{\gmhat}{\widehat{\mathbb{G}}_{m}}
\newcommand{\gmhateta}{\widehat{\mathbb{G}}_{m,\eta}}

\newcommand{\cont}{\operatorname{Cont}}
\newcommand{\Cont}{\operatorname{Cont}}
\newcommand{\scrshat}{\widehat{\scrs}}
\newcommand{\G}{\mathsf{G}}

\newcommand{\B}{\mathsf{B}}
\newcommand{\F}{\mathsf{F}}
\newcommand{\V}{\mathsf{V}}
\newcommand{\E}{\mathsf{E}}
\newcommand{\calG}{\mathcal{G}}
\newcommand{\bmu}{b_{\mu}}
\newcommand{\can}{\mathrm{can}}
\newcommand{\ad}{\mathrm{ad}}

\newcommand{\ebreve}{\breve{E}}

\newcommand{\oee}{\mathcal{O}_{E}}

\newcommand{\tildex}{\widetilde{\mbx}}

\newcommand{\tildeh}{\widetilde{\mathcal{H}}}

\newcommand{\nilp}{\operatorname{Nilp}}
\newcommand{\mbx}{\mathbb{X}}
\newcommand{\mbxcan}{\mathbb{X}^{\mathrm{can}}}
\newcommand{\oeab}{\mathcal{O}_{\eab}}
\newcommand{\an}{\mathrm{an}}

\newcommand{\smat}[1]{\left(\begin{smallmatrix} #1 \end{smallmatrix}\right)}

\newcommand{\mby}{\mathbb{Y}}

\newcommand{\mintgbmu}{\mathcal{M}^{\mathrm{int}}_{\mathcal{G},b,[\mu]}}
\newcommand{\mgbmu}{M_{G,b_{\mu},[\mu],\mathcal{G}(\zp)}}
\newcommand{\minfgbmu}{M_{G,b_{\mu},[\mu],\infty}}
\newcommand{\minftbmu}{M_{T,b_{\mu},[\mu],\infty}}
\newcommand{\mtbmu}{M_{T,b_{\mu},[\mu],\mathcal{T}(\zp)}}
\newcommand{\minttbmu}{\mathcal{M}^{\mathrm{int}}_{\mathcal{T},b_{\mu},[\mu]}}
\newcommand{\mintftbmu}{\mathscr{M}_{\mathcal{T},b_{\mu},[\mu]}}
\newcommand{\mintfgbmu}{\mathscr{M}_{\mathcal{G},b,[\mu]}}
\newcommand{\mintfgbordmu}{\mathscr{M}_{\mathcal{G},b,[\mu]}}
\newcommand{\mintgbordmu}{\mathcal{M}^{\mathrm{int}}_{\mathcal{G},b,[\mu]}}
\newcommand{\dr}{\mathrm{dR}}
\newcommand{\HT}{\mathrm{HT}}
\newcommand{\locan}{\mathrm{la}}
\newcommand{\Lie}{\operatorname{Lie}}

\newcommand{\igcsf}{\mathfrak{Ig}_{\mathrm{CS}}}
\newcommand{\igcsa}{\mathfrak{Ig}_{\mathrm{CS},\eta}}
\newcommand{\igcsd}{\mathfrak{Ig}_{\mathrm{CS},\eta}^{\diamondsuit}}

\newcommand{\igcsdE}{\mathfrak{Ig}_{\mathrm{CS},\eta,\eab}^{\diamondsuit}}
\newcommand{\igmf}{\mathfrak{Ig}_{\mathrm{M}}}
\newcommand{\igma}{\mathfrak{Ig}_{\mathrm{M},\eta}}

\newcommand{\igmaC}{\mathfrak{Ig}_{\mathrm{M},\eta,C}}

\newcommand{\perf}{\operatorname{Perf}}

\newcommand{\igmaL}{\mathfrak{Ig}_{\mathrm{M},\eta,L}}
\newcommand{\igcsdL}{\mathfrak{Ig}_{\mathrm{CS},\eta,L}^{\diamondsuit}}
\newcommand{\igcsaL}{\mathfrak{Ig}_{\mathrm{CS},\eta,L}}

\newcommand{\igm}{\operatorname{Ig}_{\mathrm{M}}}

\newcommand{\ul}[1]{\underline{#1}}
\newcommand{\et}{\mathrm{\acute{e}t}}
\newcommand{\Spec}{\mathrm{Spec}}
\newcommand{\cris}{\mathrm{cris}}
\newcommand{\mf}[1]{\mathfrak{#1}}
\newcommand{\mc}[1]{\mathcal{#1}}

\newcommand{\gr}{\mathrm{gr}}
\newcommand{\oelocp}{\mathcal{O}_{\mathsf{E},(v)}}
\newcommand{\oeplocp}{\mathcal{O}_{\mathsf{E}',(v')}}
\newcommand{\pdr}{\mathcal{P}_{\dr}}
\newcommand{\gdr}{\mathcal{G}_{\dr}}

\newcommand{\pcris}{\mathcal{P}_{\cris}}
\newcommand{\eab}{E^{\mathrm{ab}}}
\newcommand{\shtglocmu}{\mathrm{Sht}^{\mathrm{W}}_{\mathcal{G},\mu}}

\newcommand{\mbb}[1]{\mathbb{#1}}
\newcommand{\opn}[1]{\operatorname{#1}}

\newcommand\restr[2]{{
  \left.\kern-\nulldelimiterspace 
  #1 
  \vphantom{\big|} 
  \right|_{#2} 
  }}

\makeatletter
\newcommand*{\da@rightarrow}{\mathchar"0\hexnumber@\symAMSa 4B }
\newcommand*{\da@leftarrow}{\mathchar"0\hexnumber@\symAMSa 4C }
\newcommand*{\xdashrightarrow}[2][]{%
  \mathrel{%
    \mathpalette{\da@xarrow{#1}{#2}{}\da@rightarrow{\,}{}}{}%
  }%
}
\newcommand{\xdashleftarrow}[2][]{%
  \mathrel{%
    \mathpalette{\da@xarrow{#1}{#2}\da@leftarrow{}{}{\,}}{}%
  }%
}
\newcommand*{\da@xarrow}[7]{%
  \sbox0{$\ifx#7\scriptstyle\scriptscriptstyle\else\scriptstyle\fi#5#1#6\m@th$}%
  \sbox2{$\ifx#7\scriptstyle\scriptscriptstyle\else\scriptstyle\fi#5#2#6\m@th$}%
  \sbox4{$#7\dabar@\m@th$}%
  \dimen@=\wd0 %
  \ifdim\wd2 >\dimen@
    \dimen@=\wd2 %
  \fi
  \count@=2 %
  \def\da@bars{\dabar@\dabar@}%
  \@whiledim\count@\wd4<\dimen@\do{%
    \advance\count@\@ne
    \expandafter\def\expandafter\da@bars\expandafter{%
      \da@bars
      \dabar@ 
    }%
  }%
  \mathrel{#3}%
  \mathrel{%
    \mathop{\da@bars}\limits
    \ifx\\#1\\%
    \else
      _{\copy0}%
    \fi
    \ifx\\#2\\%
    \else
      ^{\copy2}%
    \fi
  }%
  \mathrel{#4}%
}
\makeatother

\newcommand{\bincoeff}[2]{\genfrac(){0pt}{0}{#1}{#2}}

\newcommand{\bgmu}{B(G,[\mu^{-1}])}
\newcommand{\bun}{\operatorname{Bun}}

\newcommand{\Hdg}{\mathrm{Hdg}}
\newcommand{\dR}{\mathrm{dR}}
\newcommand{\Fl}{\mathrm{Fl}}
\newcommand{\Fil}{\mathrm{Fil}}
\newcommand{\Spf}{\mathrm{Spf}}

\title[$p$-adic Maass--Shimura operators on $\mu$-ordinary Igusa varieties]{$p$-adic Maass--Shimura operators on $\mu$-ordinary Igusa varieties}
\author{Andrew Graham}
\address{Mathematical Institute, University of Oxford, Woodstock Road, Oxford OX2 6GG, United Kingdom}
\email{andrew.graham@maths.ox.ac.uk}
\thanks{AG was (partly) funded by UK Research and Innovation grant MR/V021931/1.}

\author{Pol van Hoften} 
\address{School of Mathematical Sciences, Zhejiang University, 866 Yuhangtang Rd, Hangzhou, 310058, P. R. China}
\email{pvhoften@zju.edu.cn}
\thanks{PvH was (partly) funded by the Dutch Research Council (NWO) under the grant VI.Veni.232.127.}
\author{Sean Howe}
\address{Department of Mathematics, University of Utah, 155 S 1400 E, Salt Lake City, UT 84112}
\email{sean.howe@utah.edu}
\thanks{SH was (partly) funded by National Science Foundation grants DMS-2201112 and DMS-2501816 and as a member at the Institute for Advanced Study during the academic year 2023-24 by the Friends of the Institute for Advanced Study Membership.}

\begin{document}
\begin{abstract}
For Shimura varieties of Hodge type, we optimally extend algebraic Maass--Shimura differential operators on $p$-integral nearly holomorphic automorphic forms to differential operators on $\mu$-ordinary Mantovan Igusa varieties. We then show that the rank one operators can be integrated to an action of an explicit formal group. Via $p$-adic Fourier theory, this provides a $p$-adic interpolation by extending the action of a symmetric algebra of differential operators to the algebra of functions on the Tate module of the dual $p$-divisible group. Passing to the generic fiber, we obtain an action of an explicit algebra of $p$-adic locally analytic functions, and we show that the action of the subalgebra of locally constant functions is equivalent to a natural Hecke action and thus preserves classical forms. In the ordinary case, we further show that the locally analytic action extends to nearly overconvergent automorphic forms. Our results extend, clarify, and recover prior constructions.
\end{abstract}
\maketitle
\tableofcontents


\section{Introduction}

\subsection{Overview} Maass--Shimura differential operators play a key role in the study of automorphic forms. For example, for the group $\opn{GL}_2$, one considers the operator
\[
\Theta_k = \frac{1}{2 \pi i} \left( \frac{d}{dz} + \frac{k}{z-\bar{z}} \right), \quad \quad z \in \mathbb{H} = \{ z \in \mbb{C} : \opn{Im}(z) > 0 \},
\]
which takes weight $k$ modular forms to weight $k+2$ nearly holomorphic modular forms. Special values of nearly holomorphic forms obtained by applying these operators have an interesting connection with $L$-values; for example, if $f$ is a cuspidal modular form of weight $k \geq 2$ and $j \geq 0$ is an integer, then values of $\Theta^j f  := \Theta_{k+2(j-1)} \circ \cdots \circ \Theta_k f$ at points in $\mbb{H} \cap F$ for an imaginary quadratic number field $F$ are closely related to special values of the Rankin--Selberg $L$-function $L(f, \chi^{-1}, s)$, where $\chi$ is an anticyclotomic character of infinity-type $(k+j, -j)$ (see \cite[Theorem 5.4]{BertoliniDarmonPrasanna}).\smallskip  

In order to study the arithmetic of such $L$-values, it is often desirable to be able to $p$-adically interpolate the Maass--Shimura operators. For $\opn{GL}_2$, the operators can be interpolated by considering $q$-expansions; this approach is used in work of Bertolini--Darmon--Prasanna \cite{BertoliniDarmonPrasanna} for example to construct anticyclotomic Rankin--Selberg $p$-adic $L$-functions for a fixed cuspidal modular form $f$ as $\chi$ varies. In the complex world, the classical Maass--Shimura operators for $\opn{GL}_2$ have been generalized to other groups that give rise to Shimura varieties (for example, in the PEL case by Shimura \cite{ShimuraArithmeticity}). It is therefore natural to ask whether these operators can be $p$-adically interpolated, and this question has been studied previously, e.g., in \cite{EischenFintzenMantovanVarma} and \cite{EischenMantovan}. \smallskip

In the case of $\mathrm{GL}_2$, there is a natural space of $p$-adic modular forms that can be interpreted as the space of functions on the Katz--Igusa formal scheme $\mf{Ig}_{\mathrm{Katz}}$ (a $\mbb{Z}_p^\times$-cover of the formal ordinary locus of the modular curve). Algebraic nearly holomorphic forms of any weight $k$ embed in $\mathcal{O}(\mf{Ig}_{\mathrm{Katz}})$, and there is a differential operator $\theta$ on $\mathcal{O}(\mf{Ig}_{\mathrm{Katz}})$ that extends the operators $\Theta_{k}$ under these embeddings. In \cite{HoweUnipotent}, one of us (SH) constructed an action of the formal group $\widehat{\mathbb{G}}_m$ on $\mf{Ig}_{\mathrm{Katz}}$  that integrates $\theta$. By $p$-adic Fourier theory, this gives a dual action of the algebra $\cont(\zp,\zp)$ on $\mathcal{O}(\mf{Ig}_{\mathrm{Katz}})$, which encodes the desired $p$-adic interpolation. The purpose of the present work is to generalize this construction to Hodge type Shimura varieties. This has potential applications to the construction of $p$-adic $L$-functions, see \S\ref{Sec:PAdicLFunctions}. \smallskip 

Our main result, Theorem \ref{Thm:IntroIntegralAction}, gives a systematic construction of formal group actions on Igusa formal schemes that integrate certain Maass--Shimura operators on nearly holomorphic automorphic forms. By the $p$-adic Fourier theory of \cite{FourierPaper}, such an action gives rise to an action of the algebra of functions on the Tate module of the dual $p$-divisible group (Corollary \ref{Cor:IntroIntegralOperators}). This resulting algebra action is the desired $p$-adic interpolation of Maass--Shimura operators in higher dimensions and extends the interpolation of \cite{EischenFintzenMantovanVarma} and \cite{EischenMantovan}. After inverting $p$, we obtain a variant of this algebra action that can be described explicitly as the action of an algebra of locally analytic functions (Corollary \ref{Cor:IntroRationalOperatorsI}), and we show that this action neatly encodes the relation between $p$-adic differential operators and $p$-adic Hecke operators on $p$-adic automorphic forms (Theorem \ref{Thm:IntroInfiniteLevelTheorem}). In the ordinary case, we also show that the action extends to nearly overconvergent forms (Theorem \ref{Thm:IntroNOCmainThm}) generalizing \cite{GrahamPilloniRodriguesJacinto}, \cite{UFJ}.

\subsection{Main results} 

Let $\gx$ be a Shimura datum of Hodge type with reflex field $\mathsf{E}$. Let $p$ be an odd\footnote{See \S \ref{Sub:IgusaI} for a comment on this hypothesis.} prime such that $G=\g_{\qp}$ is unramified and let $v \mid p$ be a prime of $\mathsf{E}$ above $p$; set $E=\mathsf{E}_v$. Let $K^p \subset \gafp$ be a neat compact open subgroup, and let $K_p=\mathcal{G}(\zp)$ for a reductive model $\mathcal{G}$ of $G$ over $\zlocp $. Let $\scrs_K$ be the canonical integral model of the Shimura variety for $\gx$ of level $K=K^pK_p$ over $\mathcal{O}_{\mathsf{E},(v)}$. Our assumptions imply (see \S\ref{sss.finite-extension-cocharacter-choice-existence}) that we can choose a finite extension $\mathsf{E}'$ of $\mathsf{E}$ contained in $E$ and a representative $\mu:\mathbb{G}_{m,\mathcal{O}_{\mathsf{E}', (v)}} \to \mathcal{G}_{\mathcal{O}_{\mathsf{E}', (v)}}$ lying in the conjugacy class of Hodge cocharacters $[\mu]$ associated with $X$; we fix such a choice and an embedding $\mathsf{E}' \to \mathbb{C}$ extending $\mathsf{E} \subset \mathbb{C}$ (to facilitate the comparison of the $p$-adic and complex theories).  

\subsubsection{} Let $\mathcal{P}_{\mu^{-1}} \subset \mathcal{G}_{\oeplocp}$ be the parabolic associated with $\mu$, see \S \ref{sss.cocharacter-conv} for our conventions, let $\mathcal{M}_{\mu} = \mathcal{M}_{\mu^{-1}} \subset \mathcal{P}_{\mu^{-1}}$ be the centralizer of $\mu^{-1}$ and let $\mathcal{U}_{\mu}$ be the unipotent radical of the opposite parabolic $\mathcal{P}_{\mu}$. We write $\mf{g},\mf{p}_{\mu^{-1}}$, $\mf{m}_{\mu}=\mf{m}_{\mu^{-1}}$, and $\mf{u}_{\mu}$, respectively, for the Lie algebras of  $\mathcal{G}$, $\mathcal{P}_{\mu^{-1}}$, $\mathcal{M}_{\mu} = \mathcal{M}_{\mu^{-1}}$, and $\mathcal{U}_{\mu}$. We use similar notation for subgroups and Lie algebras associated with other (fractional) cocharacters.

\subsubsection{} Let $\mathcal{P}_{\dr} \to \scrs_{K, \mathcal{O}_{\mathsf{E}', (v)}}$ be the canonical $\mathcal{P}_{\mu^{-1}}$-torsor. We consider $\mathcal{O}(\mathcal{P}_{\dr})$, the space of algebraic $p$-integral nearly holomorphic automorphic forms of level $K$. To motivate this terminology, we observe that for a representation $\mathcal{P}_{\mu^{-1}} \to \mathcal{M}_{\mu} \to \operatorname{GL}(W)$, we get an automorphic vector bundle $\mathcal{P}_{\dr} \times^{\mathcal{P}_{\mu^{-1}}} W$ and a pullback map (which is injective if $W$ is an irreducible representation of $\mathcal{M}_{\mu}$)
\begin{align}
    W^{\vee} \otimes H^0(\scrs_{K}, \mathcal{P}_\dR \times^{\mathcal{P}_{\mu^{-1}}} W) \to \mathcal{O}(\mathcal{P}_{\dr}).
\end{align}
Thus we can consider classical $p$-integral holomorphic automorphic forms (of varying weights) as sections of $\mathcal{O}(\mathcal{P}_{\dr})$.  More generally, we could have started with a representation $W$ of $\mathcal{P}_{\mu^{-1}}$ that does not factor through $\mathcal{M}_{\mu}$. The sections of such vector bundles are classical nearly holomorphic forms of weight $W$ as in, e.g., \cite{ZLiu19}. 
\begin{Prop}[Proposition \ref{Prop:TPdrToLieGisaniso}, Lemma \ref{Lem.inclusion-of-Lie-algebras}] \label{Prop:IntroKS}
The tangent bundle $T_{\pdr}$ of $\pdr$ is canonically isomorphic to $\mathfrak{g} \otimes \mathcal{O}_{\pdr}$ and the induced map $\mathfrak{g} \to T_{\pdr}$ is compatible with Lie brackets. 
\end{Prop}
Using Proposition \ref{Prop:IntroKS} we give $\mathcal{O}(\mathcal{P}_{\dr})$ an action of $\mathfrak{g}$. The action of $\mf{u}_{\mu} \subset \mf{g}$ provides algebraic Maass--Shimura operators that, after extending scalars along $\mathcal{O}_{\mathsf{E}',(v)} \to \mathbb{C}$, recover the action of the classical $C^{\infty}$ Maass--Shimura operators introduced in \cite{ShimuraArithmeticity}, see \S \ref{sss.intro-real} and \S \ref{sub:AlgebraicMaassShimura}.

\begin{Rem} \label{Rem:ClassicalPerspective} In previous work on (nearly) holomorphic / nearly overconvergent $p$-adic automorphic forms, algebraic Maass--Shimura operators have been constructed from a sheaf-theoretic perspective by using the Gauss-Manin connection on certain automorphic vector bundles, see e.g., \cites{EischenFintzenMantovanVarma, EischenMantovan, Urban, GrahamPilloniRodriguesJacinto, AndreattaIovita, ZLiu19, kazi2025twistedtripleproductpadic}. The perspective we adopt here appears to be new, and simplifies the theory by instead using the Gauss-Manin connection to parallelize the tangent bundle of $\mathcal{P}_\dR$. In particular, the behavior of the Maass--Shimura operators under composition is completely transparent from our perspective.
\end{Rem}

\subsubsection{} Let $b_{\mu}=\mu(p^{-1}) \in G(\qpbr)$ be the $\mu$-ordinary element. There is a dense open $\mu$-ordinary locus 
$\scrs_{K,k_{E}}^{[b_{\mu}]} \subset \scrs_{K,k_{E}}$, where $k_{E}$ is the residue field of $\mathcal{O}_{E}$; we consider its $p$-adic formal completion $(\scrshat_K)^{[b_{\mu}]}$ over $\spf \mathcal{O}_{E}$.  We use similar notation for $p$-adic formal completions of other objects. 

Let $\overline{\mu}$ be the Galois average of $\mu$, and consider $\mathcal{N}=\mathcal{M}_{\mu} \cap \mathcal{U}_{\overline{\mu}}$. We consider $\mathcal{O}(\mathcal{P}_{\dr})^{\mathcal{N}} \subset \mathcal{O}(\mathcal{P}_{\dr})$, which has an action of $(\operatorname{Sym}(\mf{u}_{\mu}))^{\mathcal{N}}$ induced by the algebraic Maass--Shimura operators. We construct in \S \ref{Sec:DifferentialOperators} a morphism
\begin{align}
\label{intro.map-to-pdr-quot}    \igmf \to \widehat{\mathcal{N}}\backslash \widehat{\mathcal{P}}_{\dr}
\end{align}
over $\scrs_K$, where $\igmf \to (\scrshat_K)^{[b_{\mu}]}$ is a pro-\'etale torsor for $\mathcal{M}_{\overline{\mu}}(\zp)$; it is the analogue in the general case of the Katz--Igusa tower for $\opn{GL}_2$ (or rather its variant $\mf{I}_{\mathrm{Mant}}$ as in \cite[Definition 4.1.2]{Howe.CompletedKirillov}), and we consider $\mathcal{O}(\igmf)$ as our space of $p$-adic automorphic forms (cf. \cite{EischenMantovan}, but note that we do not pass here to invariants under the unipotent radical of a Borel in $\mathcal{M}_{\overline{\mu}}(\zp)$). Pullback along \eqref{intro.map-to-pdr-quot} induces a natural map 
\begin{align}\label{eq.intro-functions-restriction-map}
    \mathcal{O}(\mathcal{P}_{\dr})^{\mathcal{N}}=\mathcal{O}(\mathcal{N} \backslash \mathcal{P}_{\dr}) \to \mathcal{O}(\igmf),
\end{align}
and we show that there is a compatible action of $(\operatorname{Sym}(\mf{u}_{\mu}))^{\mathcal{N}}$ on $\mathcal{O}(\igmf)$. Our main results interpolate the Maass--Shimura operators lying in the subalgebra $\operatorname{Sym}\left(\mf{u}_{\mu}^{\mathcal{N}}\right)$ on $\mathcal{O}(\igmf)$ by extending the action of $\operatorname{Sym}\left(\mf{u}_{\mu}^{\mathcal{N}}\right)$ to a larger algebra of functions. 

\begin{Rem} Our construction of the action of $(\operatorname{Sym}(\mf{u}_{\mu}))^{\mathcal{N}}$ on $\mathcal{O}(\igmf)$ compatibly with the algebraic Maass--Shimura operators on $\mathcal{O}(\mathcal{P}_{\dr})^{\mathcal{N}}$ is already new. The main ingredient is a parallelization of the tangent bundle of $\mathcal{Q}$, a natural $\mathcal{N}$-torsor over $\igmf$, see Proposition \ref{Prop:MuOrdinaryKodairaSpencer}.
\end{Rem}

\begin{mainThm}[Theorem \ref{Thm:GlobalAction}] \label{Thm:IntroIntegralAction} 
There is an explicit $p$-divisible formal group $\mathcal{H}$ over $\spf \mathcal{O}_E$ with an action of $\mathcal{M}_{\overline{\mu}}(\zp)$ such that $\operatorname{Lie} \mathcal{H}$ is $\mathcal{M}_{\overline{\mu}}(\zp)$-equivariantly isomorphic to $\mf{u}_{\mu}^\mathcal{N}$. Furthermore, there is an $\mathcal{M}_{\overline{\mu}}(\zp)$-semilinear action of $\mathcal{H}$ on $\igmf$ such that the action of $\Lie \mathcal{H}=\mf{u}_{\mu}^{\mathcal{N}}$ on $\mathcal{O}(\igmf)$ given by differentiation agrees with the extension to $\mathcal{O}(\igmf)$ of the algebraic Maass--Shimura operators described above. 
\end{mainThm}
Let $\mathcal{H}^{\vee}$ be the Serre-dual of the $p$-divisible formal group $\mathcal{H}$ of Theorem \ref{Thm:IntroIntegralAction}, and let $T_p \mathcal{H}^{\vee}$ be its $p$-adic Tate module (an affine $p$-adic formal scheme over $\spf \mathcal{O}_E$). It comes equipped with a natural Hodge--Tate map $T_p \mathcal{H}^{\vee} \to \omega_{\mathcal{H}}$. This induces a $\mathcal{M}_{\overline{\mu}}(\zp)$-equivariant algebra morphism $\operatorname{Sym}(\Lie \mathcal{H}) \to \mathcal{O}(T_p \mathcal{H}^{\vee})$. Using the integral $p$-adic Fourier theory of \cite[Section 6]{FourierPaper}, we deduce the following corollary.
\begin{Cor}[Corollary \ref{Cor:IntegralFourierConsequence}] \label{Cor:IntroIntegralOperators}
The action of $\operatorname{Sym}(\Lie \mathcal{H})=\operatorname{Sym}(\mathfrak{u}_{\mu}^{\mathcal{N}})$ on $\mathcal{O}(\igmf)$ extends to a continuous $\mathcal{M}_{\overline{\mu}}(\zp)$-semilinear algebra action of $\mathcal{O}(T_p \mathcal{H}^{\vee})$ on $\mathcal{O}(\igmf)$.
\end{Cor}

\begin{Eg}\label{example.intro-integral-ordinary}
If $E=\qp$, then the $\mu$-ordinary locus is just the ordinary locus, see \cite[Corollary 4.3.2]{LeeNewton}. In this case, $\mathcal{N}$ is trivial, $\mathcal{H}=\mf{u}_{\mu} \otimes \gmhat$ and $\mathcal{O}(T_p \mathcal{H}^{\vee})=\operatorname{Cont}(\mf{u}_{\mu}^{\ast},\zp)$. Corollary \ref{Cor:IntroIntegralOperators} then provides a $p$-adic interpolation of the algebraic Maass--Shimura operators on $\mathcal{O}(\mathcal{P}_{\dr})$. 
\end{Eg}

\begin{Rem} \label{Rem:IntroEischenMantovan}
A version of Corollary \ref{Cor:IntroIntegralOperators} for unitary Shimura varieties appears in the work of Eischen--Fintzen--Mantovan--Varma and Eischen--Mantovan, see \cite{EischenFintzenMantovanVarma} and \cite{EischenMantovan}. For a $\mathfrak{N} \subset \mathcal{M}_{\ovmu}$ the unipotent radical of a Borel, they construct certain $p$-adic differential operators on the space of invariants $\mathcal{O}(\igmf)^{\mathfrak{N}(\zp)}$. For comparison, their results concern interpolating the action of $\operatorname{Sym} \Lie \mathcal{H}^{\mathrm{m}}$, where $\mathcal{H}^{\mathrm{m}}$ is the multiplicative part of $\mathcal{H}$, see Lemma \ref{Lem:SimpleWeights}. Moreover, writing $\mathcal{O}(T_p \mathcal{H}^{m,\vee})=\cont(\zp^{\oplus k}, \zp)$ for some integer $k$, the operators constructed by Eischen--Mantovan correspond to $p$-adic limits of polynomial functions with $\mathbb{Z}_p$-coefficients, thus are a proper subset of all continuous functions (for example when $k=1$, the function $z\mapsto \binom{z}{p}$ is not such a limit). We show that our operators restrict to theirs in Proposition \ref{Prop:ComparisonEM}. The method of \cite{EischenFintzenMantovanVarma} and \cite{EischenMantovan} uses an explicit description of the action of differential operators in Serre--Tate coordinates, whereas our construction is coordinate-free.  
\end{Rem}
\subsubsection{} \label{subsub:IntroTateModule} Outside of the ordinary case, the ring $\mathcal{O}(T_p \mathcal{H}^{\vee})$ is often hard to describe explicitly. The adic generic fiber $T_p H^{\vee}=T_p \mathcal{H}^{\vee}_{\eta}$ is a (Galois-twisted) profinite set and thus $\mathcal{O}(T_p H^{\vee})$ \emph{does} have an explicit description. However, outside the ordinary case, the  map
\begin{align}
   \mathcal{O}(T_p \mathcal{H}^{\vee})[\tfrac{1}{p}] \to \mathcal{O}(T_p H^{\vee})
\end{align}
might not be injective or surjective (see \cite[Remark 7.2.4]{FourierPaper}), and so we do not generally get an action of $\mathcal{O}(T_p H^{\vee})$ on $\mathcal{O}(\igmf)[\tfrac{1}{p}]=\mathcal{O}(\igma)$.

However, by \cite[Corollary 7.2.7]{FourierPaper}, there is a subalgebra
\begin{align}
    \mathcal{O}^{\gamma-\mathrm{la}}(T_p H^{\vee}) \subset \mathcal{O}(T_p H^{\vee})
\end{align}
which naturally lifts to a subalgebra of $\mathcal{O}(T_p \mathcal{H}^{\vee})[\tfrac{1}{p}]$ and thus acts on $\mathcal{O}(\igma)$. We now describe this algebra after an extension of scalars.  \smallskip

Let $C / E$ be the completion of an algebraic closure. Let $\mathcal{U}_{\overline{\mu},0}$ be the center of $\mathcal{U}_{\overline{\mu}}$ and let $\mf{u}_{\overline{\mu},0}$ be its Lie algebra. Then, $\mf{u}_{\mu,C}^{\mathcal{N}}=\mf{u}_{\overline{\mu},0,C} \cap \mf{u}_{\mu,C}$ is a submodule of $\mf{u}_{\overline{\mu},0,C}$. After choosing a suitable $C$-point on the associated infinite level Rapoport--Zink space, we obtain an identification $T_p \mathcal{H}(\mathcal{O}_C)=T_p H(C)=\mf{u}_{\overline{\mu},0}$. Choosing a basis $\zeta$ for $\mathbb{Z}_p(1)$, we obtain also an automorphism $\Omega_\zeta$ of $\mf{u}_{\mu,C}^{\mathcal{N}}$ that scales root spaces by certain abelian periods in $C$ (see \S\ref{sss.omegazeta}). For $\iota: \mf{u}_{\mu,C}^{\mathcal{N}} \rightarrow \mf{u}_{\overline{\mu},0, C}$ the inclusion, let $\gamma = (\iota \circ \Omega_\zeta)^*|_{(\mf{u}_{\overline{\mu},0})^{\ast} }$, so that $\gamma: (\mf{u}_{\overline{\mu},0})^{\ast} \rightarrow (\mf{u}_{\mu,C}^{\mathcal{N}})^*$. By \cite[Theorem 1]{FourierPaper}, for $\otimes^\blacksquare$ the solid tensor product of \cite{CondensedNotes}, there is an isomorphism 
\[ \mathcal{O}^{\gamma-\mathrm{la}}(T_p H^{\vee}) \otimes_{E}^{\blacksquare} C \to \mathcal{O}^{\gamma-\mathrm{la}}(\mf{u}_{\overline{\mu},0}^{\ast},C),\]
where the target is the ring of functions $\mf{u}_{\overline{\mu},0}^{\ast} \to C$ that are locally pulled back along $\gamma$ from analytic functions on open subsets of $(\mf{u}_{\mu,C}^{\mathcal{N}})^*$. In particular, we can view  $\operatorname{Sym}(\mf{u}_{\mu,C}^{\mathcal{N}})$ as the polynomial functions on $(\mf{u}_{\mu,C}^\mathcal{N})^*$, and then pull these back along $\gamma$ to obtain a map $\gamma^*:\operatorname{Sym}(\mf{u}_{\mu,C}^{\mathcal{N}}) \hookrightarrow \mathcal{O}^{\gamma-\mathrm{la}}(\mf{u}_{\overline{\mu},0}^{\ast},C)$. 

\begin{Cor} \label{Cor:IntroRationalOperatorsI}
    The action of $\operatorname{Sym}(\mf{u}_{\mu,C}^{\mathcal{N}})$ on 
    \[ \mathcal{O}(\igmaC)=\mathcal{O}(\igmf)[\tfrac{1}{p}]\otimes_{E}^{\blacksquare} C \]
    extends along  $\gamma^*$:  $\operatorname{Sym}(\mf{u}_{\mu,C}^{\mathcal{N}}) \hookrightarrow  \mathcal{O}^{\gamma-\mathrm{la}}(\mf{u}_{\overline{\mu},0}^{\ast},C)$ to an action of  $\mathcal{O}^{\gamma-\mathrm{la}}(\mf{u}_{\overline{\mu},0}^{\ast},C)$.
\end{Cor}

\begin{Rem}For the definition of $\gamma$-locally analytic functions we could forget about $\Omega_{\zeta}$ and instead take $\gamma$ to just be the dual of the inclusion, but it is important to incorporate $\Omega_{\zeta}$ in order to match the $\mathsf{E}'$-rational Maass--Shimura operators. 
\end{Rem}

\begin{Eg}
    In the ordinary case as in Example \ref{example.intro-integral-ordinary}, the choice of a point on the infinite level Rapoport--Zink space can be made using the fixed basis $\zeta$ of $\mathbb{Z}_p(1)$ so that the resulting period automorphism $\Omega_\zeta$ is the identity. Thus in this case we are just recovering the action of $C^\locan(\mf{u}_{\mu}^*, C) \subseteq \mathrm{Cont}(\mf{u}_{\mu}^{*}, C)$ after inverting $p$. 
\end{Eg}

\begin{Eg} \label{Eg:IntroLT}
Let $\mathsf{E} \subseteq \mathbb{C}$ be a CM field with totally real subfield $\mathsf{F} \not=\mathbb{Q}$ and let $V$ be a Hermitian space over $\mathsf{E}$ of with signatures $(n-1,1),(n,0), \cdots (n,0)$, with $n \ge 2$. We write $\g$ for $\operatorname{GU}(V)$, and consider the usual associated PEL datum $\gx$. Suppose that $p$ is inert in $\mathsf{F}$ and splits into $p=\mf{p}_1 \mf{p}_2$ in $\mathsf{E}$. The reflex field of $\gx$ is equal to $\mathsf{E}$ in this case (since $\mathsf{F} \not=\mathbb{Q}$). Then $\mc{U}_{\overline{\mu},0}= \mathcal{U}_{\overline{\mu}}$ and we can identify $\mathcal{H}$ with 
\begin{align}
    \operatorname{LT}^{\oplus n-1},
\end{align}
where $\operatorname{LT}$ is the Lubin--Tate formal group over $\mathcal{O}_{\oee}$ of height $d=[E:\qp]$ and dimension $1$ corresponding to the place $w$ of $E$. Moreover, even though $\mc{N}$ is non-trivial, we have $\mathcal{U}_{\mu}=\mathcal{U}_{\mu}^{\mathcal{N}}$ since $\mathcal{U}_{\overline{\mu}}$ is abelian, thus $\operatorname{Lie} \mathcal{H} = \mf{u}_{\mu}$ and $\operatorname{Sym}(\mf{u}_{\mu})=(\operatorname{Sym} \mf{u}_{\mu})^{\mathcal{N}}$.

After basechange to $C$, we obtain an identification of $T_p H^{\vee}$ with $n-1$ copies of $\oee$. Moreover, this identifies $\mathcal{O}^{\gamma-\mathrm{la}}(\mf{u}_{\overline{\mu},0}^{\ast},C)$ with the ring of $C$-valued functions on $\mathcal{O}_{E}^{\oplus n-1}$ which are ${E}$-analytic via $w:{E} \to C$; here the $p$-adic Fourier theory goes back to Schneider--Teitelbaum \cite{SchneiderTeitelbaumFourier}. In particular, after $p$-depleting by multiplying by the indicator function for $(\mathcal{O}_{E}^\times)^{n-1}$, we obtain $p$-adic families of differential operators parameterized by $E$-analytic characters of $(\mathcal{O}_{E}^\times)^{n-1}$. The period automorphism $\Omega_{\zeta}$ corresponds to multiplication by a Lubin--Tate period, see \S \ref{sss.period-description}.
\end{Eg}

\subsubsection{} In the ordinary case (i.e., $E=\qp$), we also show that the action in Corollary \ref{Cor:IntroRationalOperatorsI} extends to an action on nearly overconvergent forms, generalizing \cite{GrahamPilloniRodriguesJacinto} and \cite[\S7]{UFJ}. To state this result, we define the space of nearly overconvergent automorphic forms (in the ordinary case) as
\[
\mathscr{N}^{\dagger} := \varinjlim^{\mathrm{lcv}}_U\opn{H}^0(U, \mathcal{O}_{P_{\opn{dR}}^{\opn{an}}}),
\]
where $P_{\opn{dR}}^{\opn{an}}$ is the analytification of $P_{\opn{dR}}$ and the (locally convex) inductive limit is over all open subspaces $U \subset P_{\opn{dR}}^{\opn{an}}$ containing the closure of $\mathfrak{Ig}_{M, \eta}$ (via the map $\mathfrak{Ig}_{M, \eta} \to P_{\opn{dR}}^{\opn{an}}$). This is an LB-space over $\mbb{Q}_p$ of compact-type and has the structure of a $(\mathfrak{g}, \mathcal{M}_{\mu}(\mbb{Z}_p))$-module compatible with the $(\mathfrak{g}, P_{\mu^{-1}})$-module structure on $\mathcal{O}_{P_{\opn{dR}}}$.

\begin{mainThm}[Theorem \ref{Thm:NOCmainThm}]\label{Thm:IntroNOCmainThm}
    Suppose we are in the ordinary case, i.e., $E=\qp$. Then there exists a unique continuous and $\mathcal{M}_{\mu}(\mbb{Z}_p)$-semilinear algebra action of $C^{\opn{la}}(\mathfrak{u}_{\mu}^*, \mbb{Q}_p)$ on $\mathscr{N}^{\dagger}$ that is compatible with the action in Corollary \ref{Cor:IntroRationalOperatorsI} via the natural pullback map $\mathscr{N}^{\dagger} \to \mathcal{O}(\mathfrak{Ig}_{M, \eta})$.
\end{mainThm}
Our proof of Theorem \ref{Thm:NOCmainThm} is by a reduction to the Siegel case, which we handle following the general strategy of \cite{GrahamPilloniRodriguesJacinto} and \cite[\S7]{UFJ}. However, at a certain step, we use a different choice of coordinates to eliminate the analysis of a nilpotent operator, resulting in a simpler proof and stronger control over integrality of the action; see Proposition \ref{Prop:OrdinarySiegelOCcomp} and Remark \ref{Rem:GrowthOfAction}.

\begin{Rem}
We do not prove a $\mu$-ordinary generalization of Theorem \ref{Thm:IntroNOCmainThm}, but see Conjecture \ref{Conj:NOC} for an (optimistic) conjecture. 
\end{Rem}

\subsection{Torsion points, locally constant functions, and Hecke operators.}
Consider the restriction of the action in Theorem \ref{Thm:IntroIntegralAction} to the $p^n$-torsion subgroup $\mathcal{H}[p^n]$. In terms of the algebra action of Corollary \ref{Cor:IntroIntegralOperators}, this corresponds to considering the action of characters of $T_p \mathcal{H}^\vee$ valued in $\mu_{p^n}$. In particular, after inverting $p$, extending scalars to $C$, and making the choices to obtain Corollary \ref{Cor:IntroRationalOperatorsI}, the action of the torsion subgroup $H(C)[p^\infty]=\bigcup_n\mathcal{H}[p^n](C)$ spans the action of the locally constant functions $\mathcal{O}^{\mathrm{sm}}(\mf{u}_{\overline{\mu},0}^{\ast},C) \subseteq \mathcal{O}^{\gamma-\mathrm{la}}(\mf{u}_{\overline{\mu},0}^{\ast},C)$.\smallskip 

To obtain Corollary \ref{Cor:IntroIntegralOperators}, the first choice we made was a point in the infinite level Rapoport--Zink space. After passing to the rigid analytic generic fiber, this choice provides a lift of the map \eqref{intro.map-to-pdr-quot} from the (finite) level $K^pK_p$ to the partial infinite level $K^p\mathcal{U}_{\overline{\mu}}(\mathbb{Z}_p)$. In particular, we can also evaluate nearly holomorphic automorphic forms of level $K^p\mathcal{U}_{\overline{\mu}}(\mathbb{Z}_p)$ to obtain elements of $\mathcal{O}(\mf{Ig}_{M,\eta,C})$, and the action of 
\[ H(C)[p^\infty]=T_p \mathcal{H}(C) \otimes \mathbb{Q}_p/\mathbb{Z}_p=\mf{u}_{\overline{\mu},0}\otimes \mathbb{Q}_p/\mathbb{Z}_p = \mathcal{U}_{\overline{\mu},0}(\mathbb{Q}_p)/\mathcal{U}_{\overline{\mu},0}(\mathbb{Z}_p)\]
extends the Hecke action of $\mathcal{U}_{\overline{\mu},0}(\mathbb{Q}_p)/\mathcal{U}_{\overline{\mu},0}(\mathbb{Z}_p)$ on these forms; see Theorem \ref{Thm:IntroInfiniteLevelTheorem}.(2) for a precise statement. In particular, the action of $\mathcal{O}^{\mathrm{sm}}(\mf{u}_{\overline{\mu},0}^{\ast},C)$ preserves classical (nearly) holomorphic automorphic forms. 

\begin{Rem} \label{Rem:HeckeDifferentialRemark}
That the same algebra action simultaneously interpolates the action of Hecke operators and differential operators is important for applications to $p$-adic $L$-functions. In particular, certain $p$-depletions with complicated descriptions via Hecke operators seem to have simple descriptions via this algebra action. We expect that this observation can be used in the construction of new $p$-adic $L$-functions, especially in cases where the $\mu$-ordinary locus is not the ordinary locus. Moreover, we hope that this might clarify certain choices of test vectors that are made in existing constructions of $p$-adic $L$-functions. In \S \ref{Sec:PAdicLFunctions}, we give an informal discussion of these ideas in some examples.
\end{Rem}

\subsubsection{} So far we have noted the Hecke actions of $\mathcal{U}_{\overline{\mu},0}(\mathbb{Q}_p)/\mathcal{U}_{\overline{\mu},0}(\mathbb{Z}_p)$ and of $\mathcal{M}_{\overline{\mu}}(\mathbb{Z}_p)$. In fact these are both part of a larger Hecke action of $C[\mathcal{U}_{\overline{\mu}}(\mathbb{Z}_p) \backslash\!\backslash P_{\overline{\mu}}(\mathbb{Q}_p)]$ on $\mathcal{O}(\mf{Ig}_{M,\eta,C})$ obtained by taking invariants in a larger space of functions at infinite level with a $P_{\overline{\mu}}(\mathbb{Q}_p)$-action (see \S\ref{ss.infinite-level}). The various compatibilities are summarized in the following result. Let $N$ be the generic fiber of $\mathcal{N}$, and let $P_{\dr,\infty}$ be the basechange to the Shimura variety with infinite level at $p$ of the generic fiber of $\pdr$. 

\begin{mainThm} \label{Thm:IntroInfiniteLevelTheorem}  \hfill
\begin{enumerate}
\item The restriction map
\[ \mathcal{O}(N\backslash P_{\dR,\infty,C})^{\mathcal{U}_{{\overline{\mu}}}(\mathbb{Z}_p)} \rightarrow \mathcal{O}(\mf{Ig}_{M,\eta,C}) \]
is $C[\mathcal{U}_{\overline{\mu}}(\mathbb{Z}_p) \backslash\!\backslash P_{\overline{\mu}}(\mathbb{Q}_p)]$-equivariant when the nearly holomorphic automorphic forms are equipped with the twisted Hecke action obtained by combining the usual Hecke action with the torsor action through (see \S \ref{subsub:TwistedHeckeAction})
\[ P_{\overline{\mu}}(\mathbb{Q}_p) \rightarrow M_{\overline{\mu}}(\mathbb{Q}_p)\subseteq M_{\mu}(E)\subseteq P_{\mu^{-1}}(E). \]
\item The Hecke action of 
\[ C[\mc{U}_{\overline{\mu},0}(\mathbb{Q}_p)/\mc{U}_{\overline{\mu},0}(\zp)] \subseteq C[\mathcal{U}_{\overline{\mu}}(\mathbb{Z}_p) \backslash\!\backslash P_{\overline{\mu}}(\mathbb{Q}_p)]\]
on $\mathcal{O}(\mf{Ig}_{M,\eta,C})$ agrees with the action of 
\[ \mathcal{O}^{\mathrm{sm}}(\mf{u}_{\overline{\mu},0}^{\ast},C) \subseteq \mathcal{O}^{\gamma-\locan}(\mf{u}_{\overline{\mu},0}^{\ast},C) \]
on $\mathcal{O}(\mf{Ig}_{M,\eta,C})$ under the isomorphism 
\[ 
C[\mc{U}_{\overline{\mu}, 0}(\mathbb{Q}_p)/\mathcal{U}_{\overline{\mu}, 0}(\zp)] \xrightarrow{\sim} \mathcal{O}^{\mathrm{sm}}(\mf{u}_{\overline{\mu},0}^{\ast},C)
\]
sending $u \in \mathcal{U}_{\overline{\mu}, 0}(\mathbb{Q}_p)/\mathcal{U}_{\overline{\mu}, 0}(\zp)=\mf{u}_{\overline{\mu},0} \otimes \mathbb{Q}_p/\mathbb{Z}_p$ to the associated\footnote{The fixed trivialization $\mathbb{Z}_p=\mathbb{Z}_p(1)$ allows us to identify $\mf{u}_{\overline{\mu},0}^*=\mf{u}_{\overline{\mu},0}^*(1)$ and thus interpret an element of $\mf{u}_{\overline{\mu},0} \otimes \mathbb{Q}_p/\mathbb{Z}_p$ as a function on $\mf{u}_{\overline{\mu},0}^*$ valued in $\mathbb{Q}_p(1)/\mathbb{Z}_p(1)=\mu_{p^\infty}(C).$} finite order character. 
\item For $g \in P_{\overline{\mu}}(\mathbb{Q}_p)$, $s \in \mathcal{O}(\mf{Ig}_{M,\eta,C})$, and $f \in \mathcal{O}^{\gamma-\locan}(\mf{u}_{\overline{\mu},0}^*, C)$, the Hecke action of $[g]=\mathcal{U}_{\overline{\mu}}(\mathbb{Z}_p) g \mathcal{U}_{\overline{\mu}}(\mathbb{Z}_p)  \in  C[\mathcal{U}_{\overline{\mu}}(\mathbb{Z}_p) \backslash\!\backslash P_{\overline{\mu}}(\mathbb{Q}_p)]$ satisfies 
\[ [g] \cdot (f \cdot s)= (g \cdot_! f) \cdot ([g] \cdot s), \]
where $g \cdot_! f$ is defined by, writing $\mathrm{Ad}^*$ for the coadjoint action, 
\[ (g \cdot_! f)(z) = \begin{cases} f(\mathrm{Ad}_{g^{-1}}^*(z)) & \textrm{ if $\mathrm{Ad}^*_{g^{-1}}(z) \in \mathfrak{u}_{\overline{\mu},0}^*$} \\
0 & \textrm{ otherwise}.\end{cases} \]
\end{enumerate}
\end{mainThm}

\begin{Rem} Part (3) of Theorem \ref{Thm:IntroInfiniteLevelTheorem} is most interesting when $g \in M_{\overline{\mu}}(\mathbb{Q}_p)$. In particular, for each of $+/-$, let $M_{\overline{\mu}}(\mathbb{Q}_p)^{\pm} \subseteq M_{\overline{\mu}}(\mathbb{Q}_p)$ be the sub-monoid of elements such that $m^{\pm 1} \mc{U}_{\overline{\mu}}(\mathbb{Z}_p) m^{\mp 1} \supset \mc{U}_{\overline{\mu}}(\mathbb{Z}_p)$. Then we have natural inclusions
\[ C[M_{\overline{\mu}}(\mathbb{Q}_p)^{\pm}] \subseteq C[\mathcal{U}_{\overline{\mu}}(\mathbb{Z}_p) \backslash\!\backslash P_{\overline{\mu}}(\mathbb{Q}_p)] \]
and part (3) of Theorem \ref{Thm:IntroInfiniteLevelTheorem} describes how the Hecke actions of these subalgebras intertwines with the action of $\mathcal{O}^{\gamma-\locan}(\mf{u}_{\overline{\mu},0}^*, C)$. 
\end{Rem}

\begin{Eg}
    Suppose $\gx$ is the standard Shimura datum attached to the group $\mathsf{G} = \opn{GSp}_{2g}$ with Hodge cocharacter $\mu(z) = \left( \begin{smallmatrix} z I_g & \\ & I_g \end{smallmatrix} \right)$. Then $\mathcal{O}^{\gamma-\locan}(\mf{u}_{\overline{\mu},0}^{\ast},C)$ is identified with the algebra of locally analytic functions $\mathcal{O}^{\opn{la}}(\opn{Sym}_{g \times g}(\mbb{Z}_p), C)$ on the space of $g \times g$ symmetric matrices with entries in $\mbb{Z}_p$, and this algebra acts on sections $\mathcal{O}(\mathfrak{Ig}_{M, \eta,C})$. The space $\mathcal{O}(\mathfrak{Ig}_{M, \eta,C})$ carries an action of the ``Siegel $U_p$-operator''
    \[
    U_p = \frac{1}{p^{g(g+1)/2}} \sum_{A \in \opn{Sym}_{g \times g}(\mbb{Z}_p)/\opn{Sym}_{g \times g}(p\mbb{Z}_p)} \left( \begin{smallmatrix} p I_g & A \\ & I_g \end{smallmatrix} \right) \; \in \; \mbb{Q}[P_{\mu}(\mbb{Q}_p)],
    \]
    and also an action of the Frobenius morphism $\varphi = \mu(p)^{-1} = \left( \begin{smallmatrix} p^{-1}I_g &  \\ & I_g \end{smallmatrix} \right)$. One can show that the action of $1 - \varphi \circ U_p$ on $\mathcal{O}(\mathfrak{Ig}_{M, \eta})$ coincides with the action of the indicator function $\mathbf{1}_{X} \in \mathcal{O}^{\opn{la}}(\opn{Sym}_{g \times g}(\mbb{Z}_p), C)$ of $X = \opn{Sym}_{g \times g}(\mbb{Z}_p) \setminus \opn{Sym}_{g \times g}(p \mbb{Z}_p)$ (see \cite[Lemma 6.6.1]{GrahamPilloniRodriguesJacinto} for the $g=1$ case; the general case follows from the same proof \emph{mutatis mutandis}).
\end{Eg}

\subsection{Sketch of proof} We continue with the notation as defined above. 

\subsubsection{The Caraiani--Scholze Igusa formal scheme.} Let $\igcs$ be the perfect $\mu$-ordinary Igusa variety over $k_E$ of Hamacher--Kim \cite{HamacherKim}, with Witt vector lift $\igcsf \to \spf \mathcal{O}_E$. If $\mbxcan$ is the canonical lift of the $\mu$-ordinary $p$-divisible group with $\mathcal{G}$-structure $\mbx$ in the sense of \cite{ShankarZhou}, then $\igcsf \to (\scrshat_K)^{[b_{\mu]}}$ is a torsor for the $p$-adic formal group $\Aut_{\mathcal{G}}(\mbxcan)$ of automorphisms of $\mbxcan$ that are compatible with $\mathcal{G}$-structures. There is a  natural map $\igcsf \to \igmf$ which is a torsor for the identity component $\Aut_{\mathcal{G}}(\mbxcan)^{\circ}$ of $\Aut_{\mathcal{G}}(\mbxcan)$.

\subsubsection{} There is a natural morphism $\igcsf \to \widehat{\mathcal{P}}_{\dr}$. In particular, we can evaluate nearly holomorphic automorphic forms to functions on $\igcsf$. Now, the action of $\Aut_{\mathcal{G}}(\mbxcan)$ on $\igcsf$ in fact extends to the action of a larger group $\Aut_{G}(\tildex)$ of self quasi-isogenies of $\mbxcan$ (equivalently, of $\mbx$). We describe a part of this action explicitly: For $\mathcal{U}_{\overline{\mu},0}$ the center of $\mathcal{U}_{\overline{\mu}}$, we show that $\mathcal{U}_{\overline{\mu},0}(\mathcal{O}_E)\cong \mf{u}_{\overline{\mu},0} \otimes \mathcal{O}_E$ has a natural structure of an admissible filtered Dieudonn\'{e} module induced by $\mu^{-1}$ and $b=b_{\mu}$, thus gives rise to a $p$-divisible group $\mathcal{H}/\mathcal{O}_E$. Using the semiperfect Dieudonn\'{e} theory of \cite{ScholzeWeinstein} combined with Grothendieck--Messing theory, we show that the universal cover $\widetilde{\mathcal{H}}$ of $\mathcal{H}$ maps naturally to the center of $\Aut_{G}(\tildex)^{\circ}$, and that $T_p \mathcal{H} \leq \widetilde{\mathcal{H}}$ factors through $\Aut_{\mathcal{G}}(\mbxcan)^\circ$. In particular, $\mathcal{H}=\widetilde{\mathcal{H}}/T_p \mathcal{H}$ acts on $\igmf=\Aut_{\mathcal{G}}(\mbxcan)^{\circ}\backslash \igcsf$.

\subsubsection{}
The action of $T_p \mathcal{H}$ on $\igcsf$ does not in general fix the map $\igcsf \to \widehat{\mathcal{P}}_{\dr}$, but it can only change it by an element of $\mathcal{N}$, so that we obtain the induced map $\igmf \rightarrow \widehat{\mathcal{N}}\backslash \widehat{\mathcal{P}}_{\dr}$ written in \eqref{intro.map-to-pdr-quot}. Computing the derivative of this map in order to prove Theorem \ref{Thm:IntroIntegralAction} requires a manipulation with the Gauss--Manin connection, because this is what is used to obtain the parallelization $T_{\mathcal{P}_\dR}=\mf{g} \otimes \mathcal{O}_{\mathcal{P}_\dR}$. However, because the Gauss--Manin connection agrees with the crystalline connection, this is amenable to attack by a technical but direct computation in explicit Dieudonn\'{e} theory. The basic strategy for this computation builds on the strategy used for the modular curve  in \cite[Theorem 5.3.1]{HoweUnipotent}. However, even in the case of the modular curve, our use of the torsor $\mathcal{P}_\dR$ here and the parallelization of its tangent bundle provides a more ``direct" and clean conceptual argument than the one in \cite{HoweUnipotent} --- in particular, it provides the relation with the theory of nearly holomorphic forms directly, instead of breaking it off into another step in the process (the modular curve case suffers somewhat from the embarrassment of prior computations available to compare with, which obscures the essential role of this structure). 

\subsubsection{} In fact, we make the derivative computation in a more general setting (see \S \ref{sub:GeneralPAdicDifferentialOperators}) that also applies to other Newton strata, and to other quotients of $\igcsf$ in the $\mu$-ordinary case. This applies for example when $\mathcal{U}_{\overline{\mu}}$ is not abelian, and allows us to interpolate the Maass--Shimura operators in the directions of $\mf{u}_{\mu} \backslash \mf{u}_{\mu}^{\mathcal{N}}$. From a classical perspective, the case of the Mantovan Igusa variety is the key one, and thus we do not push very far in the applications of this more general statement. However, even for proving Theorem \ref{Thm:IntroIntegralAction}, it is useful to allow this level of generality in order to facilitate a reduction of the computation to the Siegel case (as, in general, the $\mu$-ordinary stratum in $\scrs_K$ will embed into a non-ordinary stratum in the Siegel moduli space). 

\subsubsection{}\label{sss.intro-real} This relaxation is also useful from a conceptual perspective. In particular, it highlights that our computations are essentially the same as the computations comparing different constructions of $C^\infty$ Maass--Shimura operators. Indeed, the ``real Igusa variety" is the quotient $\mathsf{G}(\mathbb{Q}) \backslash \mathsf{G}(\mathbb{A})/K$, whose $C^\infty$ function theory plays host to the classical complex theory of automorphic forms on $G$. Viewing 
\[ \mathsf{X} \subseteq \Fl_{[\mu^{-1}]}(\mathbb{C})=\mathsf{G}(\mathbb{C})/\mathsf{P}_{\mu^{-1}}(\mathbb{C})\]
by sending a Hodge structure to the Hodge filtration, we can choose an element $x_\dR \in \mathsf{G}(\mathbb{C})$, and then consider the embedding 
\[\mathsf{G}(\mathbb{Q}) \backslash \mathsf{G}(\mathbb{A})/K = \mathsf{G}(\mathbb{Q}) \backslash \mathsf{G}(\mathbb{R}) \times \mathsf{G}(\mathbb{A}^{\infty})/K \to \mathsf{G}(\mathbb{Q}) \backslash \mathsf{G}(\mathbb{C}) \times \mathsf{G}(\mathbb{A}^{\infty})/K.\]
sending $(g_\infty,g^\infty)$ to $(g_\infty x_\dR, g^\infty)$. This plays the role of \eqref{intro.map-to-pdr-quot} here: Indeed, 
\[ \mathcal{P}_\dR(\mathbb{C})=\mathsf{G}(\mathbb{Q}) \backslash \mathsf{G}(\mathbb{C}) \times \mathsf{G}(\mathbb{A}^{\infty})/K,\]
and a parallelization of its tangent bundle is evident from the right multiplication action of $G(\mathbb{C})$; this analytic parallelization is the same as the algebraic parallelization used above, and encodes within it the Maass--Shimura operators. We explain this comparison in more detail in \S\ref{sub:AlgebraicMaassShimura}; the interpretation of $\mathsf{G}(\mathbb{Q}) \backslash \mathsf{G}(\mathbb{A})/K$ as a ``real Igusa variety" is developed partially in \cite[\S1]{HoweUnipotent} and \cite[Remark 1.1.1]{Howe.CompletedKirillov}.

\begin{Rem}
In certain special PEL cases, there is another construction of actions of $p$-adic formal groups on Mantovan Igusa varieties which uses Baer sums of extensions of $p$-divisible groups, see e.g. \cites{LiuZhangZhang, Fiore}. We expect that their actions agree with ours.
\end{Rem}

\subsection{Outline} In \S \ref{Sec:Preliminaries} we collect some preliminaries on crystalline Dieudonn\'e theory and Grothendieck--Messing theory, as well as recalling some perfectoid geometry. In \S \ref{Sec:IntegralModels}, we recall integral models of Shimura varieties of Hodge type, the canonical torsors over them and the Hodge period map. Next, we recall the theory of Igusa varieties in this setting, and prove some new properties of them in the $\mu$-ordinary case. In \S \ref{Sec:DifferentialOperators} we prove Theorem \ref{Thm:IntroIntegralAction} as outlined above, and give a detailed discussion of Remark \ref{Rem:IntroEischenMantovan}. In \S \ref{Sec:ApplyingFourier} we apply the $p$-adic Fourier theory of \cite{FourierPaper} to prove Theorem \ref{Thm:IntroInfiniteLevelTheorem} and Corollaries \ref{Cor:IntroIntegralOperators} and \ref{Cor:IntroRationalOperatorsI}. In \S \ref{Sec:NearlyOC}, we prove Theorem \ref{Thm:IntroNOCmainThm} and in \S \ref{Sec:PAdicLFunctions}, we speculate about  applications to $p$-adic $L$-functions in several examples. In Appendix \ref{appendix.locally-convex-inductive-limits} we discuss some subtleties of the interaction of locally convex inductive limits with condensed mathematics.

\subsection{Acknowledgements} We would like to thank Fabrizio Andreatta, Ana Caraiani, Sean Cotner, Ellen Eischen, Michael Harris, Lorenzo La Porta, Elena Mantovan, James Newton and Karl Schwede for helpful discussions and correspondence. We are particularly grateful to Michael Harris for his consistent encouragement, and for his stimulating questions surrounding Theorem \ref{Thm:IntroInfiniteLevelTheorem}. Large parts of this paper were written while the authors were visiting the University of Oxford and the Max-Planck-Institut für Mathematik (Bonn), and we thank these institutions for their hospitality. For the purpose of Open Access, the authors have applied a CC BY public copyright licence to any Author Accepted Manuscript (AAM) version arising from this submission. 

\subsection{Notation and conventions}\label{sss.cocharacter-conv}

For $\lambda$ a fractional cocharacter of a reductive group scheme $\mathcal{G}$ over a DVR, we write $\mathcal{P}_{\lambda}$ for the associated parabolic subgroup of $\mathcal{G}$, i.e., the stabilizer of the decreasing filtration $\Fil_\lambda^\bullet$ by weight spaces for $\lambda$ (for any representation $\mc{V}$ of $\mc{G}$, the filtered part $\Fil^t_{\lambda} \mc{V}$ is the sum of weight spaces where $\lambda$ acts through weight $t'\geq t$). We write  $\mathcal{M}_{\lambda}$ for the centralizer of $\lambda$, which is a Levi subgroup of $\mc{P}_{\lambda}$, and $\mathcal{U}_{\lambda}$ for the unipotent radical of $\mc{P}_{\lambda}$. We write $\mf{g},\mf{p}_{\lambda}$, $\mf{m}_{\lambda}$, and $\mf{u}_{\lambda}$, respectively, for the Lie algebras of $\mc{G},\mc{P}_\lambda$, $\mc{M}_\lambda$, and $\mc{U}_\lambda$. We will also often write $G,P_{\lambda}, M_{\lambda}$ and $U_{\lambda}$ for the respective generic fibers of $\mc{G},\mc{P}_\lambda$, $\mc{M}_\lambda$, and $\mc{U}_\lambda$.

We will denote Shimura data and their reflex fields by sans serif fonts, for example $\gx$ for a Shimura datum and $\mathsf{E}$ for its reflex field. For a prime $p$ and a place $v$ of $\mathsf{E}$ above $p$, we will write $E$ for the completion of $\mathsf{E}$ at $v$ with ring of integers $\mathcal{O}_E$ and residue field $k_E$. We will then write $G$ for the base change of $\mathsf{G}$ to $\qp$, and $\mathcal{G}$ for a reductive group scheme over $\zp$ with generic fiber $G$. By \cite[Proposition 3.14]{DanielsYoucis}), there is an essentially unique reductive group scheme $\mc{G}_{\zlocp}$ over $\zlocp$ with generic fiber $\g$ such that $\mathcal{G} = \mc{G}_{\zlocp} \otimes_{\zlocp} \zp$.

Our (conjugacy class of) Hodge cocharacters, normalized as in \cite[Section 1.3.1]{KisinPoints}, is denoted by $[\mu]$ and is defined over $\mathsf{E}$. Our normalization of the Hodge cocharacter is such that the descending Hodge filtration on de Rham homology is a $\mu^{-1}$-filtration as above for some $\mu \in [\mu]$. Under this convention, the de Rham homology of the universal abelian variety with its descending Hodge filtration defines a (left) $P_{\mu^{-1}}$-torsor over the Shimura variety. We consider $[\mu]$ as a $G(\qpbar)$-conjugacy class using the place $v$ of $\mathsf{E}$. We use $\mathrm{Gr}_G$ to denote the $B_\mathrm{dR}^+$-affine Grassmannian, with its stratification by Schubert cells as defined in \cite[Definition 19.2.2]{ScholzeWeinsteinBerkeley}. Under our conventions, the Shimura variety over $E$ with infinite level at $p$ has a Hodge--Tate period map with target in $\mathrm{Gr}_{G,{\mu^{-1}}}$. Moreover, the Białynicki-Birula map induces an identification $\grgmu \to \operatorname{Fl}_{G,[\mu]}^{\diamondsuit}$, where $\operatorname{Fl}_{G,[\mu]}$ is the variety of filtrations (or parabolic subgroups) of type $[\mu]$. 

We  define Newton cocharacters as usual in the theory of isocrystals. Thus, for $\nu$ the Newton cocharacter, any associated slope filtration (e.g., on a deformation of a $\mu$-ordinary $p$-divisible group or the Harder--Narasimhan filtration for an associated vector bundle on the Fargues--Fontaine curve) is a $\nu^{-1}$-filtration. 

\begin{Eg}
To help fix these conventions in mind, we note that a Hodge cocharacter for the Shimura datum defining the modular curve can be taken as 
\[ \mu: \mathbb{G}_m \rightarrow \mathrm{GL}_2,\; z \mapsto \begin{bmatrix} z & 0 \\ 0 & 1\end{bmatrix}.\]
For $E$ an elliptic curve we may view the standard representation of $\mathrm{GL}_2$ as the de Rham homology $H_{1,\dR}$ of $E$. The upper left $z$ in the matrix corresponds to $\Lie E$, i.e., the $(-1,0)$ part isomorphic to the non-trivial quotient $\gr^{-1}H_{1,\dR}$ for the Hodge filtration, and the lower right $1$ in the matrix corresponds to $\omega_{E^\vee}$, i.e., the $(0,-1)$ part that is the non-trivial submodule $\Fil^0H_{1,\dR}$ for the Hodge filtration --- in other words, the Hodge filtration on $H_{1,\dR}$ is the filtration $\Fil^\bullet_{\mu^{-1}}$ as defined above.  

The covariant (i.e., homological) isocrystal of an ordinary elliptic curve $E_0$ over $\overline{\mathbb{F}}_p$ is isomorphic to $(\qpbr^2, \mu^{-1}(p) \circ \sigma)$, for $\sigma$ the lift of Frobenius on coordinates, so that its Newton cocharacter is  $\nu=\mu^{-1}$. In particular, the associated slope filtration is a $\nu^{-1}=\mu$ filtration. At the level of $p$-divisible groups, this says that the multiplicative formal group appearing naturally as a sub-$p$-divisible group in any lift $E$ of $E_0$ corresponds to the upper left entry $z$ as above. This is compatible with the Lie algebra in the Hodge decomposition appearing there too --- in other words, the slope filtration in this case induces the splitting of the Hodge filtration on de Rham homology dual to the unit root splitting on de Rham cohomology. 
\end{Eg}


\section{Preliminaries} \label{Sec:Preliminaries} In this section we will collect and prove preliminary results about $p$-divisible groups. We will also recall some perfectoid geometry.

\subsection{Some perfectoid geometry} In this section we introduce some standard notation in perfectoid geometry; we refer the reader to \cite{ScholzeWeinsteinBerkeley}, \cite{FarguesScholze} and \cite[Section 2.1]{PappasRapoportShtukas} for more details. Let us write $\perf$ for the category of affinoid perfectoid spaces over $\fp$ equipped with the v-topology.

\subsubsection{Formal schemes} We follow the conventions of formal schemes over $\spf \zp$ of \cite[Section~2.2]{ScholzeWeinstein}. In particular, we think of formal schemes over $\spf \zp$ as functors on the category $\nilp$ of $\zp$-algebras in which $p$ is nilpotent, see \cite[Section~2.2]{ScholzeWeinstein}. Moreover, there is a fully faithful functor $\mathcal{X} \mapsto \mathcal{X}^{\mathrm{ad}}$ from the category of formal schemes over $\spf \zp$ to the category of pre-adic spaces over $\spa \zp$, see \cite[Proposition 2.2.1]{ScholzeWeinstein}.

\subsubsection{} If $X$ is a pre-adic space over $\spa(\zp,\zp)$ we let $X^\diamondsuit$ denote the functor on $\perf$ described in \cite[Lemma 18.1.1]{ScholzeWeinsteinBerkeley}; note that it commutes with fiber products. For a Huber pair $(A,A^+)$ we write $\spd(A,A^+)$ instead of $\spa(A,A^+)^\diamondsuit$, and we abbreviate it as $\spd(A)$ when $A^+$ is equal to the subring $A^\circ$ of power bounded elements. For analytic adic spaces over $\spa(\qp,\zp)$, for example for rigid spaces over a complete extension of $\qp$, the functor $X \mapsto X^\diamondsuit/\spd \mathbb{Q}_p$ is often fully faithful, see e.g. \cite[Lemma 3.1.12]{FourierPaper}. 

\subsubsection{} For a formal scheme $\mathfrak{X}$ over $\spf(\zp)$, we write $\mathfrak{X}^\diamondsuit$ for $(\mathfrak{X}^\mathrm{ad})^\diamondsuit$. If $\mathfrak{X}^\mathrm{ad}_{\eta} = \mathfrak{X}^\mathrm{ad} \times_{\spa(\zp,\zp)} \spa(\qp,\zp)$ denotes the adic generic fiber, then
\begin{align}
    (\mathfrak{X}^\mathrm{ad}_{\eta})^{\diamondsuit} = (\mathfrak{X})^{\diamondsuit} \times_{\spd \zp} \spd \qp,
\end{align}
we will often implicitly use this. For a $\zp$-scheme $X$, we consider $X^\diamond \simeq (\widehat{X})^\diamondsuit$, where $\widehat{X}$ denotes the formal scheme over $\spf(\zp)$ given by the $p$-adic completion of $X$. 

\subsubsection{} For $(R,R^+) \in \perf$, recall from \cite[Proposition II.1.16]{FarguesScholze} the relative Fargues--Fontaine curve $X_{(R,R^+)}/\spa (\mathbb{Q}_p, \mathbb{Z}_p)$. 

\subsubsection{} \label{subsub:BunGI}Let $G$ be a reductive group over $\qp$. Following \cite{FarguesScholze}, we denote by $\bun_G(R,R^+)$ the groupoid of $G$-torsors on $X_{(R,R^+)}$. By \cite[Theorem~III.0.2]{FarguesScholze}, the presheaf of groupoids $\bun_G$ on $\perf$ sending $(R,R^+)$ to $\bun_G(R,R^+)$ is a stack for the v-topology. 

\subsubsection{} \label{subsub:BunGII} For a choice of algebraic closure $\ovfp$ of $\fp$ we set $\zpbr=W(\ovfp)$ and $\qpbr=W(\ovfp)[\tfrac{1}{p}]$. Let $\sigma$ be the automorphism of $\qpbr$ induced by the absolute Frobenius on $\ovfp$. Let $B(G)$ be the set of $\sigma$-conjugacy classes in $G(\qpbr)$. If $[\mu]$ is a $G(\qpbar)$-conjugacy class of minuscule cocharacters, we let $\bgmu \subset B(G)$ be the subset of $[\mu^{-1}]$-admissible elements, as defined in \cite[Section 1.1.5]{KMPS}. If $k$ is a perfect field and $b \in G(W(k)[\tfrac{1}{p}])$, then there is an induced morphism $\spd k \xrightarrow{b} \bun_{G,k}$, and we let $\widetilde{G}_b$ denote the automorphism functor of $b$ over $\spd k$.

\subsection{Preliminaries on \texorpdfstring{$p$}{p}-divisible groups} Recall that $\nilp$ denotes the category of $\zp$-algebras $R$ such that $p$ is nilpotent in $R$.

\subsubsection{Tate modules and universal covers}
For $R \in \nilp$ and $\mbx$ a $p$-divisible group  over $R$, we will write $T_p \mbx$ for the $p$-adic Tate module of $\mbx$ defined by
\begin{align}
    T_p \mbx:=\varprojlim_n \mbx[p^n].
\end{align}
It is representable by a flat affine scheme over $\spec R$ by \cite[Prop. 3.3.1]{ScholzeWeinstein}. We moreover write $\tildex$ for the universal cover of $\mbx$, defined as 
\begin{align}
    \tildex:=\varprojlim_{[p]} \mbx.
\end{align}
If $\mbx$ is isogenous to the product of an \'etale $p$-divisible group and a connected $p$-divisible group, then $\tildex$ is representable by a formal scheme by \cite[Prop. 3.1.3.(iii)]{ScholzeWeinstein}. It is explained in the proof of \cite[Proposition 4.1.2]{CaraianiScholze} that there is a short exact sequence of fpqc sheaves 
\begin{align}
    0 \to T_p \mbx \to \tildex \to \mbx \to 0
\end{align}
where the map $\tildex \to \mbx$ is projection onto the first coordinate. 

\begin{Lem}\label{lemma.cover-and-tate-module-rep-properties}
Let $k \in \nilp$ be a perfect field. Let $\mbx_0$ be a $p$-divisible group over $k$, and let $\tildex_0$ be its universal cover. Then $\mbx_0=\mbx^\circ_0 \times \mbx^{\et}_0$, where $\mbx^\circ_0$ is connected and $\mbx^{\et}_0$ is \'{e}tale. For any choice for coordinates $x_1, \ldots, x_d$ on $\mbx^\circ_0$, so that $\mbx^\circ_0=\Spf k[[x_1,\ldots,x_d]]$, we have that $\tildex^\circ_0=\Spf k[[x_1^{1/p^\infty},\ldots, x_d^{1/p^\infty}]]$. In particular:
\begin{enumerate}
\item $T_p \mbx_0=\spec \left(\mathcal{O}(T_p \mbx^{\et}) [x_1^{1/p^\infty},\ldots, x_d^{1/p^\infty}]/(x_1, \ldots, x_d) \right)$, 
so that $T_p \mbx_0$ is represented by the quotient of a perfect ring by a regular sequence.
\item If $\mbx$ is a lift of $\mbx_0$ to $\Spf W(k)$, then $T_p \mbx$ is represented by the formal spectrum of a $p$-adically complete $W(k)$-algebra $\mathcal{O}(T_p \mbx)$ with $\mathcal{O}(T_p \mbx)/p$ the quotient of a perfect ring by a regular sequence.
\item If $\mbx$ is a lift of $\mbx_0$ to $\Spf W(k)$, then its connected part $\mbx^\circ$ is a lift of $\mbx_0^\circ$, and, for coordinates $y_1, \ldots, y_d$ on $\mbx^\circ$ so that 
$\mbx^\circ=\Spf(W(k)[[y_1, \ldots, y_d]])$, 
\[ \widetilde{\mbx^\circ}=\Spf W(k)[[x_1^{1/p^\infty},\ldots, x_d^{1/p^\infty}]] \]
where the $x_i$ are elements of $\mathcal{O}(\widetilde{\mbx^\circ})$ such that $x_i \equiv y_i \mod p$.
\end{enumerate}
\end{Lem}
\begin{proof}
    The presentation $\tildex^\circ_0=\Spf k[[x_1^{1/p^\infty},\ldots, x_d^{1/p^\infty}]]$ follows from  \cite[Proposition 3.1.3-(iii)]{ScholzeWeinstein}, and the claims (1)-(3) follow from this result.
\end{proof}

\subsubsection{} For $R \in \nilp$ and $\mbx$ a $p$-divisible group  over $R$, we write $E\mbx$ for the universal vector extension of $\mbx$, and $D(\mbx)=\Lie E\mbx$. If $R$ is $p$-adically complete and $\mbx$ is a $p$-divisible group over $R$, we write $D(\mbx)=\varprojlim_n D(\mbx_{R/p^n})$. 

\subsubsection{}For $R \in \nilp$ and $A$ an abelian variety over $R$, we write $EA$ for the universal vector extension of $A$, and $D(A)=\Lie EA$. If $R$ is $p$-adically complete and $A$ is an abelian variety over $\Spf R$, we write $D(A)=\varprojlim_n D(A_{R/p^n})$. In both cases, the map $A[p^\infty] \rightarrow A$ induces a canonical identification $D(A[p^\infty])=D(A)$. 

\subsubsection{}\label{GM} Let $R \in \nilp$, and recall the big fppf crystalline site $\operatorname{CRIS}(\spec R)=\operatorname{CRIS}(\spec R/\Sigma)$, see \cite[Section 2.1.1]{DeJongDieudonne} or \cite[Section 1.1]{BBM}. Here $\Sigma=(\spec \zp, p \zp, \gamma)$, where $\gamma$ is the canonical divided power structure on $p \zp$. We note that $p$ is locally nilpotent on $T'$ for any $T/\Spec R \rightarrow T'/\Spec \mathbb{Z}_p$ in $\operatorname{CRIS}(\spec R)$. For $\Spec S / \Spec R$ and $S' \twoheadrightarrow S$ an arbitrary $p$-complete divided powers thickening, we can evaluate any crystal $\mathbb{M}$ by
\begin{multline} \mathbb{M}((\Spec S / \Spec R, \Spec S \hookrightarrow \Spec S')):=\\\varprojlim_n \mathbb{M}(\Spec (S/p^n) / \Spec R \hookrightarrow \Spec (S'/p^n) / \Spec \mathbb{Z}_p). \end{multline}
We will write this evaluation as $\mathbb{M}(S')$ when the other data is clear.

\subsubsection{} For $R \in \nilp$ and a $p$-divisible group $\mbx$ over $\spec R$, as in \cite[\S3.2]{ScholzeWeinstein}, there is a crystal $\mathbb{M}(\mbx)$ on $\operatorname{CRIS}(\spec R)$ whose values on affine $\Spec S/\Spec R \hookrightarrow \Spec S'/\Spec \mathbb{Z}_p$ can be computed as $D(\mathbb{X}_{S'})$ for $\mathbb{X}_{S'}$ any lift of $\mathbb{X}_S$ to ${S'}$.

\newcommand{\QPRS}{{\mathrm{QPRS}}}
\subsubsection{}\label{sss.qprs-dieudonne} We write $\QPRS$\footnote{A quotient of a perfect ring by a regular sequence is, in particular, quasi-regular semi-perfectoid in the sense of \cite[Definition 1.1.4]{AnschutzLeBras}. We use the smaller class $\QPRS$ because already the results on explicit Dieudonn\'{e} theory for $\QPRS$ rings as established in \cite{ScholzeWeinstein} will suffice for our purposes.} for the full sub-category of $\mathbb{F}_p$-algebras whose objects are those $\mathbb{F}_p$-algebras that are isomorphic to a quotient of a perfect $\mathbb{F}_p$-algebra by a regular sequence. For $R \in \QPRS$, there is an initial $p$-adically complete divided powers thickening $A_{\mathrm{cris}}(R) \twoheadrightarrow R$ see \cite[Proposition 4.1.3]{ScholzeWeinstein}; we let $B^+_{\cris}(R)=A_{\cris}(R)[\frac{1}{p}]$. We write $\sigma:A_{\mathrm{cris}}(R) \to A_{\mathrm{cris}}(R)$ for the map induced by the absolute Frobenius on $R$. 

Recall that a \emph{covariant Dieudonn\'e module} over $R$ is a pair $(M^+, \varphi_{M^+})$, where $M^+$ is a finite locally free $A_{\mathrm{cris}}(R)$-module and where
\begin{align}
    \varphi_{M^+}:  \sigma^{\ast} M^+[\tfrac 1p] \to M^+[\tfrac 1p]
\end{align}
is an isomorphism such that  \begin{align}
    M^+ \subseteq \varphi_{M^+} (\sigma^{\ast}M^+) \subseteq \tfrac{1}{p} M^+.
\end{align}
Given a $p$-divisible group $\mbx$ over $R$, we let 
\[ \mathbb{D}(\mbx)=\mathbb{M}(\mbx)(A_{\mathrm{cris}}(R)), \]
equipped with its Dieudonn\'{e} module structure as in  \cite[\S4.1]{ScholzeWeinstein} (here our $\varphi$ is the map $\frac{F}{p}$ for $F$ as in \cite[\S4.1]{ScholzeWeinstein}). By \cite[Theorem A]{ScholzeWeinstein} (cf. \cite{AnschutzLeBras} for a strengthening), the functor $\mbx \to \mathbb{D}(\mbx)$ from $p$-divisible groups over $R$ to covariant Dieudonn\'e modules over $R$ is fully faithful, and the functor $\mbx \to \mathbb{D}(\mbx)[\tfrac{1}{p}]$ from the isogeny category of $p$-divisible groups over $R$ to $\varphi$-modules over $B^+_\cris(R)$ is fully faithful. 

\begin{Rem} \label{Rem:NormalizationsDieudonne}
 Our conventions agree with those of \cite{CaraianiScholze}. In particular, the Dieudonn\'e module of $\qp/\mathbb{Z}_p$ is $A_{\mathrm{cris}}(R)$ equipped with the Frobenius given by $\sigma$, and the Dieudonn\'e module of $\mu_{p^{\infty}}$ is $A_{\mathrm{cris}}(R)$ equipped with the Frobenius given by $\tfrac{\sigma}{p}$. This also means that the contravariant Dieudonn\'e module is isomorphic to the dual of the covariant Dieudonn\'e module, seen as an $F$-crystal (see \cite[footnote on page 692]{CaraianiScholze}).
\end{Rem}

\subsubsection{}\label{ss.D-maps} Suppose $R$ is $p$-adically complete and and $R/p \in \QPRS$. Then, for $\mathbb{X}_1$ and $\mathbb{X}_2$ $p$-divisible groups over $R$, maps $\alpha:\tildex_1 \rightarrow \tilde{\mathbb{X}}_2$ are the same as quasi-isogenies $\mathbb{X}_{1,R/p} \rightarrow \mathbb{X}_{2,R/p}$, and we write 
\[ D(\alpha): D(\mbx_1)[\tfrac{1}{p}] \rightarrow D(\mbx_2)[\tfrac{1}{p}]\]
for the map induced by base change along $B^+_\cris(R/p) \rightarrow R[\tfrac{1}{p}]$ of the map \[\mathbb{D}(\alpha): \mathbb{D}(\mbx_{1,R/p})[\tfrac{1}{p}]\rightarrow \mathbb{D}(\mbx_{2,R/p})[\tfrac{1}{p}]\]
induced by the associated quasi-isogeny mod $p$. We note that, if $\alpha$ comes from an isomorphism $\mathbb{X}_{1,R/J} \xrightarrow{\sim} \mathbb{X}_{2,R/J}$ for some divided powers ideal $J$, then $D(\alpha)$ agrees, after inverting $p$, with the integral map $D(\mbx_1) \rightarrow D(\mbx_2)$ induced by the crystalline connection, i.e., the crystal structure of $\mathbb{M}(\mbx_{1,R/J}) \simeq \mathbb{M}(\mbx_{2,R/J})$ (since $\mathbb{D}(\alpha)$ is induced by the same crystal structure). In particular, in this setting, if $R$ is also $p$-torsion free then $D(\alpha)$ is integral and equals this map.

\subsubsection{}\label{sss.gm-theory}

Let $R$ be $p$-adically complete. For $\mbb{X}$ a $p$-divisible group over $R$ and $\mbb{X}_0:=\mbx_{R/p}$, we have $\mathbb{M}(\mbx_0)(R)=D(\mbb{X})=\Lie E\mbx$. Writing $V_{\mbx}$ for the kernel of $E\mbx \rightarrow \mbx$ and similarly for $V_0$, we obtain a functor 
\[ \mbx \mapsto (\mbb{X}_0, V_\mbx \subseteq \mathbb{M}(\mbx_0)(R) ) \]
from $p$-divisible groups over $R$ to pairs consisting of a $p$-divisible group over $R/p$ and a lift of the Hodge filtration $V_0 \subseteq \mathbb{M}(\mbx_0)(R/p)$ to $V \subseteq \mathbb{M}(\mbx_0)(R)$. By Grothendieck--Messing theory, see \cite[Theorem 1.6 on p. 151]{Messing-TheCrystalsAssociatedToBarsottiTateGroups}, this is an equivalence. 

\subsubsection{} Suppose now that $R/\mathbb{Z}_p$ is $p$-adically complete and $R/p \in \QPRS$. Combining the Dieudonn\'{e} theory over $R/p$ with Grothendieck--Messing theory, we obtain:
\begin{Prop}\label{prop.filtered-dieudonne-theory}
    The functor 
    \[ (\mathbb{D}, \Fil^0): \mbb{X} \mapsto (\mathbb{D}(\mbb{X}_0), V_{\mbb{X}} \subseteq \Lie E\mbb{X}=\mathbb{M}(\mbb{X}_0)(R)=\mathbb{D}(\mbb{X}_0) \otimes_{A_\cris(R/p)} R) \] 
is fully faithful from the category of $p$-divisible groups over $R$ to the category of pairs $(M, V)$ consisting of a Dieudonn\'{e} module over $R/p$ and a locally free direct summand $V=\Fil^0(M \otimes_{A_\cris(R/p)} R)$, and there is a functorial identification 
\[ \Lie \mbb{X} = \gr^{-1}(\mathbb{D}(\mbb{X}_0) \otimes_{A_\cris(R/p)} R)= (\mathbb{D}(\mbb{X}_0) \otimes_{A_\cris(R/p)} R)/V_{\mathbb{X}}.\]
This functor is compatible with any base change $R \rightarrow R'$ for $R$ and $R'$ any $p$-complete rings with $R/p$ and $R'/p$ in $\QPRS$. Moreover, when $R/p=k$ is a perfect field so that $A_\cris(R/p)=W(k)$, the essential image of $(\mathbb{D}, V)$ consists of those $(M, V)$ such that $V_k$ is the kernel of the map
\[ M_{k} \rightarrow \sigma_*M_{k}  \]
obtained by reducing modulo $p$ the adjoint to $p\varphi: (\sigma^*M)_{W(k)} \rightarrow M_{W(k)}$. 
\end{Prop}
\begin{proof}
The full faithfulness is immediate from the discussion in \S\ref{sss.qprs-dieudonne} and \S\ref{sss.gm-theory}, and the compatibility with base change comes from the compatibility of $p$-complete base change for the formation of the universal vector extension. Thus it remains only to verify the last statement, where it suffices to observe that classical Dieudonn\'{e} theory for perfect fields implies this kernel is equal to the Hodge filtration and that any Dieudonn\'{e} module arises from a $p$-divisible group. 
\end{proof}
When $k$ is a perfect field and $M$ is Dieudonn\'e module over $k$, we will refer to locally free direct summands $\operatorname{Fil}^0 M \subset M$ as in Proposition \ref{prop.filtered-dieudonne-theory} as \emph{admissible filtrations} on $M$.

\subsubsection{}\label{sss.abelian-variety-compatibility}
Suppose $R \in \nilp$ and $A/R$ is an abelian variety. Then, by \cite[\S4, p.65, \S13]{MazurMessing}, the dual $\mathbb{M}(A[p^\infty])^*$ is canonically identified with the relative degree $1$ crystalline cohomology $\mathbb{M}(A)^*$ of $A/R$, and this matches the Hodge filtrations on $\mathbb{M}(A)(R)^*$ and $\mathbb{M}(A[p^\infty])(R)^*$ (as both can be computed using the Lie algebra of the universal vector extension). Moreover, by \cite[Proposition V.3.6.4 on p.359]{Berthelot-CohomologieCristallineDesSchemeas}, the Gauss--Manin connection and Hodge filtration on the relative de Rham cohomology $H^{1}_{\dr}(A/R)$ are canonically identified with the crystalline connection and Hodge filtration on $\mathbb{M}(A)(R)^*$. Taking duals everywhere, we find that the de Rham homology $\mathcal{H}_{1}^{\dr}(A/R)$ and $\mathbb{M}(A[p^\infty])(R)$ are canonically identified, matching the crystalline and Gauss--Manin connections and matching their Hodge filtrations.


\section{Shimura varieties and Igusa varieties} \label{Sec:IntegralModels} In \S \ref{Sub:IntegralModels} we recall some of the theory of canonical integral models of Shimura varieties of Hodge type. In particular, we define the canonical torsor and its Hodge period map, and in \S \ref{sub:canonical-P-torsor} we study the differential of the Hodge period map. In \S \ref{sub:AlgebraicMaassShimura} we then give a construction of algebraic Maass--Shimura operators on spaces of integral nearly holomorphic automorphic forms, and compare these to analytic Maass--Shimura operators over $\mathbb{C}$. In \S \ref{Sub:IgusaI} we recall the Igusa varieties of \cite{HamacherKim} and some results of \cite{DAddeziovH}. In \S \ref{sub:ComparisonIgusa}, we compare the Igusa varieties of \cite{HamacherKim} with other Igusa varieties in the literature. After that, we specialize to the setting of $\mu$-ordinary Igusa varieties in \S \ref{sub:MuOrdinaryIgusa}, and study their formal lifts and adic generic fibers in \S \ref{sub:InfiniteLevel}.

\subsection{Integral models and the canonical \texorpdfstring{$\mathcal{G}$}{G}-torsor} \label{Sub:IntegralModels} 

\subsubsection{} For a symplectic space $(V, \psi)$ over $\mathbb{Q}$, we write $\g_V \coloneqq \operatorname{GSp}(V, \psi)$ for the group of symplectic similitudes of $V$ over $\mathbb{Q}$. It admits a Shimura datum $\gvx$, where $\mathsf{H}_V$ is the union of the Siegel upper and lower half-spaces, interpreted as a conjugacy class of weight $-1$ homomorphisms from the Deligne torus as in \cite[(1.3.2)]{KisinPoints}. 

Let $\gx$ be a Shimura datum of Hodge type with reflex field $\mathsf{E}$. Fix a prime $p$ and a place $v$ above $p$ of the reflex field $\mathsf{E}$. Write $G:=\g \otimes \qp$, and $E:=\mathsf{E}_v$ for the $v$-adic completion of $\mathsf{E}$, and $\mathcal{O}_E$ for its ring of integers. Let $K_p \subseteq G(\qp)$ be a hyperspecial subgroup. We fix a choice of a Hodge embedding $\gx \to \gvx$ and a self-dual $\zlocp$-lattice $\Lambda_{(p)} \subseteq V$ such that $K_p$ is the stabilizer in $G(\qp)$ of $\Lambda \coloneqq \Lambda_{(p)} \otimes_{ \zlocp} \zp$; such choices are possible as explained in \cite[Section 2.3.15]{KisinPappas}. Let $\mathcal{G}_{\zlocp}$ be the Zariski closure of $\g$ in $\operatorname{GL}(\Lambda_{(p)})$, a reductive group scheme over $\zlocp$, and write $\mathcal{G}:=\mathcal{G}_{\zlocp}\otimes_{\zlocp}\zp$ for the corresponding reductive integral model of $G$ over $\zp$.

\subsubsection{} Fix a neat compact open subgroup $K^p \subset \gafp$ and let $K=K^pK_p$. By \cite[Main theorem]{KisinModels} and \cite[Main Theorem]{MadapusiPeraKim}, there is a smooth canonical integral model $\scrs_{K}=\scrs_{K}\gx$ over $\oelocp$ of the Shimura variety $\mathbf{Sh}_K\gx$. By \cite[Theorem 1.1.1]{XuNormalization}, we can find a compact open subgroup $U^p \subset \gv(\afp)$ such that the natural map
\begin{align}
    \mathscr{S}_{K}\gx \to \mathscr{S}_{U}\gvx \otimes_{\zlocp} \mathcal{O}_{E,(v)}
\end{align}
is a closed immersion. The pullback of the universal abelian scheme (up to prime-to-$p$ isogeny) over $\mathscr{S}_{U}\gvx$ gives rise to an abelian scheme (up to prime-to-$p$ isogeny) $A$ over $\mathscr{S}_{K}\gx$ with associated $p$-divisible group $\mby=A[p^{\infty}]$. 

\subsubsection{} We will often write $\scrs_K = \scrs_K\gx$ if the Shimura datum $\gx$ is clear from context; the same goes for all the other objects that we will introduce below.

\subsubsection{Tensors} Let us write $\Lambda_{(p)}^\otimes$ for the tensor space of $\Lambda_{(p)}$, defined as the direct sum of $\Lambda_{(p)}^{\otimes n} \otimes (\Lambda_{(p)}^{\ast})^{\otimes m}$ for all pairs of integers $m, n \ge 0$. We will also use this notation later for modules over commutative rings, or over sheaves of rings. By \cite[Lemma 1.3.2]{KisinModels}, the subgroup $\mathcal{G}_{\zlocp} \subseteq \operatorname{GL}(\Lambda_{(p)})$ is the stabilizer of a collection of tensors $\{s_\alpha \in \Lambda_{(p)}^\otimes\}_{\alpha \in \mathscr{A}}$. 

\subsubsection{} By \cite[Corollary 2.3.9]{KisinModels}, there are natural tensors 
\[
\{s_{\alpha, \dr} \in H^{\dr}_1(A/\scrs_K\gx)^{\otimes}\}_{\alpha \in \mathscr{A}},
\]
which are horizontal for the Gauss--Manin connection $\nabla$. We now define a $\mathcal{G}_{\zlocp}$-torsor $\beta:\mathcal{G}_{\dr} \to \scrs_K$ which sends $x:T \to \scrs_K$ to the set of isomorphisms
\begin{align}
   x^{\ast}H^{\dr}_1(A/\scrs_K\gx) \xrightarrow{\sim} \Lambda_{(p),T}
\end{align}
taking $s_{\alpha, \mathrm{dR}}$ to $s_{\alpha} \otimes 1$ for all $\alpha \in \mathscr{A}$, see \cite[Proposition 4.5.6]{Lovering} for the fact that this is a $\mathcal{G}_{\zlocp}$-torsor. It is moreover explained in \cite[Section 4.5.7]{Lovering} (see \cite[Lemma 3.3.2]{Lovering}) that there is a $\mathcal{G}_{\zlocp}$-equivariant (and hence smooth) morphism
\begin{align}
    \pi_{\mathrm{Hdg}}:\mathcal{G}_{\dr} \to \operatorname{Fl}_{[\mu^{-1}]}.
\end{align}
Here $\operatorname{Fl}_{[\mu^{-1}]}$ is the natural integral model over $\oelocp$ of the flag variety over $E$ associated with the inverse of the conjugacy class of the Hodge cocharacter $[\mu]$, see \cite[Section 3.2]{Lovering}.

\begin{Rem} We note that Lovering uses de Rham cohomology instead of de Rham homology in \cite[Section 4.5.7]{Lovering}. Moreover, in his definition he takes the set of isomorphisms (in our notation) $\Lambda_{(p),T} \xrightarrow{\sim} H^1_{\dr}(A/\scrs_K\gx)  $ instead of $\Lambda_{(p),T}^{\ast} \xrightarrow{\sim} H^1_{\dr}(A/\scrs_K\gx) $. This is probably either an error or due to a differing convention about the weight of the Hodge structures in the Siegel Shimura datum (which is not actually specified in loc. cit.) from our work / Kisin's work on integral models. In any case, because we use the weight $-1$ Siegel Shimura datum, the representation $V$ of $\g$ should correspond to homological realizations rather than cohomological realizations. See e.g. \cite[Section 5.1.5]{KisinShinZhu} for a comprehensive reference. \end{Rem}

\subsection{The tangent bundle of the canonical \texorpdfstring{$\mathcal{P}$}{p}-torsor} \label{sub:canonical-P-torsor}

\subsubsection{}\label{sss.finite-extension-cocharacter-choice-existence} Choose a finite extension $\mathsf{E} \subset \mathsf{E}' \subset E$ such that $[\mu]$ has a representative defined over $\mathsf{E}'$; this is possible because $G=\g \otimes \qp$ is quasi-split. We fix such a representative, which we can moreover arrange (after possibly enlarging $\mathsf{E}'$) so that it extends to a morphism $\mu:\mathbb{G}_{m,\oeplocp} \to \mathcal{G}_{\oeplocp}$. We thus obtain a parabolic subgroup $\mathcal{P}_{\mu^{-1},\mathcal{O}_{\mathsf{E}',(v')}}$ as the stabilizer of the filtration $\operatorname{Fil}^{\bullet}_{\mu^{-1}}$ (see \S \ref{sss.cocharacter-conv} for our conventions). We note that we can view $\Fil^{\bullet}_{\mu^{-1}}$ as an element of $\Fl_{[\mu^{-1}]}(\oeplocp)$, and that, viewing $\Lambda_{\oeplocp}=\Lambda_{(p)} \otimes_{\zlocp} \oeplocp$ as a representation of $\mc{G}_{\oeplocp}$, it yields a filtration $\operatorname{Fil}^{\bullet}_{\mu^{-1}}\Lambda_{\oeplocp}$.

\subsubsection{} In this section and in \S \ref{sub:AlgebraicMaassShimura} only, we will implicitly basechange $\scrs_K$ to $\mathcal{O}_{\mathsf{E}',(v')}$. We will also omit the subscript $\mathcal{O}_{\mathsf{E}',(v')}$ on $\mathcal{P}_{\mu^{-1}}$ and $\mathcal{G}$. There is a natural $\mathcal{P}_{\mu^{-1}}$-torsor $\mathcal{P}_{\dr} \subset \mathcal{G}_{\dr}$ over $\scrs_K$ representing the following subfunctor: It sends $x:T \to \scrs_K$ to the set of isomorphisms $H^{\dr}_1(A/\scrs_K\gx) \xrightarrow{\sim} \Lambda_{(p),T}$ respecting the tensors and identifying $\operatorname{Fil}^{\bullet}_{\mathrm{Hdg},T}$ with $\operatorname{Fil}^{\bullet}_{\mu^{-1}}\Lambda_T$. In other words $\mathcal{P}_\dR$ is the fiber of $\pi_{\mathrm{Hdg}}$ over $\Fil^{\bullet}_{\mu^{-1}} \in \operatorname{Fl}_{[\mu^{-1}]}(\oeplocp)$. The space of \emph{$p$-integral nearly holomorphic forms of level $K$} is defined to be the space of global sections $\opn{H}^0(\mathscr{S}_K, \mathcal{O}_{\mathcal{P}_{\dR}/\mathscr{S}_K})$. 

\subsubsection{The Gauss--Manin connection} Writing $\mf{g}=\Lie \mathcal{G}_{\zlocp}$, we have a short exact sequence of tangent bundles (where we recall $\beta$ is the structure map of $\mathcal{G}_{\dr}$) 
\begin{equation} \label{Eq:ShortExactSequenceTangentBundles}
     0 \to \mf{g} \otimes \mathcal{O}_{\mathcal{G}_{\dR}} \to T_{\mathcal{G}_{\dR}} \xrightarrow{d\beta}\beta^{\ast} T_{\mathscr{S}_K} \to 0  .
\end{equation}
By the usual theory of the Atiyah bundle (see, e.g., \cite{Biswas.Atiyah} for a detailed description in the context of differential geometry), a connection on $\mathcal{G}_\dR$ is the same as a $\mathcal{G}$-equivariant splitting $s:\beta^{\ast} T_{\mathscr{S}_K} \rightarrow T_{\mathcal{G}_{\dR}}$, and the connection is integrable if and only if the splitting is a map of Lie algebras on the $\mathcal{G}$-invariant sections. This applies, in particular, for the Gauss--Manin connection on $\mathcal{G}_\dR$ and we write $s$ below for the associated splitting. For use in later arguments, we now recall the construction of $s$ in terms of the usual algebraic interpretation of the Gauss--Manin connection.

\subsubsection{} \label{subsub:Constructions} Let us denote $D(A) = \mathcal{H}_1^{\dR}(A/\mathscr{S}_K)$. We use the Gauss--Manin construction on $D(A)$ to construct $s$: Let $x_0 \colon S = \opn{Spec}R \to \mathcal{G}_{\dR}$ be a morphism, and let $\varphi_{x_0} \colon D(A_{x_0}) \xrightarrow{\sim}  \Lambda_{(p), R} $ be the corresponding tensor preserving trivialization. An element $v \in\beta^{\ast}T_{\mathscr{S}_K} (S)$ is the same thing as a morphism 
\[ v \colon S_{\varepsilon} := S \times_{\oeplocp} \opn{Spec}(\oeplocp[\varepsilon]/\varepsilon^2) \to \mathscr{S}_K\]
such that the composition $S \to S \times_{\oeplocp} \opn{Spec}(\oeplocp[\varepsilon]/\varepsilon^2) \xrightarrow{v} \mathscr{S}_K$ agrees with $\beta \circ x_0$. Let $A_v$ denote the abelian variety over $S_{\varepsilon}$ associated with $v$ and let $A_{x_0, S_{\varepsilon}}$ denote the base-change along the map $S_{\varepsilon} \to S$ induced from $\oeplocp \to \oeplocp[\varepsilon]/\varepsilon^2$. The Gauss--Manin connection induces an isomorphism
\[
D(v) \colon D(A_v)  \xrightarrow{\sim} D(A_{x_0, S_{\varepsilon}}) = D(A_{x_0}) \otimes_R R[\varepsilon]/\varepsilon^2
\]
which is the identity modulo $\varepsilon$. The isomorphism $D(v)$ is tensor preserving because the de Rham tensors are horizontal for the Gauss--Manin connection. Thus the composition of $D(v)$ with $\varphi_{x_0} \otimes_{R} R[\varepsilon]/\varepsilon^2$ is also tensor preserving, which means that $v:S_{\varepsilon} \to \mathscr{S}_K$ lifts to a point $s(v):S_{\varepsilon} \to \mathcal{G}_{\dR}$, and this defines $s$. 

\subsubsection{} Recall that $\iota \colon \mathcal{P}_{\dR} \to \mathcal{G}_{\dR}$ is the reduction of structure to $\mathcal{P}_{\mu^{-1}}$ parametrizing frames of $D(A)$ matching the Hodge filtration with $\Fil^\bullet_{\mu^{-1}}$; equivalently, it is the fiber of $\pi_{\Hdg}$ over $\Fil^\bullet_{\mu^{-1}}$. We have a natural map given as the composition 
\begin{equation} \label{Eqn:TPdrToLieG}
T_{\mathcal{P}_{\dR}} \xrightarrow{d\iota} \iota^* T \mathcal{G}_{\dR} \xrightarrow{\iota^*(1-(s\circ d\beta))} \mf{g} \otimes \mathcal{O}_{\mathcal{P}_{\dR}}
\end{equation}
where we note that the map $1-s\circ d\beta$ pulled back for the second arrow is just the left splitting of \eqref{Eq:ShortExactSequenceTangentBundles} induced by the right splitting $s$.

\begin{Prop} \label{Prop:TPdrToLieGisaniso}
    The composition of the two arrows in \eqref{Eqn:TPdrToLieG} is an isomorphism of vector bundles $T_{\mathcal{P}_\dR} \rightarrow \mf{g} \otimes \mathcal{O}_{\mathcal{P}_\dR}$. 
\end{Prop}
\begin{proof}
We will argue by a reduction to the case that $\gx=\gvx$. There is a commutative diagram (where objects with a subscript $V$ denote the corresponding objects for $\gvx$)
    \begin{equation}
        \begin{tikzcd}
            T_{\mathcal{P}_{\dr}} \arrow{r} \arrow{d} & \mathfrak{g} \otimes \mathcal{O}_{\pdr} \arrow{d} \\
            T_{\mathcal{P}_{V,\dr}} \otimes_{\mathcal{O}_{\mathcal{P}_{V,\dr}}} \mathcal{O}_{\pdr} \arrow{r} & \mathfrak{g}_V \otimes \mathcal{O}_{\pdr}  \\
            T_{\mathcal{P}_{V,\dr}} \arrow{r} \arrow{u} & \mathfrak{g}_V \otimes \mathcal{O}_{\mathcal{P}_{V,\dr}}, \arrow{u}
        \end{tikzcd}
    \end{equation}
    where the top vertical maps are injective because $\pdr \to \mathcal{P}_{V,\dr}$ is a closed immersion. 
\begin{Claim} \label{Claim:ReductiontoGSP}
    If the bottom horizontal arrow is an isomorphism, then the top horizontal arrow is an isomorphism
\end{Claim}
\begin{proof}
The middle horizontal arrow is an isomorphism because it is a base change of the bottom horizontal arrow. The top-left vertical map is injective on closed points of $\mathcal{P}_{\dR}$ (because the map $\pdr \to \mathcal{P}_{V,\dr}$ is a closed immersion). Since the middle horizontal arrow is an isomorphism, the base change of the top horizontal map to any closed point of $\mathcal{P}_{\dr}$ stays injective. Since $T_{ \mathcal{P}_{\dr}}$ has the same rank as $\mathfrak{g} \otimes \mathcal{O}_{\mathcal{P}_{\dR}}$, it follows that the top horizontal arrow is also an isomorphism.   
\end{proof}
By Claim \ref{Claim:ReductiontoGSP}, we may and will assume that we are in the case $\gx=\gvx$. We consider the composition
    \begin{align} \label{eq:KSISO}
        \beta^{\ast} T_{\scrs_{K}} \xrightarrow{s} T_{\mathcal{G}_{\dr}} \xrightarrow{d\pi_{\mathrm{Hdg}}} \pi_{\mathrm{Hdg}}^{\ast} T_{\operatorname{Fl}_{[\mu^{-1}]}},
    \end{align}
    which can be identified with the pullback under $\beta$ of the Kodaira--Spencer isomorphism of \cite[Proposition 2.3.5.2]{Lan} by unwinding the definitions. In particular, it is injective at all closed points\footnote{Alternatively, it is injective at all closed points because a square zero deformation of an abelian variety is the same thing as a square zero deformation of its Hodge filtration.}. On the other hand, since $\pdr \subset \mathcal{G}_{\dr}$ is the inverse image of $\Fil_{\mu^{-1}}^{\bullet} \in \operatorname{Fl}_{[\mu^{-1}]}$, the tangent bundle $T_{\mathcal{P}_{\dR}}$ is the kernel of $\iota^* d\pi_{\Hdg}$. In particular, by comparing dimensions, we see that $T_{\mathcal{P}_{\dR}}$ is a complementary subbundle to $\iota^*\beta^{\ast}T_{\scrs_{K}}$ in $\iota^* T_{\mathcal{G}_{\dR}}$, i.e., $\iota^* T_{\mathcal{G}_{\dR}} = \iota^*\beta^{\ast}T_{\scrs_{K}} \oplus T_{\mathcal{P}_{\dR}}$. Thus the projection of $T_{\mathcal{P}_{\dR}}$ along $\iota^*\beta^{\ast}T_{\scrs_{K}}$ to the complementary sub-bundle $\mathfrak{g} \otimes \mathcal{O}_{\mathcal{P}_{\dR}}$ is an isomorphism. 
\end{proof}

\subsection{Algebraic Maass--Shimura operators} \label{sub:AlgebraicMaassShimura} 
Restricting the inverse of the isomorphism in Proposition \ref{Prop:TPdrToLieGisaniso} to $\mf{g}_{\oeplocp} \subseteq H^0(\mathcal{P}_\dR, \mf{g} \otimes \mathcal{O}_{\mathcal{P}_\dR})$, we obtain a canonical inclusion
\begin{equation}\label{eq.inclusion-lie-algebra-to-vector-fields} \mf{g}_{\oeplocp} \hookrightarrow H^0(\mathcal{P}_\dR, T_{\mathcal{P}_\dR}). \end{equation}
The algebraic Maass--Shimura operators will be obtained from those vector fields lying in the image of $\mf{u}_\mu \subseteq \mf{g}$. In particular, we would like to say that these vector fields commute, so that we obtain an induced map from $\operatorname{Sym}(\mf{u}_{\mu})$ to differential operators on $\mathcal{P}_\dR$. While $\mf{u}_{\mu}$ is a commutative Lie subalgebra of $\mf{g}$, we have not yet shown that \eqref{eq.inclusion-lie-algebra-to-vector-fields} respects the Lie bracket. We now establish this by base change to $\mathbb{C}$ where it is clear from a complex analytic computation.

\begin{Rem}The claim can also be deduced by using the Grothendieck--Messing period map in divided powers neighborhoods of closed points to reduce to a computation on the flag variety. In fact, this is essentially the same argument, since in the complex analytic setting we are exploiting the existence of a flat trivialization over the universal cover to reduce to the same computation on the flag variety. We prefer the complex analytic argument here because the familiar complex analytic uniformization makes the existence of the parallelization of Proposition \ref{Prop:TPdrToLieGisaniso} completely transparent, and because the same setup can then be reused in the comparison with classical Maass--Shimura operators. 
\end{Rem}

\begin{Lem}\label{Lem.inclusion-of-Lie-algebras} The map \eqref{eq.inclusion-lie-algebra-to-vector-fields} is a map of Lie algebras, where $\mf{g}$ is equipped with its usual Lie algebra structure from the bracket of left invariant vector fields on $G$ and $H^0(\mathcal{P}_\dR, T_{\mathcal{P}_\dR})$ is equipped with the Lie algebra structure coming from the bracket of vector fields on $\mathcal{P}_\dR$. 
\end{Lem}
\begin{proof}
    Because everything is flat over $\oeplocp$, it suffices to prove the comparison over $\mathbb{C}$. By the construction of Shimura varieties, there is a natural isomorphism of complex analytic spaces
    \begin{equation}\label{eq.complex-gdr} \mathcal{G}_\dR(\mathbb{C}) = \mathsf{G}(\mathbb{Q}) \backslash \left( \mathsf{X} \times \mathsf{G}(\mathbb{C}) \times \mathsf{G}(\mathbb{A}_f)/ K \right)\end{equation}
    such that $\pi_\Hdg: \mathcal{G}_\dR(\mathbb{C}) \rightarrow \Fl_{[\mu^{-1}]}(\mathbb{C})$ is given by viewing $\mathsf{X}$ as an open subset of $\Fl_{[\mu^{-1}]}(\mathbb{C})$ (by sending a Hodge structure to the associated Hodge filtration), then taking the map induced by the $\g(\mathbb{Q})$-invariant map
    \[ \mathsf{X} \times \mathsf{G}(\mathbb{C}) \times \mathsf{G}(\mathbb{A}_f)/ K \xrightarrow{(x, g, g'K) \mapsto (g^{-1}x)} \Fl_{[\mu^{-1}]}(\mathbb{C}) . \]
    The $P_{\mu^{-1}}(\mathbb{C})$-torsor $\mathcal{P}_\dR(\mathbb{C})$ is given by the preimage of $\Fil^\bullet_{\mu^{-1}}$ under this map. We can also express $\mathcal{P}_\dR(\mathbb{C})$ in another way as follows: let $U \subseteq \g(\mathbb{C})$ be the open subset of elements $g$ such that $g \cdot \Fil^{\bullet}_{\mu^{-1}} \in \mathsf{X}$. Then, we obtain a map 
    \[ \mathsf{G}(\mathbb{Q}) \backslash \left( U \times \mathsf{G}(\mathbb{A}_f)/ K \right) \xrightarrow{ (g, g'K) \mapsto (g \Fil^\bullet_{\mu^{-1}}, g, g'K) } \mathcal{G}_\dR(\mathbb{C}) \]
    that is an isomorphism onto $\mathcal{P}_\dR(\mathbb{C})$. In particular, using this presentation we obtain a natural map from $\mf{g}_{\mathbb{C}}$ to vector fields on $\mathcal{P}_\dR(\mathbb{C})$ by restriction of the left-invariant fields on $\g(\mathbb{C})$ to the open subspace $U$. This is evidently compatible with the Lie bracket, so it remains to see that it agrees with the map induced by \eqref{eq.inclusion-lie-algebra-to-vector-fields}. 

    By construction, the splitting from the Gauss--Manin connection 
    \[ T_{\mathcal{G}_\dR(\mathbb{C})}=\beta^{\ast} T_{\scrs_{K}(\mbb{C})} \oplus \left(\mf{g} \otimes \mathcal{O}_{\mathcal{G}_\dR(\mathbb{C})}\right) \]
    is precisely that induced by the product structure on $\mathsf{X} \times \g(\mathbb{C})$ via \eqref{eq.complex-gdr}. In particular, making the computation before taking the quotient by $\g(\mathbb{Q})$ and forgetting the term $\g(\mathbb{A}_f)/K$, the derivative $d\iota$ of $\iota:\pdr(\mathbb{C}) \to \mathcal{G}_{\dr}(\mathbb{C})$ sends $t \in \mf{g}$, viewed by the above method as a vector field on $U$, to the vector field that above a point $(x,g)$ where $x=g \Fil^{\bullet}_{\mu^{-1}}$ is given by $(\mathrm{Ad}(g)(t) \cdot x, g \cdot t)$. In particular, the projection to the second component is the left invariant vector field associated with $t$. Since this projection is the inverse to the algebraic construction of \eqref{eq.inclusion-lie-algebra-to-vector-fields}, we conclude the two constructions agree.  
\end{proof}

\subsubsection{} \label{subsub:AlgebraicOperatorsDef} It follows from Lemma \ref{Lem.inclusion-of-Lie-algebras}
that the map of \eqref{eq.inclusion-lie-algebra-to-vector-fields} induces a map from the universal enveloping algebra $U(\mf{g})$ into the ring of differential operators on $\mathcal{O}_{\mathcal{P}_{\dR}}$. We refer to the differential operators in the image of $\opn{Sym}(\mf{u}_{\mu})=U(\mf{u}_{\mu})$ as the \emph{algebraic Maass--Shimura operators}. 

\subsubsection{} We now want to compare them to the usual Maass--Shimura operators on nearly holomorphic $C^\infty$-automorphic forms for $G$. 

To that end, we fix $x \in \mathsf{X}$ and let $K_\infty \subset \mathsf{G}(\mathbb{R})$ be the stabilizer of $x$. If we fix a $g_0 \in \mathsf{G}(\mathbb{C})$ such that $g_0 \cdot \Fil_{\mu^{-1}}^{\bullet}=x$, then we obtain an embedding 
\[ \mathsf{G}(\mathbb{Q})\backslash \mathsf{G}(\mathbb{A})/K=  \mathsf{G}(\mathbb{Q})\backslash \mathsf{G}(\mathbb{R}) \times \mathsf{G}(\mathbb{A}^\infty)/K\xrightarrow{(g, g'K)\mapsto (gx, g g_0 , g'K)}\mathcal{P}_\dR(\mathbb{C}) \]
where in writing the coordinates of the map we have viewed $\mathcal{P}_\dR(\mathbb{C})$ as a subset of $\mathcal{G}_\dR(\mathbb{C})$ with the quotient presentation of the latter as in the proof of Lemma \ref{Lem.inclusion-of-Lie-algebras}. Under this embedding, algebraic nearly holomorphic forms (i.e. algebraic functions on $\mathcal{P}_\dR$) pull back to $K_\infty$-finite functions in 
\[ 
C^\infty(\mathsf{G}(\mathbb{Q})\backslash \mathsf{G}(\mathbb{A})/K, \mathbb{C}).
\]
Moreover, since any section of $\mathcal{P}_{\dR}$ is killed by $\mathfrak{u}_{\mu^{-1}}^M \subset \opn{Sym}\mathfrak{u}_{\mu^{-1}}$ for $M \gg 1$, sections of $\mathcal{P}_{\dR}$ pull back to $K_{\infty}$-finite functions which are also $\opn{Ad}(g_0^{-1})(\mathfrak{u}_{\mu^{-1}})$-finite. This motivates the following definition. Recall that the choice $x \in \mathsf{X}$ determines a complex structure on $\mathsf{X}$ as well as a Hodge decomposition 
\[
\mathfrak{g}_{\mbb{C}} = \mathfrak{g}_{\mbb{C}}^{(0, 0)} \oplus \mathfrak{g}_{\mbb{C}}^{(-1, 1)} \oplus \mathfrak{g}_{\mbb{C}}^{(1, -1)};
\]
we have an identification of $\mathfrak{g}_{\mbb{C}}^{(-1, 1)} = \opn{Ad}(g_0^{-1})(\mathfrak{u}_{\mu})$ (resp. $\mathfrak{g}_{\mbb{C}}^{(1, -1)} = \opn{Ad}(g_0^{-1})(\mathfrak{u}_{\mu^{-1}})$) with the holomorphic (resp. anti-holomorphic) tangent bundle on $\mathsf{X}$, see \cite[\S 2]{HarrisPartial}.

\begin{Def}
    We define the space of nearly holomorphic $C^{\infty}$-forms for $\mathsf{G}$ to be the subspace of $K_{\infty}$-finite and $\mathfrak{g}_{\mbb{C}}^{(1, -1)}$-finite elements of $C^\infty(\mathsf{G}(\mathbb{Q})\backslash \mathsf{G}(\mathbb{A})/K, \mathbb{C})$.
\end{Def}

To justify this definition, note that by multiplying by a suitable automorphy factor $j(g, x)$ and evaluating at $1 \in \mathsf{G}(\mbb{A}^{\infty})$, any nearly holomorphic $C^{\infty}$-form 
\[
F \in C^\infty(\mathsf{G}(\mathbb{Q})\backslash \mathsf{G}(\mathbb{A})/K, \mathbb{C})
\]
gives rise to a $C^{\infty}$-function $f \colon \mathsf{X} \to W$ which satisfies $f(\gamma \cdot y) = j(\gamma, y) \cdot f(y)$ for all $\gamma \in \Gamma := \mathsf{G}(\mbb{Q}) \cap K$ and $y \in \mathsf{X}$. Here $W$ is the dual of the finite-dimensional algebraic representation of $\mathcal{M}_{\mu}(\mbb{C})$ generated by the $K_{\infty}$-orbit of $F$. By the discussion above, $f$ is then killed by the anti-holomorphic differential operators of large enough degree. Thus, any nearly holomorphic $C^{\infty}$-form gives rise to a nearly holomorphic $C^{\infty}$-function on $\mathsf{X}$ in the usual sense (see \cite[Section 13]{ShimuraArithmeticity}). Furthermore, if $F$ is pulled back from $\mathcal{P}_{\dR}$, the action of the algebraic Maass--Shimura operators on $F$ corresponds to the usual Maass--Shimura operators on nearly holomorphic $C^{\infty}$-functions on $\mathsf{X}$ (cf., e.g., \cite[\S 2.5]{ZLiu19} in the Siegel case).   

\begin{Rem}
    Strictly speaking, nearly holomorphic $C^{\infty}$-forms as we have defined them may not be automorphic forms, because we have not imposed a growth condition (equivalently, a growth condition on $f$ at the boundary of $\mathsf{X}$).\footnote{It is true, however, that if a section of $\mathcal{P}_{\dR}$ extends to the canonical extension of $\mathcal{P}_{\dR}$ over a toroidal compactification of $\mathscr{S}_K$, then the associated nearly holomorphic $C^{\infty}$-form is an automorphic form (see \cite{Su19}).} Since this is not the main focus of this article, we ignore this additional condition in this section.
\end{Rem}

\subsection{Igusa varieties} \label{Sub:IgusaI} We now introduce Igusa varieties. \textbf{From now on, we will assume that $p>2$, so that we may use the results of \cite{HamacherKim}, \cite{KimCentralLeaves}, \cite{DAddeziovH}.} We will moreover implicitly base change $\scrs_{K}$ and all other objects from $\oeplocp$ to $\oee$.

\subsubsection{} Let us write $\operatorname{Sh}_{K}=\scrs_{K,k_{E}}$ for the special fiber over $k_{E}$, the residue field of $\mathcal{O}_E$. For $x \in \shg(\ovfp)$, we write $A_x$ for the abelian variety up to prime-to-$p$ isogeny over $\ovfp$ corresponding to the image of $x \in \shgv(\ovfp)$ and $\mby_x$ for its associated $p$-divisible group. It is explained\footnote{In \cite{ShankarZhou} they use contravariant Dieudonn\'e theory instead of covariant Dieudonn\'e theory, so their $\mathbb{D}(-)$ is the dual of ours. Nevertheless, a module and its dual have the same tensor space, so we may use their results.} in \cite[Section 6.3]{ShankarZhou}, that there are canonical crystalline tensors $\{s_{\alpha,\mathrm{cris},x}\}$ in $\mathbb{D}(\mby_x)^{\otimes}$, that are invariant under the Frobenius on $\mathbb{D}(\mby_x)[\tfrac{1}{p}]$. They are constructed by choosing a lift $\tilde{x}$ of $x$ to $\spf \zpbr$, and considering the de Rham tensors $\{s_{\alpha,\dr,x}\}$ in $D(\mby_{\tilde{x}})^{\otimes} = \mathbb{D}(\mby_x)^{\otimes}$, but they do not depend on the choice of $\tilde{x}$. It is moreover explained in loc. cit. that there is an isomorphism
\begin{align} \label{Eq:Basis}
    \mathbb{D}(\mby_x) \xrightarrow{\sim} \Lambda \otimes_{\zp} \zpbr 
\end{align}
taking $s_{\alpha,\mathrm{cris},x}$ to $s_{\alpha} \otimes 1$ for all $\alpha \in \mathscr{A}$. Under such an isomorphism, the Frobenius is given by an element $b_x \in G(\qpbreve)$, which is well defined up to $\sigma$-conjugacy by $\mathcal{G}(\zpbreve)$, where $\sigma:G(\qpbreve) \to G(\qpbreve)$ is induced by the Frobenius $\sigma:\qpbreve \to \qpbreve$. We will write $\llbracket b \rrbracket$ for the $\sigma$-conjugacy class of $b$ under $\mathcal{G}(\zpbr)$, and $[b]$ for its $\sigma$-conjugacy class under $G(\qpbr)$.

\subsubsection{Central leaves} Recall from \cite[Section 5.4]{DAddeziovH}, cf. \cite[Corollary 3.3.8]{HamacherKim} that for $b \in G(\qpbreve)$ there are (reduced) locally closed subschemes
\begin{align}
    \cb \subseteq \shgb \subseteq \shg
\end{align}
of $\shg$. The subscheme $\shgb$ is called the \emph{Newton stratum} attached to $[b]$, and the subscheme $\cb \subseteq \shgb$ is called the \emph{central leaf} attached to $\llbracket b \rrbracket$. We note that the natural map $\cb \to \shgb$ is a closed immersion by \cite[Corollary 3.3.8]{HamacherKim} and that the central leaf $\cb$ is smooth and equidimensional by \cite[Corollary 5.3.1]{KimCentralLeaves}.

\subsubsection{} \label{subsub:Generalb} Let $b \in G(\qpbr)$ be an element whose $\sigma$-conjugacy class $[b]$ is contained in $B(G,\mu^{-1})$ and such that
\begin{align}
    (\Lambda \otimes_{\zp} \zpbr, b \circ \mathrm{Id} \otimes \sigma),
\end{align}
is the covariant Dieudonn\'e module of a $p$-divisible group $\mbx_b$ over $\fpbar$. Associated with $\mbx_b$ is the universal cover $\tildex_b$ with automorphism formal group $\Aut(\tildex_b)$, which contains $\Aut(\mbx_b)$ as a closed subgroup. Recall from \cite[Lemma 4.4.4]{DAddeziovH} the closed subgroup
\begin{align}
    \Aut_{G}(\tildex_b) \subset \Aut(\tildex_b),
\end{align}
whose intersection with $\Aut(\mbx_b)$ we denote by $\Aut_{\mathcal{G}}(\mbx_b)$. On $\QPRS$-rings $R$, it follows from \cite[Lemma 4.4.4, Lemma 2.3.10]{DAddeziovH} that there is a commutative diagram
\begin{equation} \label{Eq:DescriptionAutomorphismGroups}
    \begin{tikzcd}
    \Aut_{\mathcal{G}}(\mbx_b)(R) \arrow{r} \arrow[d, "\sim"] & \Aut_{G}(\tildex_b)(R) \arrow[d, "\sim"] \\ 
    \mathcal{G}(A_{\cris}(R))^{\varphi_b=1} \arrow{r} & G(B^+_{\cris}(R))^{\varphi_b=1},
    \end{tikzcd}
\end{equation}
where $\varphi_b$ denotes $\sigma$-conjugation by $b$. The notation $\mathcal{G}(A_{\cris}(R))^{\varphi_b=1}$ should be interpreted as $G(B^+_{\cris}(R))^{\varphi_b=1} \cap \mathcal{G}(A_{\cris}(R))$.

\subsubsection{} \label{subsub:Igb} Now assume that $\mbx_b$ is completely slope divisible. Recall from \cite[Section 5.1]{HamacherKim} the Igusa variety
\begin{align}
    \operatorname{Ig}^b \to \operatorname{Sh}_{K,\fpbar}^{\llbracket b \rrbracket},
\end{align}
which is an $\Aut_{\mathcal{G}}(\mbx_b)$-torsor by \cite[Proposition 5.4.3]{DAddeziovH}. It is explained in \cite[Section 5.1]{HamacherKim} that $\operatorname{Ig}^b$ is a perfect scheme and that the action of $\Aut_{\mathcal{G}}(\mbx_b)$ extends to an action of $\Aut_G(\tildex_b)$. By construction, there is a closed immersion
\begin{align} \label{Eq:InclusionIgusaSiegel}
    \operatorname{Ig}^b\gx \to \operatorname{Ig}^b\gvx \times_{\operatorname{Sh}_{U,\fpbar}\gvx^{\llbracket b \rrbracket}} \operatorname{Sh}_{K,\fpbar}\gx^{\llbracket b \rrbracket}.
\end{align}
Moreover, the Siegel Igusa variety $\operatorname{Ig}^b\gvx \to \operatorname{Sh}_{U,\fpbar}\gvx^{\llbracket b \rrbracket}$ represents the functor sending $x:T \to \operatorname{Sh}_{U,\fpbar}\gvx^{\llbracket b \rrbracket}$ to the set of isomorphisms $\rho:\mby_{T} \xrightarrow{\sim} \mbx_{b,T}$ compatible with the polarizations, see \cite[Definition 4.3.1]{CaraianiScholze}. Furthermore, the closed immersion of perfect schemes \eqref{Eq:InclusionIgusaSiegel} is characterized by the fact that on $\fpbar$-points it consists precisely of those isomorphisms $\rho:\mby \xrightarrow{\sim} \mbx_{b}$ such that $\mathbb{D}(\rho):\mathbb{D}(\mby) \xrightarrow{\sim} \Lambda_{\zpbr} $ sends $s_{\alpha,\cris,x}$ to $s_{\alpha} \otimes 1$ for all $\alpha \in \mathscr{A}$. 

\subsubsection{The product formula} Let $b$ be as in \S \ref{subsub:Generalb}. We consider the $p$-adic formal scheme $\scrshat_K=\scrshat_K\gx$ which contains a formal subscheme $\left(\scrshat_K\right)^{[b]}$ coming from the Newton stratum $\shgb$; we consider it as a functor on the category $\nilp_{\mathcal{O}_E}$ of nilpotent $\mathcal{O}_E$-algebras. Let us write $\mintfgbmu \to \spf \zpbr$ for the Rapoport--Zink space of Hodge type associated with $(\mathcal{G},b, \mu)$ of \cite[Section 5.3]{HamacherKim}, note that axiom A of loc. cit. holds in the hyperspecial case by work of Kisin \cite{KisinPoints}. The $p$-divisible group $\mbx_b$ defines a tautological point $x_0 \in \mintfgbmu(\ovfp)$.\footnote{The construction of $\mintfgbmu$ in \cite[Section 5.3]{HamacherKim} agrees with the one in \cite[Section 4.10]{PappasRapoportShtukas}, where the diamond $\mintfgbmu^{\diamondsuit}$ is moreover shown to agree with the integral moduli space of local shtukas $\mintgbmu$.}

By \cite[Lemma 5.12]{HamacherKim}, there is an action of $\Aut_G(\tildex_b)$ on $\mintfgbmu$. In fact $\mintfgbmu$ is defined as a closed formal subscheme of the Rapoport--Zink space $\operatorname{RZ}_{\mbx_b}$, see \cite[Theorem 2.16]{RapoportZink}, which is the moduli space on $\nilp_{\zpbr}$ sending $R$ to the set of isomorphism classes of pairs $(X, \Xi)$ where $X$ is a $p$-divisible group over $R$ and $\Xi$ is a quasi-isogeny $\mbx_{b,R/p} \dashrightarrow X \otimes_{R} R/p$. This has a moduli-theoretic action of $\Aut(\tildex_b)$. The precise statement of \cite[Lemma 5.12]{HamacherKim} is that the action of $\Aut(\tildex_b)$ on $\operatorname{RZ}_{\mbx_b}$ restricts to an action of $\Aut_G(\tildex_b)$ on $\mintfgbmu$.

We also consider the closed formal subscheme $\operatorname{RZ}_{\mbx_{b}, \lambda} \subset \operatorname{RZ}_{\mbx_{b}}$ where $\Xi$ is compatible with the polarization on $\mbx_{b}$ up to a scalar. Its diamond $\operatorname{RZ}_{\mbx_{b}, \lambda}^{\diamondsuit}$ is the integral local Shimura variety for the integral local Shimura datum $(\operatorname{GSp}(\Lambda),b,[\mu])$.

\subsubsection{} \label{subsub:FormalLiftingGeneral} Let $\mathfrak{Ig}^b \to \spf \zpbr$ be the Witt-vector lift to $\zpbr$ of the perfect scheme $\operatorname{Ig}^b$ and, by abuse of notation, let $\Aut_G(\tildex_b)$ be the Witt-vector lift to $\zpbr$ of the perfect affine formal scheme $\Aut_G(\tildex_b)$. Recall that the action of $\Aut_{\calG}(\mbx_{b})$ on $\operatorname{Ig}^b$ extends to an action of $\Aut_G(\tildex_b)_{\fpbar}$. It follows formally that the action of $\Aut_{G}(\tildex_b)$ on $\mathfrak{Ig}^b$ extends to an action of $\Aut_G(\tildex_b)$. Indeed, this follows because both functors satisfy $\mathcal{F}(R)=\mathcal{F}(R/p)$ by the rigidity of quasi-isogenies, and thus it is enough to define the action modulo $p$.

\subsubsection{} \label{subsub:ProductFormula} There is a product formula map, see \cite[Definition 4.3.11, Lemma 4.3.12]{CaraianiScholze}, 
\begin{align}
    \mathfrak{Ig}^b\gvx \times \operatorname{RZ}_{\mbx_{b}, \lambda} \to (\scrshat_{U}\gvx)^{[b]}_{\spf \zpbr}
\end{align}
which on $R$-points for $R \in \nilp_{\zpbr}$ sends $((A, \rho), (X, \Xi))$ to the unique polarized abelian variety $B$ over $R$ up to prime-to-$p$ quasi-isogeny equipped with a quasi-isogeny $A \dashrightarrow B$ compatible with polarizations and level structures, which on $p$-divisible groups is given by $A[p^{\infty}] \xrightarrow{\sim} \mbx_{b,R} \xdashrightarrow{\Xi} X_{R}$. One can now consider the following commutative diagram
\begin{equation} \label{Eq:ProductFormulaExistsDiagram}
    \begin{tikzcd}
        \mathfrak{Ig}^b\gx \times \mintfgbmu \arrow{r} \arrow[d, dashed, "\pi_{\infty}"] & \mathfrak{Ig}^b\gvx \times \operatorname{RZ}_{\mbx_{b}, \lambda} \arrow{d} \\
        (\scrshat_{K}\gx)^{[b]}_{\spf \zpbr} \arrow{r} & (\scrshat_{U}\gvx)^{[b]}_{\spf \zpbr}
    \end{tikzcd}
\end{equation}

\subsubsection{} \label{subsub:FormalLiftingGeneralII} It follows from \cite[Theorem 6.8]{HamacherKim} that the dashed arrow in \eqref{Eq:ProductFormulaExistsDiagram} exists, and that the resulting map 
\begin{align} \label{Eq:ProductFormula}
    \pi_{\infty}:\mathfrak{Ig}^b \times \mintfgbmu \to \left(\scrshat_K\right)^{[b]}_{\spf \zpbr}
\end{align}
is $\Aut_G(\tildex_b)$-invariant and a quasi-torsor for $\Aut_G(\tildex_b)$. Moreover, evaluating the perfect special fiber of $\pi_{\infty}$ in $(-, x_0)$, recovers $ \operatorname{Ig}^b \to \mathbf{Sh}_{K,\fpbar}^{\llbracket b \rrbracket}$. In particular, for any choice of point $\tilde{x}: \spf \zpbr \to \mintfgbmu$ lifting $x_0$, we get a lift
\begin{align}
    \mathfrak{Ig}^b \to \left(\scrshat_K\right)^{[b]}_{\spf \zpbr}
\end{align}
lifting $\operatorname{Ig}^b\to \operatorname{Sh}_{K,\fpbar}^{\llbracket b \rrbracket}$. This morphism is a quasi-torsor for the closed subgroup $\Aut_{\calG}(\mbx_{\tilde{x}}) \subset \Aut_G(\tildex_b)$, which is the functor of automorphisms of the lift $\mbx_{\tilde{x}}$ of $\mbx_b$ determined by $\tilde{x}$, consisting of automorphisms that lift automorphisms lying in $\Aut_{\calG}(\mbx_b)$. It can be seen equivalently as the stabilizer inside $\Aut_G(\tildex_b)$ of the point $\tilde{x} \in \mintfgbmu(\spf \zpbr)$. 

\subsubsection{}\label{sss.pdRmapConstruction} The choice of $\tilde{x}$ defines a lift $\mbx_{\tilde{x}}$ of $\mbx_{b}$, and thus an admissible filtration on $\mathbb{D}(\mbx_b)=\Lambda \otimes_{\zp} \zpbr$, see Proposition \ref{prop.filtered-dieudonne-theory}. \textbf{We assume from now on that this filtration is given by $\operatorname{Fil}_{\mu^{-1}}^{\bullet} \Lambda_{\zpbr}$}. Therefore, there is a morphism $\mathfrak{Ig}^b\gvx \to \mathcal{P}_{\dr,\zpbr}\gvx$ taking an isomorphism $\rho$ to the induced isomorphism $D(\rho):D(\mby_{T}) \xrightarrow{\sim} D(\mbx_{\tilde{x},T})=\Lambda_T$ of filtered modules. 
\begin{Prop} \label{Prop:MapToPdr}
The dashed arrow exists in the following commutative diagram
\begin{equation}
    \begin{tikzcd}
        \mathfrak{Ig}^b\gx \arrow{r} \arrow[d, dashed] &  \mathfrak{Ig}^b\gvx \arrow{d} \\
        \mathcal{P}_{\dr,\zpbr}\gx \arrow{r} & \mathcal{P}_{\dr,\zpbr}\gvx.
    \end{tikzcd}
\end{equation}
\end{Prop}
\begin{Lem} \label{Lem:CommutativeAlgebra}
Consider a commutative diagram
\begin{equation}
\begin{tikzcd}
    A & B \arrow[l, twoheadrightarrow]  \\
    D \arrow[u, dashed] & C \arrow[l, twoheadrightarrow] \arrow{u}
\end{tikzcd}
\end{equation}
of $p$-adically complete and $p$-torsion free $\zp$-algebras with surjective horizontal maps. If the natural evaluation map $A \to \prod_{f:A \to \zpbr} \zpbr$ is injective and if the dashed arrow exists on $\zpbr$-points, then the dashed arrow exists.
\end{Lem}
\begin{proof}
This follows from a straightforward diagram chase on the commutative cube formed by the diagram and the evaluation maps to the corresponding product of copies of $\zpbr$ for each term.
\end{proof}
\begin{proof}[Proof of Proposition \ref{Prop:MapToPdr}]
The horizontal arrows are closed immersions by construction, and because the morphism $\scrs_{K}\gx \to \scrs_{U}\gvx$ is, see \cite[Theorem 1.1.1 ]{XuNormalization}; thus the dashed arrow is unique if it exists. By Lemma \ref{Lem:CommutativeAlgebra}, it suffices to show that the dashed arrow exists (uniquely) on the level of $\zpbr$-points (the hypothesis of the lemma follows from the fact that $\fpbar$-points are dense in $\operatorname{Ig}^b$, being an inverse limit of perfections of finite type $\fpbar$-algebras). In other words, we are trying to show that for $z:\spf \zpbr \to \scrs_{K}\gx$ and an isomorphism $\rho:\mby_{z} \xrightarrow{\sim} \mbx_{\tilde{x}}$ with $(z, \rho) \in \mathfrak{Ig}^b\gx(\spf \zpbr)$, the induced morphism $D(\rho)$ is tensor preserving. But it follows from the construction of $\operatorname{Ig}^b$ that this is true for $D(\rho)=\mathbb{D}(\rho_{\fpbar})$, see \S \ref{subsub:Igb}; here we use the compatibility of the crystalline tensors and the de Rham tensors.
\end{proof}

\subsubsection{} There is a natural map $\Aut(\mbx_{{\tilde{x}}}) \to \operatorname{GL}(\Lambda_{\zpbr})$ which sends an automorphism of the $p$-divisible group $\mbx_{{\tilde{x}}}$ to the induced automorphism of the module $D(\mbx_{{\tilde{x}}})=\Lambda_{\zpbr}$. 
\begin{Lem} \label{Lem:TensorPreservingI}
The subgroup $\Aut_{\mathcal{G}}(\mbx_{{\tilde{x}}}) \subset \Aut(\mbx_{{\tilde{x}}})$ is mapped to the subgroup $\mathcal{P}_{\mu^{-1}} \subset \mathcal{G}_{\zpbr} \subset \operatorname{GL}(\Lambda_{\zpbr})$.
\end{Lem}
\begin{proof}
Consider the following diagram, where the horizontal arrows are the product of the identity map with the natural action map and the right vertical arrow is the product of the map given by Proposition \ref{Prop:MapToPdr} with itself:
\begin{equation}
    \begin{tikzcd}
        \mathfrak{Ig}^b\gx \times_{\Spf \zpbr} \Aut_{\mathcal{G}}(\mbx_{{\tilde{x}}}) \arrow{r}{\sim} \arrow[d, dashed] & \mathfrak{Ig}^b\gx \times_{\scrshat_{K,\zpbr}} \mathfrak{Ig}^b\gx \arrow{d} \\
        \mathcal{P}_{\dr,\zpbr} \times_{\Spf \zpbr} \mathcal{P}_{\mu^{-1}} \arrow{r}{\sim} &  \mathcal{P}_{\dr,\zpbr} \times_{\scrshat_{K,\zpbr}} \mathcal{P}_{\dr,\zpbr}.
    \end{tikzcd}
\end{equation}
Composing the top horizontal map, the right vertical map, and the inverse of the bottom horizontal map, we thus fill in the dashed arrow. On the other hand, if we push-out the bottom row from $\mathcal{P}_{\mu^{-1}}$ to $\operatorname{GL}(\Lambda_{\zpbr})$, then the same composition is given by the product of the map given by Proposition \ref{Prop:MapToPdr} and the natural map $\Aut_{\mathcal{G}}(\mbx_{{\tilde{x}}}) \rightarrow \Aut(\mbx_{{\tilde{x}}}) \to \operatorname{GL}(\Lambda_{\zpbr})$. Thus the lemma holds after base change along $\mathfrak{Ig}^b\gx \to \spf \zpbr$, and since this is a faithfully flat map, we are done. 
\end{proof}
\begin{Rem}
We cannot prove Lemma \ref{Lem:TensorPreservingI} in the same way that we proved Proposition \ref{Prop:MapToPdr} because $\Aut_{\mathcal{G}}(\mbx_b)$ does not generally have a schematically dense set of $\fpbar$-points.
\end{Rem}

\subsubsection{} \label{subsub:UniformizationPdr} Recall that $X$ denotes the universal $p$-divisible group over $\operatorname{RZ}_{\mbx_{b}, \lambda}$. Let $\mathcal{P}_{\dr}^{\mathrm{loc}}(G_V) \to \operatorname{RZ}_{\mbx_{b}, \lambda}$ be the torsor of trivializations $D(X) \xrightarrow{\sim} \Lambda_{\zpbr}$ matching the Hodge filtration on $D(X)$ with $\operatorname{Fil}^{\bullet}_{\mu^{-1},\zpbr}$ and compatible with the polarization up to a scalar.
\begin{Lem} \label{Lem:UniformizationPdrI}
There is a natural isomorphism $\mf{Ig}^b\gvx \times_{\spf \zpbr} \mathcal{P}_{\dr}^{\mathrm{loc}}(G_V) \xrightarrow{\mathcal{I}} \pi_{\infty}\gvx^{\ast} \mathcal{P}_{\dr,\zpbr}\gvx$ over $\mf{Ig}^b\gvx \times_{\spf \zpbr} \operatorname{RZ}_{\mbx_{b}, \lambda}$.
\end{Lem}
\begin{proof}
    It follows directly from the moduli description of $\pi_{\infty}\gvx$ that there is a canonical isomorphism of $p$-divisible groups $\pi_{\infty}\gvx^{\ast} \mby \xrightarrow{\sim} \mf{Ig}^b\gvx \times_{\spf \zpbr} X$ compatible with the polarizations up to a scalar. This canonical isomorphism induces the desired isomorphism $\mathcal{I}$.
\end{proof}

\subsubsection{} \label{subsub:MapToPdrUniformization} There is a $\zpbr$-point $\xi^{\mathrm{can},\dR}$ of $\mathcal{P}_{\dr}^{\mathrm{loc}}(G_V)$ corresponding to $\mbx_{\tilde{x}}$ together with the canonical identification $D(\mbx_{\tilde{x}})=\Lambda_{\zpbr}$. The map $\mf{Ig}^b\gvx \to \mathcal{P}_{\dr,\zpbr}\gvx$ constructed in \S\ref{sss.pdRmapConstruction} is then given by
\begin{align}
    \mf{Ig}^b\gvx \xrightarrow{1 \times \xi^{\mathrm{can},\mathrm{dr}}} \mf{Ig}^b\gvx \times \mathcal{P}_{\dr}^{\mathrm{loc}}(G_V) \xrightarrow{\mathcal{I}} \mathcal{P}_{\dr,\zpbr}\gvx.
\end{align}

\subsubsection{} There are de Rham tensors $\{s_{\alpha, \dr, \mathcal{M}}\}_{\alpha \in \mathscr{A}}$ in $D(X)^{\otimes}$ over $\mintfgbmu$ such that the closed subscheme $\mathcal{P}_{\dr}^{\mathrm{loc}}(G) \subset \mathcal{P}_{\dr}^{\mathrm{loc}}(G_V) \times_{\operatorname{RZ}_{\mbx_{b}, \lambda}} \mintfgbmu$ defined by those trivializations that match $s_{\alpha, \dr, \mathcal{M}}$ with the tensors $s_{\alpha} \otimes 1$ is a torsor for $\mathcal{P}_{\mu^{-1}}$, see \cite[Remark 2.3.5.(b)]{HowardPappasRZ}.\footnote{The Rapoport--Zink space of Hodge type considered in \cite{HowardPappasRZ} agrees with $\mintfgbmu$ because its construction agrees with the construction in \cite[Section 4.3]{HamacherKim}.} We consider the commutative diagram
\begin{equation} \label{Eq:CommutativeDiagramUniformizationPdr}
    \begin{tikzcd}
        \mf{Ig}^b\gvx \arrow{r}{1 \times \xi^{\mathrm{can},\dR}} & \mf{Ig}^b\gvx \times \mathcal{P}_{\dr}^{\mathrm{loc}}(G_V) \arrow{r}{\mathcal{I}} & \mathcal{P}_{\dr,\zpbr}\gvx \\
        \mf{Ig}^b\gx \arrow[u, hook] \arrow[r, dashed] & \mf{Ig}^b\gx \times \mathcal{P}_{\dr}^{\mathrm{loc}}(G) \arrow[r, dashed] \arrow[u, hook]& \mathcal{P}_{\dr, \zpbr}\gx \arrow[u, hook].
    \end{tikzcd}
\end{equation}
\begin{Prop} \label{Prop:UniformizationPdrII}
The dashed arrows in \eqref{Eq:CommutativeDiagramUniformizationPdr} exist (uniquely).
\end{Prop}
\begin{proof}
The rightmost dashed arrow exists because the de Rham tensors $\{s_{\alpha,\dr}\}$ on $\pi_{\infty}\gvx^{\ast} \mby $ are matched with the de Rham tensors $\{s_{\alpha, \dr, \mathcal{M}}\}$ under the canonical isomorphism described in the proof of Lemma \ref{Lem:UniformizationPdrI}, see \cite[Proposition 5.12.(2)]{HamacherProduct}.\footnote{The product formula map constructed in \cite[Proposition 5.12.(2)]{HamacherProduct} agrees with the one constructed in \cite[Theorem 6.8]{HamacherKim} because it is also the unique dashed arrow making \eqref{Eq:ProductFormulaExistsDiagram} commute.} The leftmost dashed arrow exists by Proposition \ref{Prop:MapToPdr}.
\end{proof}

\begin{Rem} Strictly speaking the de Rham tensors on $\mintfgbmu$ are only defined on a restricted class of test objects, namely those objects of $\operatorname{Nilp}_{\zpbr}$ that are formally smooth and formally of finite type over $\zpbr/p^n$ for some $n$. Now $\mintfgbmu \times_{\spf \zpbr} \spec \zpbr / p^n$ is Zariski locally of this form for all $n$, which then allows us to define $\mathcal{P}_{\mathrm{dr}}^{\mathrm{loc}}(G)$. To check that the product formula preserves these tensors, we need to check this on test objects that are products of objects in the restricted class above with formal spectra of Witt vector lifts of perfect formal schemes, and this is allowed by \cite[Proposition 5.12.(2)]{HamacherProduct}.

\end{Rem}

\subsection{Comparing Igusa varieties} \label{sub:ComparisonIgusa} The goal of this section is to compare the construction of Igusa varieties of Hamacher--Kim \cite{HamacherKim} used above, with other constructions of Igusa varieties in the literature. We moreover show that the product formula map $\pi_{\infty}$ of Hamacher--Kim introduced above is compatible with the product formula map coming from the Igusa stacks of \cite{DvHKZIgusaStacks}. 

\begin{Rem}
We now briefly comment on why we use both perspectives on Igusa varieties in this paper. The techniques of Hamacher--Kim are used to show in \cite{DAddeziovH} that $\operatorname{Ig}^b \to \mathbf{Sh}_{K,\fpbar}^{\llbracket b \rrbracket}$ is a torsor for $\Aut_{\mathcal{G}}(\mbx_b)$. We can only recover this statement up to perfection from the Igusa stacks perspective, and the perfection of $\Aut_{\mathcal{G}}(\mbx_b)^{\circ}$ is trivial, so this loses a lot of information. On the other hand, the perspective coming from Igusa stacks makes passage to the adic generic fiber much easier to understand.
\end{Rem}

\subsubsection{} Recall the perfect Igusa variety $\operatorname{Ig}_{K^p}^b\gx \to \operatorname{Sh}_{K,\fpbar}^{\llbracket b \rrbracket,\mathrm{perf}}$ of \cite[Section 4.3.4]{DvHKZIgusaStacks}. It is explained in \cite[Section 2.18, before the proof of Proposition 2.20]{HamacherKimII} that the Igusa variety $\operatorname{Ig}^b \to \operatorname{Sh}_{K,\fpbar}^{\llbracket b \rrbracket,\mathrm{perf}}$ can be identified with $\operatorname{Ig}_{K^p}^b\gx \to \operatorname{Sh}_{K,\fpbar}^{\llbracket b \rrbracket,\mathrm{perf}}$. More precisely, let $\shtglocmu$ be as in \cite[Section 3.1.6]{DvHKZIgusaStacks}. Then both $\operatorname{Ig}^b$ and $\operatorname{Ig}_{K^p}^b\gx$ can be written as the fiber product of a map $\operatorname{Sh}_{K,\fpbar}^{\llbracket b \rrbracket,\mathrm{perf}} \to \shtglocmu$ with the map $\spec \fpbar \to \shtglocmu$ determined by $b$, see \cite[Remark 3.1.8]{DvHKZIgusaStacks}. The two maps $\operatorname{Sh}_{K,\fpbar}^{\llbracket b \rrbracket,\mathrm{perf}} \to \shtglocmu$ moreover agree by \cite[Proposition 5.3.3]{DvHKZIgusaStacks}.  

\subsubsection{}  Recall that associated with $b$ we have the v-sheaf $\widetilde{G}_{b}$ over $\spd \fpbar$, see \S \ref{subsub:BunGII}.
\begin{Lem} \label{Lem:DieudonneBundles}
    There is a natural isomorphism $\Aut_{G}(\tildex_b)^{\diamondsuit} \xrightarrow{\sim} \widetilde{G}_{b}$. 
\end{Lem}
\begin{proof}
In the case that $G=\operatorname{GL}(V)$ this is \cite[Corollary 9.46]{ZhangThesis}. On $(R,R^+)$-points for a perfectoid Huber pair $(R,R^+)$ over $\fpbar$, this sends a quasi-isogeny $\mbx_{b,R^+/\varpi} \dashrightarrow \mbx_{b,R^+/\varpi}$ first to the induced isomorphism $\mathbb{D}(\mbx_{b,R^+/\varpi})[\tfrac{1}{p}] \to \mathbb{D}(\mbx_{b,R^+/\varpi})[\tfrac{1}{p}]$, and then to the induced isomorphism of vector bundles on the relative Fargues--Fontaine curve $X_{(R,R^+)}$, see \cite[Section 2.5.6]{DvHKZIgusaStacks}. It is straightforward to check from the definition and \eqref{Eq:DescriptionAutomorphismGroups}, that the construction described above identifies the subgroups $\Aut_{G}(\tildex_b)^{\diamondsuit}$ and $\widetilde{G}_{b}$ of $\Aut(\tildex_b)^{\diamondsuit}$.
\end{proof}

\subsubsection{} Recall the v-sheaf Igusa variety $\operatorname{Ig}_{K^p}^{b,\mathrm{v}} \to \operatorname{Sh}_{K,\fpbar}^{\diamond}$ of \cite[First paragraph of Section 4.4]{DvHKZIgusaStacks}. Recall the inclusion $\operatorname{Ig}^{b,\diamond} \to \operatorname{Ig}^{b,\mathrm{v}}_{K^p}$ over 
$\operatorname{Sh}_{K,\fpbar}^{\diamond}$ of \cite[Corollary 4.4.3]{DvHKZIgusaStacks}, which identifies $\operatorname{Ig}^{b,\mathrm{v}}_{K^p}$ with the relative canonical compactification of $\operatorname{Ig}^{b,\diamond} \to \operatorname{Sh}_{K,\fpbar}^{\diamond}$. It is explained in \cite[Lemma 8.5.2]{DvHKZIgusaStacks} that the action of $\widetilde{G}_b$ on $\operatorname{Ig}^{b,\mathrm{v}}_{K^p}$ restricts to an action of $\widetilde{G}_b$ on $\operatorname{Ig}^{b,\diamond}$. Under the isomorphism of Lemma \ref{Lem:DieudonneBundles}, this agrees with the action of $\Aut_{G}(\tildex_b)^{\diamondsuit}$. Indeed, this can be reduced to the Siegel case, where it is \cite[Corollary 11.26]{ZhangThesis}.

\subsubsection{} The diamond $\mintfgbmu^{\diamondsuit}$ is isomorphic to the integral moduli space of shtukas $\mintgbmu$, see \cite[Theorem 2.5.4]{PappasRapoportRZSpaces}. More precisely, they are both the same closed subdiamonds of $\operatorname{RZ}_{\mbx_{b}}^{\diamondsuit} \xrightarrow{\sim} \mathcal{M}_{\operatorname{GL}_{\Lambda},b,[\mu]}^{\mathrm{int}}$, where the isomorphism is \cite[Corollary 25.1.3]{ScholzeWeinsteinBerkeley}. 

\subsubsection{} \label{subsub:CompatibilityProductFormula} Recall that $\mathfrak{Ig}^{b,\diamondsuit} = \operatorname{Ig}^{b,\diamond} \times_{\spd \fpbar} \spd \zpbr$ since $\operatorname{Ig}^b$ is a perfect scheme. It follows from the construction of $\pi_{\infty}$ (namely by reduction to the Siegel case), that $\pi_{\infty}^{\diamondsuit}$ agrees with the product formula map induced by \cite[Corollary 4.5.3]{DvHKZIgusaStacks} given by
\begin{align}
    \operatorname{Ig}^{b,\mathrm{v}}_{\spd \zpbr} \times \mintgbmu \to \scrshat_{K}^{\diamondsuit},
\end{align}
by restricting along
\begin{align}
    \mathfrak{Ig}^{b,\diamondsuit} \times \mintfgbmu^{\diamondsuit}  \to \operatorname{Ig}^{b,\mathrm{v}}_{\spd \zpbr} \times \mintgbmu. 
\end{align}
Indeed, in the Siegel case this follows from \cite[Corollary 11.26]{ZhangThesis}. Here we are using \cite[Theorem VII]{DvHKZIgusaStacks} that \cite[Conjecture 4.2.1]{DvHKZIgusaStacks} holds, which is required to invoke \cite[Corollary 4.5.3]{DvHKZIgusaStacks}. 

\subsection{Igusa varieties in the \texorpdfstring{$\mu$}{mu}-ordinary case} \label{sub:MuOrdinaryIgusa} We now specialize to the $\mu$-ordinary case. 

\subsubsection{The \texorpdfstring{$\mu$}{mu}-ordinary locus} \label{subsub:FramingObject} Recall that our choice of place $v$ of $\E$ gives us a conjugacy class of Hodge cocharacters $[\mu]$ over $E$. Recall from \S \ref{sss.finite-extension-cocharacter-choice-existence} the fixed representative $\mu:\mathbb{G}_{m,\mathcal{O}_{E}} \to \calG_{\mathcal{O}_{E}}$. Let $b_{\mu}=\mu(p^{-1}) \in G(E)$ be the $\mu$-ordinary element. The special fiber $\operatorname{Sh}_K:=\left(\scrs_K\right)_{k_E}$ has a dense open $\mu$-ordinary locus, see
\cite[Main theorem]{Wortmann} (or \cite[Theorem 3]{KMPS} for a published reference), which we will write as $\operatorname{Sh}_K^{[b_{\mu}]}$. In this case, there is an equality $\operatorname{Sh}_K^{[b_{\mu}]}=\operatorname{Sh}_K^{[[b_{\mu}]]}$, i.e., the $\mu$-ordinary locus is a central leaf. We will write $\mbx=\mbx_{b_{\mu}}$. 

\subsubsection{} \label{subsub:IgM} We remark that $\mbx$ is completely slope divisible because the Newton cocharacter $\nu_{b_{\mu}}=-\overline{\mu}$ is integral, see \cite[Definition 2.4.1, discussion after Definition 2.4.2]{KimLeaves}. It follows from this that the pullback of $\mby$ to $\operatorname{Sh}_K^{[b_{\mu}]}$ is completely slope divisible and thus admits a slope filtration $\operatorname{Fil}^{\bullet} \mby_{\operatorname{Sh}_K^{[b_{\mu}]}}$, see \cite[Corollary 2.2]{OortZink}. Let us write $\igcs \to \shg^{[b_{\mu}]}$ for the induced perfect Igusa variety, which is defined over $k_{E}$ since $b_{\mu} \in G(E)$ hence $\mbx$ is defined over $k_{E}$. 

\subsubsection{} Let $\Aut_{\calG}(\mbx)^{\mathrm{slp}} \subset \Aut_{\calG}(\mbx)$ be the closed subgroup of automorphisms that preserve the slope decomposition of $\mbx$. The natural map
\begin{align}
    \Aut_{\calG}(\mbx)^{\mathrm{slp}} \to \pi_0(\Aut_{\calG}(\mbx))
\end{align}
is a bijection, giving a semidirect product decomposition
\begin{align}
    \Aut_{\calG}(\mbx) = \Aut_{\calG}(\mbx)^{\circ} \rtimes \Aut_{\calG}(\mbx)^{\mathrm{slp}},
\end{align}
see \cite[Section 4.1.8, Lemma 4.4.4]{DAddeziovH}. We note that $\Aut_{\calG}(\mbx)^{\circ}(\fpbar)=\{1\}$, and that $\Aut_{\calG}(\mbx)^{\mathrm{slp}}_{\fpbar}$ is simply the profinite group scheme associated with the profinite group $\Aut_{\calG}(\mbx)(\fpbar)=G_{b_{\mu}}(\qp) \cap \mathcal{G}(\zpbr)=\mathcal{M}_{\overline{\mu}}(\zp)$, see \cite[Lemma 4.4.4]{DAddeziovH}. In fact, the action of $\mathcal{M}_{\overline{\mu}}(\zp)$ on $\mbx$ is already defined over $k_{E}$; this can be readily checked on the level of Dieudonn\'e modules. This shows that $\Aut_{\calG}(\mbx)^{\mathrm{slp}}(k_{E})=\Aut_{\calG}(\mbx)^{\mathrm{slp}}(\ovfp)$ and hence $\Aut_{\calG}(\mbx)^{\mathrm{slp}} \simeq \ul{\mathcal{M}_{\overline{\mu}}(\zp)}$.

\subsubsection{} We recall the Mantovan Igusa variety $\igm \to \operatorname{Sh}_K^{[b_{\mu}]}$, see \cite[Section 5.4.1]{DAddeziovH}; it is defined over $k_{E}$ since $b_{\mu} \in G(E)$ so that $\mbx$ is defined over $k_{E}$. There is a closed immersion
\begin{align}
    \igm\gx \to \igm \gvx \times_{\operatorname{Sh}_U\gvx^{[[b_{\mu}]]}}\operatorname{Sh}_K\gx^{[b_{\mu}]} ,
\end{align}
and $\igm \gvx \to \operatorname{Sh}_U\gvx^{[[b_{\mu}]]}$ is the moduli space of trivializations $\operatorname{gr}^{\bullet}_{\operatorname{slp}} \mby \to \operatorname{gr}^{\bullet}_{\operatorname{slp}} \mbx_{\igm}$ of associated graded quotients of the slope filtrations, compatible with the polarizations. It follows from the discussion in loc. cit. that $\igm \to \operatorname{Sh}_K^{[b_{\mu}]}$ is a torsor for $\ul{\mathcal{M}_{\overline{\mu}}(\zp)}$. There is a natural map
\begin{align}
    \igcs \to \igm
\end{align}
over $\operatorname{Sh}_K^{[b_{\mu}]}$, which is equivariant for the natural map
\begin{align}
    \Aut_{\calG}(\mbx) \to \ul{\mathcal{M}_{\overline{\mu}}(\zp)}
\end{align}
sending an automorphism of $\mbx$ to the induced automorphism of $\operatorname{gr}^{\bullet}_{\operatorname{slp}}$. 

\subsubsection{Formal lifting} By Grothendieck--Messing theory (see \S\ref{sss.gm-theory} and Proposition \ref{prop.filtered-dieudonne-theory}), lifts $\mbx_{\zpbr}$ of $\mbx$ to $\zpbr$ correspond to lifts of the Hodge filtration to $\mathbb{D}(\mbx)$. There is a particular lift $\mbxcan$ corresponding to the filtration induced by the cocharacter $\mu^{-1}$, called the \emph{canonical lift} of $\mbx$. By \cite[Proposition 3.4]{ShankarZhou} the induced map 
\begin{align}
    \xi^{\mathrm{can}}:\spf \zpbr \to \operatorname{RZ}_{\mathbb{X}}
\end{align}
factors through $\mintfgbmu$. Moreover as explained in Section 3.6 of \cite{ShankarZhou}, there is a unique direct sum decomposition
\begin{align} \label{eq:LiftSlopeDecomposition}
    \mbxcan:=\bigoplus_i \mbxcan_i,
\end{align}
lifting the slope decomposition of $\mbx$.
\begin{Prop} \label{Prop:FaithfullyFlat}
The morphism $\igcsf \to \left(\scrshat_K\right)^{[\bmu]}$ corresponding to the canonical lift $\mbxcan$ is representable in faithfully flat morphisms.
\end{Prop}
To prove Proposition \ref{Prop:FaithfullyFlat}, we will need the following lemma. 
\begin{Lem}[Proposition 5.1 of \cite{BhattDirectSummand}] \label{Lem:CompletelyFlat} 
Let $S \to T$ be a morphism of commutative rings where $S$ is Noetherian. Suppose that there is an element $t \in S$ such that both $S$ and $T$ are $t$-adically complete and $t$-torsion free. If $S/t \to T/t$ is (faithfully) flat, then $S \to T$ is (faithfully) flat.
\end{Lem}
\begin{proof}[Proof of Proposition \ref{Prop:FaithfullyFlat}]
Let $W \subset \left(\scrshat_K\right)^{[\bmu]}$ be an affine open formal subscheme, then the base change $\mf{Ig}_{\mathrm{CS},W}$ is an affine formal scheme because $\igcs \to \operatorname{Sh}_K^{[\bmu]}$ is affine. Now we are in the situation of Lemma \ref{Lem:CompletelyFlat} with $S=\mathcal{O}(W)$ and $T=\mathcal{O}(\mf{Ig}_{\mathrm{CS},W})$. Indeed, both source and target are $p$-adically complete and $p$-torsion free, and $S/p \to T/p$ is faithfully flat. Thus we deduce that $S \to T$ is faithfully flat and adic. Therefore $\mf{Ig}_{\mathrm{CS},W} \to W$ is representable in faithfully flat morphisms; the proposition follows. 
\end{proof}
\begin{Cor} \label{Cor:Torsor}
The morphism $\igcsf \to \left(\scrshat_K\right)^{[\bmu]}$ is a torsor for $\Aut_{\mathcal{G}}(\mbxcan)$, and $\Aut_{\mathcal{G}}(\mbxcan) \to \spf \mathcal{O}_E$ is representable in faithfully flat morphisms.
\end{Cor}

\subsubsection{} Let $\ul{\mathcal{M}_{\overline{\mu}}(\zp)} \to \spf \mathcal{O}_E$ be the unique pro-(finite \'etale) lift of the pro-(finite \'etale) group scheme $\ul{\mathcal{M}_{\overline{\mu}}(\zp)}$. We claim that it can be identified with the subgroup $\Aut_{\mathcal{G}}(\mbxcan)^{\mathrm{slp}} \subset \Aut_{\mathcal{G}}(\mbxcan)$ consisting of automorphisms preserving the direct sum decomposition \eqref{eq:LiftSlopeDecomposition}. Indeed, $\Aut_{\mathcal{G}}(\mbxcan)^{\mathrm{slp}}$ is a lift of $\Aut_{\mathcal{G}}(\mbx)^{\operatorname{slp}}=\ul{\mathcal{M}_{\overline{\mu}}(\zp)}$, and it is moreover pro-(finite \'etale) because the map
\begin{align}
    \Aut_{\mathcal{G}}(\mbxcan)^{\mathrm{slp}}(\spf \zpbr) \to \Aut_{\mathcal{G}}(\mbx)^{\mathrm{slp}}(\spec \fpbar) = \Aut_{\mathcal{G}}(\mbx)(\spec \fpbar)
\end{align}
is a bijection by \cite[Theorem 3.5]{ShankarZhou}. 

\subsubsection{} There is a unique pro-(finite \'etale) lift
\begin{align}
    \igmf \to \left(\scrshat_K\right)^{[\bmu]}
\end{align}
of $\igm \to \operatorname{Sh}_K^{[\bmu]}$ coming from the identification of the \'etale sites of $\operatorname{Sh}_K^{[\bmu]}$ and $\left(\scrshat_K\right)^{[\bmu]}$. By the uniqueness and Corollary \ref{Cor:Torsor} we see that it can equivalently be described as the pushout of $\mathfrak{Ig}_\mathrm{CS}$ along the natural map 
\begin{align}
    \Aut_{\mathcal{G}}(\mbxcan) \to \Aut_{\mathcal{G}}(\mbxcan)^{\operatorname{slp}}=\ul{\mathcal{M}_{\overline{\mu}}(\zp)}
\end{align}
sending an automorphism of $\mbxcan$ to the induced automorphism of the associated graded of the filtration induced by \eqref{eq:LiftSlopeDecomposition}, or equivalently as the quotient of $\mathfrak{Ig}_\mathrm{CS}$ by $\Aut_{\mathcal{G}}(\mbxcan)^\circ$. It follows from fpqc descent along $\igcsf \to \left(\scrshat_K\right)^{[\bmu]}$ that the slope filtration of $\mby$ on the special fiber lifts (uniquely) to a filtration of $\mby$ over $\left(\scrshat_K\right)^{[\bmu]}$. Moreover, it follows that the trivialization of $\operatorname{gr}^{\bullet}_{\operatorname{slp}} \mby_{\igm}$ lifts uniquely to a trivialization of $\operatorname{gr}^{\bullet}_{\operatorname{slp}} \mby_{\igmf}$. 

\begin{Lem} \label{Lem:InverseLimitGroupSchemes}
The $p$-adic formal group $\Aut_{\mathcal{G}}(\mbxcan) \to \spf \oee$ can be written as $\varprojlim_n U_n$, with $U_n \to \spf \oee$ representable by finite faithfully flat group schemes. 
\end{Lem}
\begin{proof}
Consider the affine finite type\footnote{As in \cite[Section 5.5]{DAddeziovH}, we note that the automorphism group of an affine finite $\oee$-scheme $\spec A$ is a closed subgroup of $\operatorname{GL}_{\oee}(A)$, hence affine and of finite type.} group scheme over $\spec \oee$ given by the automorphism group $\Aut(\mbxcan[p^n])$. The natural map $\Aut_{\mathcal{G}}(\mbxcan) \to \Aut(\mbxcan[p^n])$ corresponds to a morphism $\mathcal{O}(\Aut(\mbxcan[p^n])) \to \mathcal{O}(\Aut_{\mathcal{G}}(\mbxcan))$ which factors through a quotient $\mathcal{O}(\Aut(\mbxcan[p^n])) \to \mathcal{O}(U_n')$ via an injective map $$\mathcal{O}(U_n') \to \mathcal{O}(\Aut_{\mathcal{G}}(\mbxcan)).$$ It follows formally that $U_n' \to \Aut(\mbxcan[p^n])$ is a closed immersion of affine finite type group schemes over $\spec \oee$. Since $\mathcal{O}(U_n')$ is $p$-torsion free by Corollary \ref{Cor:Torsor}, it follows that $U_n' \to \spec \oee$ is faithfully flat. \smallskip 

We now set $U_n=\spf \mathcal{O}(U_n') = U_n' \times_{\spec \oee} \spf \oee$ so that $\mathcal{O}(U_n)$ is the $p$-adic completion of $\mathcal{O}(U_n')$. It follows from the proof of \cite[Lemma 5.5.3]{DAddeziovH} that $U_{n,k_{E}}'$ is a finite group scheme over $k_{E}$ and thus by Nakayama that $U_{n} \to \spf \oee$ is (representable in) finite morphisms. It follows from Lemma \ref{Lem:CompletelyFlat} that $U_n \to \spf \oee$ is representable in faithfully flat morphisms. It remains to show that the natural map $\Aut_{\mathcal{G}}(\mbxcan) \to \varprojlim_n U_n$ is an isomorphism. For this, we note that $\Aut_{\mathcal{G}}(\mbxcan) \subset \Aut(\mbxcan)$ is a closed immersion and that $\Aut(\mbxcan) = \varprojlim_n ( \Aut(\mbxcan[p^n]) \times_{\spec \oee} \spf \oee)$. The lemma now follows formally. 
\end{proof}

\subsection{Adic Igusa varieties} \label{sub:InfiniteLevel} We now describe $\mu$-ordinary Igusa varieties on the generic fiber. 

\subsubsection{} We continue with the $\mu$-ordinary setting of \S \ref{subsub:FramingObject}. We start with the morphisms of formal schemes
\begin{align}
    \igcsf \to \left(\scrshat_K\right)^{[\bmu]} \to \scrshat_K,
\end{align}
which induce maps on adic generic fibers
\begin{align}
    \igcsa \to \left(\scrshat_K\right)^{[\bmu]}_{\eta} \to (\scrshat_{K})_\eta=:\mathbf{Sh}_{K}^{\mathrm{an},\circ}\subseteq \mathbf{Sh}_{K}^{\mathrm{an}}. 
\end{align}

\begin{Lem} \label{Lem:TorsorGenericFiber}
    The map $\igcsa \to \left(\scrshat_K\right)^{[\bmu]}_{\eta}$ is a pro-\'etale torsor for $\left(\Aut_{\mathcal{G}}(\mbxcan)\right)^{\ad}_{\eta}$. 
\end{Lem}
\begin{proof}
It is clearly a quasi-torsor because taking adic generic fibers commutes with fiber products. We note that we can write
\begin{align}
    \igcsf \to \left(\scrshat_K\right)^{[\bmu]}
\end{align}
as an inverse limit of torsors $\mf{Y}_n \to \left(\scrshat_K\right)^{[\bmu]}$ for finite flat group schemes $U_n$, see Lemma \ref{Lem:InverseLimitGroupSchemes}. Since taking adic generic fibers commutes with inverse limits, it thus suffices to check that the rigid fiber of a torsor for a finite faithfully flat $p$-adic formal group $U_n$ gives a torsor for $U_{n,\eta}$; this is explained in \cite[proof of Lemma 4.1.6]{FourierPaper}.
\end{proof}

\subsubsection{} Next, we identify the generic fiber of $\left(\Aut_{\mathcal{G}}(\mbxcan)\right)^{\ad}_{\eta}$. Recall that the map $\igcsf \to \scrshat_K$ is induced from the product formula by the canonical lift $ \xi^{\mathrm{can}}:\spf \mathcal{O}_E \to \mintfgbordmu$. We consider the composition $\xi_{\HT}$ of $\xi^{\mathrm{can}}_{E}$ with the local Hodge--Tate period map $M_{G,b_{\mu},[\mu],\mathcal{G}(\zp)} \to [\grgmu / \ul{\mathcal{G}(\zp)}]$ of level $\mathcal{G}(\zp)$, where $\mgbmu:=\mintgbordmu \times_{\spd \oee} \spd E$. We further consider $\xi_{\HT,V}$ defined as the composition
\begin{align}
  \spd E \xrightarrow{  \xi_{\HT}} [\grgmu / \ul{\mathcal{G}(\zp)}] \to [\operatorname{Gr}_{\operatorname{GL}(V),[\mu], \spd E}/\ul{\operatorname{GL}(\Lambda)(\zp)}].
\end{align}
Using the Beauville--Laszlo map $\left[\grgmu / \ul{\mathcal{G}(\zp)} \right] \to \operatorname{Bun}_{G}$ and its $\operatorname{GL}(V)$ analogue, we get compatible closed subgroups
\begin{align}
    \Aut(\xi_{\HT,V}) \to \widetilde{\operatorname{GL}(V)}_{b_{\mu},\spd E}, \qquad \Aut(\xi_{\HT}) \to \widetilde{G}_{b_{\mu}, \spd E}.
\end{align}
\begin{Lem} \label{Lem:AutomorphismTwoWaysI}
Under the isomorphism $\Aut(\tildex)^{\diamondsuit} \xrightarrow{\sim} \widetilde{\operatorname{GL}(V)}_{b_{\mu}}$ of Lemma \ref{Lem:DieudonneBundles}, the closed subgroup $\Aut(\xi_{\HT,V})$ is identified with $\Aut(\mbxcan)^{\diamondsuit}_{\spd E}$.
\end{Lem}
\begin{proof}
This can be checked on rank one geometric points $(C, \mathcal{O}_C)$. There, it is a consequence of the Scholze--Weinstein classification of $p$-divisible groups in terms of their Hodge--Tate filtered $p$-adic Tate modules, see \cite[Theorem 5.1.4]{ScholzeWeinstein}, and its compatibility with crystalline Dieudonn\'e theory via modifications of vector bundles on the Fargues--Fontaine curve, see \cite[Proposition 5.1.6]{ScholzeWeinstein}. 
\end{proof}
\begin{Lem} \label{Lem:AutomorphismTwoWaysII}
Under the isomorphism $\Aut_{G}(\tildex)^{\diamondsuit} \xrightarrow{\sim} \widetilde{G}_{b_{\mu}}$ of Lemma \ref{Lem:DieudonneBundles}, the closed subgroup $\Aut_{\mathcal{G}}(\mbxcan)_{\spd E}^{\diamondsuit}$ is identified with $\Aut(\xi_{\HT})$. 
\end{Lem}
\begin{proof}
By definition, $\Aut_{\mathcal{G}}(\mbxcan)$ is defined to be the intersection of $\Aut(\mbxcan)$ with $\Aut_G(\tildex)$ inside $\Aut(\tildex)$. Since $\grgmu \to \operatorname{Gr}_{\operatorname{GL}(V),[\mu^{-1}], \spd E}$ is a closed immersion, it follows from the moduli description of the Beauville--Laszlo map in terms of modifications of $G$-bundles, see \cite[Section III.3]{FarguesScholze}, that $\Aut(\xi_{\HT}) \subset \Aut(\xi_{\HT,V})$ is the intersection of $\Aut(\xi_{\HT,V})$ with $\widetilde{G}_{b_{\mu}}$ inside $\Aut(\tildex)^{\diamondsuit}$. We now conclude using Lemma \ref{Lem:AutomorphismTwoWaysI}. 
\end{proof}

\subsubsection{} Next, we reinterpret the map from $\igcsa$ to the Shimura variety in terms of the Hodge--Tate period map. By the discussion in \S \ref{subsub:CompatibilityProductFormula}, we can interpret this using the product formula map $\operatorname{Ig}^{b_{\mu},\mathrm{v}}_{\spd \oee} \times \mintgbordmu \to (\scrshat_K)^{\diamondsuit}=\scrs_K^{\diamond}$ of \cite[Corollary 4.5.3]{DvHKZIgusaStacks}. Passing to the generic fiber, we get a map
\begin{align}
    \operatorname{Ig}^{b_{\mu},\mathrm{v}}_{\spd E} \times \mgbmu \to \mathbf{Sh}_K^{\circ, \diamondsuit},
\end{align}
and the map $\operatorname{Ig}^{b_{\mu},\mathrm{v}}_{\spd E} \to \mathbf{Sh}_K^{\circ, \diamondsuit}$ comes from evaluating this at the $\spd E$ point coming from $\xi^{\mathrm{can}}_{\eta}$. Using the Igusa stack diagram, see \cite[Theorem I, proof of Proposition 4.4.5]{DvHKZIgusaStacks}, we can reinterpret $\operatorname{Ig}^{b_{\mu},\mathrm{v}}_{\spd E} \to \mathbf{Sh}_K^{\circ, \diamondsuit}$ as the fiber product 
\begin{equation}
    \begin{tikzcd}
       \operatorname{Ig}^{b_{\mu},\mathrm{v}}_{K^p,\spd E} \arrow{d} \arrow{r} & \spd E \arrow{d}{\xi_{\HT}} \\
       \mathbf{Sh}_K^{\circ, \diamondsuit} \arrow{r} & \left[\grgmu / \ul{\mathcal{G}(\zp)} \right]. 
    \end{tikzcd}
\end{equation}
We thus see that $\operatorname{Ig}^{b_{\mu},\mathrm{v}}_{K^p,\spd E}$ is a torsor for $\Aut(\xi_{\HT})$ over the closed subdiamond of $\mathbf{Sh}_K^{\circ, \diamondsuit}$ corresponding to the image of $\xi_{\HT}$, where $\Aut(\xi_{\HT})$ acts via its map to $\widetilde{G}_{b_{\mu},\spd E}$. Now $\igcsd \to \operatorname{Ig}^{b_{\mu},v}_{K^p,\spd E}$ is the canonical compactification of $\igcsd$ towards $\mathbf{Sh}_K^{\circ, \diamondsuit}$, see \cite[Corollary 4.4.3]{DvHKZIgusaStacks} and \cite[Corollary 18.8]{EtCohDiam}. By Lemma \ref{Lem:AutomorphismTwoWaysII}, we thus see that $\igcsd \to (\scrshat_K)^{[b_{\mu}], \diamondsuit}_{\spd E}$ is isomorphic to the base change of 
\begin{align}
    (\scrshat_K)^{[b_{\mu}], \diamondsuit}_{\spd E} \to \mathbf{Sh}_K^{\circ, \diamondsuit} \to \left[\grgmu / \ul{\mathcal{G}(\zp)} \right]
\end{align}
along $\xi_{\HT}$. 

\subsubsection{} Recall the infinite level local Shimura variety, which is a $\ul{\mathcal{G}(\zp)}$-torsor $$\minfgbmu \to \mgbmu.$$ Let $C$ be the completion of an algebraic closure of $E$ and let $\eab \subset C$ be the completion of the maximal abelian extension of $E$. We are going to choose an $\eab$-point $x$ of $\minfgbmu $ above $\xi^{\mathrm{can}}_{\eta}$; all our further constructions will depend on this choice. To make this choice compatible with our fixed choice of $\mu$, we recall the local Hodge--Tate period map
\begin{align}
    \pi_{\HT}^{\mathrm{loc}}:\minfgbmu \to \grgmu \xrightarrow{\sim} \operatorname{Fl}_{[\mu]}^{\diamondsuit}. 
\end{align}
Our choice of $\mu$ defines a $\spa E$ point $y$ of $\operatorname{Fl}_{[\mu]}$ given by $\Fil^\bullet_\mu$ (or the parabolic $P_{\mu}$).
\begin{Prop} \label{Prop:CompatibleChoice}
There is a choice of $\eab$-point $x$ of $\mathcal{M}_{G,b_{\mu},[\mu],\infty}$ above $\xi^{\can}_{\eta}$ such that $\pi_{\HT}^{\mathrm{loc}}(x)=y$.
\end{Prop}
\begin{proof}
Let $\mathcal{T} \subset \mathcal{G}$ be a maximal torus such that $\mu$ factors through $\mathcal{T}_{\mathcal{O}_E}$. Then $b_{\mu} \in T(E)$, and we have a morphism $(\mathcal{T}, b_{\mu}, [\mu]) \to (\mathcal{G}, b_{\mu}, [\mu])$ of integral local Shimura data. By \cite[Remark 4.2.3]{PappasRapoportShtukas}, we have that $\xi^{\mathrm{can}}_{k_{E}} \in \minttbmu(\spd k_{E})$ and by the construction of the canonical lift, see \cite[Theorem 6.5]{ShankarZhou} or \cite[Theorem 3.4.5]{KisinZhou}, we in fact have $\xi^{\mathrm{can}} \in \minttbmu(\spd \oee)$. So we have reduced to the case that $T$ is an unramified torus, in which case $\operatorname{Fl}_{[\mu]}(\spd E) = \{y\}$. It thus suffices to show that $\minftbmu(\eab)$ is nonempty. Now $\minftbmu \times_{\spd E} \spd C$ is a locally profinite set noncanonically isomorphic to $\ul{T(\qp)}$, whose Galois action is described by local class field theory for $E$, see e.g. \cite[Proposition 3.15]{GleasonComponents}; it thus certainly has points over $\eab$. 
\end{proof}

\subsubsection{} Fix $E \subset L \subset C$ and let $x$ be an $L$-point of $\minfgbmu$ satisfying the conclusion of Proposition \ref{Prop:CompatibleChoice}, and consider $y=\pi_{\HT}(x)$. Then $\Aut(y)_{L}$ is identified with the stabilizer in $\ul{\mathcal{G}(\zp)}$ of the point $y$, which is just $\ul{\mathcal{P}_{\overline{\mu}}(\zp)} = \mathcal{G}(\zp) \cap P_{\mu}(E)$. Under the natural identification
\begin{align}
    \Aut_{\mathcal{G}}(\mbxcan)^{\diamondsuit}_{\spd L} \xrightarrow{\sim} \Aut(y)_{L} \xrightarrow{\sim} \ul{\mathcal{P}_{\overline{\mu}}(\zp)},
\end{align}
of Lemma \ref{Lem:AutomorphismTwoWaysII}, we remark that $\Aut_{\mathcal{G}}(\mbxcan)^{\circ,\diamondsuit}_{\spd L}$ is identified with $\ul{\mathcal{U}_{\overline{\mu}}(\zp)}$, since $\Aut_{\mathcal{G}}(\mbxcan)^{\circ}$ consists of those automorphisms acting trivially on the associated graded of the slope filtration. 

\subsubsection{} In fact, the morphism $\ul{\mathcal{P}_{\overline{\mu}}(\zp)} \to \widetilde{G}_{b_{\mu},L}$ induced by $y$ extends to a morphism $\ul{P_{\overline{\mu}}(\qp)} \to \widetilde{G}_{b_{\mu},L}$; the stabilizer in $G(\qp)$ of the point $y$. Indeed, this follows from the $G(\qp)$-equivariance of the Beauville--Laszlo map $\operatorname{BL}:\grgmu \to \bun_{G}$, see \cite[Chapter III.3]{FarguesScholze}. 

\subsubsection{} \label{subsub:MapsGenericFiber} Our choice of $x$ induces a left $\ul{P_{\ovmu}(\qp)}$-equivariant map (here and below we consider the $\ul{G(\qp)}$-action on $\mathbf{Sh}_{K^p,L}^{\circ, \diamondsuit}$ as a left action instead of a right action)
\begin{align}
    \igcsdE \to \mathbf{Sh}_{K^p,L}^{\circ, \diamondsuit}
\end{align}
lifting $\igcsdE \to \mathbf{Sh}_{K, L}^{\circ, \diamondsuit}$. Now let $P_{\dR}^{\mathrm{an}}$ be the analytification of $P_{\dr}=\mathcal{P}_{\dr,E}$  and consider the induced map
\begin{align} \label{Eq:InfiniteLevelMapToPdr}
     \igcsdE \to P_{\dr,L,\infty}^{\mathrm{an}}:=\mathbf{Sh}_{K^p,L}^{\circ, \diamondsuit} \times_{ \mathbf{Sh}_{K,L}^{\circ, \diamondsuit}}  P_{\dr,L}^{\mathrm{an}}.
\end{align}
The right hand side has a natural action of $\ul{G(\qp)} \times P_{\mu^{-1}}^{\mathrm{an}}$ (where the action of $G(\mathbb{Q}_p)$ comes from the canonical $G(\mathbb{A}_f)$-equivariant structure on $P_\dR$ over the full tower of Shimura varieties), while the left hand side has a natural action of $\widetilde{G}_{b_{\mu},L} \supset \ul{P_{\overline{\mu}}(\qp)}$. 
\begin{Lem} \label{Lem:TwistedHeckeAction}
The map in \eqref{Eq:InfiniteLevelMapToPdr} is $\ul{P_{\overline{\mu}}(\qp)}$-equivariant, where $\ul{P_{\overline{\mu}}(\qp)}$ acts on the right hand side via the map 
\begin{align}
    \ul{P_{\overline{\mu}}(\qp)} \to \ul{G(\qp)} \times P^{\mathrm{an},\diamondsuit}_{\mu^{-1}}
\end{align}
that is the natural inclusion in the first factor and in the second factor is the composition
\begin{align} \label{Eq:FactorsThrough N}
    \ul{P_{\overline{\mu}}(\qp)} \to \ul{M_{\overline{\mu}}(\qp)} \subset \ul{M_{\mu}(E)} \subset P_{\mu^{-1}}^{\mathrm{an},\diamondsuit}.
\end{align}
\end{Lem}
\begin{proof}\newcommand{\Spd}{\mathrm{Spd}\,} \newcommand{\loc}{\mathrm{loc}}
We first note that the map \eqref{Eq:InfiniteLevelMapToPdr} can be obtained as follows: First, by Proposition \ref{Prop:UniformizationPdrII}, there is a uniformization map
\begin{align} \label{Eq:InfiniteLevelPdRUniformization} \igcsdL \times_{\Spd L} P_{\dr,\loc,\infty,L}^{\diamondsuit} \to P_{\dr,\infty,L}^{\diamondsuit} \end{align}
where $P_{\dr,\loc,\infty,L}$ is the $P_{\mu^{-1}}$-torsor over the infinite level local Shimura variety. We have a canonical point $\ell \in P_{\dr,\loc,\infty,L}(x)$ lying above $x$, and \eqref{Eq:InfiniteLevelMapToPdr} is the composition of \eqref{Eq:InfiniteLevelPdRUniformization} with 
\[ \igcsdL \to \igcsdL \times_{\Spd L} P_{\dr,\loc,\infty,L}^{\diamondsuit},\; i \mapsto (i, \ell), \]
see \S \ref{subsub:MapToPdrUniformization}. Because \eqref{Eq:InfiniteLevelPdRUniformization} is a $\widetilde{G}_{b_{\mu}}$-torsor, for $j$ a section of $\widetilde{G}_{b_{\mu}}$, $(j \cdot i, \ell)$ maps to the same point as $(i, j^{-1} \cdot \ell)$. The map \eqref{Eq:InfiniteLevelMapToPdr} is equivariant for the action of $\ul{G(\mathbb{Q}_p)} \times P_{\mu^{-1}}$ where on the left the action is purely on $P_{\dr,\loc,\infty,L}^{\diamondsuit}$. We are thus reduced to verifying that, if $j$ is a section of $\ul{P_{\overline{\mu}}(\mathbb{Q}_p)} \leq \widetilde{G}_{b_{\mu}}$, then $j^{-1} \cdot_{\widetilde{G}_{b_{\mu}}} \ell = 
\overline{j} \cdot_{P_{\mu^{-1}}} (j \cdot_{\ul{G(\mathbb{Q}_p)}} \ell)$, where $\overline{j}$ is the projection into $M_{\overline{\mu}}$ in the statement. This is a purely local statement about the action of $\widetilde{G}_{b_{\mu}}$ on $P_{\dr,\loc,\infty,L}$, and can be reduced to a statement about the infinite level local Shimura variety for $(\operatorname{GL}_V, b_{\mu}, [\mu])$. \smallskip

In this case, $P_{\dr,\loc,\infty,L}^{\diamondsuit}$ is a moduli space of quadruples $(\mathfrak{G}, \rho, \beta,\alpha,)$ over perfectoid Huber pairs $(R,R^+)$ over $L$, where $(\mathfrak{G}, \rho) \in \operatorname{RZ}_{\mbx}(\spf R^+)$, where $\alpha: T_p \mathfrak{G} \xrightarrow{\sim} \Lambda$ is an isomorphism of $\zp$ local systems, and where $\beta:D(\mathfrak{G}) \xrightarrow{\sim} V_{R}$ is an isomorphism of filtered vector bundles. Indeed, this follows from the moduli interpretation of the infinite level Rapoport--Zink space in \cite[Definition 4.2.3]{CaraianiScholze}. We can change the moduli interpretation, without affecting the representing object, by taking $\mathfrak{G}$ up to isogeny over $R^+$, and by replacing $\beta$ with a rational version. The point $\ell$ then corresponds to the quadruple $$(\mathbb{X}^\can, \tilde{\mathbb{X}}^\can \xrightarrow{\rho_\can}\tilde{\mathbb{X}}, D(\mathbb{X}^\can) \xrightarrow{\beta_\can} V_{\qpbreve}, V_p \mathbb{X}^\can \xrightarrow{x} V_{\qp}),$$  
where $\rho_\can$ and $\beta_\can$ denote the canonical isomorphisms from the construction. Thus, 
\[ j^{-1} \cdot_{\widetilde{G}_{b_{\mu}}} \ell = (\mathbb{X}^\can, \tilde{\mathbb{X}}^\can \xrightarrow{j^{-1} \circ \rho_{\can}}\tilde{\mathbb{X}}, D(\mathbb{X}^\can) \xrightarrow{\beta_\can} V_{\qpbreve}, V_p \mathbb{X}^\can \xrightarrow{x} V_{\mathbb{Q}_p}). \]
Because $j$ lifts to a quasi-isogeny $j^\can$ of $\mathbb{X}^\can$ (in other words, $\rho_{\can} \circ j^\can=j \circ \rho_{\can}$), this point is equivalent to
\begin{align} 
(\mathbb{X}^\can, \tilde{\mathbb{X}}^\can \xrightarrow{j^{-1} \circ \rho_{\can} \circ j^\can}\tilde{\mathbb{X}}, D(\mathbb{X}^\can) \xrightarrow{\beta_\can \circ D(j^\can)} V_{\qpbreve}, V_p \mathbb{X}^\can \xrightarrow{x \circ V_p(j^\can)} V_{\mathbb{Q}_p})\\
=(\mathbb{X}^\can, \tilde{\mathbb{X}}^\can \xrightarrow{\rho_{\can}}\tilde{\mathbb{X}}, D(\mathbb{X}^\can) \xrightarrow{\beta_\can \circ D(j^\can)} V_{\qpbreve}, V_p \mathbb{X}^\can \xrightarrow{x \circ V_p(j^\can)} V_{\mathbb{Q}_p}).\end{align}
Since $j$ is viewed as an element of $P_{\overline{\mu}}(\mathbb{Q}_p)$ (resp. $P_{\mu^{-1}}$) via the action on $V_p \mathbb{X}^\can$  (resp. $D(\mathbb{X}^\can)$) through the given trivializations, this is equal to 
\[ \overline{j} \cdot_{P_{\mu^{-1}}} \left(j \cdot_{\ul{G(\mathbb{Q}_p)}} (\mathbb{X}^\can, \tilde{\mathbb{X}}^\can \xrightarrow{\rho_{\can}}\tilde{\mathbb{X}}, D(\mathbb{X}^\can) \xrightarrow{\beta_\can} V_{\qpbreve}, V_p \mathbb{X}^\can \xrightarrow{x} V_{\mathbb{Q}_p}) \right), \]
as claimed. 
\end{proof}


\section{\texorpdfstring{$p$}{p}-adic differential operators} \label{Sec:DifferentialOperators} In \S \ref{sub:GeneralPAdicDifferentialOperators} and \S \ref{sub:MuOrdinaryOperators}, we prove Theorem \ref{Thm:IntroIntegralAction}, see Theorem \ref{Thm:IdentificationAction} and Theorem \ref{Thm:GlobalAction}. Afterwards, we construct additional algebraic Maass--Shimura operators on $\mathcal{O}(\igmf)$, see \S \ref{sub:ActiononIgm}. In \S \ref{sub:TateModuleH}, we compute the (Hodge--Tate filtered) Tate module of $\mathcal{H}$. In \S \ref{sub:EischenMantovan}, we compare our operators to those of \cite{EischenMantovan}, \cite{EischenFintzenMantovanVarma}.

\subsection{Construction of \texorpdfstring{$p$}{p}-adic differential operators} \label{sub:GeneralPAdicDifferentialOperators}

\subsubsection{}\label{sss.action-setup} We now return to the notation established in \S \ref{Sub:IgusaI}. In particular, we consider a class $[b] \in \bgmu$. We suppose given a representative $b \in [b]$ such that $\Lambda_{\zpbr}$ is a covariant Dieudonn\'{e} module under $b \circ \mathrm{Id} \otimes \sigma$. The associated $p$-divisible group with extra structure $\mbx_b$ over $\fpbar$ gives rise to a central leaf $C_b=\cb \otimes_{k_{E}} \fpbar$ and a perfect Igusa variety $\operatorname{Ig}^b \to C_b$. We suppose given a lift $\mbx_{\tilde{x}}$ of $\mbx_b$ corresponding to $\tilde{x}:\spf \zpbr \to \mintfgbmu$, which is determined under Grothendieck--Messing theory by an admissible filtration $V_0=\Fil^0 \Lambda_{\zpbr}$. We fix a Hodge cocharacter $\mu$ over $\zpbr=W(\fpbar)$ such that $V_0=\Fil^0_{\mu^{-1}}$, which exists by e.g. \cite[Proposition 4.8]{Zhou}. \smallskip

We fix now a vector subgroup $\mc{W} \leq \mc{G}$ such that:
\begin{enumerate}
    \item $\mc{W}_{\qpbreve}$ is preserved by $\mathrm{Ad}(b)$, and $\mc{W}(\zpbreve) \subseteq \mc{W}(\zpbreve)[\tfrac{1}{p}]=\mc{W}(\qpbreve)$ is a covariant Dieudonn\'{e} module under $\mathrm{Ad}(b) \circ \sigma$, i.e., $\mathrm{Ad}(b) \circ \sigma$ on $\mc{W}(\qpbreve)$ induces a map $\mc{W}(\zpbreve) \rightarrow \frac{1}{p} \mc{W}(\zpbreve)$ whose image contains $\mc{W}(\zpbreve)$.
    \item $\Fil^0\mc{W}_{\zpbreve}:=\mc{W}_{\zpbr} \cap \mathcal{P}_{\mu^{-1}}$ is a vector subgroup of $\mc{W}_{\zpbreve}$ such that $\Fil^0\mc{W}_{\zpbr}(\zpbr)$ is an admissible filtration on $\mc{W}(\zpbreve)$. 
\end{enumerate}
By Proposition \ref{prop.filtered-dieudonne-theory}, there is a unique $p$-divisible group $\mathcal{H}/\zpbr$ equipped with an identification $\mbb{D}(\mathcal{H}_{\fpbar})=\mc{W}(\zpbr)$ (as a Dieudonn\'{e} module) with Hodge filtration $\Fil^0(\mbb{D}(\mathcal{H}_{\fpbar}))=\Fil^0(\mc{W}(\zpbr))$. 

\subsubsection{} \label{subsub:LieHGeneral} There is a canonical identification $\Lie \mathcal{H}=\gr^{-1}(\mc{W}_{\zpbreve})$, see Proposition \ref{prop.filtered-dieudonne-theory}. We will use a modified version of this latter canonical identification: We write 
\[ \mf{w}:=\Lie \mc{W} \subseteq \Lie \mc{G}=:\mf{g},\]
and $\Fil^0(\mf{w}_{\zpbr}):=\mf{w} \cap \Fil_{\mu^{-1}}^0(\mf{g}_{\zpbr})$. A vector group is naturally isomorphic to its Lie algebra (by sending an $R$-point $v$ to the $R[\varepsilon]/\varepsilon^2$-point $\varepsilon v$), and for $\mc{W}$ this construction is compatible with the filtrations. Thus we can identify $\gr^{-1}(\mc{W}_{\zpbreve})$ with $\gr^{-1}(\mf{w}_{\zpbr})$. In particular, we have canonical identifications
\begin{equation}\label{eq.canonical-idents-LieH} \Lie \mathcal{H}=\gr^{-1}(\mc{W}_{\zpbr})=\gr^{-1}(\mf{w}_{\zpbreve}).\end{equation}

\begin{Rem} \label{Rem:FakeLogarithm} The identification $\mathcal{W} \to \mathfrak{w}$ above identifies the filtered Dieudonn\'{e} module $(\mathcal{W}(\zpbreve), \Fil^0\mathcal{W}(\zpbreve))$ with $(\mf{w}_{\zpbreve}, \Fil^0(\mf{w}_{\zpbreve}))$, where the Frobenius on $\mf{w}_{\qpbreve}$ is induced by the Frobenius $\mathrm{Ad}(b) \circ \mathrm{Id}\otimes\sigma$ on $\mf{g}_{\qpbreve}$. 
\end{Rem}

\subsubsection{}\label{sss.map-construction}
We now construct a map $T_p \mc{H} \rightarrow \Aut_\mc{G}(\mbx_{\tilde{x}})$. By Lemma \ref{lemma.cover-and-tate-module-rep-properties}, $T_p \mc{H}=\Spf R$ where $R$ is $p$-adically complete and $R/p \in \QPRS$, so it suffices to give the map on $R \in \nilp$ with $R/p \in \QPRS$. For such an $R$, writing $\varphi=\mathrm{Ad}(b) \circ \sigma$, we may identify
\[ \Aut_{\mc{G}}(\mbx_{\tilde{x}})(R)=\mathcal{G}(A_\cris(R/p))^{\varphi=1} \times_{\mathcal{G}(R)} \mathcal{P}_{\mu^{-1}}(R)\]
using \eqref{Eq:DescriptionAutomorphismGroups} together with Proposition \ref{prop.filtered-dieudonne-theory}. We similarly identify 
\[ T_p \mc{H}(R)=\Hom((\mathbb{Q}_p/\mathbb{Z}_p)_R, \mathcal{H}_R)=\mc{W}(A_{\cris}(R/p))^{\varphi=1} \times_{\mc{W}(R)} \Fil^0(\mc{W}(R)), \]
so that we may take the map $T_p \mc{H}(R) \rightarrow \Aut_{\mc{G}}(\mbx_{\tilde{x}})(R)$ induced by $\mc{W} \hookrightarrow \mc{G}$. 

We now extend this to a map $\tildeh \rightarrow \Aut_G(\tildex)$: Since $\tildeh=\bigcup_k \frac{1}{p^k} T_p\mc{H}$, it suffices to describe compatible maps from $\frac{1}{p^k} T_p \mc{H}$ as $k$ varies. Identifying $\frac{1}{p^k} T_p \mc{H}$ with $T_p \mc{H}$ by multiplication by $p^k$, it suffices to describe such maps on $R$ as above. For such an $R$, noting that $\Aut_G(\tildex)(R)=G(B^+_\cris(R/p))^{\varphi=1}$, we take the map
\[ T_p \mc{H}(R) \subseteq \mc{W}(A_{\cris}(R/p))^{\varphi=1} \xrightarrow{\frac{1}{p^k}} \mc{W}(B^+_\cris(R/p)) \subseteq G(B^+_\cris(R/p))^{\varphi=1}. \]

We write this assignment $\tildeh \rightarrow \Aut_G(\tildex)$ as $\tilde{h}\mapsto \alpha_{\tilde{h}}$. If $R$ is $p$-adically complete and $p$-torsion free with $R/p \in \QPRS$ then, for $D(\alpha_{\tilde{h}})$ as in \S\ref{ss.D-maps} we have $D(\alpha_{\tilde{h}}) \in \mc{W}(R[\tfrac{1}{p}])$. If there is a divided powers ideal $J$ such that 
\[ \tilde{h} \in T_p\mathcal{H}(R/J)\subseteq \tildeh(R/J)=\tildeh(R) \]
then, since $\alpha_{\tilde{h}}$ comes from an isomorphism mod $J$, we find $D(\alpha_{\tilde{h}})$ lies in $\mc{W}(R)$. 

\subsubsection{}\label{sss.action-and-canonical-identification} Suppose now that $t \in \Lie \mathcal{H} \subset \mathcal{H}(\zpbreve[\varepsilon]/\varepsilon^2)$. Then, by taking 
\[ S=\zpbreve[[\varepsilon^{1/p^\infty}, x_1^{1/p^\infty}, \ldots, x_n^{1/p^\infty}]]/(\varepsilon^2, y_1-\varepsilon e_1, \ldots, y_n - \varepsilon e_n) \]
where $y_i$ and $x_i$ are as in \ref{lemma.cover-and-tate-module-rep-properties}.(3) and the $\varepsilon e_i$ are the $y$-coordinates of $t$, we obtain a lifting  $\widetilde{t}$ of $t$ to $\tildeh(S)$. 

Let $R=S/\varepsilon$, so both $S/p$ and $R/p$ are quotients of perfect rings by regular sequences. Because $t$ is zero modulo $\varepsilon$, the image of $\tilde{t}$ via the reduction map $\widetilde{\mathcal{H}}(S) \xrightarrow{\sim} \widetilde{\mathcal{H}}(R)$ is an element of $T_p \mathcal{H}(R)$. In particular since $S$ is $p$-torsion free, by the above discussion, $D(\alpha_{\tilde{t}})$ lies in $\mc{W}(S)$. 

\begin{Lem} \label{Lem:Identification}
The image of $D(\alpha_{\tilde{t}})$ in $\gr^{-1} \mc{W}(S)$ is the element associated to $t$ by the canonical identification $\Lie \mc{H}=\gr^{-1}\mc{W}$.
\end{Lem}
\begin{proof}
Under the canonical identification of $\mc{W}({\zpbr})$ with the Dieudonn\'{e} module of $\mathcal{H}$, by \cite[Lemma 3.5.1]{ScholzeWeinstein}, the map
\begin{equation}\label{eq.qlog} \widetilde{\mathcal{H}}(S)=\mc{W}(B^+_\cris(S/p))^{\varphi=1} \rightarrow \mc{W}(S[\tfrac{1}{p}]) \end{equation}
is the quasi-logarithm. The logarithm on $\mathcal{H}$ can be computed by choosing a lift to $\widetilde{\mathcal{H}}$, applying this map, projecting to $\gr^{-1}\mc{W}(S[\tfrac{1}{p}])$, and applying the canonical identification of $\Lie \mathcal{H}$ with $\gr^{-1}\mc{W}$, see \cite[Remark 3.2.6]{ScholzeWeinstein}. We thus see that the image of $\mathbb{D}(\alpha_{\tilde{t}})$ is the logarithm of $t$, viewed as an element of $\mathcal{H}(\zpbr[\varepsilon]/\varepsilon^2)$, which is $t$ itself. 
\end{proof}

\subsubsection{} We consider the partial uniformization map $u: \mathfrak{Ig}^b \rightarrow \scrshat_{\zpbr}$ associated with $\mbx_{\tilde{x}}$ as in \S \ref{subsub:FormalLiftingGeneral}. Since $T_p \mathcal{H} \subset \Aut_{\mathcal{G}}(\mbx_{\tilde{x}})$ and $u$ is $\Aut_{\mathcal{G}}(\mbx_{\tilde{x}})$-invariant by construction, the map $u$ descends to a map $\overline{u}: \overline{\mathfrak{Ig}}^b:=T_p \mathcal{H} \backslash \mathfrak{Ig}^b  \rightarrow \scrshat_{\zpbr}$, where we are taking the quotient in the fpqc topology. By Proposition \ref{Prop:MapToPdr}, there is a map 
\[ v: \mathfrak{Ig}_{b} \rightarrow \mathcal{P}_{\dr,\zpbr} \]
lying above $u$, which is equivariant for the action of $T_p \mathcal{H}$, where $T_p \mathcal{H}$ acts on $\mathcal{P}_{\dR,\zpbr}$ through the composition of the torsor action of $\mc{P}_{\mu^{-1}}$ and (see Lemma \ref{Lem:TensorPreservingI})
\[ T_p \mathcal{H} \rightarrow \ul{\mathrm{Aut}}_{\mathcal{G}}(\mbx_{\tilde{x}}) \to \mathcal{P}_{\mu^{-1}}.\]
By construction, this action factors through the torsor action of the subgroup 
\[ \Fil^0(\mc{W}_{\zpbr}) = \mc{W}_{\zpbr} \cap \mathcal{P}_{\mu^{-1}} \subseteq \mathcal{P}_{\mu^{-1}}.\]
In particular, if we write 
\[ \overline{\mathcal{P}}_\dR:= \Fil^0(\mc{W}_{\zpbr}) \backslash \mathcal{P}_{\dr,\zpbr}, \]
we obtain a lift $\overline{v}$ of $\overline{u}$ fitting in the commutative diagram

\[\begin{tikzcd}
	{\mathfrak{Ig}^{b}} & {\mathcal{P}_{\dR,\zpbr}} \\
	{\overline{\mathfrak{Ig}}^b} & {\overline{\mathcal{P}}_\dR} \\
	{\scrs_{\zpbr}.}
	\arrow["v", from=1-1, to=1-2]
	\arrow[from=1-1, to=2-1]
	\arrow["u"', shift right=3, curve={height=12pt}, from=1-1, to=3-1]
	\arrow[from=1-2, to=2-2]
	\arrow["{\overline{v}}", from=2-1, to=2-2]
	\arrow["{\overline{u}}"', from=2-1, to=3-1]
	\arrow[from=2-2, to=3-1]
\end{tikzcd}\]
\subsubsection{} \label{subsub:MapToTangentBundle} We also write $\overline{\mathcal{G}}_\dR= \Fil^0(\mc{W}_{\zpbr}) \backslash \mathcal{G}_{\dr,\zpbr} $. Since $\mc{W}$ is abelian, we obtain an induced action of $\gr^{-1}(\mc{W}_{\zpbr})$ on $\overline{\mathcal{G}}_\dR$. Differentiating this action gives a canonical inclusion
\[ \mathrm{gr}^{-1}\mf{w}_{\zpbr} \otimes \mathcal{O}_{\overline{\mathcal{G}}_\dR} \hookrightarrow T_{\overline{\mathcal{G}}_\dR/\scrs_K}. \]
Alternatively, this can be deduced from the identity $T_{\mathcal{G}_\dR/\scrs_K}=\mathfrak{g} \otimes \mathcal{O}_{\mathcal{G}_\dR}$ coming from differentiation: Then $T_{\overline{\mathcal{G}}_\dR/\scrs_K}$ is canonically identified with the descent of $(\mathfrak{g}_{\zpbr}/ \Fil^0\mf{w}_{\zpbreve}) \otimes \mathcal{O}_{\mathcal{G}_\dR}$ for the adjoint action of $\Fil^0(\mc{W}_{\zpbr})$, which is trivial on the sub-module $\mathrm{gr}^{-1}\mf{w}_{\zpbr} \otimes \mathcal{O}_{\mathcal{G}_\dR}$.

Now, because the splitting $T_{\mathcal{G}_\dR}=T_{\mathcal{G}_\dR/\scrs_K} \oplus \rho^*T_{\scrs_K}$ induced by the Gauss--Manin connection, see \S \ref{subsub:Constructions}, is equivariant for the action of $\mathcal{G}$ (and thus of $\Fil^0(\mc{W}_{\zpbr})$), we still obtain an identification $T_{\overline{\mathcal{P}}_\dR}=T_{\overline{\mathcal{G}}_\dR/\scrs_K}$. Therefore, we have a canonical map 
\begin{equation}\label{eq.can-inclusion-tangent-pdr} \mathrm{gr}^{-1}\mf{w}_{\zpbr} \hookrightarrow T_{\overline{\mathcal{P}}_\dR}.\end{equation}    

\subsubsection{} Let $\rho$ denote the projection $\mathcal{G}_\dR \rightarrow \scrs_K$ and let $s$ denote the Gauss--Manin lift of \S \ref{subsub:Constructions}. 
\begin{Lem} \label{Lem:IdentificationII}
    For $w \in \mathrm{gr}^{-1}\mf{w}_{\zpbr}$, the vector field $\partial_w$ on $\overline{\mathcal{P}}_\dR$ induced by \eqref{eq.can-inclusion-tangent-pdr} can be characterized as the unique map $\partial_w: \overline{\mathcal{P}}_\dR[\varepsilon]\rightarrow\overline{\mathcal{P}}_\dR$ such that: Viewed as a map $\overline{\mathcal{P}}_\dR[\varepsilon]\rightarrow\overline{\mathcal{G}}_\dR$ (i.e., a vector field in $\overline{\mathcal{G}}_\dR$ along $\overline{\mathcal{P}}_\dR$), we have
\begin{equation}\label{eq.GM-characterization} \partial_w =  w \cdot s(d\rho(\partial_w)) \end{equation}
Here $w$ acts by viewing it as an element of $\gr^{-1}(\mc{W}_{\zpbreve})(\zpbreve[\varepsilon]/\varepsilon^2) \subseteq \gr^{-1}(\mc{W}_{\zpbreve})(\overline{\mathcal{P}}_\dR[\varepsilon])$.
\end{Lem}
\begin{proof}
    This is tautological after noting that multiplying by $w$, viewed as an element of $\gr^{-1}(\mc{W}_{\zpbreve})(\overline{\mathcal{P}}_\dR[\varepsilon])$, is the same as adding $w$ viewed as an element of the vertical tangent bundle $T_{\overline{\mathcal{G}}_\dR/\scrs_K}$. 
\end{proof}

\subsubsection{} Now, we have an induced action of $\mathcal{H}=\widetilde{\mathcal{H}}/T_p \mathcal{H}$ on $\overline{\mathfrak{Ig}}^b$; we write 
\[a: \mathcal{H} \times_{\Spf \zpbr} \overline{\mathfrak{Ig}}^b \rightarrow \overline{\mathfrak{Ig}}^b\]
for the action map. In particular, differentiating the action gives, for any $t \in \Lie \mathcal{H}$, a vector field $\partial_t$ on $\overline{\mathfrak{Ig}}^b$ expressed as the composition 
\[ \overline{\mathfrak{Ig}}^b[\varepsilon]=\Spf \zpbr[\varepsilon]/\varepsilon^2  \times_{\Spf \zpbr} \overline{\mathfrak{Ig}}^b \xrightarrow{(t \circ \pi_1, \pi_2) } \mathcal{H} \times_{\Spf \zpbr}  \overline{\mathfrak{Ig}}^b \xrightarrow{a} \overline{\mathfrak{Ig}}^b. \]
Here $\pi_i$ denotes projection to the $i$-th component of $\overline{\mathfrak{Ig}}^b[\varepsilon]=\Spf \zpbr[\varepsilon]/\varepsilon^2  \times_{\Spf \zpbr} \overline{\mathfrak{Ig}}^b$. By composition with $\overline{v}$, we obtain
\[ d\overline{v}(\partial_t): \overline{\mathfrak{Ig}}^b[\varepsilon] \rightarrow \overline{\mathcal{P}}_\dR, \]
a vector field on $\overline{\mathcal{P}}_\dR$ along the map $\overline{v}$, i.e., a section of $H^0(\overline{\mf{Ig}}^b, \overline{v}^*T_{\overline{\mathcal{P}}_\dR}).$

\subsubsection{} 
The following computes $d\overline{v}(\partial_t)$ and, as we will explain in the next section, specializes to Theorem \ref{Thm:IntroIntegralAction}.
\begin{Thm} \label{Thm:IdentificationAction}
Let $t \in \Lie \mathcal{H}$, and let $w \in \gr^{-1}(\mf{w}_{\zpbr})$ be the associated element under the canonical isomorphism $\Lie \mathcal{H} \xrightarrow{\sim} \gr^{-1}(\mf{w}_{\zpbr})$ in \eqref{eq.canonical-idents-LieH}. Then 
\[ d\overline{v}(\partial_t) = \overline{v}^* \partial_w.\]
In other words, $d\overline{v}|_{\Lie \mathcal{H}}$ is the composition of the canonical maps 
\[ \Lie \mathcal{H} \overset{\eqref{eq.canonical-idents-LieH}}{\xrightarrow{\sim}}\mathrm{gr}^{-1}\mf{w}_{\zpbr}=\mf{w}_{\zpbr}/\Fil^0 \mf{w}_{\zpbr}\lhook\joinrel\xrightarrow{\eqref{eq.can-inclusion-tangent-pdr}}T_{\overline{\mathcal{P}}_\dR}. \]
\end{Thm}
\begin{proof}
Recall the commutative diagram
\begin{equation}
    \begin{tikzcd}
        \mathfrak{Ig}^b\gx \arrow{r} \arrow[d] & \mathfrak{Ig}^b\gvx \arrow{d} \\
        \mathcal{P}_{\dr,\zpbr}\gx \arrow{r} & \mathcal{P}_{\dr,\zpbr}\gvx.
    \end{tikzcd}
\end{equation}
of Proposition \ref{Prop:MapToPdr}, where the horizontal arrows are closed immersions. There is an induced commutative diagram
\begin{equation}
    \begin{tikzcd}
        \overline{\mathfrak{Ig}}^b\gx \arrow{r} \arrow{d} & \overline{\mathfrak{Ig}}^b\gvx \arrow{d} \\
        \overline{\mathcal{P}}_{\dr}\gx \arrow{r} & \overline{\mathcal{P}}_{\dr}\gvx
    \end{tikzcd}
\end{equation}
where the horizontal arrows are again closed immersions. Since the bottom horizontal map induces an injective map on tangent bundles, and the top row is $\mathcal{H}$ equivariant, it follows that it suffices to prove the theorem when $\gx=\gvx$, and we will henceforth assume that we are in this case. In this case, the map $\mf{Ig}^b \to \scrshat_K\gx$ represents the functor sending $z:T \to \scrshat_K\gx$ to the set of trivializations $A[p^{\infty}]_T \xrightarrow{\sim} \mbx_{\tilde{x},T}$ compatible with the polarizations, see \cite[Lemma 4.3.10]{CaraianiScholze}; this will be used below. \smallskip 

Fix $t \in \operatorname{Lie} \mathcal{H}$. By Lemma \ref{Lem:IdentificationII}, we need to show that, as maps $\overline{\mf{Ig}}^b[\varepsilon] \rightarrow \overline{\mathcal{G}}_\dR$,  
\begin{equation}\label{eq.in-proof-desired-gm-equality} d\overline{v}(t)=w\cdot \overline{s}(d\overline{u}(t)) \end{equation}
where $\overline{s}$ denotes the Gauss--Manin lift from $\overline{u}^*T_{\scrs}$ to $\overline{v}^* T_{\overline{\mathcal{G}}_\dR}$.

To that end, we consider $\zpbr[\varepsilon]/\varepsilon^2 \to S$ as in \S\ref{sss.action-and-canonical-identification}, and choose a lift $\widetilde{t}$ of $t$ from $\mathcal{H}(\zpbr[\varepsilon]/\varepsilon^2)$ to $\widetilde{\mathcal{H}}(S)$. Then, we have a commutative diagram

\begin{equation}\label{eq.proof-of-derivative-cd}\begin{tikzcd}
	{\Spf S \times_{\Spf \zpbr}} \mathfrak{Ig}^{b} & {\widetilde{\mathcal{H}}\times_{\Spf \zpbr}\mathfrak{Ig}^{b}} & {\mathfrak{Ig}^{b}} & {\mathcal{P}_\dR} \\
	{ \Spf \zpbr[\varepsilon]/\varepsilon^2 \times_{\Spf \zpbr} \overline{\mathfrak{Ig}}^{b}} & { \mathcal{H}\times_{\Spf \zpbr} \overline{\mathfrak{Ig}}^{b} } & {\overline{\mathfrak{Ig}}^{b}} & {\overline{\mathcal{P}}_\dR} ,
	\arrow["{(\tilde{t}\circ\pi_1, \pi_2)}", from=1-1, to=1-2]
	\arrow[from=1-1, to=2-1]
	\arrow["a", from=1-2, to=1-3]
	\arrow["v", from=1-3, to=1-4]
	\arrow[from=1-4, to=2-4]
	\arrow["{(t\circ\pi_1,  \pi_2)}", from=2-1, to=2-2]
	\arrow["a", from=2-2, to=2-3]
	\arrow["{\overline{v}}", from=2-3, to=2-4]
\end{tikzcd}\end{equation}
where the composition on the bottom row is $d\overline{v}(t)$. We write our universal object over $\mathfrak{Ig}^b$ as $(A, \rho: A[p^\infty] \xrightarrow{\sim}  \mbx_{\tilde{x},\mf{Ig}^b})$. Since the classifying map for $A$ factors over $\overline{\mf{Ig}}^b$, the universal abelian variety up to prime-to-$p$ isogeny $A$ over $\mathfrak{Ig}^b$ is pulled back from $\overline{\mf{Ig}}^b$. In particular, writing $B$ for the pullback to $\mf{Ig}^b[\varepsilon]$ of this universal family along the composition of the first two bottom horizontal arrows and the map $\mf{Ig}^b[\varepsilon]\rightarrow \overline{\mf{Ig}}^b[\varepsilon]$, we find that the composition of the first two horizontal arrows in the top row classifies some $(B_{\mathfrak{Ig}^b_S}, \rho')$. It follows, in particular, that, after pullback along $\mf{Ig}^b_S \rightarrow \overline{\mf{Ig}}^b[\varepsilon]$, we can lift $d\overline{v}(t)$ to the classifying map for $(B_{\mathfrak{Ig}^b_S},D(\rho'))$. 

Now, after pullback to $\mathfrak{Ig}^b$, we can also lift $\overline{s}(d\overline{u}(t))$ to $v^*T_{\mathcal{G}_\dR}$ as the Gauss--Manin lift $s(du(t))$, which we view as a map $\mf{Ig}^b[\varepsilon] \rightarrow \mathcal{G}_\dR$. Writing $B_0$ for the reduction of $B$ mod $\varepsilon$ and $B_0[\varepsilon]=B_0 \times_{\mathfrak{Ig}^b}\mathfrak{Ig}^b[\varepsilon]$, the Gauss--Manin connection evaluated on $t$ gives an isomorphism $\nabla_t: D(B) \rightarrow D(B_0[\varepsilon])$ compatible with the extra structures, and the associated lift based at a point $(B_0, s)$ is $(B, s \circ \nabla_t)$.  As explained in \S\ref{sss.abelian-variety-compatibility}, the Gauss--Manin connection agrees with the crystalline connection for the reduction mod $p$ of $A[p^\infty]$, and this is how we will compute it.

To relate this to our first computation, we note that, by the definition of the action, working mod $p$, there is a quasi-isogeny $\psi: A_{\mathrm{Ig}^b[\varepsilon]} \rightarrow B_{\mathrm{Ig}^b[\varepsilon]}$ 
such that, writing $\tilde{\psi}$ for the map on universal covers, the following diagram commutes: 
\begin{equation}\label{eq.action-reps-diagram}
    \begin{tikzcd}
       \widetilde{A_{{\mathfrak{Ig}^b_S}}[p^\infty]} \arrow[d, "\tilde{\psi}_S"] \arrow[r, "\tilde{\rho}_S"]  &  \tildex_{b}  \arrow[d, "\alpha_{\tilde{t}}"] \\
        \widetilde{B_{{\mathfrak{Ig}^b_S}}[p^\infty]} \arrow[r, "\tilde{\rho'}"] &  \tildex_{b} .
    \end{tikzcd}
\end{equation}
Since $\alpha_{\tilde{t}}$ is an isomorphism mod $\varepsilon$, the map $\tilde{\psi}_S$ is an isomorphism mod $\varepsilon$, and thus $\psi_0:=\psi \mod \varepsilon$ is an isomorphism (since this can be checked on the fpqc cover $\mf{Ig}^b_{S/\varepsilon}$). Thus the crystalline connection induces an isomorphism $\gamma: D(B_0) \rightarrow D(A)$. The trivialization $s$ in question is $D(\rho) \circ \gamma$, so the Gauss--Manin lift classifies
\[ (B, D(\rho[\varepsilon]) \circ \gamma[\varepsilon]  \circ \nabla_t ).\]
Note that $D(\rho[\varepsilon]) \circ \gamma[\varepsilon]  \circ \nabla_t$ can be thought of as the map coming from the crystalline connection associated to the isomorphism $\rho \circ \psi_0^{-1} \circ \mathrm{Id}_{B_0} $ modulo $(p,\varepsilon)$. This agrees with $\rho[\varepsilon] \circ \psi^{-1}$ modulo $(p, \varepsilon)$. Thus, if we pullback to $\mathfrak{Ig}^b_S$ so that we have our Dieudonn\'{e} theory available as in \S\ref{ss.D-maps}, we obtain that
\[ s(du(t))|_{\mathfrak{Ig}^b_S}=(B_{\mathfrak{Ig}^b_S},  D(\tilde{\rho}_S) \circ D(\tilde{\psi}_S)^{-1}). \]
But since $\alpha_{\tilde{t}} \circ \tilde{\rho}_S \circ \tilde{\psi}_S^{-1}  =\tilde{\rho}'$,  $d\overline{v}(t)|_{\mf{Ig}^b_S}$ is lifted by
\begin{align*} (B_{\mathfrak{Ig}^b_S},D(\rho'_S)) &= (B_{\mathfrak{Ig}^b_S}, D(\alpha_{\tilde{t}}) \circ D(\tilde{\rho}_S) \circ D(\tilde{\psi}_S)^{-1}  ) \\&= D(\alpha_{\tilde{t}}) \cdot (B_{\mathfrak{Ig}^b_S}, D(\tilde{\rho}_S) \circ D(\tilde{\psi}_S)^{-1}).
\\&= D(\alpha_{\tilde{t}}) \cdot s(du(t))|_{\mathfrak{Ig}^b_S}.
\end{align*}
Projecting down to $\overline{\mathcal{G}}_\dR$, we thus obtain, 
\[ d\overline{v}(t)|_{\mf{Ig}^b_S} = w \cdot \overline{s}(d\overline{u}(t))|_{\mf{Ig}^b_S} \]
where on the right to obtain $w$ from $D(\alpha_{\tilde{t}})$ we have used Lemma \ref{Lem:Identification}. Since $\mf{Ig}^b_S \rightarrow \mf{Ig}^b[\varepsilon]$ is an fpqc cover, we obtain the desired equality \eqref{eq.in-proof-desired-gm-equality}.

\end{proof}    

\subsection{Construction of \texorpdfstring{$p$}{p}-adic differential operators in the \texorpdfstring{$\mu$}{mu}-ordinary case} \label{sub:MuOrdinaryOperators} We now specialize the operators constructed above to the $\mu$-ordinary setting, and to the Mantovan Igusa variety.

\subsubsection{} \label{Subsub:Hinmuordinarysetting}

We continue with the notation from \S \ref{subsub:FramingObject}, in particular there is a fixed representative $\mu:\mathbb{G}_{m,\oee} \to \mathcal{G}_{\oee}$ of $[\mu]$. We let $\mbx=\mbx_{b_{\mu}}$, which is defined over $k_{E}$ since $b_{\mu} \in G(E)$, and we let $\mbxcan$ be the canonical lift of $\mbx$ over $\oee$. Let $\mathcal{U}_{\overline{\mu},0} $ be the center of $\mathcal{U}_{\overline{\mu}}$ and consider the vector group
\begin{align}
    \mathcal{W}=\mathcal{U}_{\overline{\mu},0} \subset \mathcal{G},
\end{align}
together with $\Fil^0\mc{W}_{\oee}:=\mc{W}_{\oee} \cap \mathcal{P}_{\mu^{-1}}$. We will write $\mf{w}=\Lie \mathcal{W}$ and $\operatorname{Fil}^0 \mf{w}_{\oee} = \operatorname{Lie} \Fil^0\mc{W}_{\oee}$, and let $\mathcal{N}=\mathcal{M}_{\mu} \cap \mathcal{U}_{\overline{\mu}}$. 

\begin{Lem} \label{Lem:LieAlgebraH}
There is an inclusion $\operatorname{Fil}^{0} \mf{w}_{\oee} \subset \Lie \mathcal{M}_{\mu}$, and the induced natural ($M_{\overline{\mu}}(\zp)$-equivariant) map $\operatorname{gr}^{-1} \mf{w}_{\oee} \to \mf{g}/\Lie \mathcal{M}_{\mu}$ has image $\mf{u}_{\mu}^{\mathcal{N}}$. 
\end{Lem}
\begin{proof}
The first claim follows from the fact that $\mf{w} \subset \mf{p}_{\overline{\mu}} \subset \mf{p}_{\mu}$ so that $ \mf{w}_{\oee} \cap \mf{p}_{\mu^{-1}}=  \mf{w}_{\oee} \cap \Lie \mathcal{M}_{\mu^{-1}} =  \mf{w}_{\oee} \cap \Lie \mathcal{M}_{\mu}$. We also deduce that the image $\operatorname{gr}^{-1} \mf{w}_{\oee}$ of $\mf{w}_{\oee}$ in $\mf{g}/\Lie \mathcal{M}_{\mu}$ lies in $\mf{p}_{\mu}/\Lie \mathcal{M}_{\mu}$ and so is contained in $\mf{u}_{\mu}$. It moreover lands in the invariants under $\mathcal{N}$, since conjugation by $\mathcal{N}$ acts trivially on $\mf{u}_{\overline{\mu},0}$ since it is central in $\mf{u}_{\overline{\mu}}$. 

To show that the induced map $\operatorname{gr}^{-1} \mf{w}_{\oee} \to \mf{u}_{\mu}^{\mathcal{N}}$ is an isomorphism, we recall that $\mf{u}_{\mu}$ is abelian. Now $\mf{u}_{\mu} \subset \mf{u}_{\overline{\mu},\oee}$, and a complement for this inclusion is given by the intersection with $\Lie \mathcal{M}_{\mu}$, which is $\mathfrak{n}=\operatorname{Lie} \mathcal{N}$. Thus an element of $\mf{u}_{\mu}$ is central in $\mf{u}_{\overline{\mu},\oee}$ precisely when it commutes with $\mathfrak{n}$, or equivalently when it is fixed by the adjoint action of $\mathcal{N}$. This shows that $\mf{u}_{\overline{\mu},0,\oee} \cap \mf{u}_{\mu}=\mf{u}_{\mu,\oee}^{\mathcal{N}}$, and hence we have a decomposition 
\[
\mf{u}_{\overline{\mu},0,\oee} = \mf{u}_{\mu,\oee}^{\mathcal{N}} \oplus (\mf{u}_{\overline{\mu},0,\oee} \cap \mathfrak{p}_{\mu^{-1}}) = \mf{u}_{\mu,\oee}^{\mathcal{N}} \oplus \opn{Fil}^0 \mathfrak{w}_{\mathcal{O}_E}.
\]
From this, we deduce the claim.
\end{proof}

\begin{Lem} \label{Lem:AssumptionsMuOrdinary}
The following hold:
    \begin{enumerate}
    \item $\mc{W}_{E}$ is preserved by $\mathrm{Ad}(b_{\mu})$, and $\mc{W}(\oee) \subseteq \mc{W}(\oee)[\tfrac{1}{p}]=\mc{W}(E)$ is a covariant Dieudonn\'{e} module under $\mathrm{Ad}(b_{\mu}) \circ \sigma$. 
    \item $\Fil^0\mc{W}_{\oee}$ is a vector subgroup of $\mc{W}_{\oee}$ such that $\Fil^0\mc{W}(\oee)$ is an admissible filtration on $\mc{W}(\oee)$.
\end{enumerate}
\end{Lem}
\begin{proof}
As in Remark \ref{Rem:FakeLogarithm}, there is a canonical identification of filtered $F$-crystals $(\mathcal{W}(\oee), \Fil^0\mathcal{W}(\oee))$ with $(\mf{w}_{\oee}, \Fil^0(\mf{w}_{\oee}))$, where the Frobenius on $\mf{w}_{E}$ is induced by the Frobenius $\mathrm{Ad}(b_{\mu}) \circ \mathrm{Id}\otimes\sigma$ on $\mf{g}_{E}$. It follows from the discussion in \cite[Section 3]{KimCentralLeaves}, see \cite[Section 4.4]{DAddeziovH} that the slope $<0$ part of the F-isocrystal $(\mathfrak{g}_{E}, \mathrm{Ad}(b) \circ \operatorname{Id} \otimes \sigma)$ is given by $\mf{u}_{\overline{\mu},E}$. Since the Lie bracket is stable under the adjoint action, it follows that $\mf{u}_{\overline{\mu},0,E} \subset \mf{u}_{\overline{\mu},E}$ is a sub $F$-isocrystal. \smallskip

By Lemma \ref{Lem:LieAlgebraH}, on every (absolute) root subspace in $\mf{u}_{\overline{\mu},E}$, conjugation by $b_{\mu}=\mu(p^{-1})$ acts either as the identity or as $p^{-1}$. It follows that $\mf{u}_{\overline{\mu},0}$ is a covariant Dieudonn\'e module. Now note that $\Fil^0\mc{W}_{\oee}:=\mc{W}_{\oee} \cap \mathcal{P}_{\mu^{-1}}=\mc{W}_{\oee} \cap \mathcal{M}_{\mu^{-1}}$. Thus $\Fil^0\mc{W}(\oee)$ is simply the subspace where conjugation by $\mu(p^{-1})$ acts as $1$, and so it is clearly an admissible filtration on $\mc{W}(\oee)$. 
\end{proof}

\subsubsection{} Let $\mathcal{H} /\spf \oee$ be the $p$-divisible group determined by the filtered Dieudonn\'e module $(\mathcal{W}(\oee), \mathrm{Ad}(b_{\mu}) \circ \sigma, \Fil^0\mathcal{W}(\oee))$ under Proposition \ref{prop.filtered-dieudonne-theory}. Note that the map $T_p \mathcal{H} \to \Aut_{\mathcal{G}}(\mbxcan)$ is defined over $\spf \oee$, since both $\mathcal{H}$ and $\mbxcan$ are defined over $\spf \oee$. Note moreover that this map factors through $\Aut_{\mathcal{G}}(\mbxcan)^{\circ}$ since $\mathcal{H}$ is connected because all the slopes of the $F$-isocrystal $\mf{u}_{\overline{\mu},0,E}$ are strictly negative, see the proof of Lemma \ref{Lem:AssumptionsMuOrdinary}. The natural map $\Aut_{\mathcal{G}}(\mbxcan) \to \mathcal{P}_{\mu^{-1}}$ sending an automorphism of $\mbxcan$ to the induced automorphism of its filtered covariant Dieudonn\'e crystal, see Lemma \ref{Lem:TensorPreservingI}, can also be defined over $\spf \oee$. The image of $\Aut_{\mathcal{G}}(\mbxcan)^{\circ} \to \mathcal{P}_{\mu^{-1}}$ lands in $\mathcal{N}:=\mathcal{U}_{\overline{\mu}} \cap \mathcal{P}_{\mu^{-1}}$ because it preserves the Hodge filtration and acts trivially on the associated graded pieces of the slope filtration. Since $\mathcal{P}_{\overline{\mu}} \subset \mathcal{P}_{\mu}$ we see that $\mathcal{N}=\mathcal{U}_{\overline{\mu}} \cap \mathcal{M}_{\mu^{-1}}$, and that $\Aut_{\mathcal{G}}(\mbxcan)^{\circ} \to \mathcal{P}_{\mu^{-1}}$ lands in $\mathcal{M}_{\mu^{-1}}$. There is thus a commutative diagram (the map of Proposition \ref{Prop:MapToPdr} is also naturally defined over $\oee$)
\begin{equation}
    \begin{tikzcd}
        \igcsf \arrow{r} \arrow{d} & \overline{\igcsf} \arrow{r} \arrow{d} & \igmf \arrow{d} \\
        \mathcal{P}_{\dr} \arrow{r} & \overline{\mathcal{P}}_{\dr} \arrow{r} & \mathcal{N} \backslash \pdr.
    \end{tikzcd}
\end{equation}

\subsubsection{} Recall from Lemma \ref{Lem:LieAlgebraH} the (visibly $\mathcal{M}_{\overline{\mu}}(\zp)$-equivariant) identification 
\begin{align}
 \Lie \mathcal{H} \xrightarrow{\sim} \mf{u}_{\mu}^{\mathcal{N}}.
\end{align}
Recall the action of $\mathfrak{u}_{\mu}$ on $\mathcal{O}_{\mathcal{P}_{\dr}}$ from \S \ref{sub:AlgebraicMaassShimura}. By Lemma \ref{Lem.inclusion-of-Lie-algebras}, this induces an action of the invariants $\left(\operatorname{Sym} \mathfrak{u}_{\mu}\right)^{\mathcal{N}}$ on $\mathcal{O}_{\mathcal{N} \backslash \pdr}$ by passing to iterates of the morphism above and passing to $\mathcal{N}$-invariants. We can now state and prove Theorem \ref{Thm:IntroIntegralAction}.
\begin{Thm} \label{Thm:GlobalAction}
    The restriction map
    \begin{align}
        H^0(\mathcal{N} \backslash \pdr, \mathcal{O}) \to H^0(\igmf, \mathcal{O})
    \end{align}
    is $\operatorname{Lie} \mathcal{H}$-equivariant, where $\Lie \mathcal{H}$ acts via $\Lie \mathcal{H} \xrightarrow{\sim} \left(\operatorname{Lie}\mathcal{U}_{\mu}\right)^{\mathcal{N}}$ on the left-hand side and via differentiating the action of $\mathcal{H}$ on $\igmf$ on the right-hand side.
\end{Thm}
\begin{proof}
    This is a direct consequence of Theorem \ref{Thm:IdentificationAction} and Lemmas \ref{Lem:AssumptionsMuOrdinary} and \ref{Lem:LieAlgebraH}.
\end{proof}
 
\subsection{Constructing additional algebraic differential operators} \label{sub:ActiononIgm} 

We will in fact show that we can construct an action of $\left(\operatorname{Sym} \mf{u}_{\mu} \right)^{\mathcal{N}}$ extending the action of $\operatorname{Sym} (\mf{u}_{\mu}^{\mathcal{N}})$ on $\mathcal{O}(\igmf)$ coming from Theorem \ref{Thm:GlobalAction}. However, we will generally not be able to interpolate these extra differential operators.

\subsubsection{} Consider the commutative diagram
\[
\begin{tikzcd}
\mathcal{Q} \arrow[d, "f"'] \arrow[r, "\chi"] & \mathcal{P}_{\opn{dR}} \arrow[d, "x"] \arrow{r}{\iota} & \mathcal{G}_{\dr} \arrow{dr}{\pi_{\mathrm{Hdg}}} \arrow[ddl, "\beta"] \\
\igmf \arrow[r, "\bar{\chi}"]  \arrow{dr}{g}     & \mathcal{N} \backslash \pdr & & \operatorname{Fl}_{[\mu^{-1}]} \\
& \scrs_{K},
\end{tikzcd}
\]
where $\mathcal{Q}$ is defined such that the left square is Cartesian. Recall the isomorphism $T_{\mathcal{P}_{\opn{dR}}} \xrightarrow{\sim} \mathfrak{g} \otimes \mathcal{O}_{\pdr}$ of Proposition \ref{Prop:TPdrToLieGisaniso}.
\begin{Prop} \label{Prop:MuOrdinaryKodairaSpencer}
The map $T_{\mathcal{Q}} \xrightarrow{d\chi} \chi^{\ast} T_{\mathcal{P}_{\opn{dR}}} \xrightarrow{\sim} \mathfrak{g} \otimes \mathcal{O}_{\mathcal{Q}}$ induces an isomorphism between $T_{\mathcal{Q}}$ and $\mf{u}_{\overline{\mu}} \otimes \mathcal{O}_{\mathcal{Q}} \subset \mathfrak{g} \otimes \mathcal{O}_{\mathcal{Q}}$.
\end{Prop}
\begin{proof}
We consider the pushout $\pcris$ of $\igcsf$ along $\Aut_{\mathcal{G}}(\mbxcan) \to \mathcal{P}_{\overline{\mu}}$, which admits a natural map $\pcris \xrightarrow{j} \gdr$. The map $\iota \circ \chi:\mathcal{Q} \to \mathcal{G}_{\dr}$ factors through $\mathcal{P}_{\operatorname{cris}} \subset \mathcal{G}_{\dr}$ via a map $c:\mathcal{Q} \to \pcris$. Indeed, this follows from the fact that $\mathcal{Q}$ is the push-out of $\mf{Ig}_\mathrm{CS}$ along the map $\mathbf{Aut}^\circ_{\mathcal{G}}(\mbb{X}^\can) \rightarrow \mathcal{N}=\mc{U}_{\overline{\mu}} \cap \mc{P}_{\mu^{-1}}$, that $\mathcal{P}_\dR$ is the push-out of $\mf{Ig}_\mathrm{CS}$ along the map to $\mathcal{P}_{\mu^{-1}}$ and that $\mathcal{G}_\dR$ is the push-out along the map to $\mathcal{G}$.  \smallskip

We first show that $T_{\pcris} \xrightarrow{dj} j^{\ast} T_{\gdr} \xrightarrow{j^{\ast} (1-s)} \mf{g} \otimes \mathcal{O}_{\pcris}$ factors through $\mf{p}_{\overline{\mu}} \otimes \mathcal{O}_{\pcris}$. For this, we go back to the construction of $s$ in terms of the Gauss--Manin connection in \S \ref{subsub:Constructions} and use the notation from there. So fix $x_0:S \to \pcris$ inducing $\varphi_{x_0}:D(A_{x_0}) \xrightarrow{\sim} \Lambda_{S}$ be the corresponding trivialization, and let $v \in \beta^\ast T_{\scrs_K}(S)$. It suffices to show that $D(v):D(A_v) \to D(A_{x_0})$ preserves the slope filtration. For this, we simply note that the isomorphism $D(v)$ induced by the Gauss--Manin connection equals the isomorphism $D(v)$ induced by the crystalline connection, see \S\ref{sss.abelian-variety-compatibility}. The isomorphism induced by the crystalline connection preserves the slope filtration, because the slope filtration is a filtration on the Dieudonn\'e crystal. \smallskip 

Next, we show that the image of $dc:T_{\mathcal{Q}} \to c^{\ast} T_{\pcris} \to \mf{p}_{\overline{\mu}} \otimes \mathcal{O}_{\mathcal{Q}}$ lands in $\mf{u}_{\overline{\mu}} \otimes \mathcal{O}_{\mathcal{Q}}$. We first note that there is a canonical trivialization
\begin{align}
    \rho:\operatorname{gr}_{\mathrm{slp}} \mby_{\igmf} \xrightarrow{\sim} \operatorname{gr}_{\mathrm{slp}} \mbxcan_{\igmf}
\end{align}
of the associated graded of the slope filtrations. The universal trivialization $\varphi:D(\mby_{\mathcal{Q}}) \xrightarrow{\sim} \Lambda \otimes \mathcal{O}_{\mathcal{Q}}$  is compatible with this canonical trivialization, in the sense that $\varphi$ induces
\begin{align}
    D(f^{\ast}\rho): \operatorname{gr}_{\mathrm{slp}} D(\mby_{\mathcal{Q}}) \to  \operatorname{gr}_{\mathrm{slp}} D(\mbxcan_{\mathcal{Q}}). 
\end{align}
Now we argue from the construction of $s$ again as above. It suffices to show that $D(v)$ also induces $D(f^{\ast}\rho)$ on the associated graded of the slope filtration. Now the crystalline connection on $D(\operatorname{gr}_{\mathrm{slp}} \mbxcan_{\igmf})$ is trivial since $\operatorname{gr}_{\mathrm{slp}} \mbxcan_{\igmf}$ is a constant $p$-divisible group. This equals the crystalline connection on $\operatorname{gr}_{\mathrm{slp}} D(\mbxcan_{\mathcal{Q}})$ induced by the one on $D(\mbxcan_{\mathcal{Q}})$, and thus $D(v)$ is trivial on $\operatorname{gr}_{\mathrm{slp}} D(\mbxcan_{\mathcal{Q}})$. \smallskip 

We now have a map $T_{\mathcal{Q}} \to \mf{u}_{\overline{\mu}} \otimes \mathcal{O}_{\mathcal{Q}}$ which we would like to show is an isomorphism. It is $\mathcal{N}$-equivariant by construction, and it thus suffices to show that the induced map $T_{\mathcal{Q}}/(\mf{n} \otimes \mathcal{O}_{\mathcal{Q}})=f^{\ast} T_{\igmf} = f^{\ast} g^{\ast} T_{\scrs_{K}} \to (\mf{u}_{\overline{\mu}} \otimes \mathcal{O}_{\mathcal{Q}})/ (\mf{n} \otimes \mathcal{O}_{\mathcal{Q}}) = \mf{u}_{\mu} \otimes \mathcal{O}_{{\mathcal{Q}}}$ is an isomorphism. 

For this, we note that as a consequence of Proposition \ref{Prop:TPdrToLieGisaniso}, we get an isomorphism $\beta^{\ast} T_{\scrs_{K}} \xrightarrow{\sim} \pi_{\mathrm{Hdg}}^{\ast} T_{\operatorname{Fl}_{[\mu^{-1}]}}$. Restricting to $\mathcal{P}_{\dr}$ this gives an isomorphism $\pi^{\ast} T_{\scrs_{K}} \xrightarrow{\sim} \iota^{\ast} \beta^{\ast} T_{\scrs_{K}}$, and since $\pdr$ maps to a point in $\operatorname{Fl}_{[\mu^{-1}]}$ with tangent space $\mathfrak{g}/\mf{p}_{\mu^{-1}}$, we can identify
\begin{align}
    \iota^{\ast} \beta^{\ast} T_{\scrs_{K}}= \mf{g}/\mf{p}_{\mu^{-1}} \otimes \mathcal{O}_{\mathcal{P}_{\dr}} = \mf{u}_{\mu} \otimes \mathcal{O}_{\mathcal{P}_{\dr}}.
\end{align}
The pullback of this map along $\chi:\mathcal{Q} \to \pdr$, is by construction, the map \[ T_{\mathcal{Q}}/(\mf{n} \otimes \mathcal{O}_{\mathcal{Q}})=f^{\ast} T_{\igmf} = f^{\ast} g^{\ast} T_{\scrs_{K}} \to \mf{u}_{\mu} \otimes \mathcal{O}_{\mathcal{Q}}\] described above, which is therefore an isomorphism. This concludes the proof.
\end{proof}

\subsubsection{} We thus get an $\mathcal{N}$-equivariant isomorphism $\mf{u}_{\overline{\mu}} \otimes \mathcal{O}_{\mathcal{Q}} \to T_{\mathcal{Q}}$. Since the map $\mf{g}_{\oee} \to T_{\pdr}$ is a homomorphism of Lie algebras, see Lemma \ref{Lem.inclusion-of-Lie-algebras}, it follows that $\mf{u}_{\overline{\mu}} \to T_{\mathcal{Q}}$ is also a homomorphism of Lie algebras. Thus, there is an induced $\mathcal{N}$-equivariant action of $U(\mf{u}_{\overline{\mu}})$ on $\mathcal{O}_{\mathcal{Q}}$ and thus an action of $(U(\mf{u}_{\overline{\mu}}))^{\mathcal{N}}$ on $(f_{\ast}\mathcal{O}_{\mathcal{Q}})^{\mathcal{N}}=\mathcal{O}_{\igmf}$. 

\begin{Cor} \label{Cor:AlgebraicActionIgM}
The action of $\left(\mathfrak{u}_{\mu}\right)^\mathcal{N} \subset (U(\mf{u}_{\overline{\mu}}))^{\mathcal{N}}$ on $\mathcal{O}_{\igmf}$ described above recovers the differentiation action of $\Lie \mathcal{H} \simeq \left(\mathfrak{u}_{\mu}\right)^\mathcal{N}$.
\end{Cor}
\begin{proof}
This is a direct consequence of the fact that the map $T_{\mathcal{Q}} \xrightarrow{d\chi} \chi^{\ast} T_{\pdr}$ is $\mf{u}_{\mu}$-equivariant, by construction, together with Theorem \ref{Thm:GlobalAction}.
\end{proof}

\subsection{Computing the Tate module} \label{sub:TateModuleH} 

We let the notation be as in \S \ref{Subsub:Hinmuordinarysetting} and consider the $p$-divisible group $\mathcal{H}$ over $\spf \mathcal{O}_E$. Its adic generic fiber defines a $p$-divisible rigid analytic group $H$ over $\spa E$. Given a choice of $x \in \minfgbmu(\eab)$ as in Proposition \ref{Prop:CompatibleChoice}, we obtain an identification $\Aut_{\mathcal{G}}(\mbxcan)_{\eta,\eab}^{\diamondsuit} \xrightarrow{\sim} \ul{\mathcal{P}_{\overline{\mu}}(\zp)}$, see Lemma \ref{Lem:AutomorphismTwoWaysII}. The goal of this section is to compute the image of $T_p H_{\eab}^{\diamondsuit} \to \ul{\mathcal{P}_{\overline{\mu}}(\zp)}$ induced by this map.

\begin{Rem} \label{Rem:}
By \cite[Main Theorem]{FarguesI}, $p$-divisible rigid analytic groups over $E$ are classified by triples $(T,W, \alpha)$, where $T$ is a continuous representation of $\gal_{E}$ on a finite free $\zp$-module, where $W$ is an $E$-vector space and where $\alpha:W \to T \otimes_{\zp} C(-1)$ is a continuous Galois equivariant morphism of $C$-vector spaces. As a byproduct of the arguments in this section, we will determine the triple $(T,W,\alpha)$ associated with $H$, see Corollary \ref{Cor:HodgeTateTripleH}. 
\end{Rem}

\subsubsection{} Choose a maximal torus $\mathcal{T}$ of $\mathcal{G}$ such that $\mu$ factors through $\mathcal{T}_{\oee}$. Recall from the proof of Proposition \ref{Prop:CompatibleChoice} that we get an induced morphism $(\mathcal{T},b_{\mu},[\mu]) \to (\mathcal{G},b_{\mu},[\mu])$ of integral local Shimura data, and that $\xi^{\can}:\spf \oee \to \mintfgbmu$ factors through $\mintftbmu \to \mintfgbmu$. 

\subsubsection{} Choose tensors $\{t_{\beta} \in \Lambda^{\otimes}\}_{\beta \in \mathscr{B}}$ cutting out $\mathcal{T}$. Since $b \in T(\qpbr) \subset G(\qpbr)$, the tensors $t_{\beta, \cris}=t_{\beta} \otimes 1 \in \Lambda^{\otimes}_{\oee} = \mathbb{D}(\mbx)^{\otimes}$ are Frobenius invariant. By construction, see \cite[Theorem 3.4.5]{KisinZhou}, the canonical lift $\mbxcan$ is $(\mathcal{T},\mu^{-1})$-adapted in the sense of \cite[Definition 3.2.4]{ShankarZhou} or \cite[Section 1.1.8]{KisinPoints}. In particular, there are Galois-invariant \'etale tensors $\{t_{\beta,\text{\'et}}\}_{\beta \in \mathscr{B}} \in (T_p \mbxcan_{\eta})^{\otimes}$. Let $x \in \minftbmu(\eab)$ be a point lifting $\xi^{\can}_E$ along $\minftbmu \to \mtbmu$, which exists by the proof of Proposition \ref{Prop:CompatibleChoice}.

\subsubsection{} \label{subsub:EtaleCrystalline} The point $x$ gives us a trivialization $\Lambda \to T_p \mbxcan_{\eta,\eab}$ sending the tensors $t_{\beta}$ to the \'etale tensors $t_{\beta, \text{\'et}}$. The \'etale--crystalline comparison theorem gives a canonical isomorphism
\begin{align}
   T_p \mbxcan_{\eta,\eab} \otimes_{\zp} B_{\cris, \eab}  \xrightarrow{\sim} \mathbb{D}(\mbxcan) \otimes_{\mathcal{O}_E} B_{\cris, \eab}.
\end{align}
By \cite[Proposition 1.1.13]{KisinPoints}, this isomorphism moreover matches $t_{\beta, \text{\'et}} \otimes 1$ with $t_{\beta, \cris} \otimes 1$ for all $\beta \in \mathscr{B}$. We can compose this isomorphism with the inverse of the canonical identification $\Lambda_{\oee} \xrightarrow{\sim} \mathbb{D}(\mbxcan)$ and the isomorphism $\Lambda \xrightarrow{\sim} T_p \mbxcan_{\eta,\eab}$ to obtain a tensor-preserving automorphism of
\begin{align}
    \Lambda \otimes_{\zp} B_{\cris, \eab},
\end{align}
giving an element $c \in \mathcal{T}(B_{\cris, \eab})$.

\subsubsection{} There is a natural identification $T_p H(\eab) \xrightarrow{\sim} T_p \mathcal{H}(\mathcal{O}_{\eab})$. The morphism $f:T_p \mathcal{H} \to \Aut_{\mathcal{G}}(\mbxcan)$ on $\oeab$ points can be written as
\begin{equation}
    \begin{tikzcd}
        T_p \mathcal{H}(\mathcal{O}_{\eab}) \arrow{r} \arrow[d] \arrow[dr,"f"] & \Aut_{\mathcal{G}}(\mbxcan)(\oeab) \arrow{d}\\
        \mathcal{U}_{\overline{\mu},0}(A_\cris(\oeab/p)) \arrow{r} &  G(B_{\cris,\eab}).
    \end{tikzcd}
\end{equation}
It follows that the map $T_p H(\eab) \to G(\zp)$ can be written as
\begin{equation}
    \begin{tikzcd}
        T_p H(\eab) \arrow{d} \arrow{r} &  G(\zp) \arrow{dr} \\
        T_p \mathcal{H}(\mathcal{O}_{\eab}) \arrow{r}{f} & G(B_{\cris,\eab}) \arrow{r}{\operatorname{Ad} c} & G(B_{\cris,\eab}).
    \end{tikzcd}
\end{equation}

\subsubsection{} By full faithfulness of filtered Dieudonn\'e theory, see Proposition \ref{prop.filtered-dieudonne-theory}, we know that the image of $f$ can be identified with the intersection
\begin{align}
    \Aut_{\mathcal{G}}(\mbxcan)(\oeab) \cap U_{\overline{\mu},0}(B_{\cris,\eab})
\end{align}
inside $G(B_{\cris,\eab})$. Therefore, we are identifying the image of $T_p H(\eab) \to G(\zp)$ with the intersection
\begin{align}
    \Aut_{\mathcal{G}}(\mbxcan)(\oeab) \cap c U_{\overline{\mu},0}(B_{\cris,\eab}) c^{-1} = P_{\overline{\mu}}(\zp) \cap c U_{\overline{\mu},0}(B_{\cris,\eab}) c^{-1}. 
\end{align}
Because $c \in T(B_{\cris,\eab})$ and thus normalizes $U_{\overline{\mu},0}(B_{\cris,\eab})$, this intersection agrees with $U_{\overline{\mu},0}(\zp) \subset P_{\overline{\mu}}(\zp)$.

\subsubsection{} Next, we identify the Hodge--Tate filtration $W \subset T_p H(C)$. There is an \'etale--crystalline comparison
\begin{align}
    T_p H_{\eab} \otimes_{\zp} B_{\cris,\eab} \xrightarrow{\sim} \mathbb{D}(\mathcal{H}) \otimes_{\oee} B_{\cris,\eab}.
\end{align}
By definition, there is an identification $\mathbb{D}(\mathcal{H}) = \mathcal{U}_{\overline{\mu},0}(\oee)$. Above, we have constructed an isomorphism $T_p H_{\eab} \xrightarrow{\sim} \mathcal{U}_{\overline{\mu},0}(\zp)$. Thus we may identify the \'etale--crystalline comparison with an automorphism $h$ of $\mathcal{U}_{\overline{\mu},0}(\zp) \otimes_{\zp} B_{\cris,\eab}$. 
\begin{Lem} \label{Lem:IdentificationIsomorphisms}
There is an equality $h=\operatorname{Ad} c$.
\end{Lem}
\begin{proof}
After passing to Lie algebras and identifying $\mathcal{U}_{\overline{\mu},0} = \mathfrak{u}_{\overline{\mu},0}$, this follows from the tensor-functoriality of the \'etale--crystalline comparison. 
\end{proof}

\subsubsection{} Recall that the Hodge filtration on $\mathbb{D}(\mathcal{H})=\mathcal{U}_{\overline{\mu},0}(\oee)$ is given by the intersection with $\mathcal{P}_{\mu^{-1}}$. By Lemma \ref{Lem:LieAlgebraH}, this can be identified with $\mathcal{U}_{\mu}^{\mathcal{N}}(\oee)$. Viewing
\begin{align}
    T_p H_{\eab} \otimes B_{\dr,\eab}^+ \subset T_p H_{\eab} \otimes_{\zp} B_{\dr,\eab} = \mathbb{D}(\mathcal{H}) \otimes_{\oee} B_{\dr,\eab} = U_{\ovmu,0}(B_{\dr,\eab}),
\end{align}
identifies $T_p H_{\eab} \otimes B_{\dr,\eab}^+$ with $\mu(t) U_{\ovmu,0}(B^+_{\dr,\eab}) \mu(t^{-1})$, where $t \in B^+_{\dr,\eab}$ is a uniformizer. Here we are using the compatibility of the \'etale--crystalline comparison with filtrations after tensoring up to $B_{\dr,\eab}$. 

\subsubsection{} On the other hand, viewing
\begin{align}
     \mathbb{D}(\mathcal{H}) \otimes_{\oee} B^+_{\dr,\eab} \subset \mathbb{D}(\mathcal{H}) \otimes_{\oee} B_{\dr,\eab} \xrightarrow{\sim} T_p H_{\eab} \otimes_{\zp} B_{\dr,\eab} = U_{\ovmu,0}(B_{\dr,\eab})
\end{align}
identifies $\mathbb{D}(\mathcal{H}) \otimes_{\oee} B^+_{\dr,\eab}$ with  $c U_{\ovmu,0}(B^+_{\dr,\eab}) c^{-1}$. From the above computation, we also get an identification of $$T_p H_{\eab} \otimes B_{\dr,\eab}^+ = U_{\ovmu,0}(B^+_{\dr,\eab})$$ with $c \mu(t) U_{\ovmu,0}(B^+_{\dr,\eab}) \mu(t)^{-1} c^{-1}$. Now since $c$ and $\mu(t)$ commute since they both lie in $T$, this equals $\mu(t) c U_{\ovmu,0}(B^+_{\dr,\eab})c^{-1} \mu(t)^{-1}$. Conjugating by $\mu(t)^{-1}$, we find that
\begin{align}
    c U_{\ovmu,0}(B^+_{\dr,\eab})c^{-1} =  \mu(t^{-1}) U_{\ovmu,0}(B^+_{\dr,\eab}) \mu(t).
\end{align}
Thus $\mathbb{D}(\mathcal{H}) \otimes_{\oee} B^+_{\dr,\eab}$ is identified with $\mu(t^{-1}) U_{\ovmu,0}(B^+_{\dr,\eab}) \mu(t)$. Since the Hodge-Tate filtration is given by setting (cf., e.g., \cite[top of p.666]{CaraianiScholze})
\[ \Fil^i_\HT(T_p \mathcal{H} \otimes \eab) = \textrm{Image of }t^i \mathbb{D}(\mathcal{H}) \otimes_{\oee} B^+_{\dr,\eab} \cap T_p \mathcal{H} \otimes B^+_{\dr,\eab}\] 
it is then an immediate computation using the above identifications
\[ T_p \mathcal{H} \otimes  B^+_{\dr,\eab} =U_{\ovmu,0}( B^+_{\dr,\eab}) \textrm{ and } \mathbb{D}(\mathcal{H}) \otimes_{\oee} B^+_{\dr,\eab}=\mu(t^{-1}) U_{\ovmu,0}(B^+_{\dr,\eab}) \mu(t)\]
that the Hodge-Tate filtration satisfies
\[ \Fil^1 \left(U_{\ovmu,0}(\eab)\right) = \mathcal{U}_{\mu}(\eab) \cap U_{\overline{\mu},0}(\eab) = \mathcal{U}_{\mu}^{\mathcal{N}}(\eab). \]

\begin{Cor} \label{Cor:HodgeTateTripleH}
The $p$-divisible rigid analytic group $H_{\eab}$ over $\eab$ is classified by the triple $(T',W,\iota(-1))$, where $T'=\mathcal{U}_{\overline{\mu},0}(\zp)$ with the trivial Galois action, where $W=U_{\mu}^{\mathcal{N}}(\eab)(-1)$ and where $\iota$ is the canonical inclusion $U_{\mu}^{\mathcal{N}}(\eab) \to U_{\ovmu,0}(\eab)$.  
\end{Cor}

\subsubsection{} We consider the map
\[ \Omega: U_{\mu}^\mathcal{N}(\eab)\otimes_{\mathbb{Z}_p} \mathbb{Z}_p(1) \xrightarrow{\sim} U_{\mu}^\mathcal{N}(\eab). \]
induced by composing the scalar multiplication 
\[ U_{\mu}^\mathcal{N}(B^+_{\dR,\eab})\otimes_{\mathbb{Z}_p} \mathbb{Z}_p(1) \subseteq  U_{\mu}^\mathcal{N}(B^+_{\dR,\eab})\otimes_{B^+_{\dR,\eab}}\Fil^1 B^+_{\dR,\eab}  \rightarrow U_{\mu}^\mathcal{N}(B^+_{\dR,\eab}) \]
and conjugation by $c$ on $\mathcal{U}_{\mu}^{\mathcal{N}}(B^+_{\dR, \eab})$ and then projecting to $\mathcal{U}_{\mu}^{\mathcal{N}}(\eab)$.

Under the identification of the first copy of $U_{\mu}^\mathcal{N}(\eab)$ with $\Lie \mathcal{H} \otimes_{\oee} \eab$ given by Lemma \ref{Lem:LieAlgebraH} and the second copy as $U_{\mu}^\mathcal{N}(\eab)\subseteq U_{\overline{\mu},0}(\eab)=T_p\mathcal{H} \otimes \eab$ as above, $\Omega$ is (by construction) the Hodge-Tate comparison 
\begin{align}\label{eq.HT-comp-Omega} \Lie \mathcal{H}\otimes \eab(1)= \gr^{-1}_{\Hdg}(\mathbb{D}(\mathcal{H}))\otimes \eab(1) \xrightarrow{\sim} \gr^{1}_{\HT}(T_p \mathcal{H} \otimes \eab) = \Fil^1_{\HT} (T_p \mathcal{H} \otimes \eab). \end{align}

\begin{Cor} \label{Cor:HodgeTateTripleH-overE}
The $p$-divisible rigid analytic group $H$ over $E$ is classified by the triple $(T',W, (\iota \circ \Omega)(-1))$, where $T'=\mathcal{U}_{\overline{\mu},0}(\zp)$ with the Galois action given by its identification with $T_p H(\eab)$ as above, where $W=U_{\mu}^{\mathcal{N}}(\eab)$, and where $\iota$ is the canonical inclusion $U_{\mu}^{\mathcal{N}}(\eab) \to \mathcal{U}_{\overline{\mu},0}(\eab)$.  
\end{Cor}

The Galois action on the Tate module is computed in \S\ref{sss.galois-action} below. 

\subsubsection{}\label{sss.omegazeta}
If we choose a trivialization $\zeta$ of $\mathbb{Z}_p(1)$, then we obtain from $\Omega$ an associated
\[ \Omega_{\zeta}: U_{\mu}^\mathcal{N}(\eab) \xrightarrow{\sim} U_{\mu}^\mathcal{N}(\eab). \]
This period automorphism appears in the introduction and several of our results. 

\begin{Rem}On each absolute root space, $\Omega_\zeta$ acts by a scalar in $\eab$. The inverses of these scalars lie in $\mathcal{O}_{\eab}$ (since the canonical map from $\Lie \mathcal{H}$ to $T_p \mathcal{H}\otimes \eab(-1)$ is induced from the integral Hodge-Tate map $T_p \mathcal{H}^\vee \rightarrow \omega_{\mathcal{H}}$ by extending scalars and dualizing). However, they are not typically units. In fact, these scalars have absolute value one only on the roots in the multiplicative part of $\mathcal{H}$ (because any map from $\mathcal{H}$ to $\widehat{\mathbb{G}}_m$ over $\mathcal{O}_{\eab}$ is trivial modulo $\mathfrak{m}_{\eab}$ except on the multiplicative part, and the integral Hodge-Tate map $T_p \mathcal{H}^\vee \rightarrow \omega_{\mathcal{H}}$ is defined by interpreting elements of $T_p \mathcal{H}^\vee$ as such maps then pulling back the differential form $\frac{dz}{z}$).  
\end{Rem}

\begin{Eg} \label{Eg:OrdinaryOmegaZeta}
    In the ordinary case (i.e. $E=\mathbb{Q}_p$ so $\overline{\mu}=\mu$), we have a canonical identification $\mathbb{X}^\can=\left( \Lambda_0 \otimes_{\mathbb{Z}_p} \mathbb{Q}_p/\mathbb{Z}_p \right) \oplus \left( \Lambda_1 \otimes_{\mathbb{Z}_p} \widehat{\mathbb{G}}_m \right)$, where the subscripts denote the weight spaces for $\mu$. Then $T_p \mathbb{X}^\can= \Lambda_0 \oplus \left(\Lambda_1 \otimes_{\mathbb{Z}_p} \mathbb{Z}_p(1)\right)$. In particular, if we choose a basis $\zeta$ for $\mathbb{Z}_p(1)$, we obtain a trivialization of $T_p \mathbb{X}^\can$ such that the associated $c$ is given by $\mu(t^{-1})$ for $t$ the usual Fontaine element associated to this basis for $\mathbb{Z}_p(1)$. In particular, if we use the same choice of basis $\zeta$ for $\mathbb{Z}_p(1)$ in the above construction, then the resulting period automorphism $\Omega_\zeta$ is the identity. 
\end{Eg}

\subsubsection{}\label{sss.galois-action} Finally, we compute the Galois action 
of $\gal(\eab/\ebreve)$ on $T_p H_{\eab} = \mathcal{U}_{\overline{\mu},0}(\zp)$. We recall the morphism
\begin{align}
    \operatorname{Res}_{\oee/\zp} \mathbb{G}_m \xrightarrow{r_{\mu}} \mathcal{T}
\end{align}
of \cite[Proof of Lemma 4.8]{DanielsTori}. The crystalline Galois representation of $\gal(\eab/\ebreve)$ corresponding to $T_p \mbxcan_{\eta,\eab}$ corresponds to the composition
\begin{align}
    \gal(\eab/\ebreve) \to \mathcal{O}_E^{\times} \xrightarrow{r_{\mu}} \mathcal{T}(\zp) \to \operatorname{GL}(\Lambda)(\zp),
\end{align}
where the first map is the local Artin map, see \cite[discussion before Proposition 4.9]{DanielsTori}. Since the map $T_p H \to \Aut_{\mathcal{T}}(\mbxcan)_{\eta}$ is Galois-equivariant, we find that the Galois action on $T_p H_{\eab} = \mathcal{U}_{\overline{\mu},0}(\zp)$ is induced by 
\begin{align}
     \gal(\eab/\ebreve) \to \mathcal{O}_E^{\times} \xrightarrow{r_{\mu}} \mathcal{T}(\zp) \xrightarrow{\operatorname{Ad}} \operatorname{GL}(\mathcal{U}_{\overline{\mu},0})(\zp).
\end{align}

\begin{Rem}
We expect that the action of $\gal(\eab/E)$ on $T_p \mbxcan_{\eta,\eab}$ factors through the Lubin--Tate extension of $E$ corresponding to the uniformizer $p$ (because we used $b_{\mu}=\mu^{-1}(p)$ instead of $\mu^{-1}(\varpi)$ for another uniformizer $\varpi$ of $E$). We show that this is indeed the case in an example, see Section \ref{sub:ExampleII}.
\end{Rem}

\subsection{Comparison to Eischen--Fintzen--Mantovan--Varma and Eischen--Mantovan} \label{sub:EischenMantovan} The goal of this section is to compare with \cites{EischenFintzenMantovanVarma, EischenMantovan} and to make precise Remark \ref{Rem:IntroEischenMantovan}.

\subsubsection{} \label{subsub:PELNotation} Let $\B$ be a finite-dimensional simple algebra over $\mathbb{Q}$ with center $\F$. Fix a positive involution $\ast$ and let $\F^+=\F^{\ast=1}$. We note that $\F^+$ is a totally real field, and we assume that $\F$ is an imaginary CM extension of $\F^+$. Let $(\V, \langle\rangle)$ be a Hermitian $\B$-module and let $\g=\operatorname{GU}(\V)$ be its similitude group, which we assume is of type $A$. Fix a prime $p>2$ such that $G=\G_{\qp}$ is quasi-split and splits over an unramified extension; we will similarly write $B=\B \otimes_{\mathbb{Q}} \qp$, $V=\V \otimes_{\mathbb{Q}} \qp$, $F=\F \otimes_{\mathbb{Q}} \qp$ and $F^+=\F^+ \otimes_{\mathbb{Q}} \qp$. Fix a prime $v$ of the reflex field $\mathsf{E} \subset \F$ above $p$ and write $E=\mathsf{E}_v$. Our choice of Hodge cocharacter $\mu$ determines a Hodge decomposition $V_{E} = V_{0} \oplus V_{-1}$. 

Let $\mathcal{O}_{\B,(p)}$ be a $\mathbb{Z}_{(p)}$-order in $\B$ whose $p$-adic completion is a maximal order $\mathcal{O}_B$ inside $B$. Fix an $\mathcal{O}_{\B, (p)}$-stable lattice $\Lambda_{(p)}$ in $V$ which is preserved by $\ast$ and self-dual under the pairing $\langle  \rangle$. Note that the PEL moduli space of abelian varieties up to prime-to-$p$ isogeny with extra structures is noncanonically a finite disjoint union of $\scrs_{K}\gx$, indexed by $\ker^1(\mathbb{Q},\g)$. 

\subsubsection{} Eischen and Mantovan define in \cite[Section 3.2]{EischenMantovan} an inverse system of Igusa varieties
\begin{align}
    \operatorname{Ig}_{\mu, n,m} \to \scrs_K\gx \otimes_{\oee} \zpbr/p^m
\end{align}
for a fixed choice of embedding $\oee \to \zpbr$. These are \'etale torsors for\footnote{Eischen--Mantovan write $J_{\mu}$ for what we call $\mathcal{M}_{\overline{\mu}}$.} $\mathcal{M}_{\overline{\mu}}(\mathbb{Z}/p^n \mathbb{Z})$ and it follows from their definition (in the PEL case, the Igusa varieties of \cite{HamacherKim} agree with those of \cite{CaraianiScholze}) that for $m=1$ there is an isomorphism of $\mathcal{M}_{\overline{\mu}}(\zp)$-torsors
\begin{align}
    \varprojlim_n \operatorname{Ig}_{\mu, n,1} \xrightarrow{\sim} \operatorname{Ig}_M \otimes_{k_E} \ovfp
\end{align}
over $\operatorname{Sh}_K\gx \otimes_{k_E} \ovfp$. Therefore, there is an isomorphism
\begin{align}
    \igmf \times_{\spf \oee} \spf \zpbr \xrightarrow{\sim} \varinjlim_m \varprojlim_n \operatorname{Ig}_{\mu, n,m}
\end{align}
of $\underline{\mathcal{M}_{\overline{\mu}}(\zp)}$-torsors over the formal Shimura variety $\scrshat_K\gx \times_{\spf \oee} \spf \zpbr$. They consider the space 
\begin{align}
    \mathcal{O}(\mf{Ig}_{\mathrm{M},\zpbr})=\varprojlim_m \varinjlim_n H^0(\operatorname{Ig}_{\mu,n,m}, \mathcal{O}).
\end{align}
Fix a Borel subgroup $\mathcal{B}_{\overline{\mu}} \subset \mathcal{M}_{\overline{\mu}}$ with unipotent radical $\mathfrak{N} \subset \mathcal{M}_{\overline{\mu}}$.\footnote{These are denoted by $B_{\mu}$ and $N_{\mu}$ in \cite{EischenMantovan}.} Then Eischen--Mantovan consider the space $\mathcal{O}(\mf{Ig}_{\mathrm{M},\zpbr})^{\mathfrak{N}(\zp)}$, see \cite[Definition 4.1.1]{EischenMantovan}, and define certain operators $\Theta^{\lambda}$ on it.

\subsubsection{} We now describe the weights $\lambda$ for which Eischen--Mantovan construct differential operators. For $G^0$ the kernel of the similitude map over $\qp$, there is a canonical embedding
\begin{align}
    G^0_{\zpbr} \subset \prod_{\tau:\mathcal{O}_{\mathsf{F}} \to \zpbr} \operatorname{GL}_{(\mathcal{O}_{B} \otimes \zpbr)_{\tau}}(\Lambda_{\zpbr, \tau})
\end{align}
identifying $G^0_{\zpbr}$ with the fixed points of the involution $x \mapsto x^{\ast}$. It contains a Levi $\mathcal{M}_{\mu}^0 \otimes_{\oee} \zpbr$ which we can identify with
\begin{align} \label{eq:IdentificationLeviPEL}
     \mathcal{M}_{\mu}^0 \otimes_{\oee} \zpbr \simeq \prod_{\tau:\mathcal{O}_F \to \zpbr} \operatorname{GL}_{a({\tau})}
\end{align}
where $a({\tau})$ is the $B \otimes \qpbr$-rank of $\left(V_{-1} \otimes_E \qpbr \right)_{\tau}$. More precisely, here we think of $\operatorname{GL}_{a(\tau)}=\operatorname{GL}((V_{-1}^{\ast} \otimes_{\mathcal{O}_E} \zpbr)_{\tau})$.

\subsubsection{} Fix a maximal torus $\mathcal{T} \subset \mathcal{G}$ such that $\mathcal{T}_{\mathcal{O}_E}$ contains the image of $\mu$, and fix a Borel subgroup $\mathcal{B} \subset \mathcal{G}$ containing $\mathcal{T}$. We may and will assume that under the isomorphisms \eqref{eq:IdentificationLeviPEL}, these are identified with the standard maximal torus of the right hand side and furthermore that $\mathcal{B}_{\oee} \cap \mathcal{M}_{\mu}^0$ is identified with the standard upper triangular Borel. 
We may thus identify the cone of dominant (for $\mathcal{M}_\mu^0$) characters with
\begin{align}
    X^{\ast}(T)^+ = \left\{(\lambda_{1, \tau}, \cdots, \lambda_{a(\tau), \tau}) \in \prod_{\tau:F \to \qpbr} \mathbb{Z}^{\oplus a(\tau)}  \; | \; \lambda_{i, \tau} \ge \lambda_{i+1, \tau} \text{ for all } i\right\}.
\end{align}
We introduce the following terminology of \cite{EischenMantovan}.
\begin{itemize}
    \item We call $\lambda \in X^{\ast}(T)^+$ \emph{positive} if all the $\lambda_{a(\tau)}$ are positive, see \cite[Section 2.3]{EischenMantovan}.
    
    \item We call $\lambda \in X^{\ast}(T)^+$ \emph{sum-symmetric} if $\lambda$ is positive and, for $\tau$ with complex conjugate $\overline{\tau}$,
    \begin{align}
        \sum_{i=1}^{a(\tau)} \lambda_{i,\tau} = \sum_{i=1}^{a(\overline{\tau})} \lambda_{i, \overline{\tau}},
    \end{align}
    see \cite[Section 2.3]{EischenMantovan}. The depth of such a character is defined to be $\frac{1}{2}\sum_{\tau:F \to \qpbr}\sum_{i=1}^{a(\tau)} \lambda_{i, \tau}$.
    \item We call $\lambda \in X^{\ast}(T)^+$ \emph{symmetric} if $\lambda$ is sum-symmetric and, for $\tau$ with complex conjugate $\overline{\tau}$, we have
    \[
    \lambda_{i, \tau} = \lambda_{i, \overline{\tau}}
    \]
    for all $i = 1, \dots, \operatorname{min}(a(\tau), a(\overline{\tau}))$.
\end{itemize}

\subsubsection{} We now want to explain how these notions are related to the inclusion $\operatorname{Lie} \mathcal{H} \subset \mathfrak{u}_{\mu}$. For the rest of this section, fix a CM-type $\Sigma \subset \opn{Hom}(\mathcal{O}_{F},\zpbr)$, that is, a set of embeddings of size $[\mathsf{F}^+ : \mathbb{Q}]$ with the property $\Sigma \cap \overline{\Sigma} = \varnothing$. We first observe that $\mathfrak{u}_{\mu} \otimes_{\oee} \zpbr$ can be described as a representation of $\mathcal{M}_{\mu}^0 \otimes_{\oee} \zpbr$ as follows
\begin{align} \label{Eq:DifferentialIdentification}
    \mathfrak{u}_{\mu} \otimes_{\oee} \zpbr \simeq \bigoplus_{\tau \in \Sigma} (\operatorname{Std}_{\tau}^{\ast} \otimes \operatorname{Std}_{{\bar{\tau}}}^{\ast}).
\end{align}
Here $\operatorname{Std}_{\tau}$ is the standard representation of $\operatorname{GL}_{a(\tau)}$, and similarly for $\operatorname{Std}_{\bar{\tau}}$. 
\begin{Lem} \label{Lem:HighestWeights}
    If $\lambda \in X^{\ast}(T)^+$ and $e$ is a positive integer, then the following are equivalent:
    \begin{itemize}
        \item $\lambda$ is the highest weight of an irreducible $\mathcal{M}_{\mu}^0 \otimes_{\mathcal{O}_E} \qpbr$-representation occurring in $\left(\mf{u}_{\mu, \qpbr}^{\ast}\right)^{\otimes e}$ (resp. $\operatorname{Sym}^e \mf{u}^{\ast}_{\mu, \qpbr}$).

        \item $\lambda$ is sum-symmetric (resp. symmetric) of depth $e$.
    \end{itemize}
    Moreover, these representations occur with multiplicity one in $\operatorname{Sym}^e \mf{u}^{\ast}_{\mu, \qpbr}$ in the symmetric case. 
\end{Lem}
\begin{proof}
This is \cite[Remark 2.4.5]{EischenFintzenMantovanVarma}, cf. \cite[Theorem 12.7]{ShimuraArithmeticity}. For completeness, we give a proof. First note that we can reduce to the case where $\#\Sigma = 1$. Since $(\opn{Std}_{\tau} \otimes \opn{Std}_{\bar{\tau}})^{\otimes e} = \opn{Std}_{\tau}^{\otimes e} \otimes \opn{Std}_{\bar{\tau}}^{\otimes e}$, the claim about sum-symmetric weights follows from the fact that any weight of $\opn{GL}_{a(\tau)}$ which appears in $\opn{Std}_{\tau}^{\otimes e}$ must be positive and of depth $e$. For the claim about symmetric weights, this follows from \cite[Exercise 6.11(b)]{FultonHarris}. 
\end{proof}

\subsubsection{} For each sum-symmetric weight $\lambda \in X^{\ast}(T)^+$, Eischen and Mantovan construct an algebraic differential operator
\begin{align}
    \Theta^{\lambda}:\mathcal{O}(\mf{Ig}_{\mathrm{M},\zpbr})^{\mathfrak{N}(\zp)} \to \mathcal{O}(\mf{Ig}_{\mathrm{M},\zpbr})^{\mathfrak{N}(\zp)},
\end{align}
see \cite[Theorem 6.3.3]{EischenMantovan} (which generalizes \cite[Theorem 5.1.3]{EischenFintzenMantovanVarma} in the ordinary case). To compare with our differential operators, we need to introduce some notation. For $\kappa \in X^{\ast}(T)^{+}$ a (not necessarily sum-symmetric) dominant weight we let $d_{\kappa, \tau} = \sum_{i=1}^{a(\tau)} \kappa_{i, \tau}$ and consider the right action of the symmetric group on $\operatorname{Std}_{\tau}^{\otimes d_{\kappa,\tau}}$, which commutes with the action of $\operatorname{GL}_{a(\tau)}$. Using the symmetric group action, we get Young symmetrizer $c_{\kappa,\tau}$ acting on the right, and we consider 
\begin{align}
    W_{\kappa,\tau}:=(\operatorname{Std}_{\tau})^{\otimes d_{\kappa,\tau}} \cdot c_{\kappa,\tau}. 
\end{align}
This is a representation of $\operatorname{GL}_{a(\tau)}$ whose basechange to $\qpbr$ is an irreducible representation of highest weight $\kappa_{\tau}$. We note that the Young symmetrizers define a canonical surjective map $\pi_{\kappa,\tau}:\operatorname{Std}_{\tau}^{\otimes d_{\kappa, \tau}} \to W_{\kappa,\tau}$. We similarly consider 
\begin{align}
    W_{\kappa}= \bigotimes_{\tau} W_{\kappa,\tau} \subset \bigotimes_{\tau} \operatorname{Std}^{\otimes d_{\kappa_{\tau}}}_{\tau}
\end{align}
and $\pi_{\kappa}:\otimes_{\tau} \operatorname{Std}_{\tau}^{\otimes d_{\kappa, \tau}} \to W_{\kappa}$. 
\begin{Def}
    Let $\kappa \in X^{\ast}(T)^+$ be any positive dominant weight. For any $\tau \colon F \hookrightarrow \qpbreve$, let $\{ b_{1,\tau}, \dots, b_{a(\tau), \tau} \}$ denote the standard basis of the representation $\opn{Std}_{\tau}$ with dual basis $b_{i,\tau}^{\vee}$. We define following \cite[Definition 2.4.2]{EischenFintzenMantovanVarma}
   \begin{align}
        \tilde{\ell}^{\kappa} &= \bigotimes_{\tau} \left(\prod_{i=1}^{a(\tau)} (\kappa_{i, \tau}!)^{-1} \cdot \left[\otimes_{i=1}^{a(\tau)} (b_{i, \tau}^{\vee})^{\otimes \kappa_{i, \tau}} \right] \cdot c_{\kappa,\tau} \right) \in \bigotimes_{\tau}(\operatorname{Std}_{\tau}^{\ast})^{\otimes d_{\kappa_{\tau}}} \\
        \ell^{\kappa} &=  \left.\kern-\nulldelimiterspace \left(\bigotimes_{\tau} \left(\prod_{i=1}^{a(\tau)} (\kappa_{i, \tau}!)^{-1} \cdot \left[\otimes_{i=1}^{a(\tau)} (b_{i, \tau}^{\vee})^{\otimes \kappa_{i, \tau}} \right] \right)\right)\right|_{W_{\kappa}}.
    \end{align}
 \end{Def}
By construction, the morphism $\ell^{\kappa}$ defines a $T_{\zpbr}$-equivariant morphism $\ell^{\kappa}:W_{\kappa} \to \zpbr[\kappa]$ satisfying $\tilde{\ell}^{\kappa} = \ell^{\kappa} \circ \pi_{\kappa}$. To state our comparison result, we note that for sum-symmetric $\lambda$ of depth $e$ we have (using \eqref{Eq:DifferentialIdentification}) $$\tilde{\ell}^{\lambda} \in \bigotimes_{\tau \in \Sigma} (\operatorname{Std}_{\tau}^{\ast} \otimes \operatorname{Std}_{{\bar{\tau}}}^{\ast})^{d_{\lambda_{\tau}}} \subset \mf{u}_{\mu}^{\otimes e},$$
and that $\mf{u}_{\mu}^{\otimes e}$ acts on $\mathcal{O}(\igmf)$ via the action described in \S \ref{sub:ActiononIgm}.

\begin{Prop} \label{Prop:ComparisonEM}
    Let $\lambda \in X^{\ast}(T)^+$ be a sum-symmetric weight of depth $e$. Then the operator $\Theta^{\lambda}$ coincides with the action of $\tilde{\ell}^{\lambda} \in \mf{u}_{\mu}^{\otimes e}$ via the action described in \S \ref{sub:ActiononIgm}. 
\end{Prop}
\begin{proof}
For a positive dominant weight $\kappa$, Eischen and Mantovan first construct an algebraic differential operator between automorphic vector bundles
\begin{align}
    D^{\lambda}_{\kappa}: \mathcal{W}_{\kappa} \to \mathcal{W}_{\kappa} \otimes_{\mathcal{O}_{\scrs_{\zpbr}}} \mathcal{W}_{\lambda} \to \mathcal{W}_{\kappa+\lambda},
\end{align}
see \cite[Definition 6.2.2]{EischenMantovan}, where the second morphism is as in \cite[Lemma 2.4.6]{EischenFintzenMantovanVarma}.\footnote{Note that Eischen--Mantovan define $\mathcal{W}_{\kappa}$ as the subrepresentation of $(\operatorname{Fil}^1 H^1_{\dr})^{\otimes e}$ cut out by $\otimes_{\tau} c_{\kappa,\tau}$, where $e=\sum_{\tau} d_{\kappa,\tau}$. This is consistent with our conventions, since the automorphic vector bundle $(\operatorname{Fil}^1 H^1_{\dr})$ corresponds to $V_{-1}^{\ast}$.} We will now reinterpret the first morphism in terms of our algebraic differential operators. Recall the $\mathcal{N}$-torsor $f:\mathcal{Q} \to \igmf$ of \S \ref{sub:ActiononIgm}. After pullback along $\mathcal{Q}_{\zpbr} \to \scrs_{\zpbr}$, the vector bundles trivialize and we can interpret the first morphism as a $\mathcal{N} \rtimes \mathcal{M}_{\bar{\mu}}(\mbb{Z}_p)$-equivariant morphism 
\begin{align}
    D_{\kappa,\lambda}: W_{\kappa} \otimes \mathcal{O}_{Q,\zpbr} \to  W_{\kappa} \otimes W_{\lambda}\otimes \mathcal{O}_{Q,\zpbr}. 
\end{align}
\begin{Claim}
The map $D_{\kappa,\lambda}$ is the tensor product of the identity with the morphism
\begin{align}
    \mathcal{O}_{Q,\zpbr} \to W_{\lambda} \otimes \mathcal{O}_{Q,\zpbr}
\end{align}
given by the $\mathcal{N} \rtimes \mathcal{M}_{\bar{\mu}}(\mbb{Z}_p)$-equivariant map (here we implicitly use the identification in \eqref{Eq:DifferentialIdentification})
\[
\mathcal{O}_{\mathcal{Q}} \to \left(\mf{u}_{\mu,\zpbr}^{\ast}\right)^{\otimes e} \otimes \mathcal{O}_{\mathcal{Q}} \xrightarrow{\pi_{\lambda} \otimes 1} W_{\lambda} \otimes \mathcal{O}_{Q,\zpbr}
\]
where the first map is induced by the action of $\mathfrak{u}_{\mu,\zpbr} \subset \mathfrak{g}_{\zpbr}$ on $\mathcal{O}_{\mathcal{Q}}$. 
\end{Claim}
\begin{proof}
The Gauss--Manin connection $\nabla:H^1_{\dr}(A/\scrs) \to H^1_{\dr}(A/\scrs) \otimes \Omega^1_{\scrs_{K}}$ pulls back to a map (note that $ \Omega^1_{\scrs_{K}}$ pulls back to $\mf{u}_{\mu}^{\ast} \otimes \mathcal{O}_{\mathcal{Q}}$ by the proof of Proposition \ref{Prop:MuOrdinaryKodairaSpencer})
\begin{align}
    \widetilde{\nabla}: (V_{0}^{\ast} \oplus V_{-1}^{\ast}) \otimes \mathcal{O}_{\mathcal{Q}} \to (V_{0}^{\ast} \oplus V_{-1}^{\ast}) \otimes \mathcal{O}_{\mathcal{Q}} \otimes \mf{u}_{\mu}^{\ast}. 
\end{align}
The differential operator $D$ of \cite[Section 6.2]{EischenMantovan} is induced by precomposing $\widetilde{\nabla}$ with the inclusion of $V_{-1}^{\ast}$ and postcomposing with the unit-root splitting
\begin{align}
    \widetilde{U}:(V_{0}^{\ast} \oplus V_{-1}^{\ast}) \otimes \mathcal{O}_{\mathcal{Q}} \to V_{-1}^{\ast} \otimes \mathcal{O}_{\mathcal{Q}}. 
\end{align}
We claim that the induced map
\begin{align}
    V_{-1}^{\ast} \otimes \mathcal{O}_{\mathcal{Q}} \to V_{-1}^{\ast} \otimes \mathcal{O}_{\mathcal{Q}} \otimes \mf{u}_{\mu}^{\ast}
\end{align}
is induced by the action map $\mathcal{O}_{\mathcal{Q}} \to \mathcal{O}_{\mathcal{Q}} \otimes \mf{u}_{\mu}^{\ast}$ by tensor product with the identity of $V_{-1}^{\ast}$. This can be deduced from Propositions \ref{Prop:MuOrdinaryKodairaSpencer} and \ref{Prop:TPdrToLieGisaniso} and their proofs.
\end{proof}

The operator $\Theta^{\lambda}$ is uniquely determined by the operators $D^{\lambda}_{\kappa}$, as we will now recall: There is an injective pullback map
\begin{align}
    \Psi_{\kappa}:H^0(\scrs_{\zpbr}, \mathcal{W}_{\kappa}) \to H^0(\igmf, \mathcal{O})^{\mathfrak{N}(\zp)}[\kappa],
\end{align}
where $[\kappa]$ denotes the $\kappa$-eigenspace for the torus $\mathcal{T}(\zpbr)$. It is defined as
\begin{align}
    H^0(\scrs_{\zpbr}, \mathcal{W}_{\kappa}) \to H^0(\igmf, W_{\kappa}) \xrightarrow{\ell^{\kappa}} H^0(\igmf, \mathcal{O}). 
\end{align}
Then the operator $\Theta^{\lambda}$ is uniquely characterized by the commutative diagram\footnote{In \cite{EischenMantovan} this is justified by the claim that the image of $\bigoplus_{\kappa} H^0(\scrs_{\zpbr}, \mathcal{W}_{\kappa}) $ in $H^0(\igmf, \mathcal{O})^{\mathfrak{N}(\zp)}$ is dense (\cite[Proposition 4.2.5]{EischenMantovan}). This claim is not quite correct; one needs to consider instead divided congruences as in \cite[Theorem 8.3]{HidaBook} or \cite[Theorem 2.6.1]{EischenFintzenMantovanVarma}.} 
\begin{equation}
    \begin{tikzcd}
        H^0(\scrs_{\zpbr}, \mathcal{W}_{\kappa}) \arrow{d}{\Psi_{\kappa}} \arrow{r}{D^{\lambda}_{\kappa}} & H^0(\scrs_{\zpbr}, \mathcal{W}_{\kappa+\lambda}) \arrow{d}{\Psi_{\kappa+\lambda}} \\
        H^0(\igmf, \mathcal{O})^{\mathfrak{N}(\zp)}[\kappa] \arrow{r}{\Theta^{\lambda}[\kappa]} & H^0(\igmf, \mathcal{O})^{\mathfrak{N}(\zp)}[\kappa+\lambda]
    \end{tikzcd}
\end{equation}
for all $\kappa$. So it suffices to show that $\Theta^{\lambda}[\kappa]$ agrees with the action of $\tilde{\ell}^{\lambda}$ for all $\kappa$, which we can do after pulling back along $\mathcal{Q} \to \igmf$. There we have the commutative diagram
\begin{equation}
    \begin{tikzcd}
          \mathcal{O}_{\mathcal{Q}} \arrow{r} & \left(\mf{u}_{\mu,\zpbr}^{\ast}\right)^{\otimes e} \otimes \mathcal{O}_{\mathcal{Q}} \arrow{r}{\pi_{\lambda} \otimes 1} \arrow{d}{\tilde{\ell}^{\lambda} \otimes 1} & W_{\lambda} \otimes \mathcal{O}_{Q,\zpbr} \arrow{d}{\ell^{\lambda} \otimes 1} \\
         & \mathcal{O}_{\mathcal{Q}} \arrow[r, equals] & \mathcal{O}_{\mathcal{Q}}
    \end{tikzcd}
\end{equation}
which sits as the middle square in the commutative diagram (see \cite[Lemma 2.4.6]{EischenFintzenMantovanVarma} for the right hand square)
\begin{equation}
    \begin{tikzcd}
        W_{\kappa}  \otimes \mathcal{O}_{\mathcal{Q}} \arrow{d}{\ell^{\kappa}} \arrow{r}{} & W_{\kappa} \otimes \left(\mf{u}_{\mu,\zpbr}^{\ast}\right)^{\otimes e} \otimes \mathcal{O}_{\mathcal{Q}} \arrow{d}{\ell^{\kappa} \otimes 1} \arrow{r} &  W_{\kappa} \otimes W_{\lambda} \otimes \mathcal{O}_{\mathcal{Q}}  \arrow[d, "\ell^{\kappa} \otimes \ell^{\lambda}"] \arrow{r}{\pi_{\kappa, \lambda}} & W_{\kappa+\lambda} \otimes \mathcal{O}_{\mathcal{Q}} \arrow{d}{\ell^{\kappa+\lambda}} \\
        \mathcal{O}_{\mathcal{Q}} \arrow[r] & \mathcal{O}_{\mathcal{Q}} \otimes \left(\mf{u}_{\mu,\zpbr}^{\ast}\right)^{\otimes e} \arrow{r}{\tilde{\ell}^{\lambda}} & \mathcal{O}_{\mathcal{Q}} \arrow[r, equals] & \mathcal{O}_{\mathcal{Q}}. 
    \end{tikzcd} 
\end{equation}
The proposition is now a direct consequence of the last commutative diagram. Indeed, the bottom row describes the action of $\tilde{\ell}^{\lambda}$ on $\mathcal{O}_{\mathcal{Q}}$, the top row describes the pullback of $D^{\lambda}_{\kappa}$ to $\mathcal{Q}$, and the left and right vertical maps induce $\Psi_{\kappa}$ and $\Psi_{\kappa+\lambda}$. 
\end{proof}

The following corollary should be compared with \cite[Remark 5.2.5]{EischenFintzenMantovanVarma}.

\begin{Cor}
If $\lambda$ is sum-symmetric but not symmetric, then the operator $\Theta^{\lambda}=0$. 
\end{Cor}
\begin{proof}
This follows from the fact that the action of $\mf{u}_{\mu}$ on $\mathcal{O}(\pdr)$ is a Lie-algebra action by Lemma \ref{Lem.inclusion-of-Lie-algebras}, and thus the action of the $e$-th tensor power factors through the symmetric power (because $\mf{u}_{\mu}$ is an abelian Lie algebra).
\end{proof}

\subsubsection{} In what follows, we will consider the inclusion $\mathcal{M}_{\overline{\mu}, \zpbr} \subset \mathcal{M}_{\mu} \otimes_{\oee} \zpbr$. We assume that our maximal torus $\mathcal{T}$ has been chosen to be contained in $\mathcal{M}_{\overline{\mu}, \zpbr}$, and that our Borel $\mathcal{B}$ has been chosen such that $\mathcal{B} \cap \mathcal{M}_{\overline{\mu}, \zpbr}$ is the fixed Borel $\mathcal{B}_{\overline{\mu}}$. Then we get an inclusion
\begin{align}
    X^{\ast}(T)^{+} \subset X^{\ast}(T)^{+,\overline{\mu}}
\end{align}
of $\mathcal{B}$-dominant characters of $T$ into $\mathcal{B}_{\overline{\mu}}$-dominant characters of $T$. We will say a character $\lambda \in X^{\ast}(T)^{+}$ is \emph{$\mathcal{H}$-relevant} if the corresponding irreducible representation over $\qpbr$ occurs in $\operatorname{Sym}^e \operatorname{Lie} \mathcal{H}_{\qpbr}$ for some $e$. Similarly, we say that a character is $\mathcal{H}^{m}$-relevant if the corresponding irreducible representation over $\qpbr$ occurs in $\operatorname{Sym}^e \operatorname{Lie} \mathcal{H}_{\qpbr}^{m}$ for some $e$ (here $\mathcal{H}^{m}$ is the multiplicative part of $\mathcal{H}$). 
\begin{Lem} \label{Lem:SimpleWeights}
A character $\lambda \in X^{\ast}(T)^{+,\overline{\mu}}$ is simple in the sense of \cite[Definition 6.3.5]{EischenMantovan} if and only if it is $\mathcal{H}^{m}$-relevant (cf. \cite[Remark 6.3.7]{EischenMantovan}). 
\end{Lem}
\begin{proof}
In the notation of \cite{EischenMantovan}, we have 
\[
\mathbb{X}_b = \bigoplus_{\mathfrak{o}} \mathbb{X}(\mathfrak{o}, n, \mathfrak{f})^{\oplus r}
\]
where $r$ is the rank of the simple $\mathsf{F}$-algebra $\mathsf{B}$ appearing in the PEL datum (see \cite[\S 2.2]{EischenMantovan}). Suppose that $\mathfrak{o} \neq \mathfrak{o}^*$. If $\lambda = \lambda(\mathfrak{o}) = (\lambda_{s_{\mathfrak{o}}}, \dots, \lambda_0)$ is a non-trivial simple weight for the factor $\opn{Aut}(D(\mbb{X}(\mathfrak{o}, n, \mathfrak{f}))[\tfrac{1}{p}]) \cong \opn{Res}_{W(\kappa(\mathfrak{o}))[\tfrac{1}{p}]/\mbb{Q}_p} \opn{GL}_n$, then we must have $\mathfrak{f}(\tau) \neq 0, n$ for all $\tau \in \mathfrak{o}$. In particular, let $(a({\tau}), a(\bar{\tau}))$ denote the signature of $V$ at $\tau \in \mathfrak{o}$. Then $a({\tau}) \neq 0, n$ for all $\tau \in \mathfrak{o}$. Let $a_{\mathfrak{o}}^+ = \opn{min}_{\tau \in \mathfrak{o}} a({\tau}) \geq 1$ and $a_{\mathfrak{o}}^- = \opn{max}_{\tau \in \mathfrak{o}} a({\tau}) \leq n-1$. Then 
\[
\mathbb{X}(\mathfrak{o}, n, \mathfrak{f})^m = (\mu_{p^{\infty}} \otimes W(\kappa(\mathfrak{o})))^{\oplus a^+_{\mathfrak{o}}}, \quad \quad \mathbb{X}(\mathfrak{o}, n, \mathfrak{f})^{\acute{e}t} = (\mbb{Q}_p/\mbb{Z}_p \otimes W(\kappa(\mathfrak{o})))^{\oplus n-a^{-}_{\mathfrak{o}}} .
\]
A simple weight has $\lambda_1 = \cdots = \lambda_{s_{\mathfrak{o}}-1} = 0$, which corresponds to weights which appear in symmetric powers of the unipotent block 
\[
\opn{Hom}(D(\mathbb{X}(\mathfrak{o}, n, \mathfrak{f})^{\acute{e}t})[\tfrac{1}{p}], D(\mathbb{X}(\mathfrak{o}, n, \mathfrak{f})^{m})[\tfrac{1}{p}]).
\]
This unipotent block is the slope $-1$ part of $U_{\bar{\mu}}$ and lies in the center of $U_{\bar{\mu}}$, hence identifies with $\opn{Lie}\mathcal{H}_{\qpbr}^m$ after base-changing to $\qpbr$. This identifies simple weights with $\mathcal{H}^m$-relevant weight in the case $\mathfrak{o} \neq \mathfrak{o}^*$.

Now suppose that $\mathfrak{o} = \mathfrak{o}^*$. Then 
\[
\opn{Aut}(D(\mbb{X}(\mathfrak{o}, n, \mathfrak{f}))[\tfrac{1}{p}]) \subset \opn{Res}_{W(\kappa(\mathfrak{o}))[\tfrac{1}{p}]/\mbb{Q}_p} \opn{GL}_n
\]
is the unitary group of matrices fixed by the involution $x \mapsto x^*$. Suppose that $\lambda = \lambda(\mathfrak{o}) = (\lambda_{s_{\mathfrak{o}}}, \dots, \lambda_0)$ is a non-trivial simple weight for this factor; in particular, this implies that $a_{\mathfrak{o}}^+ := \opn{min}_{\tau \in \mathfrak{o}} a(\tau) \geq 1$. One can check that
\[
\mathbb{X}(\mathfrak{o}, n, \mathfrak{f})^m = (\mu_{p^{\infty}} \otimes W(\kappa(\mathfrak{o})))^{\oplus a^+_{\mathfrak{o}}}, \quad \quad \mathbb{X}(\mathfrak{o}, n, \mathfrak{f})^{\acute{e}t} = (\mbb{Q}_p/\mbb{Z}_p \otimes W(\kappa(\mathfrak{o})))^{\oplus n-a^{+}_{\mathfrak{o}}},
\]
and so simple weights appear in symmetric powers of the unipotent block
\[
\opn{Hom}(D(\mathbb{X}(\mathfrak{o}, n, \mathfrak{f})^{\acute{e}t})[\tfrac{1}{p}], D(\mathbb{X}(\mathfrak{o}, n, \mathfrak{f})^{m})[\tfrac{1}{p}]).
\]
Again, this is the slope $-1$ part of $U_{\bar{\mu}}$ and contained in the center, hence simple weights match up with $\mathcal{H}^m$-relevant weights.
\end{proof}

\subsubsection{} We now discuss the $p$-adic interpolation of the Maass--Shimura operators in \cite{EischenMantovan},\cite{EischenFintzenMantovanVarma}, reinterpreted in the language of this article. We note our main $p$-adic interpolation result in this article (see Corollary \ref{Cor:IntegralFourierConsequence}) is stronger than what is discussed in this subsection.

We first note that there is an algebra embedding
\[
\opn{Sym} (\opn{Lie} \mathcal{H}) := \bigoplus_{e \geq 0} \opn{Sym}^e \opn{Lie}\mathcal{H}  \hookrightarrow \mathcal{O}(T_p \mathcal{H}^{\vee})
\]
given by precomposing an element $f \in \opn{Sym}(\opn{Lie} \mathcal{H})$ (viewed as a polynomial function on $\omega_{\mathcal{H}} = (\opn{Lie} \mathcal{H})^{\vee}$) with the Hodge--Tate map $T_p \mathcal{H}^{\vee} \to \omega_{\mathcal{H}}$.

We may fix elements $\{ v_{\lambda} \in \opn{Sym}(\opn{Lie}\mathcal{H}^m_{\zpbr}) : \lambda \text{ is } \mathcal{H}^m\text{-relevant} \}$ such that: $v_{\lambda}$ is a highest weight vector for the algebraic representation of $\mathcal{M}_{\mu}$ of highest weight $\lambda$ appearing in $\opn{Sym}(\opn{Lie} \mathcal{H}^m_{\qpbr})$ (recall that this algebraic representation appears with multiplicity one); and one has $v_{\lambda + \lambda'} = v_{\lambda} \cdot v_{\lambda'}$ for all $\mathcal{H}^m$-relevant $\lambda$, $\lambda'$. Indeed, the set of $\mathcal{H}^{\opn{m}}$-relevant weights forms a cone under multiplication of characters (which we write additively); if we fix a basis of this cone $\{\lambda_1, \dots, \lambda_c\}$ and highest weight vectors $\{ v_{\lambda_1}, \dots, v_{\lambda_c}\}$, then we can define $v_{\lambda} := v_{\lambda_1}^{n_1} \cdots v_{\lambda_c}^{n_c}$ for a general $\mathcal{H}^{\opn{m}}$-relevant weight $\lambda = n_1\lambda_1 + \cdots + n_c \lambda_c$ with $n_i \in \mbb{Z}_{\geq 0}$ (this is well-defined because $\{ \lambda_1, \dots, \lambda_c \}$ is a basis).

Let $L/\mbb{Q}_p$ be a finite unramified extension over which $\mathcal{T}$ splits, and let $q$ denote the order of its residue field. We may assume the elements $\{ v_{\lambda} \}$ are defined over $L$. If we write $\lambda = (\kappa_{1, \tau}, \dots, \kappa_{a(\tau), \tau})_{\tau}$, and similarly for $\lambda'$, then the property $v_{\lambda + \lambda'} = v_{\lambda} \cdot v_{\lambda'}$ implies that if $\kappa_{i, \tau} \equiv \kappa'_{i, \tau}$ modulo\footnote{In \cite[Proposition 6.3.9 and Theorem 6.3.10]{EischenMantovan} the congruence is written as modulo $p^{k-1}(p-1)$, but this weaker congruence seems to be sufficient only when $\mathcal{T}$ splits over $\mathbb{Z}_p$.} $q^{k-1}(q-1)$ for all $(i, \tau)$, then
\[
v_{\lambda} \equiv v_{\lambda'} \; \text{ modulo } p^k T_p \mathcal{H}^{m,\vee}.
\]
In particular, if $(\lambda_i)_{i \geq 1}$ is a sequence of $\mathcal{H}^m$-relevant weights such that the corresponding characters $\mathcal{T}(\zp) \to \mathcal{O}_L^{\times}$ converge to a continuous character $\lambda_{\infty} \colon  \mathcal{T}(\mbb{Z}_p) \to \mathcal{O}_L^{\times}$ in the $p$-adic topology, then $v_{\lambda_i}$ converges to a well-defined element $\lim v_{\lambda_i} \in \mathcal{O}(T_p \mathcal{H}^{m,\vee})$.
    
\begin{Cor}
    Suppose that $\lambda_i$ are simple and $\kappa_{i, \tau}$ converges to $+\infty$ in the archimedean topology, for all $(i, \tau)$. With notation as in Proposition \ref{Prop:ComparisonEM}, the action of $\lim v_{\lambda_i}$ on $\mathcal{O}(\mf{Ig}_{\mathrm{M},\zpbr})^{\mathfrak{N}(\mbb{Z}_p)}$ through the algebra action of $\mathcal{O}(T_p\mathcal{H}^{\vee})$ in Corollary \ref{Cor:IntegralFourierConsequence} agrees with the operator ``$\Theta^{\lambda_{\infty}}$'' in \cite[Corollary 6.3.12]{EischenMantovan}.
\end{Cor}
\begin{proof}
    Note that the action of $\lim v_{\lambda_i}$ preserves the $\mathfrak{N}(\mbb{Z}_p)$-invariants, since the action of each individual $v_{\lambda_i}$ does. The corollary now follows from Proposition \ref{Prop:ComparisonEM} and the continuity of the action of $\mathcal{O}(T_p \mathcal{H}^{\vee})$. 
\end{proof}

\begin{Rem} \label{Rem:Caveat}
Note that if $\lambda_{\infty}$ is an algebraic $\mathcal{H}^m$-relevant character, then it is (almost) never the case that $\lim v_{\lambda_i} = v_{\lambda_{\infty}}$ if $\kappa_{i, \tau}$ converges to $+\infty$ in the archimedean topology. Thus the operator denoted by ``$\Theta^{\lambda_{\infty}}$'' does not agree with the algebraic Maass--Shimura operator $\Theta^{\lambda_{\infty}}$ considered in Proposition \ref{Prop:ComparisonEM} (the same issue arises also in \cite[Corollary 5.2.8]{EischenFintzenMantovanVarma}). 
For example, for the group $\opn{GL}_2$, simple weights are just non-negative algebraic weights $\lambda \in X^*(\mbb{G}_m)$, and the action of the operator $\Theta^{\lambda}$ on $q$-expansions is given by the $\lambda$-iterate of the operator $q \frac{d}{dq}$. Identifying $T_p \mathcal{H}^{\vee} \cong \mbb{Z}_p$, the operator $\Theta^{\lambda}$ corresponds to the action of the function
\[
v_{\lambda} \colon \mbb{Z}_p \to \mbb{Z}_p, \quad \quad x \mapsto x^{\lambda} .
\]
We can consider the sequence $\lambda_i = p^{i}(p-1)$ which converges to $+\infty$ in the archimedean topology, and converges to $\lambda_{\infty} = 0$ in the $p$-adic topology, however ``$\Theta^{\lambda_{\infty}}$'' (the action of $\lim v_{\lambda_i}$) does not coincide with $\Theta^0 = \opn{Id}$. Indeed, the function $\lim v_{\lambda_i}$ is precisely the indicator function of $\mbb{Z}_p^{\times}$. In particular, one should be careful in applying \cite[Corollary 6.3.12]{EischenMantovan}/\cite[Corollary 5.2.8]{EischenFintzenMantovanVarma}, where these two different operators at algebraic weights are not distinguished in the notation.
\end{Rem}

\begin{Rem}
Our $p$-adic interpolation result, see Corollary \ref{Cor:IntegralFourierConsequence}, is stronger than being able to define $\Theta^{\chi}$ for $p$-adic weights $\chi \colon \mathcal{T}(\mbb{Z}_p) \to \mathcal{O}_L^{\times}$. For example, when $p$ is totally split in $\mathsf{F}$ so that $E=\qp$ and we may take $L=\qp$, we obtain an action of $\Cont(\mathfrak{u}_{\mu}^*, \zp)$ on $\mathcal{O}(\igmf)$ and the action of the operators $\Theta^{\chi}$ corresponds to $p$-adic limits of polynomial functions in $\Cont(\mathfrak{u}_{\mu}^*, \zp)$. As explained in Remark \ref{Rem:IntroEischenMantovan}, such $p$-adic limits form a proper subset of all continuous functions. 
\end{Rem}

\begin{Rem}
    For applications to $p$-adic $L$-functions, it is often the case that one fixes an $\mathcal{H}$-relevant character $\lambda \in X^*(T)^+$ and aims to raise the (algebraic) operator $\Theta^{\lambda}$ induced from the action of $v_{\lambda}$ to $p$-adic powers. In order to do this, one must perform a certain ``$p$-depletion'', which often corresponds to acting by a suitable indicator function in $\mathcal{O}(T_p\mathcal{H}^{\vee})[\tfrac{1}{p}] \widehat{\otimes}_{E} C$  for a completed algebraic closure $C$ of $E$; this is discussed in more detail in \S \ref{Sec:PAdicLFunctions}.
\end{Rem}


\section{Applying \texorpdfstring{$p$}{p}-adic Fourier theory} \label{Sec:ApplyingFourier}
In this section we will apply the $p$-adic Fourier theory of \cite{FourierPaper}. In \S \ref{sub:IntegralFourier}, we prove Corollary \ref{Cor:IntroIntegralOperators} and two refinements, see Corollary \ref{Cor:IntegralFourierConsequence} and Corollaries \ref{Cor:BddAction} and \ref{Cor:IntegralSmooth}. In \S \ref{Sub:FourierOnGenFiber}, we prove Corollary \ref{Cor:IntroRationalOperatorsI} and a refinement of Theorem \ref{Thm:IntroInfiniteLevelTheorem}, see Corollary \ref{Cor:RationalFourierApplication} and Theorem \ref{Thm:InfiniteLevelTheorem}. In \S \ref{sub:Ordinary}, we restate all our results in the ordinary case, see Corollaries \ref{Cor:IntegralFourierConsequenceOrdinary} and \ref{Cor:BddActionOrdinary}, and Theorem \ref{Thm:InfiniteLevelTheoremOrdinary}. Finally, we briefly discuss two non-ordinary examples in Sections \ref{sub:ExampleII} and \ref{sub:ExampleIII}.

\subsection{Applying integral Fourier theory} \label{sub:IntegralFourier} Let the notation be as in \S \ref{sub:MuOrdinaryOperators}. In particular, we let $\mathcal{H}$ be as in \S \ref{Subsub:Hinmuordinarysetting} and we recall the isomorphism $\Lie \mathcal{H} \xrightarrow{\sim} \mf{u}_{\mu}^{\mathcal{N}}$ of Lemma \ref{Lem:LieAlgebraH}. Note that $M_{\overline{\mu}}(\zp)$ acts on $\mathcal{U}_{\overline{\mu},0}$ and thus on $\mathcal{H}$. The identification $\Lie \mathcal{H} \to \mf{u}_{\mu}^{\mathcal{N}}$ is moreover $M_{\overline{\mu}}(\zp)$-equivariant, see Lemma \ref{Lem:LieAlgebraH}.

\subsubsection{} Recall from \cite[Section 7.1.11]{FourierPaper} that $\mathcal{O}(\mathcal{H})$ is a commutative and co-commutative topological Hopf algebra over $\mathcal{O}_E$. The action of $\mathcal{H}$ on $\igmf$ gives a continuous map (see \cite[Section 7.1.3]{FourierPaper} for our conventions of taking global sections on formal schemes, and \cite[Lemma 7.1.4]{FourierPaper} for the K\"unneth isomorphism)
\begin{align}
a:\mathcal{O}(\igmf) \to \mathcal{O}(\mathcal{H} \times_{\spf \mathcal{O}_E} \igmf) \simeq \mathcal{O}(\mathcal{H}) \widehat{\otimes}_{\mathcal{O}_E} \mathcal{O}(\igmf),
\end{align}
which we think of as a co-action. Consider the dual topological Hopf algebra $\mathcal{O}(\mathcal{H})^{\ast}$ equipped with the weak topology, see \cite[Section 7.1.11]{FourierPaper}. The co-action $a$ defines a continuous algebra action of $\mathcal{O}(\mathcal{H})^{\ast}$ on $\mathcal{O}(\igmf)$, see \cite[Section 7.1]{HoweUnipotent}. 

\subsubsection{} Consider the Serre dual $p$-divisible group $\mathcal{H}^{\vee}$ and let $T_p \mathcal{H}^{\vee}$ be its $p$-adic Tate-module. This is a $p$-adic formal scheme over $\spf \mathcal{O}_E$ and $\mathcal{O}(T_p \mathcal{H}^{\vee})$ is also a topological $\mathcal{O}_E$-Hopf algebra, see \cite[Section 7.1.11]{FourierPaper}. The integral $p$-adic Fourier theory of \cite[Proposition 7.1.12]{FourierPaper} gives us a canonical isomorphism of topological $\mathcal{O}_E$-Hopf algebras
\begin{align}
\mathbb{F}^{\mathcal{H}}:\mathcal{O}(\mathcal{H})^{\ast} \xrightarrow{\sim} \mathcal{O}(T_p \mathcal{H}^{\vee}).
\end{align}
It follows that there is a continuous algebra action of $\mathcal{O}(T_p \mathcal{H}^{\vee})$ on $\mathcal{O}(\igmf)$. 

\subsubsection{} We now state the following corollary of Theorem \ref{Thm:GlobalAction}. Consider the integral Hodge--Tate map $\HT:T_p \mathcal{H}^{\vee} \to \omega_{\mathcal{H}}$, see e.g. \cite[Section 7.1.5]{FourierPaper}. 
\begin{Cor} \label{Cor:IntegralFourierConsequence}
The continuous algebra action of $\mathcal{O}(T_p \mathcal{H}^{\vee})$ on $\mathcal{O}(\igmf)$ has the following properties:
\begin{enumerate}
    \item The action of $\mathcal{O}(T_p \mathcal{H}^{\vee})$ on $\mathcal{O}(\igmf) $ is semilinear for the natural continuous action of $M_{\overline{\mu}}(\zp)$ on $\mathcal{O}(T_p \mathcal{H}^{\vee})$ and $\mathcal{O}(\igmf)$.
    
    \item The action of $\operatorname{Lie} \mathcal{H}$ on $\mathcal{O}(\igmf)$ given by differentiation coincides with the action of $\operatorname{Lie} \mathcal{H}$ via the map
    \begin{align}
        \operatorname{Lie} \mathcal{H} \subset \mathcal{O}(\omega_{\mathcal{H}}) \xrightarrow{\HT^{\ast}} \mathcal{O}(T_p \mathcal{H}^{\vee}).
    \end{align}

    \item Under pullback along $\mathcal{O}(\mathcal{N} \backslash \mathcal{P}_{\dr}) \to \mathcal{O}(\igmf)$, the action of $\mf{u}_{\mu}^{\mathcal{N}}$ on the source by algebraic Maass--Shimura operators, see \S \ref{subsub:AlgebraicOperatorsDef}, intertwines with the action of $\Lie \mathcal{H} \xrightarrow{\sim} \mf{u}_{\mu}^{\mathcal{N}}$ on the target described in (2). 
\end{enumerate}
\end{Cor}
\begin{proof}
Part (1) follows from the construction. Part (2) follows from \cite[Proposition 7.0.2]{FourierPaper}, and Part (3) follows from Theorem \ref{Thm:GlobalAction}.
\end{proof}

\subsubsection{} We now discuss a variant of Corollary \ref{Cor:IntegralFourierConsequence} for smooth functions on $\mathfrak{Ig}_{\opn{CS}}$. Recall that $\igcsf \to \igmf$ is a torsor for $\Aut_{\mathcal{G}}(\mbxcan)^{\circ}$, and that $\igcsf$ carries an action of $\Aut_{G}(\tildex)$. Consider the element $t := \bar{\mu}(p)^d \in M_{\bar{\mu}}(\mbb{Q}_p) \subset \Aut_{G}(\tildex)(\spf \oee)$, where $d = [E:\mbb{Q}_p]$, and for any integer $m \geq 0$, consider the subgroup 
\begin{equation} \label{Eqn:Autmsubgroup}
    \Aut_{\mathcal{G}}(\mbxcan)^{\circ}_m := t^m \Aut_{\mathcal{G}}(\mbxcan)^{\circ} t^{-m}.
\end{equation}
This subgroup is isomorphic to $\Aut_{\mathcal{G}}(\mbxcan)^{\circ}$ via conjugation by $t^{-m}$, hence for $\mathfrak{Ig}_{\mathrm{M}, m} := \Aut_{\mathcal{G}}(\mbxcan)^{\circ}_m \backslash \mathfrak{Ig}_{\opn{CS}}$ we have an isomorphism
\[
\mathfrak{Ig}_{\mathrm{M}, m} \xrightarrow{t^{-m}} \igmf
\]
induced from the (left) action of $t^{-m}$ on $\igcsf$. 
\begin{Lem} \label{Lem:ContractingAction}
For $m' \ge m$, there is a natural inclusion $\Aut_{\mathcal{G}}(\mbxcan)^{\circ}_{m'} \subset \Aut_{\mathcal{G}}(\mbxcan)^{\circ}_m$. Moreover, we have that $\cap_m \Aut_{\mathcal{G}}(\mbxcan)^{\circ}_{m}=\{1\}$. 
\end{Lem}
\begin{proof}
Write $\Aut(\mbxcan)^{\circ}_m = t^m \Aut(\mbxcan)^{\circ} t^{-m}$. It suffices to prove the analogous assertions for $\Aut(\mbxcan)^{\circ}_m$, since $\Aut_{\mathcal{G}}(\mbxcan)^{\circ}_m = \Aut_{G}(\tildex) \cap \Aut(\mbxcan)^{\circ}_m$, and $\Aut_{G}(\tildex)$ is stable under conjugation by $t^m$ by construction. It moreover suffices to prove the lemma for $\Aut(\mbx)^{\circ}_m = \Aut(\mbxcan)^{\circ}_{m,k_{E}}$, since the natural map
\begin{align}
    \Aut(\mbxcan)(R) \to \Aut(\mbxcan)(R/p)
\end{align}
is injective for $R \in \nilp_{\oee}$. Recall that $\mbx$ is completely slope divisible since it is $\mu$-ordinary, see \S \ref{subsub:IgM}, and write $\mbx=\oplus_i \mbx_i$ for its slope decomposition. We can then write the connected part of its automorphism group as (cf. \cite[Section 4.1.9]{DAddeziovH}) the upper triangular matrix group
\begin{align}
\Aut(\mbx)^{\circ}=\begin{pmatrix}
    1 & \Hom(\mbx_1, \mbx_2) & \cdots &  \Hom(\mbx_1, \mbx_r)   \\ 
      & 1 & \cdots & \Hom(\mbx_2, \mbx_r) \\
       &  & \ddots & \vdots \\
      &    & & 1
    \end{pmatrix}.
\end{align}
By construction\footnote{Note that $\ovmu(p)^d=\mu(p)\sigma(\mu(p)) \cdots \sigma^{d-1}(\mu(p))$ and thus corresponds to the $d$-th power of the inverse of Frobenius on the covariant Dieudonn\'e module.}, the element $t \in \Aut(\tildex)$ is given by the $d$-th power of $\operatorname{Frob}_{\mbx}$. Conjugation by $\operatorname{Frob}_{\mbx}$ induces the Frobenius endomorphisms of $\Hom(\mbx_i, \mbx_j)$, and conjugation by $t$ is given by its $d$-th power. Since $\Hom(\mbx_i, \mbx_j)$ is the $p$-adic Tate-module of the internal hom $p$-divisible group, see \cite[Lemma 4.1.7]{CaraianiScholze}, it follows that the Frobenius is injective and thus conjugation by $t$ induces an injective map $\Hom(\mbx_i, \mbx_j) \to \Hom(\mbx_i, \mbx_j)$; it follows that $\Aut(\mbx)^{\circ}_{1} \subset \Aut(\mbx)^{\circ}$. To show that $\cap_{m} \Aut(\mbx)^{\circ}_{m}=\{1\}$, it suffices to note that $\cap_{m} \operatorname{Frob}^{dm} \Hom(\mbx_i, \mbx_j)=\{0\}$ since the internal hom $p$-divisible group has strictly positive slope in this case.
\end{proof}

By Lemma \ref{Lem:ContractingAction}, we have natural maps $\mathfrak{Ig}_{\mathrm{M}, m+1} \to \mathfrak{Ig}_{\mathrm{M}, m}$ arising from the inclusions $\Aut_{\mathcal{G}}(\mbxcan)^{\circ}_{m+1} \subset \Aut_{\mathcal{G}}(\mbxcan)^{\circ}_m$.

\begin{Def}
    We define the $\mathcal{O}_E$-algebra of smooth sections on $\igcsf$ as the colimit 
    \[
    \mathcal{O}(\igcsf)^{\opn{sm}} := \varinjlim_m \mathcal{O}(\mathfrak{Ig}_{\mathrm{M}, m})
    \]
    where the transition maps are induced from pullback along $\mathfrak{Ig}_{\mathrm{M}, m+1} \to \mathfrak{Ig}_{\mathrm{M}, m}$. We equip  $\mathcal{O}(\igcsf)^{\opn{sm}}$ with the $p$-adic topology.
    
\end{Def}
\subsubsection{} Each layer of this tower $\mathfrak{Ig}_{\mathrm{M}, m}$ carries an action of the quotient 
\[
\mathcal{H}_m := \widetilde{\mathcal{H}}/(\widetilde{\mathcal{H}} \cap \Aut_{\mathcal{G}}(\mbxcan)^{\circ}_m) .
\]
Since the vector group $\mathcal{W} \subset \mathcal{G}$ underlying the construction of $\mathcal{H}$ is stable under conjugation by $t^m$, the universal cover $\widetilde{\mathcal{H}}$ is stable under conjugation by $t^m$ via the morphism $\widetilde{\mathcal{H}} \to \Aut_{G}(\tildex)$. In particular, conjugation by $t^m$ induces an isomorphism $\mathcal{H}_m \xrightarrow{\sim} \mathcal{H}$. Note that it follows from Lemma \ref{Lem:ContractingAction} that $\varprojlim_m \mathcal{H}_m = \tildeh$. Passing to duals and Tate modules, the natural map $\mathcal{H}_m \twoheadrightarrow \mathcal{H}$ induces an inclusion $T_p\mathcal{H}^{\vee} \subset T_p \mathcal{H}_m^{\vee}$ and $\varinjlim_m T_p \mathcal{H}_m^{\vee} = \widetilde{\mathcal{H}}^{\vee}$. It thus follows by definition, see \cite[Section 7.1.3]{FourierPaper}, that $\varprojlim_m \mathcal{O}(T_p\mathcal{H}_m^{\vee})=\mathcal{O}(\tildeh^{\vee}) $ as topological rings. 
\begin{Cor} \label{Cor:IntegralSmooth}
The action of $\mathcal{O}(T_p\mathcal{H}^{\vee})$ on $\mathcal{O}(\igmf)$ extends to a continuous algebra action of $\mathcal{O}(\tildeh^{\vee})$ on $\mathcal{O}(\igcsf)^{\opn{sm}}$ via the maps $\mathcal{O}(\tildeh^{\vee}) \to \mathcal{O}(T_p \mathcal{H}^{\vee})$ and $\mathcal{O}(\igmf) \to \mathcal{O}(\igcsf)^{\opn{sm}}$.
\end{Cor}
\begin{proof}
The co-action maps assemble to give a continuous map
\begin{align}
    \varinjlim_m \mathcal{O}(\mf{Ig}_{\mathrm{M},m}) \to \varinjlim_m \left(\mathcal{O}(\mf{Ig}_{\mathrm{M},m}) \widehat{\otimes}_{\oee} \mathcal{O}(\mathcal{H}_m) \right). 
\end{align}
We thus get an algebra action of $\varprojlim_m \mathcal{O}(\mathcal{H}_m)^{\ast}=\varprojlim_m \mathcal{O}(T_p\mathcal{H}_m^{\vee})=\mathcal{O}(\tildeh^{\vee})$ on $\varinjlim_m \mathcal{O}(\mf{Ig}_{\mathrm{M},m})=\mathcal{O}(\igcsf)^{\opn{sm}}$. It is continuous for the $p$-adic topology because each $\mathcal{O}(\mathcal{H}_m)^{\ast}$ has the $p$-adic topology. 
\end{proof}

\subsubsection{} We now want to state a version of Corollary \ref{Cor:IntegralSmooth} where we invert $p$ before taking the direct limit over $m$. To compare with our other results stated later, we will state it in terms of condensed mathematics; we refer to \cite[Section 2]{FourierPaper} (or \cite{CondensedNotes}) for our conventions. In particular for a topological $\zp$-module $M$, we will write $\ul{M}$ for the induced condensed $\ul{\zp}$-module. Now we set 
\[
\mathcal{O}(\mathfrak{Ig}_{\opn{CS}, \eta})^{\opn{sm}} := \varinjlim_m (\ul{\mathcal{O}(\mathfrak{Ig}_{\mathrm{M}, m})[\tfrac{1}{p}]})
\]
and $\mathcal{O}[\tfrac{1}{p}](V_p \mathcal{H}^{\vee}) := \varprojlim_m (\ul{\mathcal{O}(T_p \mathcal{H}_m^{\vee})[\tfrac{1}{p}]})$. 
\begin{Rem} \label{Rem:FailureNormality}
Writing $T_p H_{m}^{\vee}$ for the adic generic fiber of $T_p \mathcal{H}_m^{\vee}$, there is a natural map 
\begin{align} \label{Eq:AdicGenericFiberRestr}
    r^{\mathrm{int}}:\mathcal{O}(T_p \mathcal{H}_m^{\vee})[\tfrac{1}{p}] \to \mathcal{O}(T_p H_m^{\vee})
\end{align}
inducing $\mathcal{O}[\tfrac{1}{p}](V_p \mathcal{H}^{\vee}) \to \ul{\mathcal{O}(V_p H^{\vee})}$, where $V_p H^{\vee} = \varinjlim_m T_p H_{m}^{\vee}$. Note that the natural map \eqref{Eq:AdicGenericFiberRestr} is typically not injective and typically not surjective, see \cite[Remark 7.2.4]{FourierPaper} for a discussion. 
\end{Rem}
\subsubsection{} We note that $\ul{\mathcal{O}(T_p \mathcal{H}_m^{\vee})[\tfrac{1}{p}]}$ and thus $\mathcal{O}[\tfrac{1}{p}](V_p \mathcal{H}^{\vee})$ has the structure of a $\ul{\qp}$-algebra since $M \mapsto \ul{M}$ commutes with limits. The compatible maps $\HT^{\ast}[\tfrac{1}{p}]:\Lie\mathcal{H}[\tfrac{1}{p}] \rightarrow \mathcal{O}(T_p \mathcal{H}_m^{\vee})[\tfrac{1}{p}]$ induce a map $\ul{\Lie\mathcal{H}[\tfrac{1}{p}]} \to \mathcal{O}[\tfrac{1}{p}](V_p \mathcal{H}^{\vee})$. 

\begin{Cor} \label{Cor:BddAction}
    There is an algebra action of $\mathcal{O}[\tfrac{1}{p}](V_p \mathcal{H}^{\vee})$ on $\mathcal{O}(\mathfrak{Ig}_{\opn{CS}, \eta})^{\opn{sm}}$ compatible with the actions of $\mathcal{O}(T_p \mathcal{H}_m^{\vee})[\tfrac{1}{p}]$ on $\mathcal{O}(\mathfrak{Ig}_{\mathrm{M}, m})[\tfrac{1}{p}]$ from Corollary \ref{Cor:IntegralFourierConsequence}. In particular, it recovers the algebraic action of $\Lie\mathcal{H}[\tfrac{1}{p}]$ on $\mathcal{O}(\mathfrak{Ig}_{\opn{CS}, \eta})^{\opn{sm}}$ (see \S \ref{sub:ActiononIgm}) via the natural map $\ul{\Lie\mathcal{H}[\tfrac{1}{p}]} \to \mathcal{O}[\tfrac{1}{p}](V_p \mathcal{H}^{\vee})$.
\end{Cor}
\begin{proof}
The co-action maps assemble to give a map of direct systems\footnote{Note that the completed tensor product here does not agree with the solid tensor product, see \cite[Remark 2.1.12]{FourierPaper}.}
\begin{align}
   \{\mathcal{O}(\mf{Ig}_{\mathrm{M},m})[\tfrac{1}{p}]\}_{m} \to \{\left(\mathcal{O}(\mf{Ig}_{\mathrm{M},m}) \widehat{\otimes}_{\oee} \mathcal{O}(\mathcal{H}_m)\right)[\tfrac{1}{p}]\}_m. 
\end{align}
This shows that the continuous algebra action of $\mathcal{O}(\mathcal{H}_m)^{\ast}[\tfrac{1}{p}]$ on $\mathcal{O}(\mf{Ig}_{\mathrm{M},m-1}) \subset \mathcal{O}(\mf{Ig}_{\mathrm{M},m})$ restricts to the action on $\mathcal{O}(\mf{Ig}_{\mathrm{M},m-1})$. Passing to the condensed world, we thus get an algebra action of $\varprojlim_m \ul{\mathcal{O}(\mathcal{H}_m)^{\ast}[\tfrac{1}{p}]}=\varprojlim_m \ul{\mathcal{O}(T_p\mathcal{H}_m^{\vee})[\tfrac{1}{p}]}=\mathcal{O}[\tfrac{1}{p}](V_p \mathcal{H}^{\vee})$ on $\varinjlim_m \ul{\mathcal{O}(\mf{Ig}_{\mathrm{M},m})[\tfrac{1}{p}]}=\mathcal{O}(\igcsa)^{\opn{sm}}$. The second claim of the corollary follows from Corollary \ref{Cor:IntegralFourierConsequence}.
\end{proof}

\begin{Eg} \label{Eg:OrdinaryBDD}
When $\mathcal{H}=\gmhat$, then we get
\begin{align}
    \mathcal{O}[\tfrac{1}{p}](V_p \mathcal{H}^{\vee})&=\varprojlim_m \left(\cont(\tfrac{1}{p^m} \zp, \zp)[\tfrac{1}{p}]\right) = \varprojlim_m \cont(\tfrac{1}{p^m}\zp, \qp)  \\
    \cont(\qp,\qp) &= \mathcal{O}(V_p H^{\vee}),
\end{align}
which we note is much bigger than
\begin{align}
    \mathcal{O}(\tildeh^{\vee})[\tfrac{1}{p}]=\left(\varprojlim_m \cont(\tfrac{1}{p^m} \zp, \zp)\right)[\tfrac{1}{p}] = \cont(\qp, \zp)[\tfrac{1}{p}]= \mathcal{O}^+(V_p H^{\vee})[\tfrac{1}{p}].
\end{align}
If we give $\cont(\qp,\qp) = \varprojlim_m \cont(\tfrac{1}{p^m}\zp, \qp)$ the inverse limit topology (of the natural Banach topologies), then $\mathcal{O}[\tfrac{1}{p}](V_p \mathcal{H}^{\vee})= \ul{\cont(\qp,\qp)}$, since $M \mapsto \ul{M}$ commutes with inverse limits. 
\end{Eg}

\subsection{Applying Fourier theory on the generic fiber} \label{Sub:FourierOnGenFiber}

In this section we will apply the Fourier theory of \cite{FourierPaper} to further study the action of Corollary \ref{Cor:BddAction}. 

\subsubsection{} Let $H$ be the adic generic fiber of $\mathcal{H}$; this is a $p$-divisible rigid analytic group in the sense of Fargues \cite[Definition 4.1]{FarguesII} over $\spa E$. Write $\mathcal{O}(H)=\ul{H^0(H, \mathcal{O})}$, which is a solid Hopf algebra over $\ul{E}$, and consider its condensed dual $\mathcal{D}(H)$, which is also a solid Hopf algebra over $\ul{E}$, see \cite[Section 5.1.4, Proposition 5.1.1, Lemma 6.5.2]{FourierPaper}. 

\subsubsection{} Let $T_p H^{\vee}$ be the adic generic fiber of $T_p \mathcal{H}^{\vee}$, equipped with its Hodge--Tate map $\gamma:T_p H^{\vee} \to \omega_{H}$. We consider the solid Hopf algebra $\mathcal{O}(T_p H^{\vee})^{\gamma-\mathrm{la}}:=H^0_{\mathrm{cond}}(\spd E, \mathcal{O}^{\gamma-\mathrm{la}}_{T_p H^{\vee}/\spd E})$ of $\gamma$-locally analytic functions on $T_p H^{\vee}$, see \cite[Section 6.3.3]{FourierPaper}. We note that there is a natural monomorphism $\mathcal{O}(T_p H^{\vee})^{\gamma-\mathrm{la}} \to \mathcal{O}(T_p H^{\vee})$, where $\mathcal{O}(T_p H^{\vee})=\ul{H^0(T_p H^{\vee}, \mathcal{O})}$. 

\subsubsection{} We will now construct an action of $\mathcal{O}(T_p H^{\vee})^{\gamma-\mathrm{la}}$ on $\mathcal{O}(\igma) = \ul{H^0(\igma, \mathcal{O})}$ in two different ways. First, we take the rigid generic fiber of the action of $\mathcal{H}$ on $\igmf$ to get an action of $H$ on $\igma$. Applying the K\"unneth formula\footnote{This can be proved by writing $H$ as an increasing union of (uniform) affinoid adic spaces over $\spa E$, proving the K\"unneth formula for these affinoid adic spaces using \cite[Proposition A.68]{Bosco}, and then passing the projective limit through the solid tensor product using \cite[Corollary A.67]{Bosco}.}, the action of $H$ on $\igma$ gives rise to a co-action
\begin{align}
a_{\eta}:\mathcal{O}(\igma) \to \ul{\mathcal{O}(H \times_{\spa E} \igma)} \simeq \ul{\mathcal{O}(H)} \otimes_{\ul{E}}^{\blacksquare} \mathcal{O}(\igma).
\end{align}
The co-action $a_{ \eta}$ defines an action of the solid algebra $\mathcal{D}(H)$ on $\mathcal{O}(\igma)$. By \cite[Corollary 6.3.4]{FourierPaper}, there is a natural isomorphism of solid Hopf algebras
\begin{align}
\mathbb{F}^{H}:\mathcal{D}(H) \xrightarrow{\sim} \mathcal{O}(T_p H^{\vee})^{\gamma-\mathrm{la}},
\end{align}
giving the desired algebra action of $\mathcal{O}(T_p H^{\vee})^{\gamma-\mathrm{la}}$ on $\mathcal{O}(\igma)$. 

\subsubsection{} \label{subsub:rla} For the second construction, we recall from \cite[Section 7.2.1]{FourierPaper} the monomorphism $r^{\mathrm{la}}:\mathcal{O}(T_p H^{\vee})^{\gamma-\mathrm{la}} \to \mathcal{O}(T_p H^{\vee})=\ul{H^0(T_p H^{\vee}, \mathcal{O})}$. We furthermore recall from \cite[Theorem 7.2.6]{FourierPaper} the commutative diagram of condensed\footnote{The condensed structure is not addressed in the statement of \cite[Theorem 7.2.6]{FourierPaper}. However, to get the map on $S$-points for $S$ a profinite set, we can apply \cite[Theorem 7.2.6]{FourierPaper} with $(R,R^+)=(\cont(S,E), \cont(S,\oee))$. To do this, we have to check that $(R,R^+)$ is fiercely v-complete in the sense of \cite[Definition 3.1.9]{FourierPaper}, note that by \cite[Lemma 3.1.10]{FourierPaper} it is enough to check that it is diamantine in the sense of \cite[Definition 11.1]{HansenKedlaya}. It is plus-sheafy in the sense of \cite[Definition 6.8]{HansenKedlaya}, because the topological space $\spa(R,R^+)$, which is homeomorphic to $S$, has no higher cohomology. It is moreover v-complete since $(E,\oee)$ is, see \cite[proof of Lemma 3.1.6]{FourierPaper}; it is thus diamantine and we conclude.} $\ul{E}$-algebras
\begin{equation} \label{Eq:BigDiagram}
    \begin{tikzcd}
        \ul{\mathcal{O}(\mathcal{H})^{\ast}[\tfrac{1}{p}]} \arrow{r}{\mathbb{F}^{\mathcal{H}}[\tfrac{1}{p}]} & \ul{\mathcal{O}(T_p \mathcal{H}^{\vee})[\tfrac{1}{p}]} \arrow{d}{r^{\mathrm{int}}}\\
        & \mathcal{O}(T_p H^{\vee})  \\
        \mathcal{D}(H) \arrow{r}{\mathbb{F}^{H}} \arrow{uu}{r_{\mathcal{H}}^{\ast}} & \mathcal{O}(T_p H^{\vee})^{\gamma-\mathrm{la}} \arrow{u}{r^{\mathrm{la}}}.
    \end{tikzcd}
\end{equation}
Since $\ul{\mathcal{O}(T_p \mathcal{H}^{\vee})[\tfrac{1}{p}]}$ acts on $\mathcal{O}(\igma)=\ul{\mathcal{O}(\igmf)}[\tfrac{1}{p}]$, we get an action of $\mathcal{O}(T_p H^{\vee})^{\gamma-\mathrm{la}}$ by applying 
\begin{align}
   \mathbb{F}^{\mathcal{H}}[\tfrac{1}{p}] \circ r_{\mathcal{H}}^{\ast} \circ (\mathbb{F}^{H})^{-1}.
\end{align}
This agrees with the previous action, by construction, because both are induced from the action of $\mathcal{H}$ on $\igmf$ by applying $p$-adic Fourier theory, using \eqref{Eq:BigDiagram}. 

\subsubsection{} We consider the natural map of adic spaces (where $P_{\dr}^{\mathrm{an}}$ is the analytification of $P_{\dr}=\mathcal{P}_{\dR,E}$ and $N=\mathcal{N}_E$ with analytification $N^{\mathrm{an}}$)
\begin{align}
\igma \to \mathcal{N}_{\eta} \backslash \mathcal{P}_{\dr,\eta} \to N^{\mathrm{an}} \backslash P_{\dr}^{\mathrm{an}} 
\end{align}
The Hodge--Tate map on the generic fiber $\gamma:T_p H^{\vee} \to \omega_H$ induces a pullback map
\begin{align}
    \gamma^{\ast}: \ul{\operatorname{Lie} H} \subset \ul{\mathcal{O}(\omega_H)} \to \mathcal{O}(T_p H^{\vee})^{\gamma-\mathrm{la}}.
\end{align}
We have the following corollary of Theorem \ref{Thm:GlobalAction}.
\begin{Cor} \label{Cor:RationalFourierApplication}
The algebra action of $\mathcal{O}(T_p H^{\vee})^{\gamma-\mathrm{la}}$ on $\mathcal{O}(\igma)$ constructed above has the following properties: 
\begin{enumerate}
    \item The action of $\ul{\operatorname{Lie} H}$ on $\mathcal{O}(\igma)$ given by differentiation coincides with the action of $\ul{\operatorname{Lie} \mathcal{H}}$ via $\gamma^{\ast}$.

    \item The natural action of $\ul{\operatorname{Lie} H} \xrightarrow{\sim} \ul{\mf{u}_{\mu,E}^{N}}$ on $\ul{\mathcal{O}(P_{\dr}^{\mathrm{an}})}^{N^{\mathrm{an}}}$ by algebraic Maass--Shimura operators intertwines with the action of $\ul{\operatorname{Lie} H}$ on $\mathcal{O}(\igma)$ described in (1) via pullback along
    \begin{align}
        \ul{\mathcal{O}(P_{\dr}^{\mathrm{an}})}^{N^{\mathrm{an}}} \to \mathcal{O}(\igma).
    \end{align}
\end{enumerate}
\end{Cor}
\begin{proof}
Part (1) follows from \cite[Proposition 6.4.2, 6.4.4]{FourierPaper}, and part (2) follows from Theorem \ref{Thm:GlobalAction}.    
\end{proof}
We note that Corollary \ref{Cor:IntroRationalOperatorsI} is a consequence of Corollary \ref{Cor:RationalFourierApplication} together with Corollary \ref{Cor:HodgeTateTripleH}.

\subsubsection{} Define
\begin{align}
    \mathcal{O}(V_p H^{\vee})^{\gamma-\mathrm{la}}:=\varprojlim_{m} \mathcal{O}(T_p H_{m}^{\vee})^{\gamma-\mathrm{la}},
\end{align}
which admits a morphism (see \S \ref{subsub:rla})
\begin{align}
     \mathcal{O}(V_p H^{\vee})^{\gamma-\mathrm{la}} \to \mathcal{O}[\tfrac{1}{p}](V_p \mathcal{H}^{\vee}).
\end{align}
Thus the action of $\mathcal{O}[\tfrac{1}{p}](V_p \mathcal{H}^{\vee})$ on $\mathcal{O}(\igcsa)^{\mathrm{sm}}$ of Corollary \ref{Cor:BddAction} induces an action of $ \mathcal{O}(V_p H^{\vee})^{\gamma-\mathrm{la}}$ on $\mathcal{O}(\igcsa)^{\mathrm{sm}}$. This action can also be constructed by using the compatible actions of $H_m= \mathcal{H}_{m,\eta}$ on $\mf{Ig}_{M,m,\eta}$.

\begin{Cor} \label{Cor:BddActionII}
The algebra action of $\mathcal{O}(V_p H^{\vee})^{\gamma-\mathrm{la}}$ on $\mathcal{O}(\igcsa)^{\mathrm{sm}}$ is compatible with the action of $\mathcal{O}(T_p H^{\vee})^{\gamma-\mathrm{la}}$ on $\mathcal{O}(\igma)$ from Corollary \ref{Cor:RationalFourierApplication}. Moreover, it recovers the algebraic action of $\ul{\Lie\mathcal{H}[\tfrac{1}{p}]}$ on $\mathcal{O}(\mathfrak{Ig}_{\opn{CS}, \eta})^{\opn{sm}}$ (see \S \ref{sub:ActiononIgm}) via the natural map $\ul{\Lie\mathcal{H}[\tfrac{1}{p}]} \to \mathcal{O}(V_p H^{\vee})^{\gamma-\mathrm{la}}$.
\end{Cor}

\subsection{Going to infinite level}\label{ss.infinite-level} We now pass to partial infinite level on our nearly holomorphic forms, and compare the algebra action to the resulting Hecke action.

\subsubsection{} Let $E \subset L \subset C$ be a complete field containing $\qp(\mu_{p^{\infty}})$ and choose an $L$-point $x \in \minfgbmu$ satisfying the conclusion of Proposition \ref{Prop:CompatibleChoice} (this in particular assumes $L$ is big enough such that such a point exists). Recall from \S \ref{subsub:MapsGenericFiber} that this induces a map
\begin{align}\label{Eq:InfiniteLevelMapToSh}
    \igcsdL \to \mathbf{Sh}_{K^p,L}^{\circ,\diamondsuit}.
\end{align}
By \S \ref{subsub:MapsGenericFiber}, this is $\ul{P_{\overline{\mu}}(\qp)}$-equivariant via the natural inclusion $\ul{P_{\overline{\mu}}(\qp)} \to \widetilde{G}_{b_{\mu},L}$ coming from our choice of $x$, which identifies $\ul{P_{\overline{\mu}}(\zp)}$ with $\Aut_{\mathcal{G}}(\mbxcan)_{\eta,L}^{\diamondsuit}$. By the discussion in \S \ref{sub:TateModuleH}, the natural map $T_p {H}_L \to \ul{\mathcal{G}(\zp)}$ is identified with the inclusion of $\ul{\mathcal{U}_{\overline{\mu},0}(\zp)}$, and therefore $V_p H_L \subset \ul{G(\qp)}$ is identified with $\ul{U_{\overline{\mu},0}(\qp)}$. 

\subsubsection{} \label{subsub:TwistedHeckeAction} Consider $P_{\dr,\infty,L}^{\diamondsuit}:=\mathbf{Sh}_{K^p,L}^{\circ,\diamondsuit} \times_{\mathbf{Sh}_{K,L}^{\circ, \diamondsuit}} P_{\dr,K,L}^{\mathrm{an},\diamondsuit}$. Recall from \S \ref{subsub:MapsGenericFiber} the map 
\begin{align}
\igcsdL \to P_{\dr,\infty,L}^{\diamondsuit}
\end{align}
and its $\ul{P_{\ovmu}(\qp)}$-equivariance via Lemma \ref{Lem:TwistedHeckeAction}. 

\subsubsection{} For each compact open subgroup $K_p' \subset G(\qp)$ with intersection $J_{K_p'}=K_p' \cap \mathcal{U}_{\overline{\mu}}(\zp)$, we consider the natural map
\begin{align} \label{Eq:DiamondMap}
    \ul{J_{K_p'}} \backslash \igcsdL \to N^{\mathrm{an}} \backslash P_{\dr,\infty,L}^{\mathrm{an},\diamondsuit}/\ul{K_p'},
\end{align}
which is well defined by Lemma \ref{Lem:TwistedHeckeAction} and the observation that under the map $\ul{P_{\overline{\mu}}(\qp)} \to P_{\mu^{-1}}^{\mathrm{an},\lozenge}$ of \eqref{Eq:FactorsThrough N} the subgroup $\ul{U_{\overline{\mu}}(\qp)}$ factors through $N^{\mathrm{an}} \subset P_{\mu^{-1}}^{\mathrm{an},\lozenge}$. 

\begin{Lem} \label{Lem:MapDescends}
    If $K_{p,m} \subset G(\qp)$ is a compact open subgroup with $J_{K_{p,m}}=J_m$ for some $m$, so that $\ul{J_{K_{p,m}}} \backslash \igcsdL=\mf{Ig}_{M,m,\eta,L}^{\diamondsuit}$, then there is a unique dashed arrow making the following diagram commute
    \begin{equation}
        \begin{tikzcd}
            \mf{Ig}_{M,m,\eta,L} \arrow[r, dashed] \arrow{d} & N^{\mathrm{an}} \backslash P^\mathrm{an}_{\dr,L} \times_{\mathbf{Sh}_{K,L}^{\circ}} \mathbf{Sh}_{K^pK_{p,m},L}^{\circ} \arrow{d} \\
            \mf{Ig}_{M,\eta,L} \arrow{r} & N^{\mathrm{an}} \backslash P^\mathrm{an}_{\dr,L},
        \end{tikzcd}
    \end{equation}
which induces \eqref{Eq:DiamondMap} after applying $\diamondsuit$.
\end{Lem}
\begin{proof}
By the universal property of the fiber product, it suffices to show that the morphism
\begin{align}
    \mf{Ig}_{M,m,\eta,L}^{\diamondsuit} \to \mf{Ig}_{M,\eta,L}^{\diamondsuit} \times_{\mathbf{Sh}_{K,L}^{\circ,\diamondsuit}} \mathbf{Sh}_{K^pK_{p,m},L}^{\circ,\diamondsuit}
\end{align}
is uniquely induced from a map of adic spaces. But this is a morphism of finite \'etale covers of $\mf{Ig}_{M,\eta,L}^{\diamondsuit}$, and the statement thus follows from the fact that the (finite) \'etale site of an analytic adic space is equivalent to that of its diamond, see \cite[Lemma 15.6]{EtCohDiam}. 
\end{proof}

\subsubsection{} Choosing a decreasing sequence of compact open subgroups $K_{p,m}$ with $J_{K_{p,m}}=J_m$ and $\cap_{m} K_{p,m}=\{1\}$, we define 
\begin{align} \label{Eq:PDrAnSm}
   \mathcal{O}(N^{\mathrm{an}} \backslash P_{\dr,\infty}^{\mathrm{an}})^{\mathrm{sm}}:=\varinjlim_{m} \ul{ \mathcal{O}( N^{\mathrm{an}} \backslash P_{\dr,L}^\mathrm{an} \times_{\mathbf{Sh}_{K,L}^{\circ}} \mathbf{Sh}_{K^pK_{p,m},L}^{\circ})}.
\end{align}
Note moreover that
\begin{align}
     \mathcal{O}(\igcsa)^{\mathrm{sm}} \otimes_{\ul{E}}^{\blacksquare} \ul{L} = \varinjlim_m\left(\ul{\mathcal{O}(\mf{Ig}_{M,m,\eta})} \otimes_{\ul{E}}^{\blacksquare} \ul{L}\right) \xrightarrow{\sim} \varinjlim_m \ul{\mathcal{O}(\mf{Ig}_{M,m,\eta,L})}=:\mathcal{O}(\igcsaL)^{\mathrm{sm}},
\end{align}
so that there is an induced pullback map
\begin{align}
    \mathcal{O}(N^{\mathrm{an}} \backslash P_{\dr,\infty}^{\mathrm{an}})^{\mathrm{sm}} \to \mathcal{O}(\igcsaL)^{\mathrm{sm}}.
\end{align}
This map is $P_{\ovmu}(\qp)$-equivariant for the twisted action on the domain and the natural action on the codomain, see Lemma \ref{Lem:TwistedHeckeAction}. Note that the algebraic Maass--Shimura operators define an action of $\ul{\mf{u}_{\mu,L}^{N}}$ on $\mathcal{O}(N^{\mathrm{an}} \backslash P_{\dr,\infty}^{\mathrm{an}})^{\mathrm{sm}}$, see \S \ref{subsub:AlgebraicOperatorsDef}.

\subsubsection{} Using Corollary \ref{Cor:HodgeTateTripleH}, we identify $T_p H_L=\ul{\mathcal{U}_{\overline{\mu},0}(\zp)}=\ul{\mf{u}_{\overline{\mu},0}}$. We now fix a choice $\zeta$ of identification $\ul{\zp} \to \ul{\zp}(1)$ over $L$, which gives us the further identification $T_p H^{\vee}_{L}= \ul{\mf{u}_{\overline{\mu},0}^{\ast}}$ and identifies $\gamma:T_p H^{\vee}_L \to \omega_{H}=(\mf{u}_{\mu,L}^{\mathcal{N}})^{\ast}$ with $(\iota \circ \Omega_\zeta)^*$, where $\iota$ is the canonical inclusion $\mf{u}_{\mu,L}^{\mathcal{N}} \subset \mf{u}_{\overline{\mu},0,L}$. Therefore (see \cite[Proposition 5.2.9]{FourierPaper})
\begin{align}
    \mathcal{O}(T_p H^{\vee})^{\gamma-\mathrm{la}} \otimes_{\ul{E}}^{\blacksquare} \ul{L} = \mathcal{O}^{\gamma-\mathrm{la}}(\ul{\mf{u}_{\overline{\mu},0}^{\ast}},L).
\end{align}
Note that we have written $\mathcal{O}^{\gamma-\mathrm{la}}(\ul{\mf{u}_{\overline{\mu},0}^{\ast}},L)$ instead of just $\mathcal{O}^{\gamma-\mathrm{la}}(\ul{\mf{u}_{\overline{\mu},0}^{\ast}})$ (as in \cite{FourierPaper}) to stress that we are looking at $L$-valued locally analytic functions. Now $\mathcal{O}(\igcsaL)^{\mathrm{sm}}$ has an action of
\begin{align}
     \mathcal{O}(V_p H^{\vee})^{\gamma-\mathrm{la}} \otimes_{\ul{E}}^{\blacksquare} \ul{L}.
\end{align}
Using \cite[Proposition 5.2.9]{FourierPaper} and \cite[Corollary A.67]{Bosco}, we identify this with
\begin{align}
\varprojlim_m \mathcal{O}^{\gamma-\mathrm{la}} (\tfrac{1}{p^m}\ul{\mf{u}_{\overline{\mu},0}^{\ast}},L)
=:\mathcal{O}^{\gamma-\mathrm{la}} (\ul{\mf{u}_{\overline{\mu},0,\qp}^{\ast}},L).
\end{align}
There is a pullback map
\begin{align}
    \gamma^{\ast}:\ul{\operatorname{Lie} H_L} \to \mathcal{O}(V_p H_L^{\vee})^{\gamma-\mathrm{la}},
\end{align}
which is the inverse limit of the corresponding pullback maps for $\mathcal{O}(T_p H_L^{\vee})^{\gamma-\mathrm{la}}$.

\subsubsection{} Recall the Weil pairing $V_p H \times V_p H^{\vee} \to \mathbb{Q}_p(1)$. To an element $h \in V_p H$ we can thus associate a locally constant function $f_{h}:V_p H^{\vee} \to \qp(1) \to \qp(1)/\zp(1)=\mu_{p^{\infty}} \subset L$. We can finally state our main theorem. 

\begin{Thm} \label{Thm:InfiniteLevelTheorem}
The action of $\mathcal{O}^{\gamma-\mathrm{la}}(\ul{\mf{u}_{\overline{\mu},0,\qp}^{\ast}},L)$ on $\mathcal{O}(\igcsaL)^{\mathrm{sm}}$ constructed above has the following properties: 
\begin{enumerate}
    \item The action on $\ul{\mathcal{O}(\igmaL)} \subset \mathcal{O}(\igcsaL)^{\mathrm{sm}}$ factors through the restriction map $\mathcal{O}^{\gamma-\mathrm{la}}(\ul{\mf{u}_{\overline{\mu},0,\qp}^{\ast}},L) \to \mathcal{O}^{\gamma-\mathrm{la}}(\ul{\mf{u}_{\overline{\mu},0}^{\ast}},L)$ and recovers the action of Corollary \ref{Cor:RationalFourierApplication} after (solid) tensor product to $L$. 

    \item The action of $h \in U_{\overline{\mu},0}(\qp) \subseteq P_{\overline{\mu}}(\qp)$ corresponds to the action of the locally constant function $f_h \in \mathcal{O}^{\gamma-\mathrm{la}}(\ul{\mf{u}_{\overline{\mu},0,\qp}^{\ast}},L)(\ast)$.

    \item The action of $\ul{\operatorname{Lie} H_L}$ on $\mathcal{O}(\igcsaL)^{\mathrm{sm}}$ given by differentiation coincides with the action of $\ul{\operatorname{Lie} H_L}$ via $\gamma^{\ast}= (\iota \circ \Omega_\zeta)^*$.

    \item The action of $\mathcal{O}^{\gamma-\mathrm{la}}(\ul{\mf{u}_{\overline{\mu},0,\qp}^{\ast}},L)$ on $\mathcal{O}(\igcsaL)^{\mathrm{sm}}$ is $M_{\overline{\mu}}(\qp)$-semilinear for the action of $M_{\overline{\mu}}(\qp)$ on $\mathcal{O}(\igcsaL)^{\mathrm{sm}}$ and $\ul{\mf{u}_{\overline{\mu},0,\qp}^{\ast}}$. 

    \item The differentiation action of $\ul{\operatorname{Lie} H} \xrightarrow{\sim} \ul{\mf{u}_{\mu,E}^{N}}$ on $\mathcal{O}(\igcsaL)^{\mathrm{sm}}$ intertwines with the action of $\ul{\mf{u}_{\mu,E}^{N}}$ on $\mathcal{O}(N^{\mathrm{an}} \backslash P_{\dr,\infty}^{\mathrm{an}})^{\mathrm{sm}}$ through algebraic Maass--Shimura operators, via the pullback map
    \begin{align}
        \mathcal{O}(N^{\mathrm{an}} \backslash P_{\dr,\infty}^{\mathrm{an}})^{\mathrm{sm}} \to \mathcal{O}(\igcsaL)^{\mathrm{sm}}.
    \end{align}

    \item The natural action of $P_{\overline{\mu}}(\qp)$ on $\mathcal{O}(\igcsaL)^{\mathrm{sm}}$ intertwines with the twisted action of $P_{\overline{\mu}}(\qp)$ on $\mathcal{O}(N^{\mathrm{an}} \backslash P_{\dr,\infty}^{\mathrm{an}})^{\mathrm{sm}}$ described in \S \ref{subsub:TwistedHeckeAction} via the pullback map
    \begin{align}
        \mathcal{O}(N^{\mathrm{an}} \backslash P_{\dr,\infty}^{\mathrm{an}})^{\mathrm{sm}} \to \mathcal{O}(\igcsaL)^{\mathrm{sm}}.
    \end{align}
\end{enumerate}
\end{Thm}
\begin{proof}
Parts (1), (4) and (6) are a direct consequence of the construction. Part (3) follows as in the proof of part (1) of Corollary \ref{Cor:RationalFourierApplication} by taking the direct limit over $m$, and similarly part (5) follows as in the proof of part (2) of Corollary \ref{Cor:RationalFourierApplication} by taking the direct limit over $K$.  For part (2), we note that at level $J_m$ the action of $U_{\overline{\mu},0}(\qp)$ factors through the action of $\ul{U_{\overline{\mu},0}(\qp)} / T_p H_{m,L} \simeq H_{m,L}[p^{\infty}]$. We consider the natural map
\begin{align}
    \ul{\mathcal{O}(H_{m,L}[p^{\infty}])}^{\ast} \to \ul{\mathcal{O}(H_{m,L})}^{\ast}.
\end{align}
It suffices to identify this map with the inclusion $\mathcal{O}^{\mathrm{sm}}(T_p H_{m,L}^{\vee}) \to \mathcal{O}(T_p H_{m,L}^{\vee})^{\gamma-\mathrm{la}}$, where $\mathcal{O}^{\mathrm{sm}}(T_p H_{m,L}^{\vee})$ denotes the locally constant functions on $T_p H_{m,L}^{\vee}$. This is a direct consequence of the functoriality of the Fourier transform, see \cite[Theorem 2]{FourierPaper}, since locally constant functions correspond to the locally analytic character datum with (in the notation of \cite[Definition 4.3.6]{FourierPaper}) $V=0$. 
\end{proof}

We now deduce Theorem \ref{Thm:IntroInfiniteLevelTheorem} from Theorem \ref{Thm:InfiniteLevelTheorem}. 
\begin{proof}[Proof of Theorem \ref{Thm:IntroInfiniteLevelTheorem}]
Note that $\mathcal{O}(N \backslash P_{\dr,\infty}) \subset \mathcal{O}(N^{\mathrm{an}} \backslash P_{\dr,\infty}^{\mathrm{an}})^{\mathrm{sm}}$. Then, part (1) of Theorem~\ref{Thm:IntroInfiniteLevelTheorem} follows from part (6) of Theorem~\ref{Thm:InfiniteLevelTheorem} and part (2) in Theorem~\ref{Thm:IntroInfiniteLevelTheorem} follows from part (2) in Theorem~\ref{Thm:InfiniteLevelTheorem}. 

Part (3) of Theorem \ref{Thm:IntroInfiniteLevelTheorem} will follow from part (4) of Theorem \ref{Thm:InfiniteLevelTheorem}. For $g \in P_{\overline{\mu}}(\mathbb{Q}_p)$, $s \in \mathcal{O}(\igmf)\widehat{\otimes} C$, and $f \in \mathcal{O}^{\gamma-\locan}(\mf{u}_{\overline{\mu},0}^*, C)$, the Hecke action of $[g]$ on $f \cdot s$ can be described as follows: First we write
\[ U_{\overline{\mu}}(\mathbb{Z}_p) g U_{\overline{\mu}}(\mathbb{Z}_p) = g \left(g^{-1}U_{\overline{\mu}}(\mathbb{Z}_p)g  \cdot U_{\overline{\mu}}(\mathbb{Z}_p)\right) =\sqcup_i g u_i U_{\overline{\mu}}(\mathbb{Z}_p) \]
where the $u_i$ are representatives for $\left(g^{-1}U_{\overline{\mu}}(\mathbb{Z}_p)g  \cdot U_{\overline{\mu}}(\mathbb{Z}_p)\right) / U_{\overline{\mu}}(\mathbb{Z}_p)$. Then, identifying $f$ with its extension by zero to $\mf{u}_{\overline{\mu},0}^*[\tfrac{1}{p}]$ and treating $s$ as an element of $\mathcal{O}(\igcsaL)^{\mathrm{sm}}$ (for $L=C$), 
\begin{align} [g] \cdot (f \cdot s)=\sum_i (gu_i) \cdot (f \cdot s)&= \sum_i (gu_i \cdot f) \cdot ( gu_i \cdot s)\\
& = \sum_i (g \cdot f) \cdot (gu_i \cdot s)\\
& =(g \cdot f) \cdot \left( \sum_i (gu_i) \cdot s \right)\\
&=(g \cdot f) \cdot ([g] \cdot s)\end{align}
The formula of part (3) of Theorem \ref{Thm:IntroInfiniteLevelTheorem} then follows since the action of $g \cdot f$ on $[g] \cdot s$ is the same as the action of $(g \cdot f)1_{\mf{u}_{\ovmu,0}^*}$ as $1_{\mf{u}_{\ovmu,0}^*}$ acts trivially on $[g] \cdot s$ (which is again a function at Mantovan level). 
\end{proof}

\subsection{The ordinary case} \label{sub:Ordinary} Let the notation be as above, and assume that the local reflex field $E$ is equal to $\qp$. In this case, the $\mu$-ordinary locus is equal to the ordinary locus, see \cite[Corollary 1.0.2]{LeeNewton}. 
We have an equality $\overline{\mu}=\mu$ and thus $$\mathcal{M}_{\overline{\mu}}=\mathcal{M}_{\mu} \qquad \mathcal{P}_{\overline{\mu}}=\mathcal{P}_{\mu} \qquad \mathcal{U}_{\overline{\mu}}=\mathcal{U}_{\mu} \qquad b_{\mu}=\mu(p^{-1}).$$ Moreover, the group scheme $\mathcal{N}=\mathcal{M}_{\mu} \cap \mathcal{U}_{\overline{\mu}}$ is trivial and thus $\overline{\mathcal{P}}_{\dr}=\mathcal{P}_{\dr}$. Furthermore, the formal group $\mathcal{H}$ can be identified with
    \begin{align}
        \mathcal{U}_{\mu}(\zp) \otimes_{\zp} \gmhat
    \end{align}
    or $\mf{u}_{\mu} \otimes_{\zp} \gmhat$. Under these identifications, the natural isomorphism $\operatorname{Lie} \mathcal{H} \xrightarrow{\sim} \operatorname{Lie} \mathcal{U}_{\mu}$ is the identity map. Let us write $\mathcal{M}$ for $\mathcal{M}_{\mu}$ and $M$ for $M_{\mu}$ going forward.

\subsubsection{} The automorphism group $\Aut_{\mathcal{G}}(\mbxcan)$ can be identified with $\mathcal{H }\rtimes \ul{\mathcal{M}(\zp)}$, which acts on $\igmf$. Recall the morphism $\igmf \to \mathcal{P}_{\dR}$, and recall the induced pullback map
\begin{align}
    \mathcal{O}(\mathcal{P}_{\dR}) \to \mathcal{O}(\igmf). 
\end{align}
The integral Fourier theory is simply the isomorphism
\begin{align}
    \mathcal{O}(\mathcal{H})^{\ast} = \mathcal{O}(T_p \mathcal{H}^{\vee}) \simeq \cont(\mf{u}_{\mu}^{\ast}, \zp),
\end{align}
and the Hodge--Tate map is the natural map $\ul{\mathfrak{u}_{\mu}}^{\ast} \to \mathfrak{u}_{\mu}^{\ast} \otimes_{\zp} \mathbb{G}_{a,\spf \zp}$, inducing $\HT^{\ast}:\mathfrak{u}_{\mu} \to \cont(\mf{u}_{\mu}^{\ast}, \zp)$. 
The action of $\mathcal{H}$ on $\igmf$ induces a continuous algebra action of $\cont(\mf{u}_{\mu}^{\ast}, \zp)$ on $\mathcal{O}(\igmf)$. We have the following specialization of Corollary \ref{Cor:IntegralFourierConsequence}.
\begin{Cor} \label{Cor:IntegralFourierConsequenceOrdinary}
The continuous algebra action of $\cont(\mf{u}_{\mu}^{\ast}, \zp)$ on $\mathcal{O}(\igmf)$ satisfies the following properties:
\begin{enumerate}
    \item The action map is semilinear for the natural action of $\mathcal{M}(\zp)$ on $\cont(\mf{u}_{\mu}^{\ast}, \zp)$ and $\mathcal{O}(\igmf)$.
    
    \item The action of $\operatorname{Lie} \mathcal{H} \xrightarrow{\sim} \mf{u}_{\mu}$ on $\mathcal{O}(\igmf)$ given by differentiation coincides with the action of $\mf{u}_{\mu}$ via the natural map $\HT^{\ast}: \mf{u}_{\mu} \to \cont(\mf{u}_{\mu}^{\ast}, \zp)$.

    \item Under pullback along $\mathcal{O}(\mathcal{P}_{\dr}) \to \mathcal{O}(\igmf)$, the action of $\mf{u}_{\mu} $ on the source by algebraic Maass--Shimura operators intertwines with the action of $\mf{u}_{\mu}$ on the target described in (2). 
\end{enumerate}
\end{Cor}
We have the following specialization of Corollaries \ref{Cor:IntegralSmooth} and \ref{Cor:BddAction}. Define (where the locally convex inductive limit is in the sense of \cite[Section 11.1]{PGS})
\begin{align}
   \mathcal{O}(\igcsa)^{\opn{sm}}_{\mathrm{lcv}}:= \varinjlim_m^{\mathrm{lcv}} \mathcal{O}(\mf{Ig}_{\mathrm{M},m})[\tfrac{1}{p}],
\end{align}
and topologise $\cont(\mf{u}_{\mu,\qp}^{\ast}, \qp)$ as in Example \ref{Eg:OrdinaryBDD}. Note that $\ul{\mathcal{O}(\igcsa)^{\opn{sm}}_{\mathrm{lcv}}} = \mathcal{O}(\igcsa)^{\opn{sm}}$ by Lemma \ref{Lem:InductiveLimitTopologyI} since the inductive limit is strict with closed transition maps, see \S \ref{subsub:strict}.
\begin{Cor} \label{Cor:BddActionOrdinary} \noindent 
\begin{enumerate}
    \item There is a continuous action of $\cont(\mf{u}_{\mu,\qp}^{\ast}, \zp)$ on $\mathcal{O}(\igcsf)^{\mathrm{sm}}$, such that the action on $\mathcal{O}(\igmf) \subset \mathcal{O}(\igcsf)^{\mathrm{sm}}$ factors through the restriction map $$\cont(\mf{u}_{\mu,\qp}^{\ast}, \zp) \to \cont(\mf{u}_{\mu}^{\ast}, \zp)$$ via the action of Corollary \ref{Cor:IntegralFourierConsequenceOrdinary}. 

    \item There is a continuous action of $\cont(\mf{u}_{\mu,\qp}^{\ast}, \qp)$ on $\mathcal{O}(\igcsa)^{\opn{sm}}_{\mathrm{lcv}}$, such that the action on $\mathcal{O}(\igma) \subset \mathcal{O}(\igcsa)^{\opn{sm}}_{\mathrm{lcv}}$ factors through the restriction map $$\cont(\mf{u}_{\mu,\qp}^{\ast}, \qp) \to \cont(\mf{u}_{\mu}^{\ast}, \qp)$$ via the action of Corollary \ref{Cor:IntegralFourierConsequenceOrdinary}.
\end{enumerate}
\end{Cor}
\begin{proof}
Part (1) is a direct consequence of Corollary \ref{Cor:IntegralSmooth}. For part (2), we simply note that the proof of Corollary \ref{Cor:BddAction} also establishes that there is a continuous action of $\cont(\mf{u}_{\mu,\qp}^{\ast},\qp)$ on
\begin{align}
    \varinjlim_m^{\mathrm{lcv}} \mathcal{O}(\mf{Ig}_{\mathrm{M},m})[\tfrac{1}{p}]=\mathcal{O}(\igcsa)^{\opn{sm}}_{\mathrm{lcv}}.
\end{align}
Indeed, to check that the action map 
\[ a:\cont(\mf{u}_{\mu,\qp}^{\ast},\qp) \times \mathcal{O}(\igcsa)^{\opn{sm}}_{\mathrm{lcv}} \to \mathcal{O}(\igcsa)^{\opn{sm}}_{\mathrm{lcv}}\] is continuous, it suffices to check that $a^{-1}(U)$ is open for $U$ running over a basis of neighborhoods of $0$ in $\mathcal{O}(\igcsa)^{\opn{sm}}_{\mathrm{lcv}}$. By \cite[Theorem 11.1.2]{PGS}, there is such a basis consisting of $\zp$ lattices $U=\cup_m U_m$, with each $U_m$ an open $\zp$ lattice in $\mathcal{O}(\mf{Ig}_{\mathrm{M},m})[\tfrac{1}{p}]$. Now $a^{-1}(U)=\cup_m a^{-1}(U_m)$, and it thus suffices to show that $a^{-1}(U_m)$ is open. But the action of $\cont(\mf{u}_{\mu,\qp}^{\ast},\qp)$ on $U_m$ factors through the (continuous) action of $\cont(\tfrac{1}{p^m} \mf{u}_{\mu}^{\ast},\qp)$, completing the proof.
\end{proof}

\subsubsection{Going to infinite level} We make the choice of $x$ using a fixed basis $\zeta$ of $\zp(1)$ as in Example \ref{Eg:OrdinaryOmegaZeta}, so that $\Omega_{\zeta}$ is the identity. This gives identifications $T_p H_L^{\vee} = \mf{u}_{\mu}^{\ast}$ and $V_p H_L^{\vee}= \mf{u}_{\mu,\qp}^{\ast}$, so that $\mathcal{O}(V_p H_L^{\vee})^{\gamma-\mathrm{la}}$ is simply the space $\mathcal{O}(\mf{u}_{\mu,\qp}^{\ast},L)^{\mathrm{la}}$ of $L$-valued locally analytic functions on $\mf{u}_{\mu,\qp}^{\ast}$. 

\subsubsection{} The action of $\cont(\mf{u}_{\mu,\qp}^{\ast}, \qp)$ on $\mathcal{O}(\igcsa)^{\opn{sm}}_{\mathrm{lcv}}$ induces an action of (where the tensor products are completed projective tensor products)
\begin{align}
    \cont(\mf{u}_{\mu,\qp}^{\ast}, \qp) \widehat{\otimes}_{\qp} L \qquad \text{on} \qquad \mathcal{O}(\igcsa)^{\opn{sm}}_{\mathrm{lcv}} \widehat{\otimes}_{\qp} L.
\end{align}
The natural map of $L$-Fr\'echet spaces $\cont(\mf{u}_{\mu,\qp}^{\ast}, \qp) \widehat{\otimes}_{\qp} L \to \cont(\mf{u}_{\mu,\qp}^{\ast}, L)$ is an isomorphism, this follows from \cite[Proposition 1.1.29]{EmertonLocallyAnalytic}. Similarly, the natural map $\mathcal{O}(\igcsa)^{\opn{sm}}_{\mathrm{lcv}} \widehat{\otimes}_{\qp} L \to \mathcal{O}(\mf{Ig}_{\mathrm{CS}, \eta, L})^{\opn{sm}}_{\mathrm{lcv}}$ is an isomorphism, where $\mathcal{O}(\mf{Ig}_{\mathrm{CS}, \eta, L})^{\opn{sm}}_{\mathrm{lcv}}$ is defined by taking the locally convex direct limit over $\mathcal{O}(\mf{Ig}_{\mathrm{M},m, \eta, L})$, this follows from \cite[Lemma 1.1.30]{EmertonLocallyAnalytic}.

\subsubsection{} We recall the natural map $\gamma^{\ast}:\mf{u}_{\mu,L} \to \cont(\mf{u}_{\mu,\qp}^{\ast},L)$. We have the following variant of Theorem \ref{Thm:InfiniteLevelTheorem}.

\begin{Thm} \label{Thm:InfiniteLevelTheoremOrdinary}
The action of $\cont(\mf{u}_{\mu,\qp}^{\ast},L)$ on $\mathcal{O}(\mf{Ig}_{\mathrm{CS}, \eta, L})^{\opn{sm}}_{\mathrm{lcv}}$ induced by Corollary \ref{Cor:BddActionOrdinary}.(2) after completed tensor product to $L$ has the following properties: 
\begin{enumerate}
    \item The action of $h \in U_{\mu}(\qp)$ on $\mathcal{O}(\mf{Ig}_{\mathrm{CS}, \eta, L})^{\opn{sm}}_{\mathrm{lcv}}$ corresponds to the action of the locally constant function $f_h \in \cont(\mf{u}_{\mu,\qp}^{\ast},L)$.

    \item The action of $\operatorname{Lie} H_{L} \xrightarrow{\sim} \mf{u}_{\mu,L}$ on $\mathcal{O}(\mf{Ig}_{\mathrm{CS}, \eta, L})^{\opn{sm}}_{\mathrm{lcv}}$ given by differentiation coincides with the action of $\mf{u}_{\mu,L}$ via $\gamma^{\ast}$.

   \item The differentiation action of $\operatorname{Lie} H_L \xrightarrow{\sim} \mf{u}_{\mu,L}$ on $\mathcal{O}(\mf{Ig}_{\mathrm{CS}, \eta, L})^{\opn{sm}}_{\mathrm{lcv}}$ intertwines with the action of $\mf{u}_{\mu,L}$ on $\mathcal{O}(P_{\dr,\infty}^{\mathrm{an}})^{\mathrm{sm}}_{\mathrm{lcv}}$ through algebraic Maass--Shimura operators, via the pullback map
    \begin{align}
        \mathcal{O}(P_{\dr,\infty}^{\mathrm{an}})^{\mathrm{sm}}_{\mathrm{lcv}} \to \mathcal{O}(\mf{Ig}_{\mathrm{CS}, \eta, L})^{\opn{sm}}_{\mathrm{lcv}}.
    \end{align}

    \item The action of $\cont(\mf{u}_{\mu,\qp}^{\ast}, L)$ on $\mathcal{O}(\mf{Ig}_{\mathrm{CS}, \eta, L})^{\opn{sm}}_{\mathrm{lcv}}$ is $M(\qp)$-semilinear for the action of $M(\qp)$ on $\mathcal{O}(\mf{Ig}_{\mathrm{CS}, \eta, L})^{\opn{sm}}_{\mathrm{lcv}}$ and $\cont(\mf{u}_{\mu,\qp}^{\ast}, L)$. 

    \item The natural action of $P_{\mu}(\qp)$ on $\mathcal{O}(\mf{Ig}_{\mathrm{CS}, \eta, L})^{\opn{sm}}_{\mathrm{lcv}}$ intertwines with the twisted action of $P_{\mu}(\qp)$ on $\mathcal{O}(P_{\dr,\infty}^{\mathrm{an}})^{\mathrm{sm}}_{\mathrm{lcv}}$ described in \S \ref{subsub:TwistedHeckeAction} via the pullback map
    \begin{align}
        \mathcal{O}(P_{\dr,\infty}^{\mathrm{an}})^{\mathrm{sm}}_{\mathrm{lcv}} \to \mathcal{O}(\mf{Ig}_{\mathrm{CS}, \eta, L})^{\opn{sm}}_{\mathrm{lcv}}.
    \end{align}
 
\end{enumerate}
\end{Thm}
\begin{proof}
All statements follow from Theorem \ref{Thm:InfiniteLevelTheorem} using the natural isomorphisms of $L$-vector spaces\footnote{It is not clear to us if this identification is continuous for the natural topology on the right hand side, see Question \ref{Question:Appendix} in Appendix \ref{appendix.locally-convex-inductive-limits}  for a discussion. This does not affect the proof, since all statements can be checked on the underlying $L$-vector spaces.}
\begin{align}
    \mathcal{O}(P_{\dr,\infty}^{\mathrm{an}})^{\mathrm{sm}}_{\mathrm{lcv}} &\simeq \mathcal{O}(P_{\dr,\infty}^{\mathrm{an}})^{\mathrm{sm}}(\ast) \\
     \mathcal{O}(\mf{Ig}_{\mathrm{CS}, \eta, L})^{\opn{sm}}_{\mathrm{lcv}} &\simeq \mathcal{O}(\igcsdL)^{\opn{sm}}(\ast).
\end{align}
\end{proof}

\subsection{A non-ordinary example} \label{sub:ExampleII} Consider the setup of \S \ref{subsub:PELNotation} and furthermore assume that $\mathsf{F}^+ \not=\mathbb{Q}$ and that the signatures of $\V$ are $(n-1,1),(n,0), \cdots (n,0)$, with $n \ge 2$. Suppose that $p$ is inert in $\mathsf{F}_0$ and splits into $p=\mf{p}_1 \mf{p}_2$ in $\mathsf{F}$. The reflex field $\mathsf{E}$ of $\gx$ is equal to $\mathsf{F}$ in this case (since $\mathsf{F}^+ \not=\mathbb{Q}$). We will now compute $\mathcal{H}$ and $\Omega$ very explicitly in this example: Write $F=\mathsf{F}_{\mathfrak{p}_1}=F_{0,p}$ for the local reflex field, and 
\begin{align}
    \mathcal{O}_F \otimes_{\zp} \mathcal{O}_F = \prod_{i=1}^r \mathcal{O}_F,
\end{align}
with Frobenius sending the $i$-th factor to $i+1$ modulo $r$. 
\subsubsection{} We have $\mathcal{G}=\operatorname{Res}_{\mathcal{O}_{F}/\zp} \operatorname{GL}_{n} \times \mathbb{G}_{m}$. If we write $$\mathcal{G}_{\mathcal{O}_F} \simeq \mathbb{G}_{m} \times \prod_{i=1}^r \operatorname{GL}_{n},$$ then the cocharacter $\mu$ corresponds to $(\mu_1, 0, \cdots, 0)$, where $\mu_1$ is the first fundamental cocharacter of $\operatorname{GL}_{n}$. The parabolic subgroup $\mathcal{P}_{\mu} \subset \mathcal{G}_{\mathcal{O}_{F}}$ is
\begin{align}
    \mathbb{G}_{m} \times \mathcal{P}_{1} \times \prod_{i=2}^r \operatorname{GL}_{n}  \subset \mathbb{G}_{m} \times \prod_{i=1}^r \operatorname{GL}_{n}
\end{align}
and $\mathcal{P}_{\overline{\mu}}$ is $\mathbb{G}_{m} \times \operatorname{Res}_{\mathcal{O}_{F}/\zp} \mathcal{P}_{1}$. The unipotent radical of $P_{1,F}$ can be identified with $\mathcal{O}_{F}^{\oplus n-1}$, on which conjugation by $\mu_1(p^{-1})$ acts by $p^{-1}$. Thus the Dieudonn\'e module of $\mathcal{H}$ is given by $\mathcal{O}_{F}^{\oplus n-1} \otimes_{\zp} \mathcal{O}_F = \prod_{i=1}^r \mathcal{O}_F^{\oplus n-1}$, with Frobenius acting by
\begin{align}
    (v_1, \cdots, v_r) \mapsto (p^{-1}v_r, v_1, \cdots, v_{r-1}).
\end{align}
The admissible filtration corresponds to the subspace $\{0\}\times \prod_{i=2}^r \mathcal{O}_F^{\oplus n-1}$.  We thus see that $\mathbb{D}(\mathcal{H})$ is isomorphic as a filtered covariant Dieudonn\'e module to $\mathbb{D}(\mathscr{G})^{\oplus n-1}$, where $\mathscr{G}$ is the Lubin--Tate formal group over $\spf \mathcal{O}_F$ corresponding to the uniformizer $p \in \mathcal{O}_F$; note that $\mathcal{O}_F$ acts on $\mathscr{G}$. 

\subsubsection{} Let $L$ be the (complete) Lubin--Tate extension of $\mathcal{O}_F$ corresponding to the choice of uniformizer $p$. Then by Lubin--Tate theory, there is an $\mathcal{O}_F$-linear isomorphism $\mathcal{O}_F \xrightarrow{\sim} T_p \mathscr{G}_{\eta,L}$, well defined up to $\mathcal{O}_F^{\times}$. This induces a similar isomorphism $\mathcal{O}_F^{\oplus n-1} \xrightarrow{\sim} T_p \mathcal{H}_{\eta,L}$, defining a point $x$ as in Proposition \ref{Prop:CompatibleChoice}. The ring $\mathcal{O}^\locan(T_p H^\vee, L)^{\gamma-\mathrm{la}}$ can be identified with a (solid) tensor product of $n-1$ copies of the ring of $L$-valued functions on $\mathcal{O}_{F}$ which are $\mathcal{O}_{F}$-analytic via the embedding $w:\mathcal{O}_{F} \to L$. The group $\mc{N}$ is nontrivial, but it acts trivially on $U_{\mu}$ since $\mathcal{U}_{\overline{\mu}}$ is abelian (see the proof of Lemma \ref{Lem:LieAlgebraH}). We thus have $U_{\mu}=U_{\mu}^{\mathcal{N}}$ and $\operatorname{Sym}(\mf{u}_{\mu})=(\operatorname{Sym} \mf{u}_{\mu})^{\mathcal{N}}$.

\subsubsection{}\label{sss.period-description} We now describe the automorphism $\Omega$ of \S \ref{subsub:EtaleCrystalline}. First note that we have the Hodge-Tate comparison for $\mathscr{G}$, defined over $L$ in this case, 
\[ \HT: (\Lie \mathscr{G})\otimes L(1) \xrightarrow{\sim} \Fil^1_{\HT} (T_p \mathscr{G} \otimes L). \]
Then, as in \eqref{eq.HT-comp-Omega}, $\Omega$ is  the composition
\[ \mathcal{O}_F^{n-1} \otimes_{\mathcal{O}_F} L(1)=(\Lie \mathscr{G})^{n-1}\otimes L(1) \xrightarrow{\HT^{\oplus n-1}} \Fil^1_{\HT} (T_p \mathscr{G} \otimes L)^{n-1}=\mathcal{O}_F^{n-1} \otimes L. \]
In particular, if we choose a trivialization $\zeta$ of $\mathbb{Z}_p(1)$, then $\Omega_\zeta$ is multiplication by a Lubin--Tate period in each coordinate: concretely, letting $1_\et \in \mathcal{O}_F=T_p \mathscr{G}$ be the fixed basis vector as a $\mathcal{O}_F$-module, then we obtain a basis element $q(1_\et)$ for 
\[ \Fil^1_{\HT} (T_p \mathscr{G} \otimes_{\mathbb{Z}_p} L) \subseteq   T_p \mathscr{G} \otimes_{\mathbb{Z}_p} L  = \prod_{\mathcal{O}_F \rightarrow L} L\] 
by projection of $1_\et$ to the $w$-component, i.e., by the element in the product that is $1$ for the embedding $w$ and $0$ for all other embeddings. We also have a basis $1_\dR \in \mathcal{O}_F=\Lie \mathscr{G}$, and $\Omega_\zeta$ is multiplication in each coordinate by $\frac{\HT(1_\dR \otimes \zeta)}{q(1_\et)} \in L$. 

\begin{Rem}
    It is not clear to us how $\Omega_{\zeta}$ is related to the Lubin--Tate period of \cite[Appendix]{SchneiderTeitelbaumFourier}.
\end{Rem}

\begin{Rem}
    When $n=2$, the Shimura variety is a unitary Shimura curve. The action of the Lubin--Tate group $\mathscr{G}$ on $\igmf$ should be closely related to the action of $\mathscr{G}$ on $\mathcal{X}'(\infty)$ constructed in \cite[Section 2.5]{LiuZhangZhang}.
\end{Rem}

\subsection{A second non-ordinary example} \label{sub:ExampleIII} We now give a simple example where the natural inclusion $\mf{u}_{\mu}^{\mathcal{N}} \hookrightarrow \mf{u}_{\mu}$ is strict.

\subsubsection{} Consider the setup of \S \ref{subsub:PELNotation} and furthermore assume that $\mathsf{F}^+ =\mathbb{Q}$ and that the signature of $\V$ is $(2,1)$; then $\mathsf{E}=\mathsf{F}$. Suppose that $p$ is inert in $\mathsf{F}$. In this case the canonical lift of the $\mu$-ordinary $p$-divisible group is isomorphic to
\begin{align}
(\qp/\zp \otimes_{\zp} \mathcal{O}_E) \oplus \operatorname{LT} \oplus (\mu_{p^{\infty} } \otimes_{\zp} \mathcal{O}_E),
\end{align}
where $\operatorname{LT}$ is the dimension one height two Lubin--Tate formal group for $E$. We can write its $\mathcal{O}_E$-linear automorphism group as
\begin{align}
\operatorname{Aut}_{\mathcal{O}_E}(\mbxcan)=
\begin{pmatrix}
\ul{\mathcal{O}_E^{\times}} & \hom_{}(\qp/\zp, \operatorname{LT}) & T_p \mu_{p^{\infty}} \otimes_{\zp} \mathcal{O}_E \\
0 & \ul{\mathcal{O}_E^{\times}} & \hom_{\mathcal{O}_E}(\operatorname{LT}, \mu_{p^{\infty}} \otimes_{\zp} \mathcal{O}_E) \\
0 & 0 & \ul{\mathcal{O}_E^{\times}}
\end{pmatrix}.
\end{align}
It follows that we have a noncommutative extension
\begin{align}
1 \to T_p \mu_{p^{\infty}} \to \operatorname{Aut}_{\mathcal{G}}(\mbxcan)^{\circ} \to  T_p \operatorname{LT} \to 1.
\end{align}

\subsubsection{} In this case $\mathcal{H}=\gmhat$, and the natural map $\Lie \mathcal{H} \xrightarrow{\sim} \mf{u}_{\mu}^{\mathcal{N}} \to \mf{u}_{\mu}$ is not surjective, because the latter is two-dimensional.


\section{Near overconvergence of the action} \label{Sec:NearlyOC}

In this section, we discuss the near overconvergence of the action in Corollary \ref{Cor:IntegralFourierConsequenceOrdinary} in the ordinary case, following \cite{GrahamPilloniRodriguesJacinto}. We will use locally convex functional analysis over $\qp$ as in Section \ref{sub:Ordinary}. We will write $\mathcal{O}^{\opn{la}}(\mathfrak{u}_{\mu}^*, \mbb{Q}_p)$ for the locally convex $\qp$ vector space of locally analytic functions on $\mathfrak{u}_{\mu}^*$. After passing to condensed $\ul{\qp}$-vector spaces, this recovers the previous definition by Lemma \ref{Lem:InductiveLimitTopologyII} (see \S \ref{subsub:compactoid} and the proof of \cite[Lemma 6.5.2]{FourierPaper} for the compactoidness of the transition maps).

\subsection{Main result} Suppose that we are in the setting of \S \ref{subsub:FramingObject} and suppose moreover that $E=\qp$ so that $\overline{\mu}=\mu$ and $\mathcal{N}=\{1\}$. Let $P^{\opn{an}}_{\opn{dR}}$ be the analytification of the generic fiber of $\pdr$, which receives a morphism 
$\mathfrak{Ig}_{M, \eta} \to P_{\opn{dR}}^{\opn{an}}$. Let 
\begin{align}
    \mathscr{N}^{\dagger} := \varinjlim^{\mathrm{lcv}}_U \opn{H}^0(U, \mathcal{O}_{P^{\opn{an}}_{\opn{dR}}}),
\end{align}
where the locally convex inductive limit (see \cite[Section 11.1]{PGS}) is over all open neighborhoods $U \subset P^{\opn{an}}_{\opn{dR}}$ containing the closure of $\mathfrak{Ig}_{M, \eta}$. We refer to $\mathscr{N}^{\dagger}$ as the space of nearly overconvergent automorphic forms. 

Consider $\mathcal{M}_{\dr} = \mathcal{P}_{\dr} \times^{\mathcal{P}_{\mu^{-1}}} \mathcal{M}_{\mu^{-1}}$ with generic fiber $M_{\dr}$. We define the space of overconvergent automorphic forms as the locally convex inductive limit
\begin{align}
    \mathscr{M}^{\dagger} := \varinjlim^{\mathrm{lcv}}_U \opn{H}^0(U, \mathcal{O}_{M^{\opn{an}}_{\opn{dR}}}),
\end{align}
over all neighborhoods $U \subset M_{\opn{dR}}^{\opn{an}}$ containing the closure of $\mathfrak{Ig}_{M, \eta}$ (via the natural map $\mathfrak{Ig}_{M, \eta} \to M_{\opn{dR}}^{\opn{an}}$). We have a natural inclusion $\mathscr{M}^{\dagger} \subset \mathscr{N}^{\dagger}$ induced from pullback along $P_{\opn{dR}}^{\opn{an}} \to M_{\opn{dR}}^{\opn{an}}$.

\begin{Thm} \label{Thm:NOCmainThm}
The $\qp$-vector space $\mathscr{N}^{\dagger}$ is an LB-space of compact type which is a topological $(\mathfrak{g}, \mathcal{M}_{\mu}(\mbb{Z}_p))$-module, and comes equipped with a continuous $\mathcal{M}_{\mu}(\mbb{Z}_p)$-semilinear algebra action of $\mathcal{O}^{\opn{la}}(\mathfrak{u}_{\mu}^*, \mbb{Q}_p)$ on $\mathscr{N}^{\dagger}$ with the following properties:
    \begin{enumerate}
        \item The action of $\opn{Sym} \mathfrak{u}_{\mu}$ on $\mathscr{N}^{\dagger}$ through its $\mathfrak{g}$-module structure coincides with the action through the map $\opn{Sym} \mathfrak{u}_{\mu} \subset \mathcal{O}^{\opn{la}}(\mathfrak{u}_{\mu}^*, \mbb{Q}_p)$.
        \item The natural pullback map $\mathscr{N}^{\dagger} \to \mathcal{O}(\mathfrak{Ig}_{M, \eta})$ is $\mathcal{O}^{\opn{la}}(\mathfrak{u}_{\mu}^*, \mbb{Q}_p)$-equivariant.
        \item The natural map $\mathcal{O}(P^{\opn{an}}_{\opn{dR}}) \to \mathscr{N}^{\dagger}$ is $(\mathfrak{g}, \mathcal{M}_{\mu}(\mbb{Z}_p))$-equivariant.
        \item $\mathscr{M}^{\dagger} \subset \mathscr{N}^{\dagger}$ is identified with the subspace of nearly overconvergent automorphic forms killed by the action of $\mathfrak{u}_{\mu^{-1}}$.
    \end{enumerate}
\end{Thm}

The LB-space $\mathscr{N}^{\dagger}$ is important for the construction of $p$-adic $L$-functions. Moreover, because $\mathscr{N}^{\dagger}$ is of compact type, there is a natural isomorphism $$ \varinjlim_U \ul{\opn{H}^0(U, \mathcal{O}_{P^{\opn{an}}_{\opn{dR}}})} \xrightarrow{\sim} \ul{\mathscr{N}^{\dagger}}$$ of condensed $\qp$-vector spaces by Lemma \ref{Lem:InductiveLimitTopologyII}, see \S \ref{subsub:compactoid}. 

\subsection{The Siegel case}

We first discuss the proof of Theorem \ref{Thm:NOCmainThm} in the setting where $\gx=\gvx$ equals the standard Siegel datum for $\mathsf{G} = \opn{GSp}_{2g}$. The $g=1$ case is \cite{GrahamPilloniRodriguesJacinto}. Write $\scrshat=\scrshat_U\gvx$. Throughout, let $\opn{Spf}A^+ \subset (\scrshat)^{[b_{\mu}]}$ be an open affine over which $\mathcal{P}_{\opn{dR}}$ has a section. Let $\opn{Spa}(A, A^+)$ denote its adic generic fiber (note $A = A^+[1/p]$).

\subsubsection{Locally analytic functions}
Recall that $\mathcal{O}^{\opn{la}}(\mathfrak{u}_{\mu}^*, \mbb{Q}_p)$ denotes the space of locally analytic functions. More precisely, we consider this as the locally convex inductive limit of the Banach spaces of $\varepsilon$-analytic functions, as $\varepsilon \to 0$. For convenience, we fix coordinates for $\mathfrak{u}_{\mu}^*$. More precisely, set $d = g(g+1)/2$ and fix an identification $\mathfrak{u}_{\mu}^* \simeq \mbb{Z}_p^{d}$. Then by Mahler's theorem, $\Cont(\mathfrak{u}_{\mu}^*, \mbb{Q}_p)$ is identified with all multivariable power series $f(X_1, \dots, X_{d})$
\[
\sum_{k_1, \dots, k_d=0}^{\infty} a_{\ul{k}} \cdot \bincoeff{X_1}{k_1} \cdots \bincoeff{X_d}{k_d}  
\]
where $\ul{k} = (k_1, \dots, k_d)$ and $a_{\ul{k}} \in \mbb{Q}_p$, such that $|a_{\ul{k}}| \to 0$ as $k_1 + \cdots + k_d \to +\infty$. For a real number $\varepsilon > 0$, we let $C_{\varepsilon}(\mathfrak{u}_{\mu}^*, \mbb{Q}_p) \subset \Cont(\mathfrak{u}_{\mu}^*, \mbb{Q}_p)$ denote the Banach space of functions such that $p^{(k_1+ \cdots + k_d)\varepsilon}|a_{\ul{k}}| \to 0$ as $k_1 + \cdots + k_d \to +\infty$. This is independent of the choice of coordinates, and we have $\mathcal{O}^{\opn{la}}(\mathfrak{u}_{\mu}^*, \mbb{Q}_p) = \varinjlim^{\mathrm{lcv}}_{\varepsilon} C_{\varepsilon}(\mathfrak{u}_{\mu}^*, \mbb{Q}_p)$.

\subsubsection{Computations over the ordinary locus}

We introduce some notation. Let $\mathcal{P}_{\opn{dR}}^{\opn{ord}} = \mathcal{P}_{\opn{dR}} \times_{\widehat{\mathscr{S}}} (\scrshat)^{[b_{\mu}]}$, and note that $\mathfrak{Ig}_M \to \mathcal{P}_{\opn{dR}}^{\opn{ord}}$ is a reduction of structure. For an integer $n \geq 1$, let $\mathcal{P}_{\mu^{-1}, n} \subset \mathcal{P}_{\mu^{-1},\eta}$ denote the affinoid subgroup of elements which reduce to the identity modulo $p^n$ and, additionally, lie in $\mathcal{M}_{\mu}$ modulo $p^{2n}$. Set $\mathcal{P}_{\square, n} := \mathcal{M}_{\mu}(\mbb{Z}_p) \mathcal{P}_{\mu^{-1}, n} \subset \mathcal{P}_{\mu^{-1},\eta}$. Let 
\[
P^{\opn{ord}}_{\opn{dR}, n} := \mathfrak{Ig}_{M, \eta} \times^{\mathcal{M}_{\mu}(\mbb{Z}_p)} \mathcal{P}_{\square, n} 
\]
denote the pushout, which is an \'{e}tale $\mathcal{P}_{\square, n}$-torsor over $\mathcal{S}_{\mu} := (\scrshat^{[b_{\mu}]})_{\eta}$. Let $\mathfrak{Ig}_{M, n}$ denote the quotient of $\mathfrak{Ig}_M$ by $\mathcal{M}_{\mu}(\mbb{Z}_p) \cap \mathcal{P}_{\mu^{-1}, n}$ (not to be confused with $\mathfrak{Ig}_{M,m}$ introduced in \S \ref{sub:IntegralFourier}); then the space $P^{\opn{ord}}_{\opn{dR}, n}$ naturally lives over $\mathfrak{Ig}_{M, n, \eta}$ as a $\mathcal{P}_{\mu^{-1},n}$-torsor. Set $P^{\opn{univ}}_{\opn{dR}, n} = P^{\opn{ord}}_{\opn{dR}, n} \times_{\mathfrak{Ig}_{M, n, \eta}} \mathfrak{Ig}_{M,\eta}$, an \'{e}tale $\mathcal{P}_{\mu^{-1}, n}$-torsor over $\mathfrak{Ig}_{M, \eta}$.

Note that Proposition \ref{Prop:TPdrToLieGisaniso} induces an isomorphism $T_{P^{\opn{ord}}_{\opn{dR}, n}} \simeq \mathfrak{g} \otimes \mathcal{O}_{P^{\opn{ord}}_{\opn{dR}, n}}$ because $P^{\opn{ord}}_{\opn{dR}, n}$ is an open subspace of $P_{\opn{dR}}^{\opn{an}}$. Similarly $T_{P^{\opn{univ}}_{\opn{dR}, n}} \simeq \mathfrak{g} \otimes \mathcal{O}_{P^{\opn{univ}}_{\opn{dR}, n}}$ because $P^{\opn{univ}}_{\opn{dR}, n} \to P^{\opn{ord}}_{\opn{dR}, n}$ is pro-(finite-\'{e}tale). As in Lemma \ref{Lem:IdentificationII}, for any $w \in \mathfrak{g} \otimes \mathcal{O}_{P^{\opn{univ}}_{\opn{dR}, n}}$ we let $\partial_w \colon P^{\opn{univ}}_{\opn{dR}, n}[\epsilon] \to P^{\opn{univ}}_{\opn{dR}, n}$ denote the corresponding vector field. If $w \in \mathfrak{u}_{\mu}$, then we let $\theta_w \colon \mathfrak{Ig}_{M}[\epsilon] \to \mathfrak{Ig}_{M}$ denote the vector field obtained from differentiating the action of $\opn{Lie}\mathcal{H}$ on $\mathfrak{Ig}_M$.

\begin{Lem} \label{Lem:PartialWequivThetaW}
    Let $w \in \mathfrak{u}_{\mu}$. The vector field $\partial_w$ is integral, i.e., its associated derivation on $\mathcal{O}_{P^{\opn{univ}}_{\opn{dR}, n}}$ preserves $\mathcal{O}^+_{P^{\opn{univ}}_{\opn{dR}, n}}$. Moreover, the section $\mathfrak{Ig}_{M, \eta} \to P_{\opn{dR}}^{\opn{an}}$ induces a trivialization
    \[
    P^{\opn{univ}}_{\opn{dR}, n} \simeq \mathfrak{Ig}_{M, \eta} \times \mathcal{P}_{\mu^{-1}, n} 
    \]
    and $\partial_w \equiv \theta_w \times 1$ modulo $p^n$ under this identification.
\end{Lem}
\begin{proof}
    Let $s \colon \mathfrak{Ig}_{M, \eta} \to P^{\opn{univ}}_{\opn{dR}, n}$ denote the section 
    induced from the section $\mathfrak{Ig}_{M, \eta} \to P_{\opn{dR}}^{\opn{an}}$ constructed before Proposition \ref{Prop:MapToPdr} (which we can think of as arising from the unit root splitting). Clearly this induces the above trivialization. By Theorem \ref{Thm:IdentificationAction}, we have $ds(\theta_w) = s^* \partial_w$. More generally, for any $g \in \mathcal{P}_{\mu^{-1}, n}(\mathfrak{Ig}_{M, \eta})$, we have 
    \begin{equation} \label{Eqn:Adgpartialrelation}
    (g \cdot -) \circ \partial_w = \partial_{\opn{Ad}(g)w} \circ (g \cdot -)
    \end{equation}
    hence $(g \cdot s)^* \partial_{\opn{Ad}(g)w} = g \cdot ds(\theta_w)$. This implies that $\theta_w \times 1$ is identified with the vector field
    \begin{align*}
    (\mathfrak{Ig}_{M, \eta} \times \mathcal{P}_{\mu^{-1}, n})[\epsilon] &\to \mathfrak{Ig}_{M, \eta} \times \mathcal{P}_{\mu^{-1}, n} \\
    (x, g) &\mapsto \partial_{\opn{Ad}(g)w}(x)
    \end{align*}
    via the trivialization $P^{\opn{univ}}_{\opn{dR}, n} \simeq \mathfrak{Ig}_{M, \eta} \times \mathcal{P}_{\mu^{-1}, n}$. Since $\theta_w \times 1$ is integral, we see that $\partial_{\opn{Ad}(g)w}$ is integral, and hence $\partial_w$ is integral by \eqref{Eqn:Adgpartialrelation}.

    Note that $g \mapsto \opn{Ad}(g)w$ defines a section of $\mathfrak{g} \otimes \mathcal{O}^+_{\mathcal{P}_{\mu^{-1}, n}}$ (because $w \in \mathfrak{u}_{\mu}$ is integral). Consider the representation $\rho \colon \mathcal{P}_{\mu^{-1},n} \to \opn{GL}(\mathfrak{u}_{\mu} \otimes \mathcal{O}^+_{\mathcal{P}_{\mu^{-1}, n}})$ given by transporting the representation $\opn{Ad}$ via the identification $\mathfrak{g}/\mathfrak{p}_{\mu^{-1}} \simeq \mathfrak{u}_{\mu}$. We have a decomposition
    \[
    \opn{Ad}(g)w = \rho(g) \cdot w + [\opn{Ad}(g)w]
    \]
    where $[\opn{Ad}(g)w] \in \mathfrak{p}_{\mu^{-1}} \otimes \mathcal{O}^+_{\mathcal{P}_{\mu^{-1}, n}}$ is the image of $\opn{Ad}(g)w$ under the splitting 
    \[
    \mathfrak{g} \otimes \mathcal{O}^+_{\mathcal{P}_{\mu^{-1}, n}} \twoheadrightarrow \mathfrak{p}_{\mu^{-1}} \otimes \mathcal{O}^+_{\mathcal{P}_{\mu^{-1}, n}}
    \]
    induced from the decomposition $\mathfrak{g} = \mathfrak{u}_{\mu} \oplus \mathfrak{p}_{\mu^{-1}}$. Since $\partial_{w'}$ is integral for any $w' \in \mathfrak{u}_{\mu}$ (by above), we see that $\partial_{\rho(g)\cdot w}$ is integral and $\partial_{\rho(g)\cdot w} \equiv \partial_w$ modulo $p^n$ (because $\rho(g) \equiv 1$ modulo $p^n$). Furthermore, $\partial_{[\opn{Ad}(g)w]}$ is integral because it is equal to $(\theta_w \times 1) - \partial_{\rho(g) \cdot w}$.  

    To conclude the proof of the lemma, we claim that the vector field $\partial_{[\opn{Ad}(g)w]}$ is congruent to the trivial vector field modulo $p^n$. Note that $\partial_{[\opn{Ad}(g)w]}$ is in the vertical bundle of the (trivial) torsor $\mathfrak{Ig}_{M, \eta} \times \mathcal{P}_{\mu^{-1}, n}$ hence is obtained by differentiating the $\mathcal{P}_{\mu^{-1}, n}$-action on the second factor.
    
    For the rest of the proof, all matrices $\smat{* & * \\ * & *}$ denote block $2g \times 2g$-matrices with block sizes $g \times g$. Identify $\mathfrak{m}_{\mu^{-1}} \subset \mathfrak{p}_{\mu^{-1}}$ with sub-$\mbb{Z}_p$-Lie algebras $V' \subset V \subset M_{2g \times 2g}$, where $M_{2g \times 2g}$ denotes the $\mbb{Z}_p$-algebra of $2g \times 2g$-matrices, by fixing the symplectic pairing $J = \smat{ & 1 \\ -1 & }$. In particular, all elements in $V$ (resp. $V'$) are of the form $\smat{* & \\ * & *}$ (resp. $\smat{* & \\ & *}$). The isomorphism 
    \[
    \mathfrak{p}_{\mu^{-1}} \otimes \mathcal{O}_{\mathcal{P}_{\mu^{-1}, n}} = \mathfrak{p}_{\mu^{-1}} \otimes \mathcal{O}_{P^{\opn{an}}_{\mu^{-1}}}|_{\mathcal{P}_{\mu^{-1}, n}} \simeq T_{P_{\mu^{-1}}^{\opn{an}}}|_{\mathcal{P}_{\mu^{-1}, n}} = T_{\mathcal{P}_{\mu^{-1}, n}}
    \]
    identifies the $\opn{Spa}(R, R^+)$-points of $\opn{Lie}^+\mathcal{P}_{\mu^{-1}, n}$ (the integral Lie algebra) with
    \[
    \{ v \in V(R^+) : v \equiv 0 \text{ modulo } p^n, v \text{ mod } p^{2n} \in V'(R^+/p^{2n}) \} \subset \mathfrak{p}_{\mu^{-1}}(R^+) 
    \]
    for any $p$-adically complete, separated, $p$-torsion free $\mbb{Z}_p$-algebra $R^+$. 
    
    To conclude the proof, it therefore suffices to show that 
    \[
    [\opn{Ad}(g)w] \in p^{n}\opn{Lie}^+\mathcal{P}_{\mu^{-1}, n}(\opn{Spa}(R, R^+)) 
    \]
    for any $g \in \mathcal{P}_{\mu^{-1}, n}(\opn{Spa}(R, R^+))$ with $R^+$ being $p$-adically complete, separated, and $p$-torsion free, i.e., we need to show that for any $g \in \mathcal{P}_{\mu^{-1}}(R^+)$ that is congruent to the identity modulo $p^n$ and lies in $\mathcal{M}_{\mu}(R^+/p^{2n})$ modulo $p^{2n}$, the element $[\opn{Ad}(g)w] \in V(R^+)$ is zero modulo $p^{2n}$ and lies in $V'(R^+/p^{3n})$ modulo $p^{3n}$. 
    
    Since $\opn{Ad}(\mathcal{M}_{\mu})(\mathfrak{u}_{\mu}) \subset \mathfrak{u}_{\mu}$, without loss of generality, we may assume that $g = \smat{1 & \\ Y & 1}$ is unipotent, where $Y$ has entries in $p^{2n}R^+$. Write $w = \smat{ 0 & W \\ 0 & 0}$ with $W$ having entries in $R^+$. Then an easy computation shows
    \[
    [\opn{Ad}(g)w] = \smat{-WY & 0 \\ -Y W Y & YW}
    \]
    which lies in $p^{n}\opn{Lie}^+\mathcal{P}_{\mu^{-1}, n}(\opn{Spa}(R, R^+))$ as required. This completes the proof.
\end{proof}

As a consequence of this lemma, we obtain the following:

\begin{Prop} \label{Prop:OrdinarySiegelOCcomp}
    Let $\varepsilon > 0$. Then there exists $n(\varepsilon) \geq 1$ such that for all $n \geq n(\varepsilon)$, there exists a unique continuous $\mbb{Q}_p$-algebra action 
    \[
    C_{\varepsilon}(\mathfrak{u}_{\mu}^*, \mbb{Q}_p) \times \opn{H}^0(\opn{Spa}(A, A^+), \mathcal{O}_{P^{\opn{ord}}_{\opn{dR}, n}/\mathcal{S}_{\mu}}) \to \opn{H}^0(\opn{Spa}(A, A^+), \mathcal{O}_{P^{\opn{ord}}_{\opn{dR}, n}/\mathcal{S}_{\mu}})
    \]
    extending the action of $\mathfrak{u}_{\mu}$ (through the inclusion $\mathfrak{u}_{\mu} \subset C_{\varepsilon}(\mathfrak{u}_{\mu}^*, \mbb{Q}_p)$).
\end{Prop}
\begin{proof}
    Let $\opn{Spa}(A_n, A_n^+)$ denote the pullback of $\opn{Spa}(A, A^+)$ under the finite \'{e}tale morphism $\mathfrak{Ig}_{M, n, \eta} \to \mathcal{S}_{\mu}$, and we let $\opn{Spa}(A_{\infty}, A_{\infty}^+)$ denote the pullback of $\opn{Spa}(A, A^+)$ under $\mathfrak{Ig}_{M, \eta} \to \mathcal{S}_{\mu}$ (it is affinoid because $\igmf \to (\scrshat)^{[b_{\mu}]}$ is pro-(finite-\'{e}tale)). The natural pullback map
    \begin{equation} \label{eq:isometry}
    \opn{H}^0(\opn{Spa}(A, A^+), \mathcal{O}_{P^{\opn{ord}}_{\opn{dR}, n}/\mathcal{S}_{\mu}}) \to \opn{H}^0(\opn{Spa}(A, A^+), \mathcal{O}_{P^{\opn{ord}}_{\opn{dR}, n}/\mathcal{S}_{\mu}}) \widehat{\otimes}_{A_n} A_{\infty}
    \end{equation}
    is an isometry because $\mathfrak{Ig}_{M} \to (\scrshat)^{[b_{\mu}]}$ is pro-(finite-\'{e}tale) integrally. As above, the action of $\mathfrak{u}_{\mu}$ extends to the right-hand side, again because the torsor is pro-finite-\'{e}tale. 

    Let $\theta_w$ denote the derivation on $A_{\infty}$ which corresponds to the action of $w \in \mathfrak{u}_{\mu}$ via Corollary \ref{Cor:IntegralFourierConsequenceOrdinary}. Using the universal trivialization in Lemma \ref{Lem:PartialWequivThetaW}, we obtain an identification 
    \[
    \opn{H}^0(\opn{Spa}(A, A^+), \mathcal{O}_{P^{\opn{ord}}_{\opn{dR}, n}/\mathcal{S}_{\mu}}) \widehat{\otimes}_{A_n} A_{\infty} \simeq \mathcal{O}(\mathcal{P}_{\mu^{-1}, n, A_{\infty}}) = A_{\infty} \widehat{\otimes}_{\mbb{Q}_p} \mathcal{O}(\mathcal{P}_{\mu^{-1}, n}) .
    \]
    By Lemma \ref{Lem:PartialWequivThetaW}, the action of $\partial_{w}$ is integral and satisfies
    \[
    \partial_w \cdot f \equiv \theta_w \cdot f 
    \]
    modulo $p^n \mathcal{O}^+(\mathcal{P}_{\mu^{-1}, n, A_{\infty}})$, where $f \in \mathcal{O}^+(\mathcal{P}_{\mu^{-1}, n, A_{\infty}})$. Since $\theta_w$ extends to a continuous algebra action of $\mathcal{O}^{\opn{la}}(\mbb{Z}_p, \mbb{Q}_p)$ on $\mathcal{O}(\mathcal{P}_{\mu^{-1}, n, A_{\infty}})$ (Corollary \ref{Cor:IntegralFourierConsequenceOrdinary}, which we translate to the world of locally convex functional analysis using the discussion in \cite[Section 6.5]{FourierPaper}), we see that by \cite[Lemma 3.2.1]{UFJ} there exists $n(\varepsilon) \geq 1$ such that $\partial_w$ extends to a continuous algebra action of $C_{\varepsilon}(\mbb{Z}_p, \mbb{Q}_p)$ on $\mathcal{O}(\mathcal{P}_{\mu^{-1}, n, A_{\infty}})$ for any $n \geq n(\varepsilon)$. Since \eqref{eq:isometry} is an isometry, we see that the same is true for the action of $\partial_{w}$ on $\opn{H}^0(\opn{Spa}(A, A^+), \mathcal{O}_{P^{\opn{ord}}_{\opn{dR}, n}/\mathcal{S}_{\mu}})$.

    This completes the proof of the proposition because 
    \[
    C_{\varepsilon}(\mathfrak{u}_{\mu}^*, \mbb{Q}_p) \simeq C_{\varepsilon}(\mbb{Z}_p, \mbb{Q}_p)^{\widehat{\otimes} d} 
    \]
    using the fixed basis $\mathfrak{u}_{\mu}^* \cong \mbb{Z}_p^{\oplus d}$. The resulting action is unique because it can be computed using Mahler expansions.
\end{proof}

\begin{Rem}\label{remark.gprj-comparison}
    In \cite{GrahamPilloniRodriguesJacinto}, one worked with reductions of structure of $P_{\opn{dR}}^{\opn{an}}$ to \'{e}tale torsors for the affinoid subgroup of $\mathcal{P}_{\mu^{-1}, \eta}$ of elements which reduce to the identity modulo $p^n$. With the benefit of hindsight, it is better to reduce further to the smaller groups $\mathcal{P}_{\mu^{-1}, n}$ since this avoids the explicit computations involving coordinates in Proposition 5.2.1 \emph{op.cit.}.
\end{Rem}

\begin{Rem} \label{Rem:GrowthOfAction}
    Let $C_{\varepsilon}(\mathfrak{u}_{\mu}^*, \mbb{Q}_p)^{\circ}$ denote the unit ball of $C_{\varepsilon}(\mathfrak{u}_{\mu}^*, \mbb{Q}_p)$ with respect to the Banach norm $|\!|-|\!|_{\varepsilon}$ given by:
    \[
    \left|\!\left| \sum_{\ell_1, \dots, \ell_d=0}^{\infty} a_{\ul{\ell}} \cdot \bincoeff{X_1}{\ell_1} \cdots \bincoeff{X_d}{\ell_d}   \right|\!\right|_{\varepsilon} = \opn{sup}_{\ul{\ell}}(p^{(\ell_1+ \dots +\ell_d)\varepsilon}|a_{\ul{\ell}}|) .
    \]
    For applications to $p$-adic $L$-functions (in particular, understanding the growth of such $p$-adic $L$-functions), it is important to understand when $C_{\varepsilon}(\mathfrak{u}_{\mu}^*, \mbb{Q}_p)^{\circ}$ maps $\opn{H}^0(\opn{Spa}(A, A^+), \mathcal{O}^+_{P^{\opn{ord}}_{\opn{dR}, n}/\mathcal{S}_{\mu}})$ into itself via the action in Proposition \ref{Prop:OrdinarySiegelOCcomp}. If we let $n(\varepsilon) \geq 1$ denote the smallest integer for which this happens, then one can show that
    \[
    n(\varepsilon) = -[\opn{log}_p(\varepsilon)] + O(1)
    \]
    as $\varepsilon \to 0$, where $\opn{log}_p$ denotes the logarithm in base $p$ and $[-]$ denotes the integer part.

    Indeed, if we let $T = \partial_w$ and $T_1 = \theta_w$, then in the notation of \cite[Proposition 4.2.4]{GrahamPilloniRodriguesJacinto} this amounts to understanding when $p^{-k \varepsilon} |\!|f_k(T)|\!| \leq 1$ for all $k \geq 0$, where $|\!|-|\!|$ denotes the norm on $\opn{H}^0(\opn{Spa}(A, A^+), \mathcal{O}_{P^{\opn{ord}}_{\opn{dR}, n}/\mathcal{S}_{\mu}})$. Since $|\!|f_k(\theta_w)|\!| \leq 1$ for all $k$, this implies that $|\!|z_I|\!| \leq 1$ for all $I$. We therefore see from \cite[(4.2.5)]{GrahamPilloniRodriguesJacinto} that\footnote{Here we use the fact that $\binom{k}{a}\binom{a}{k_1, \dots, k_r} = \binom{k}{k-a, k_1, \dots, k_r}$ to simplify the expression slightly.}
    \[
    p^{-k \varepsilon}|\!|f_k(T)|\!| \leq \opn{max}(1, p^{\opn{log}_p(k)-k\varepsilon - n(\varepsilon) + 1/(p-1)}) .
    \]
    We must therefore take $n(\varepsilon)$ such that $\opn{log}_p(k)-k\varepsilon - n(\varepsilon) + 1/(p-1) \leq 0$. Taking $n(\varepsilon) = -[\opn{log}_p(\varepsilon)] + C$, with $C$ a constant independent of $\varepsilon$, clearly works.
\end{Rem}

\subsubsection{Overconvergence in the Shimura variety direction}

We now overconverge along the embedding $(\scrshat)^{[b_{\mu}]} \subset \widehat{\mathscr{S}}$. To ease notation, let $\mathcal{S}_{\mu} \subset \mathcal{S}$ denote the adic generic fiber of this embedding. Let $\opn{Spf}B^+ \subset \widehat{\mathscr{S}}$ be an open affine over which $\mathcal{P}_{\opn{dR}}$ has a section. Let $U=\opn{Spa}(B, B^+) \subset \mathcal{S}$ denote its adic generic fiber. Let $h \in B^+$ denote a local lift of the Hasse invariant, and note that 
\[
\opn{Spa}(B^{\opn{ord}}, B^{+, \opn{ord}}) := \opn{Spa}(B\langle 1/h \rangle, B^+\langle 1/h \rangle) \subset \mathcal{S}_{\mu}
\]
is the pullback of $\opn{Spa}(B, B^+)$. Let $\mathcal{S}_r$ denote the locus where $|h|^{p^{r+1}} \geq |p|$ for each local lift $h$ of the Hasse invariant, and set $U_r = U \cap \mathcal{S}_r$ which is affinoid of the form $\opn{Spa}(B_r, B_r^+)$.

\begin{Lem}
    The torsor $P_{\opn{dR}, n}^{\opn{ord}}$ overconverges, that is, there exists an integer $r \geq 1$ (depending on $n$) and an \'{e}tale $\mathcal{P}_{\square, n}$-torsor $P_{\opn{dR}, n, r} \to \mathcal{S}_r$ which fits into a Cartesian diagram:
    \[
\begin{tikzcd}
{P_{\opn{dR}, n}^{\opn{ord}}} \arrow[d] \arrow[r] & {P_{\opn{dR}, n, r}} \arrow[d] \\
\mathcal{S}_{\mu} \arrow[r]                       & \mathcal{S}_r                 
\end{tikzcd}
    \]
    where the top vertical arrow is $\mathcal{P}_{\square, n}$-equivariant.
\end{Lem}
\begin{proof}
    This follows from exactly the same argument as in \cite[Proposition 6.2.1]{GrahamPilloniRodriguesJacinto} noting that $\mathfrak{Ig}_{M, n}$ overconverges (because the canonical subgroup overconverges). 
\end{proof}

We now consider the following chain of Banach spaces 
\[
V_r \to V_{r+1} \to \cdots \to V_{\infty}
\]
given by $V_s = \opn{H}^0(\opn{Spa}(B_s, B_s^+), \mathcal{O}_{P_{\opn{dR}, n, s}/\mathcal{S}_s})$ (with lattice $\opn{H}^0(\opn{Spa}(B_s, B_s^+), \mathcal{O}^+_{P_{\opn{dR}, n, s}/\mathcal{S}_s})$) and
\[
V_{\infty} = \opn{H}^0(\opn{Spa}(B^{\opn{ord}}, B^{+, \opn{ord}}), \mathcal{O}_{P_{\opn{dR}, n}^{\opn{ord}}/\mathcal{S}_{\mu}})
\]
Let $|\!|-|\!|_s$ denote the Banach norm on $V_s$. By exactly the same argument as in \cite[Lemma 7.2.3]{UFJ}, we have the following property of this chain: For any real number $0 < \delta < 1$, there exists an integer $s = s(\delta) \geq r$ such that for all $v \in V_r$, $m \in \mbb{N}$, $c \in \mbb{Q}$, one has
\[
|\!| v |\!|_{\infty} \leq p^{c-m} \text{ and } |\!| v |\!|_r \leq p^c \; \Rightarrow \; |\!|v|\!|_s \leq p^{c-\delta m} .
\]
We obtain the following:

\begin{Prop} \label{Prop:OCII}
    Let $\varepsilon > 0$. Then there exists an integer $n(\varepsilon) \geq 1$ such that for all $n \geq n(\varepsilon)$ the action of $\mathfrak{u}_{\mu}$ on $\mathcal{O}_{P_{\opn{dR}}^{\opn{an}}/\mathcal{S}}$ extends to a unique algebra action of $C_{\varepsilon}(\mathfrak{u}_{\mu}^*, \mbb{Q}_p)$ on 
    \[
    \varinjlim_s \opn{H}^0(\opn{Spa}(B_s, B_s^+), \mathcal{O}_{P_{\opn{dR}, n, s}/\mathcal{S}_s}) .
    \]
\end{Prop}
\begin{proof}
    This follows by combining Proposition \ref{Prop:OrdinarySiegelOCcomp} with \cite[Proposition 3.4.1]{UFJ}. Note that it suffices to extend each basis element $e_i \in \mathfrak{u}_{\mu}$ to an action of $C_{\varepsilon}(\mbb{Z}_p, \mbb{Q}_p)$ separately.
\end{proof}

\begin{proof}[Proof of Theorem \ref{Thm:NOCmainThm} in the Siegel case]
Let $\opn{Spa}(B, B^+) \subset \mathcal{S}$ be as above. By Propositions \ref{Prop:OrdinarySiegelOCcomp} and \ref{Prop:OCII}, we see (by uniqueness) that the action $\mathfrak{u}_{\mu}$ extends to an action of $\mathcal{O}^{\opn{la}}(\mathfrak{u}_{\mu}^*, \mbb{Q}_p)$ on 
\[
\mathscr{N}^{\dagger}_B := \varinjlim^{\mathrm{lcv}}_V \opn{H}^0(V, \mathcal{O}_V)
\]
where $V$ runs over all open neighborhoods in $P_{\opn{dR}}^{\opn{an}} \times_{\mathcal{S}} \opn{Spa}(B, B^+)$ of the closure of $\mathfrak{Ig}_{M, \eta} \times_{\mathcal{S}} \opn{Spa}(B, B^+)$. Again by uniqueness, this glues to an action on $\mathscr{N}^{\dagger}$. The remaining properties in Theorem \ref{Thm:NOCmainThm} are easily verified, following \cite[\S 6]{GrahamPilloniRodriguesJacinto}.
\end{proof}

\subsection{The general case} We now prove Theorem \ref{Thm:NOCmainThm} in the general case.
\begin{proof}[Proof of Theorem \ref{Thm:NOCmainThm}]
We write objects for the Siegel case decorated with $\gvx$ or a subscript $V$. Consider the commutative diagram
\begin{equation}
    \begin{tikzcd}
        \igma\gx \arrow{r} \arrow{d} & P_{\dr}^{\an}\gx \arrow{d} \\
        \igma\gvx \arrow{r} & P_{\dr}^{\an}\gvx,
    \end{tikzcd}
\end{equation}
where the vertical arrows are closed immersions. Note that the left vertical arrow is $\mf{u}_{\mu} \otimes \gmhateta$ equivariant via the closed immersion $\mf{u}_{\mu} \otimes \gmhateta \to \mf{u}_{\mu, V} \otimes \gmhateta$. Let $\opn{Spf}B^+ \subset \mathscr{S}(\mathsf{G}_V, \mathsf{H}_V)$ be an open affine subspace, and let $\opn{Spa}(B, B^+)$ denote its adic generic fiber. We work locally over $\opn{Spa}(B, B^+)$ and add the subscript $B$ to denote the base change $- \times_{\mathscr{S}(\mathsf{G}_V, \mathsf{H}_V)} \opn{Spa}(B, B^+)$. We have a diagram of pullback maps
\begin{equation}
    \begin{tikzcd}
        \mathcal{O}( P_{\dr}^{\an}\gvx_B) \arrow[r, twoheadrightarrow] \arrow{d} & \mathcal{O}( P_{\dr}^{\an}\gx_B)\arrow{d} \\
        \mathscr{N}^{\dagger}\gvx_B \arrow{r} \arrow{d} & \mathscr{N}^{\dagger}\gx_B \arrow{d} \\
        \mathcal{O}(\igma\gvx_B) \arrow[r, twoheadrightarrow] & \mathcal{O}(\igma\gx_B).
    \end{tikzcd}
\end{equation}
The top horizontal arrow is $(\mathfrak{g}, \mathcal{M}_{\mu}(\mbb{Z}_p))$ equivariant via the inclusion $(\mathfrak{g}, \mathcal{M}_{\mu}(\mbb{Z}_p)) \to (\mathfrak{g}_V, M_{\mu,V}(\mbb{Z}_p))$. The bottom horizontal arrow is $\mathcal{O}^{\opn{la}}(\mathfrak{u}_{\mu}^*, \mbb{Q}_p)$-equivariant via the natural map $\mathcal{O}^{\opn{la}}(\mathfrak{u}_{\mu}^*, \mbb{Q}_p) \to \mathcal{O}^{\opn{la}}(\mathfrak{u}_{\mu,V}^*, \mbb{Q}_p)$. \smallskip

We first show that $\mathscr{N}^{\dagger}\gvx_B \to \mathscr{N}^{\dagger}\gx_B$ is surjective. For this, we simply note that given a cofinal collection of open neighborhoods $U \subset P^{\an}_{\dr}\gvx_B$ containing the closure of $\igma\gvx_B$, the intersections $U \cap P^{\an}_{\dr}\gx_B$ form a cofinal collection of open neighborhoods $U'=U \cap P^{\an}_{\dr}\gx_B \subset P^{\an}_{\dr}\gx_B$ containing the closure of $\igma\gx_B$. Since $P^{\an}_{\dr}\gx_B \to P^{\an}_{\dr}\gvx_B$ is a closed immersion, it follows that
\begin{align}
    H^0(U, \mathcal{O}_{P^{\an}_{\dr}\gvx_B}) \to H^0(U', \mathcal{O}_{P^{\an}_{\dr}\gx_B})
\end{align}
is surjective for each $U$, and thus in the colimit we get a surjective map $\mathscr{N}^{\dagger}\gvx_B \to \mathscr{N}^{\dagger}\gx_B$. \smallskip 

Note that the kernel of this map is a closed ideal in each term in the colimit. Since the action of $\opn{Sym}\mathfrak{u}_{\mu}$ preserves this ideal, by the density of polynomial functions, we see that the action of $\mathcal{O}^{\opn{la}}(\mathfrak{u}_{\mu}^*, \mbb{Q}_p)$ descends to a unique algebra action on $\mathscr{N}^{\dagger}\gx_B$ extending the action of $\opn{Sym}\mathfrak{u}_{\mu}$.

To conclude, we see by uniqueness that the action of $\mathcal{O}^{\opn{la}}(\mathfrak{u}_{\mu}^*, \mbb{Q}_p)$ on $\mathscr{N}^{\dagger}\gx_B$ glues to an action on $\mathscr{N}^{\dagger}\gx$, as required.
\end{proof}

\subsection{The \texorpdfstring{$\mu$}{mu}-ordinary case} We end this section by making some remarks on the $\mu$-ordinary case. 

\subsubsection{} Let the notation be as in Section \ref{Sub:FourierOnGenFiber}. Define 
\begin{align}
    \mathscr{N}^{\dagger} := \varinjlim^{\mathrm{lcv}}_U \opn{H}^0(U, \mathcal{O}_{N\backslash P^{\opn{an}}_{\opn{dR}}}),
\end{align}
where the locally convex inductive limit is over all open neighborhoods $U \subset N\backslash P^{\opn{an}}_{\opn{dR}}$ containing the closure of the image of $\igma$. Recall that $\mathcal{O}^{\gamma\opn{-la}}(T_p H^{\vee})$ acts on $\mathcal{O}(\igma)$, see Section \ref{Sub:FourierOnGenFiber}. We have the following conjectural generalization of Theorem \ref{Thm:NOCmainThm}.
\begin{Conj} \label{Conj:NOC}
    The action of $\mathcal{O}^{\gamma\opn{-la}}(T_p H^{\vee})$ on $\mathcal{O}(\igma)$ extends to an action of $\mathcal{O}^{\gamma\opn{-la}}(T_p H^{\vee})$ on $\mathscr{N}^{\dagger}$. 
\end{Conj}
It is not clear (to the authors) whether the proof of Theorem \ref{Thm:NOCmainThm} generalizes. This is because, in several places of the proof, one uses the fact that the binomial functions $\binom{x}{k}$ ($k \geq 0$) form an orthonormal basis for $\opn{Cont}(\mbb{Z}_p, \mbb{Q}_p)$, and that we can concretely describe the topology on $\mathcal{O}^{\opn{la}}(\mbb{Z}_p, \mbb{Q}_p)$ (and hence $\mathcal{O}^{\opn{la}}(\mathfrak{u}_{\mu}^*, \mbb{Q}_p)$) via a family of Banach norms depending on $\varepsilon > 0$ and this orthonormal basis. To construct an action of $\mathcal{O}^{\gamma\opn{-la}}(T_p H^{\vee})$ by the same method, it seems that one needs a suitably explicit description of the sections $\mathcal{O}(T_p \mathcal{H}^\vee)$.


\section{Relation to \texorpdfstring{$p$}{p}-adic \texorpdfstring{$L$}{L}-functions, an outlook} \label{Sec:PAdicLFunctions}

We end by describing how our theory produces the $p$-adic families which feature in the construction of several examples of $p$-adic $L$-functions. In particular, we speculate what these families should be in the unitary GGP setting when the prime $p$ is inert in the totally real subfield and split in the CM extension.

\subsection{Differential operators on unitary groups} \label{Sub:DiffOpsonunitary} Let the notation be as in \S \ref{subsub:PELNotation} with $\mathsf{F}^0=\mathbb{Q}$ and $\mathsf{F}$ an imaginary quadratic field. Let $(r,s)$ be the signature of $\V$, let $p$ be a prime split in $\mathsf{F}$, and choose an isomorphism $\mathcal{G} \xrightarrow{\sim} \operatorname{GL}_n \times \mathbb{G}_m$ such that the Hodge cocharacter is conjugate to $$\mu=\mathrm{diag}(\overbrace{z, \ldots, z}^r, \overbrace{1, \ldots, 1}^s) \times z.$$ In particular, we have
\[ 
\mathcal{M}=\mathcal{M}_{\mu}=\begin{bmatrix} \ast_{r\times r} & 0 \\ 0 & \ast_{s \times s}\end{bmatrix} \times \mbb{G}_m \textrm{ and } \mathcal{U}_\mu=\begin{bmatrix} \mathrm{Id}_{r\times r} & *_{r\times s}\\0& \mathrm{Id}_{s \times s} \end{bmatrix} \times \{1\}
\]
and we may therefore identify $\mf{u}_{\mu}$ with the space $M_{r \times s}$ of $r \times s$ matrices. We thus obtain a dual identification of $\mf{u}_{\mu}^*$ with the space of $r\times s$ matrices such that $(g_r, g_s; s) \in \mathcal{M}$ acts by 
\[ 
(g_r,g_s;s) \cdot A= (g_r^t)^{-1} A g_s^t .
\]

\subsubsection{EFMV operators} \label{Subsub:EFMVoperators} We now discuss the operators of \cite{EischenFintzenMantovanVarma} from our perspective. Let $\mf{N} \leq \mathcal{M}$ be the unipotent subgroup consisting of upper triangular unipotent matrices in both blocks (and the trivial group in the similitude factor). The $p$-adic differential operators acting on $\mathcal{O}(\igmf)^{\mf{N}(\mbb{Z}_p)}$ arise from functions in $\Cont(\mf{u}_{\mu}^*, \mathbb{Z}_p)^{\mf{N}(\mbb{Z}_p)}$ (see Corollary \ref{Cor:IntegralFourierConsequenceOrdinary}); in particular, the usual algebraic differential operators that preserve the subspace $\mathcal{O}(\mf{Ig}_M)^{\mf{N}(\mbb{Z}_p)}$ are the highest weight vectors in $\mathrm{Sym} (\mf{u}_{\mu})$. We can compute these explicitly: They are given by polynomials in $\det A_i$ for $1 \leq i \leq s$, where $A_i$ is the $i \times i$ submatrix growing from the top right corner of $A$. If we let $\mathcal{T} \subset \mathcal{G}$ denote the standard diagonal torus, the element $\det A_i \in \opn{Sym}(\mf{u}_{\mu})$ is a highest weight vector of weight
\begin{equation} \label{Eqn:HighestWeightsEMFV}
(\overbrace{1,\ldots,1}^i, \overbrace{0,\ldots,0}^{r-i}, \overbrace{0, \ldots, 0}^
{s-i}, \overbrace{-1, \ldots, -1}^i; 0).
\end{equation}

Let $\mathcal{T}^{\opn{sym}}$ denote the $s$-dimensional torus obtained as the quotient of $\mathcal{T}$ by the intersection of the kernels of the weights in \eqref{Eqn:HighestWeightsEMFV} (as $1 \leq i \leq s$ varies). In \cite[Main Result 1, Corollary 5.2.8]{EischenFintzenMantovanVarma}, it is stated that these operators can be interpolated $p$-adically into operators corresponding to characters $\kappa$ of $T^{\mathrm{sym}}$. The notation in \emph{loc.cit.}\ is slightly misleading however, see Remark \ref{Rem:Caveat}. The discrepancy can be explained by reinterpreting the construction of these operators in the language of this article, which we now describe. \smallskip

Let $W \subset M_{r \times s}(\zp)$ denote the compact open subset of $r \times s$ matrices such that $\det A_i \in \mathbb{Z}_p^\times$ for each $1 \le i \le s$. Then, the $p$-depleted differential operators that are interpolated correspond to the action of the functions 
\begin{equation} \label{Eqn:PDepletedOps}
\mathbf{1}_W\cdot f(\det A_1, \ldots, \det A_i) \in \opn{Cont}(\mf{u}_{\mu}^*, R)^{\mf{N}(\mbb{Z}_p)}
\end{equation}
for $R$-valued polynomials $f$, where $R$ is a $p$-adically complete and separated $\mbb{Z}_p$-algebra. Here $\mathbf{1}_W$ denotes the indicator function of $W$. 

The interpolation of these operators can be obtained immediately from the algebra action of $\opn{Cont}(\mf{u}_{\mu}^*, R)^{\mf{N}(\mbb{Z}_p)}$ on $\mathcal{O}(\mf{Ig}_M)^{\mf{N}(\mbb{Z}_p)}$. More precisely, note that $W$ is the pre-image of the unique open ${\mf{N}}$-orbit mod $p$, i.e., the set of $r \times s$ matrices $A$ whose $s \times s$ submatrix $A_s$ lies in the open Schubert cell mod $p$ respect to the action of upper triangular unipotent matrices. Ignoring the similitude factor (which is trivial on the unipotent radical), a matrix $A \in W$ can therefore be written uniquely in the form 
\[ 
A= m \cdot \begin{bmatrix} w  \\ 0_{(r-s) \times{s}}  \end{bmatrix},\quad \quad w=\begin{bmatrix}0 & \ldots & 0 & t_1\\ 0 & \ldots & t_2 & 0 \\
\vdots & \ldots & \vdots & \vdots \\
t_s & 0 & \ldots & 0 \end{bmatrix},\; m \in {\mf{N}}(\mbb{Z}_p), t_i \in \mathbb{Z}_p^\times. 
\]
Given a continuous character $\kappa \colon \mathcal{T}^{\opn{sym}}(\mbb{Z}_p) \to R^{\times}$, which is of the form 
\[
(\kappa_1, \dots, \kappa_s, 0, \dots, 0, \kappa_s^{-1}, \dots, \kappa_1^{-1}; 0)
\]
for continuous characters $\kappa_i \colon \mbb{Z}_p^{\times} \to R^{\times}$, we define a function $\mathbf{1}_W \cdot \Theta^{\kappa} \in \opn{Cont}(\mf{u}_{\mu}^*, R)$ as
\[
\mathbf{1}_W \cdot \Theta^{\kappa}(A) = \left\{ \begin{array}{cc} \prod_{i=1}^s \kappa_i  ((-1)^{s-i}t_i) & \text{ if } A= m \cdot \begin{bmatrix} w  \\ 0_{(r-s) \times{s}}  \end{bmatrix} \in W \\ 0 & \text{ otherwise } \end{array} \right. .
\]
By construction, this is invariant under the action of $\mf{N}(\mbb{Z}_p)$, and $p$-adically interpolates the $p$-depleted operators in \eqref{Eqn:PDepletedOps}. As an example, if we fix $1 \leq i \leq s$, and take $\kappa$ to be the weight in \eqref{Eqn:HighestWeightsEMFV}, then $\mathbf{1}_W \cdot \Theta^{\kappa}$ is just $\mathbf{1}_W \cdot \det A_i$ (the multiplication of the indicator function of $W$ with the algebraic differential operator $\opn{det}A_i \in \opn{Sym}^{\bullet}\mf{u}_{\mu}$). 
    
\begin{Rem}
There is a similar interpolation in the general ordinary case. Indeed, since $\mu$ is minuscule, $\mf{u}_\mu^*$ is a spherical variety for the adjoint action of $\mathcal{M}_{\mu}$.\footnote{Recall we are in the ordinary case, so $E = \mbb{Q}_p$. Then, to see that $\mathfrak{u}_{\mu}^*$ is a spherical variety, let $\mathcal{B} \subset \mathcal{G}$ be a Borel subgroup such that $\mathcal{B}_{\mu} = \mathcal{B} \cap \mathcal{M}_{\mu}$ is also a Borel subgroup of $\mathcal{M}_{\mu}$. Then $\mathfrak{u}_{\mu} \subset \opn{Lie}\mathcal{B}$ is an abelian ideal, hence there are only finitely many orbits for the adjoint action of $\mathcal{B}$ on $\mathfrak{u}_{\mu}$ \cite[Intro.]{Panyushev}. In particular, there exists an open orbit. Since this action factors through $\mathcal{B}_{\mu}$, this implies that there exists an open orbit for $\mathcal{B}_{\mu}$, so $\mathfrak{u}_{\mu}$ (and hence $\mathfrak{u}_{\mu}^*$) is a spherical $\mathcal{M}_{\mu}$-variety. } Fix a representative $x_0 \in \mathfrak{u}_{\mu}^*(\mbb{Z}_p)$ for the open orbit of $\mathcal{B}_{\mu}$ (the Borel subgroup with unipotent radical $\mathfrak{N}$), and let $\opn{Stab}_{\mathcal{B}_{\mu}}(x_0) \subset \mathcal{B}_{\mu}$ denote the stabilizer of the point $x_0$. Let $W \subset \mathfrak{u}_{\mu}^*(\mbb{Z}_p)$ denote the preimage of the unique Borel orbit modulo $p$, which in this case is the same thing as $\mathcal{B}_{\mu}(\mbb{Z}_p) \cdot x_0$. Then for any continuous character $\kappa \colon \mathcal{B}_{\mu}(\mbb{Z}_p)/\opn{Stab}_{\mathcal{B}_{\mu}}(x_0) \to R^{\times}$, we can construct a function $\mathbf{1}_W \cdot \Theta^{\kappa} \in \opn{Cont}(\mathfrak{u}_{\mu}^*(\mbb{Z}_p), R)$ as the extension by zero of the function on $W$ given by $b \cdot x_0 \mapsto \kappa(b)$.
\end{Rem}

\subsubsection{Operators adapted to the doubling method} In the notation of \S \ref{subsub:PELNotation}, consider a direct sum $\mathsf{W} = \mathsf{V} \oplus \mathsf{V}'$ of Hermitian $\F$-modules with $\F$ imaginary quadratic, where $\mathsf{V}$ and $\mathsf{V}'$ have signatures $(r,s)$ and $(s,r)$, respectively. Then there is a natural embedding $\operatorname{G}(\operatorname{U}(\V) \times \operatorname{U}(\V')) \to \operatorname{GU(\mathsf{W})}$. One can apply the EFMV operators of \cite{EischenFintzenMantovanVarma} (see \S \ref{Subsub:EFMVoperators}) on the Mantovan Igusa variety for $\operatorname{GU(\mathsf{W})}$ then restrict to the Mantovan Igusa variety for $\operatorname{G}(\operatorname{U}(\V) \times \operatorname{U}(\V'))$. However, this does not preserve ordinarity. An alternative is to perhaps replace the EFMV operators with a further $p$-depletion: Instead of considering functions on the (mod $p$) open $n\times n$ Schubert cell, we restrict further to functions on the subset of $n \times n$ matrices consisting of the product of the open $r \times r$ Schubert cell in the top right and open $s \times s$ Schubert cell in the bottom left. In particular, one can ask whether it is possible to obtain the Eisenstein measure of EHLS \cite{EischenHarrisLiSkinner} by considering a Hida family of \emph{Hida-ordinary} Eisenstein series on $\operatorname{GU(\mathsf{W})}$, then applying the interpolation via these differential operators and restricting. 

\subsubsection{The unitary Gan--Gross--Prasad setting} \label{Subsub:UGGPsetting} In the notation of \S \ref{subsub:PELNotation}, consider a direct sum $\mathsf{W} = \mathsf{V} \oplus \mathsf{V}'$ of Hermitian $\F$-modules with $\F$ imaginary quadratic, where $\mathsf{V}$ has signature $(1,1)$ and where $\mathsf{V}'$ has signature $(1,0)$. Consider the embedding $\mathsf{H}:=\operatorname{G}(\operatorname{U}(\V) \times \operatorname{U}(\V')) \xrightarrow{\iota} \operatorname{GU(\mathsf{W})}=:\mathsf{G}$ as above, which induces a morphism of Shimura data $\iota:\hy \to \gx$. Choose a prime $p>2$ split in $\mathsf{F}$ and identify
\begin{align}
    \mathcal{G} &\xrightarrow{\sim} \operatorname{GL}_{3} \times \mathbb{G}_{m} \\
    \mathcal{H} &\xrightarrow{\sim} \operatorname{GL}_{2} \times \mathbb{G}_{m} \times \mathbb{G}_{m},
\end{align}
such that the induced embedding of Levi subgroups attached to the Hodge cocharacters is given by
\[ 
\begin{bmatrix} g & 0 \\ 0 & g' \end{bmatrix} \times g'' \times s \rightarrow \begin{bmatrix} g''	& 0 & 0  \\ 0 & g & 0  \\ 0 & 0 & g' \end{bmatrix} \times s, 
\]
where $s$ denotes the similitude factor (recall that in $\mathcal{G}$, the Levi subgroup $\mathcal{M}$ is the subgroup of block diagonal matrices with block sizes $2 \times 2$ and $1 \times 1$).

Let $f \in \mathcal{O}(\mf{Ig}_M\gx)^{\mf{N}(\mbb{Z}_p)}$ be a $p$-adic automorphic form for $\mathsf{G}$ which is Hida-ordinary for the $U_p$-operators. Then we can consider the one-parameter $p$-adic family of automorphic forms on $\mathsf{H}$ given by
\[
\iota^* \left( \mathbf{1}_W \cdot \Theta^{\kappa}(f) \right) \in \mathcal{O}(\mf{Ig}_{M}\hy)^{\mf{N}_H(\mbb{Z}_p)} \widehat{\otimes} R
\]
where 
\[
W = \begin{bmatrix} \mbb{Z}_p^{\times} \\ \mbb{Z}_p \end{bmatrix}, \quad \quad \kappa \colon \mbb{Z}_p^{\times} \to R^{\times}, \quad \quad \mf{N}_H = \mf{N} \cap H,
\]
and the operator $\mathbf{1}_W \cdot \Theta^{\kappa}$ is defined in \S \ref{Subsub:EFMVoperators}. It should be the case that this family of $p$-adic automorphic forms has non-zero image under the ordinary projector for the subgroup $\mathsf{H}$, and is related to the construction of the $p$-adic $L$-functions in \cite{HarrisSquareRoot}.

Note that one should not $p$-deplete further, by taking $W$ to be $\begin{bmatrix} \mbb{Z}_p^{\times} & \mbb{Z}_p^{\times} \end{bmatrix}^t$ for example. Indeed, even though one can still $p$-adically interpolate the Maass--Shimura operators after applying this $p$-depletion, the pullback $\iota^* \left( \mathbf{1}_W \cdot \Theta^{\kappa}(f) \right)$ will be killed by the ordinary projector on $\mathsf{H}$. Essentially, the $p$-depletion arising from $W = \begin{bmatrix} \mbb{Z}_p^{\times} & \mbb{Z}_p \end{bmatrix}^t$ is the optimal choice in the construction of the $p$-adic $L$-function.

\subsection{The case when \texorpdfstring{$p$}{p} is inert in \texorpdfstring{$\F^+$}{F+}}

We end this section by explaining how one can extend the construction of the one-parameter families in \S \ref{Subsub:UGGPsetting} to the setting of unitary groups over non-imaginary quadratic CM fields with $p$ inert in the totally real subfield.

We place ourselves in the setting of \S \ref{sub:ExampleII}, and specialise to the case $n=3$, i.e., the signatures of $\V$ are $(2,1), (3, 0), \dots, (3, 0)$ and $p$ is inert in $\mathsf{F}^+$ but splits in $\mathsf{F}/\mathsf{F}^+$. Write $\V = \mathsf{W} \oplus \mathsf{W}'$ as Hermitian $\mathsf{F}$-modules, and assume that $\mathsf{W}$ has signatures $(1, 1), (2, 0), \dots, (2,0)$. Write $\mathsf{H}=\mathsf{G}(\mathsf{U}(W)\times \mathsf{U}(Z))$ and consider the natural map $\iota: \hy \to \gx$. We have an action of $C^{F^+\opn{-la}}(\mathcal{O}_{F^+}^{\oplus 2}, \mbb{C}_p)$ on $\mathcal{O}(\mf{Ig}_{M, \eta}\gx)$, and we note that $\mathcal{U}_{\overline{\mu}}(\mbb{Z}_p)$ is abelian and given by
\[
\begin{bmatrix}
    1 & 0 & \mathcal{O}_{F^+} \\ 0 & 1 & \mathcal{O}_{F^+} \\ 0 & 0 & 1
\end{bmatrix} \times \{ 1 \} \subset \mathsf{G}_{F^+} = \left( \prod_{\tau \colon F^+ \hookrightarrow \overline{\mathbb{Q}}_p} \opn{GL}_{3, F^+} \right) \times \mbb{G}_{m, F^+} 
\]
embedded diagonally. The differential operators one would like to interpolate should come from the top-right entry. More precisely, let $R$ be an affinoid $\mbb{C}_p$-algebra. For any $f \in \mathcal{O}(\mf{Ig}_{M, \eta}\gx)$ and $F^+$-analytic character $\kappa \colon \mathcal{O}_{F^+}^{\times} \to R^{\times}$, we can consider the family of $p$-adic automorphic forms on $\mathsf{H}$ given by
\[
\iota^*\left( \mathbf{1}_W \cdot \Theta^{\kappa} (f)\right) \in \mathcal{O}(\mf{Ig}_{M, \eta}\hy) \widehat{\otimes} R
\]
where $W = \begin{bmatrix} \mathcal{O}_{F^+}^{\times} & \mathcal{O}_{F^+} \end{bmatrix}^t \subset \mathcal{O}_{F^+}^{\oplus 2}$, and $\mathbf{1}_W \cdot \Theta^{\kappa} \in C^{F^+\opn{-la}}(\mathcal{O}_{F^+}^{\oplus 2}, R)$ is the function
\[
\mathbf{1}_W \cdot \Theta^{\kappa} ( (x_1, x_2)^t) = \left\{ \begin{array}{cc} \kappa(x_1) & \text{ if } x_1 \in \mathcal{O}_{F^+}^{\times} \\ 0 & \text{ otherwise } \end{array} \right. .
\]
If $f$ is Hida-ordinary, this family should have an interesting image under the ordinary projector on $\mathsf{H}$, which we expect will feature in the construction of square-root $p$-adic $L$-functions in this setting.


\appendix 

\section{Some remarks about locally convex inductive limits}\label{appendix.locally-convex-inductive-limits}

\subsection{Locally convex inductive limits and condensed mathematics} We fix a nonarchimedean field $K$ (not necessarily spherically complete). Our goal is to compare locally convex inductive limits with condensed inductive limits. 

\subsubsection{Tool and computational resource disclosure} The authors made use of ChatGPT 5.5 Pro while writing this appendix. More specifically, the counterexamples in Lemmas A.2.2, A.2.7 and the proofs of Lemmas A.1.7 and A.2.5 were initially suggested to us by that model. We take full responsibility for the written arguments. 

\subsubsection{} Throughout this appendix we consider inductive systems $ \{V_n\}_{n \in \mathbb{Z}_{\ge 0}}$ of locally convex $K$-vector spaces with injective transition maps $\iota_{n}:V_{n} \to V_{n+1}$, and write $V_{\infty}$ for its locally convex inductive limit, see \cite[Section 11.1]{PGS}. If we write $\varinjlim_n V_n$ for the topological inductive limit, then there is a natural continuous bijection $\varinjlim_n V_n \to V_{\infty}$. Note that $\varinjlim_n V_n$ is not necessarily a topological group because the natural
continuous bijection $\varinjlim_n (V_n \times V_n) \to (\varinjlim_n V_n) \times (\varinjlim_n V_n)$ might not be a homeomorphism; see \cite{TopologicalInductiveLimit} for an in-depth discussion of this point. 

\subsubsection{} We say that $\{V_n\}_{n \in \mathbb{Z}_{\ge 0}}$ is \emph{cond-regular} if the natural map
\begin{align}
    \varinjlim_n \ul{V_n} \to \ul{V_{\infty}}
\end{align}
is an isomorphism. In other words, we are asking that for every profinite set $S$, the natural map 
\begin{align}
\varinjlim_n C^0(S, V_n) \to C^0(S, V_{\infty})
\end{align}
is a bijection. 

\subsubsection{} \label{subsub:strict}  We now compare this to notions from classical locally convex functional analysis. Recall that the inductive system $ \{V_n\}_{n \in \mathbb{Z}_{\ge 0}}$ is called \emph{strict}, see \cite[Definition 11.1.3]{PGS}, if the transition maps $\iota_n:V_{n} \to V_{n+1}$ endow $V_{n}$ with the subspace topology from $V_{n+1}$. Recall that $ \{V_n\}_{n \in \mathbb{Z}_{\ge 0}}$ is called \emph{regular}, see \cite[Definition 11.1.3]{PGS}, if every bounded subset $B \subset V_{\infty}$ lies inside $V_n \subset V_{\infty}$ for some $n$ and is bounded inside $V_n$. If an inductive system $ \{V_n\}_{n \in \mathbb{Z}_{\ge 0}}$ is strict and each $V_{n}$ is closed in $V_{n+1}$, then the inductive system is regular, see \cite[Theorem 11.1.6]{PGS}. 

\begin{Lem} \label{Lem:InductiveLimitTopologyI}
If the inductive system $ \{V_n\}_{n \in \mathbb{Z}_{\ge 0}}$ is regular and strict, then it is cond-regular.
\end{Lem}
\begin{proof}
Let $S$ be a profinite set. To show that the natural map
\begin{align}
\varinjlim_n C^0(S, V_n) \to C^0(S, V_{\infty})
\end{align}
is a bijection, we have to show that any continuous map $f:S \to V_{\infty}$ factors through some $V_n \subset V_{\infty}$ via a continuous map $f_n:S \to V_n$. Since the image of $S \to V_{\infty}$ is quasicompact, hence bounded, we know by regularity that $f(S) \subset V_n$ for some $n$. Since the inductive system is strict, \cite[Theorem 11.1.5]{PGS} implies $V_n \subset V_{\infty}$ has the subspace topology, and it follows that the induced map $f_n:S \to V_n$ is continuous, as required.
\end{proof}

\subsubsection{} \label{subsub:compactoid} Recall that the inductive system $ \{V_n\}_{n \in \mathbb{Z}_{\ge 0}}$ is called \emph{compactoid}, see \cite[Definition 11.3.1]{PGS}, if the transition maps $V_{n} \to V_{n+1}$ are compactoid in the sense of \cite[Definition 8.1.1]{PGS}. If an inductive system $ \{V_n\}_{n \in \mathbb{Z}_{\ge 0}}$ of Banach spaces is compactoid, then the inductive system is regular, see \cite[Theorem 11.3.5]{PGS}.

\begin{Lem} \label{Lem:InductiveLimitTopologyII}
If each $V_n$ is Banach and the inductive system $ \{V_n\}_{n \in \mathbb{Z}_{\ge 0}}$ is compactoid, then it is cond-regular. 
\end{Lem}
\begin{proof}
Let $S$ be a profinite set. To show that the natural map
\begin{align}
\varinjlim_n C^0(S, V_n) \to C^0(S, V_{\infty})
\end{align}
is a bijection, we have to show that any continuous map $f:S \to V_{\infty}$ factors through some $V_n$ via a continuous map $f_n:S \to V_n$. Since the image of $S \to V_{\infty}$ is quasicompact, hence bounded, we know by regularity that $f(S) \subset V_n$ for some $n$ and is bounded inside of $V_n$. We are going to show that the map $f_{n+1}:S \to V_{n+1}$ is continuous. \smallskip 

Since $f(S)$ is bounded in $V_n$, it follows that $\iota_n(f(S)) \subset V_{n+1}$ is compactoid by \cite[Theorem 8.3.2]{PGS}. Let us write $\tau_{n+1}$ for the topology on $V_{n+1}$ and $\tau_{n+1}'$ for the subspace topology of $V_{n+1}$ inside $V_{\infty}$. Since each $V_n$ is Hausdorff, it follows by regularity that $V_{\infty}$ is Hausdorff, see \cite[Theorem 11.2.4]{PGS}. We thus see that $\tau_{n+1}'$ is a Hausdorff locally convex topology on $V_{n+1}$. It suffices to show that $\tau_{n+1}$ and $\tau_{n+1}'$ induce the same topology on $\iota_n(f(S))$, but this is a direct consequence of \cite[Corollary 3.8.39]{PGS}.
\end{proof}

\begin{Rem}
For $K$ a finite extension of $\qp$, Lemma \ref{Lem:InductiveLimitTopologyII} is \cite[Lemma 2.19]{colmez2023arithmeticdualitypadicproetale} (note that compactoid maps between Banach spaces are the same as compact maps, see \cite[first paragraph of Section 8.8]{PGS}). However, the proof there only shows that $f(S) \subset V_n$ and not that the induced map $S \to V_n$ is continuous.
\end{Rem}

\subsection{A question} If $ \{V_n\}_{n \in \mathbb{Z}_{\ge 0}}$ is a cond-regular inductive system, then one could ask if $\ul{V_{\infty}}$ determines $V_{\infty}$. There is an adjunction map
\begin{align} \label{eq:Adjunction}
    \ul{V_{\infty}}(\ast)_{\mathrm{top}} \to V_{\infty},
\end{align}
see \cite[Proposition 1.7]{CondensedNotes}. By \cite[Proposition 1.7]{CondensedNotes}, this map is an isomorphism if $V_{\infty}$ is compactly generated as a topological space. Using cond-regularity, we can identify the left-hand side with (where $\varinjlim_n$ denotes the topological inductive limit)
\begin{align}
    \ul{V_{\infty}}(\ast)_{\mathrm{top}} = (\varinjlim_n \ul{V_n})(\ast)_{\mathrm{top}} = \varinjlim_n \left(\ul{V_n}(\ast)_{\mathrm{top}}\right),
\end{align}
since $Y \mapsto Y(\ast)_{\mathrm{top}}$ is a left adjoint and thus commutes with colimits. If each $V_n$ is moreover compactly generated (as a topological space), then this identifies with $\varinjlim_n V_n$, see \cite[Proposition 1.7]{CondensedNotes}. 

\begin{Question} \label{Question:Appendix}
Suppose that $ \{V_n\}_{n \in \mathbb{Z}_{\ge 0}}$ is a cond-regular inductive system where each $V_n$ is compactly generated as a topological space. When is the adjunction map \eqref{eq:Adjunction} an isomorphism? In other words, when does the locally convex inductive limit topology on $V_{\infty}$ agree with the topological inductive limit topology?
\end{Question}

\subsubsection{} \label{subsub:Counterexample} We first give an example where Question \ref{Question:Appendix} has a negative answer: Consider the inductive system where $V_n = \oplus_{i=0}^{n} E_i$ and $V_{n} \to V_{n+1}$ the natural inclusion, with $E_i=\hat{\oplus}_{\mathbb{Z}} K$ the $K$-Banach space with countable orthonormal basis $\{e_{i,j}\}_{j \in \mathbb{Z}}$. Let $V_{\infty}$ be the locally convex inductive limit as above; it is the locally convex direct sum of the $E_i$. In this case the inductive system is strict and $V_{n} \subset V_{n+1}$ is closed, so the inductive system is cond-regular by Lemma \ref{Lem:InductiveLimitTopologyI}.
\begin{Lem} \label{Lem:Counterexample}
Let $V_{\infty}$ be as in \S \ref{subsub:Counterexample}. The topological inductive limit topology on $V_{\infty}$ has strictly more closed subsets than the locally convex inductive limit topology. In particular, \eqref{eq:Adjunction} is not a homeomorphism.
\end{Lem}
\begin{proof}
We will construct a subset $A \subset V_{\infty}$ that is not closed in $V_{\infty}$ but is closed in the topological inductive limit topology: Let $\varpi \in \mathcal{O}_K$ be an element of norm strictly less than one. Consider $a_{n,k} = \varpi^n e_{1,k} + \varpi^k e_{n,1} \in V_n \subset V_{\infty}$, and let $A=\{a_{n,k}\}_{n,k \in \mathbb{Z}_{\ge 2}}$.

\textbf{Step 1: The closure of $A$ in $V_{\infty}$ contains $0$:} By \cite[Corollary 3.3.16]{PGS} we know that $V_{\infty}$ has a basis of zero neighborhoods consisting of absolutely convex subsets in the sense of \cite[Definition 3.1.3]{PGS}; note that absolutely convex subsets are just $\mathcal{O}_K$-submodules.  By \cite[Theorem 11.1.2]{PGS} an $\mathcal{O}_K$-submodule $U$ of $V_{\infty}$ is open if and only if it is of the form $\sum_i W_i$, with $W_i$ an open $\mathcal{O}_K$-submodule of $V_i$. It follows that a basis of neighborhoods of $0$ in $V_{\infty}$ is given by $U=\bigoplus_{i} U_i$, where each $U_i=\varpi^{k_{i}} \hat{\oplus}_{\mathbb{Z}} \mathcal{O}_K$ for some $k_{i} \in \mathbb{Z}_{\ge 0}$. If we take $n$ such that $n \ge k_{1}$ and $k$ such that $k \ge k_{n}$, then $\varpi^n e_{1,k} \in U_1$ and $\varpi^k e_{n,1} \in U_n$, so that $a_{n,k} \in U$. Thus $A$ has nonempty intersection with each member of a basis of neighborhoods of $0$, and so the closure of $A$ in $V_{\infty}$ contains $0$. Note that $0 \not \in A$ so that $A$ is not closed in $V_{\infty}$.  \smallskip 

\textbf{Step 2: The set $A$ is closed in the inductive limit topology on $V_{\infty}$:} It suffices to show that $A_n=A \cap V_n$ is closed for all $n$. Since $V_n$ is a metric space, it suffices to show that convergent sequences in $A_n$ contain a limit in $A_n$. For this, we simply note that the topology on $V_n$ can be defined by the norm $\|(v_0, \cdots, v_n)\| = \operatorname{Max}_{i=0, \cdots,n}\left(\|v_i\|\right)$ and by looking at the $E_1$-coordinate we see that $\|a_{m,k} - a_{m',k'}\| \ge |\varpi^n|$ for $(m,k) \not=(m',k')$ and $2 \le m,m' \le n$. Thus any convergent sequence in $A_n$ is eventually constant.

We conclude that the inductive limit topology on $V_{\infty}$ does not agree with the locally convex inductive limit topology, so that the natural map of \eqref{eq:Adjunction} is not a homeomorphism.
\end{proof}

\begin{Rem}
Lemma \ref{Lem:Counterexample} contradicts the assertion made in \cite[proof of Lemma 3.32]{RJRC} that $\ul{V_{\infty}}(\ast)_{\mathrm{top}}$ is a classical LF space. 
\end{Rem}

\subsubsection{} Next, we give an example where Question \ref{Question:Appendix} does have a positive answer (cf. the last sentence of \cite[Remark 4.9]{RJRC}). As written, the proof of \cite[Lemma 2.20]{colmez2023arithmeticdualitypadicproetale} seems to use Lemma \ref{Lem:LocallyCompactCompactoid}.
\begin{Lem} \label{Lem:LocallyCompactCompactoid}
Suppose that $K$ is locally compact. If each $V_n$ is Banach and the inductive system $ \{V_n\}_{n \in \mathbb{Z}_{\ge 0}}$ is compactoid, then the inductive limit topology on $V_{\infty}$ agrees with the locally convex inductive limit topology. 
\end{Lem}
\begin{proof}
It suffices\footnote{Indeed, it follows directly from this that addition is continuous in the topological inductive limit topology. This shows that the topological inductive topology defines a topological group, and we can then compare the two topologies by comparing zero neighborhoods as usual.} to show that every zero neighborhood $U \subset \varinjlim_n V_n$ in the topological inductive limit topology, contains a subset $L \subset U$ which is a zero neighborhood in the locally convex inductive limit topology. Such $U$ has the defining property that $U_n = U \cap V_n$ is open in $V_n$ for each $n$. We are going to inductively construct bounded open $\mathcal{O}_K$-submodules $L_n \subset U_n$ with $\iota_n(L_n) \subset L_{n+1}$ and $\overline{\iota_n(L_n)} \subset U_{n+1}$. We will then let $L$ be the union of the $L_n$; which will be open by \cite[Theorem 11.1.2]{PGS}. \smallskip 

Since $U_1$ is open, it contains an bounded open $\mathcal{O}_K$-submodule $W_1 \subset U_1$, which is also closed. Let $L_0$ be a bounded open $\mathcal{O}_K$-submodule in $\iota_0^{-1}(W_1) \cap U_0$, then $\iota_0(L_0) \subset W_1 \subset U_1$ and hence its closure is also contained in $U_1$. For the inductive step, suppose we are given a bounded open $\mathcal{O}_K$-submodule $L_{n-1} \subset U_{n-1}$ such that $\iota_{n-2}(L_{n-2}) \subset L_{n-1}$ and $\overline{\iota_{n-1}(L_{n-1})} \subset U_{n}$. Note that $L_{n-1}$ is bounded hence $\iota_{n-1}(L_{n-1})$ is compactoid by \cite[Theorem 8.3.2]{PGS} and thus has compact closure $C_{n-1}$ by \cite[Theorem 3.8.4.(ii) and Theorem 3.8.3]{PGS}. 

The compactness of $C_{n-1}$ together with the fact that $C_{n-1} \subset U_{n}$, shows that there is an bounded open $\mathcal{O}_K$-submodule $W_n \subset V_n$ with $C_{n-1} + W_n \subset U_n$. Now $\iota_{n}(C_{n-1})$ is also compact and contained in $U_{n+1}$, and we can choose an bounded open $\mathcal{O}_K$-submodule $W_{n+1}' \subset V_{n+1}$ with $\iota_{n}(C_{n-1}) + W_{n+1}' \subset U_{n+1}$. Finally, choose $M_{n} \subset W_{n}$ an bounded open $\mathcal{O}_K$-submodule with $\iota_{n}(M_n) \subset W_{n+1}'$, and define $L_{n}=\iota_{n-1}(L_{n-1}) + M_n$. \smallskip

To show the desired properties of $L_n$, we first note that $L_n$ is open since it is an $\mathcal{O}_K$-submodule containing the bounded open $\mathcal{O}_K$-submodule $M_n$. By construction $\iota_n(L_n) \subset \iota_n(C_{n-1}) + \iota_n(M_n) \subset \iota_n(C_{n-1}) + W_{n+1}' \subset U_{n+1}$. Moreover, the closure of $\iota_n(L_n)$ is contained in the closed\footnote{Note that $\overline{\iota_n(M_n)}$ is compact and $\iota_n(C_{n-1})$ is compact, so their sum is compact.} subset $\iota_n(C_{n-1}) + \overline{\iota_n(M_n)}$, which is contained in $\iota_n(C_{n-1}) + W_{n+1}' \subset U_{n+1}$. This completes the induction step. 
\end{proof}

\subsubsection{} Next, we give an example showing that Lemma \ref{Lem:LocallyCompactCompactoid} is optimal. 
\begin{Lem} \label{Lem:CompactoidII}
Suppose that $K$ is an extension of $\qp$ that is not locally compact. Let $\{V_n\}_{n \in \mathbb{Z}_{\ge 0}}$ be the inductive system where $V_n$ is the Banach space over $K$ of functions $\zp \to K$ which have convergent power series expansions on all discs of radius $|p^{n+1}|$, and the transition maps are the natural ones. Then the inductive limit topology does $V_{\infty}$ does not agree with the locally convex inductive limit topology.
\end{Lem}
\begin{proof}
Since $K$ is not locally compact, we can choose a sequence $\lambda_1, \lambda_2, \cdots$ of elements of $\mathcal{O}_K$ with $|\lambda_i| \ge |p|$ and $|\lambda_i - \lambda_j| \ge |p|$ for $i \not=j$. Consider 
\begin{align}
    a_{n,k} = p^n \lambda_k \mathbbm{1}_{\zp} + p^k \mathbbm{1}_{1+p^n \mathbb{Z}_p},
\end{align}
which we note lies in $V_{n-1}$ but not in $V_{n-2}$, and define $A=\{a_{n,k} \; | \; n,k \in \mathbb{Z}_{\ge 2}$\}.

One checks as in the proof of Lemma \ref{Lem:Counterexample} that the closure of $A$ inside $V_{\infty}$ contains $0$
and that $0 \not \in A$. Finally, we check that $A \cap V_{n-1}$ contains all its limit points hence is closed, by looking at the value at $0$ of the functions $a_{n,k}$ to see that any convergent sequence is eventually constant. This shows that $A$ is closed in the topological inductive limit topology but not in the locally convex inductive limit topology. 
\end{proof}

\subsection{Full faithfulness} In cases where Question \ref{Question:Appendix} has a negative answer, the following lemma is a reasonable replacement (cf. \cite[Lemma 3.32]{RJRC} and \cite[Theorem 2.11]{li2025dualityarithmeticpadicproetale}).
\begin{Lem} \label{Lem:FullyFaithfull}
    Let $\{V_n\}_{n \in \mathbb{Z}_{\ge 0}}$ and $\{W_m\}_{m \in \mathbb{Z}_{\ge 0}}$ be cond-regular inductive systems with locally convex inductive limits $V_{\infty}$ and $W_{\infty}$. If all $V_n$ and $W_m$ are Fr\'echet, then the natural map
\begin{align}
    \operatorname{Hom}_K(V_{\infty}, W_{\infty}) \to \operatorname{Hom}_{\ul{K}}(\ul{V_{\infty}},\ul{W_{\infty}})
\end{align}
is a bijection. 
\end{Lem}
\begin{proof}
It follows formally that 
\begin{align}
     \operatorname{Hom}_K(V_{\infty}, W_{\infty}) \xleftarrow{\sim} \varprojlim_n \operatorname{Hom}_K(V_n, W_{\infty})
\end{align}
and from \cite[Corollary 8.9]{SchneiderFA} that 
\begin{align}
    \operatorname{Hom}_K(V_n, W_{\infty})  \xleftarrow{\sim} \varinjlim_m \operatorname{Hom}_K(V_n, W_{m}).
\end{align}
Using cond-regularity of $\{V_n\}_{n \in \mathbb{Z}_{\ge 0}}$, we see that 
\begin{align}
    \operatorname{Hom}_{\ul{K}}(\ul{V_{\infty}},\ul{W_{\infty}})  \simeq  \operatorname{Hom}_{\ul{K}}(\varinjlim_n \ul{V_n}, \ul{W_{\infty}}) \xleftarrow{\sim} \varprojlim_n \operatorname{Hom}_{\ul{K}}(\ul{V_n}, \ul{W_{\infty}}).
\end{align}
Using cond-regularity of $\{W_m\}_{m \in \mathbb{Z}_{\ge 0}}$, we identify
\begin{align}
    \operatorname{Hom}_{\ul{K}}(\ul{V_n}, \ul{W_{\infty}}) = \operatorname{Hom}_{\ul{K}}(\ul{V_n}, \varinjlim_m \ul{W_m}).
\end{align}
Now we show that
\begin{align}
    \operatorname{Hom}_{\ul{K}}(\ul{V_n}, \varinjlim_m \ul{W_m}) \xleftarrow{\sim}  \varinjlim_m  \operatorname{Hom}_{\ul{K}}(\ul{V_n},\ul{W_m})
\end{align}
following \cite[proof of Lemma 3.32]{RJRC}: Any morphism $f:\ul{V_n} \to \ul{W_{\infty}}$ induces a morphism\footnote{This is a morphism of $K$-vector spaces that is continuous, although the right hand side might not be a topological $K$-vector space.}
\begin{align}
    V_n \to \ul{V_n}(\ast)_{\mathrm{top}} \to \ul{W_{\infty}}(\ast)_{\mathrm{top}}.
\end{align}
Since $V_n$ is Fr\'echet, the first arrow is an isomorphism. Now by cond-regularity (and using that each $W_m$ is Fr\'echet and thus metrizable hence compactly generated in the last step)
\begin{align}
    \ul{W_{\infty}}(\ast)_{\mathrm{top}} = (\varinjlim_m \ul{W_m})(\ast)_{\mathrm{top}} = \varinjlim_m \ul{W_m}(\ast)_{\mathrm{top}} = \varinjlim_m W_m,
\end{align}
where the final colimit is taken in the category of topological spaces. We can thus compose $f$ with the natural continuous bijection $\varinjlim_m W_m \to W_{\infty}$ to get an induced continuous map $V_n \to W_{\infty}$ of $K$-vector spaces, and thus a morphism in the category of locally convex $K$-vector spaces. This map factors through a morphism $V_n \to W_m$ for some $m$ by \cite[Corollary 8.9]{SchneiderFA}. The induced map $\ul{V_n} \to \ul{W_m}$ thus factors $f$ (this can be checked after evaluating on $\ast$). 

We have reduced the lemma to showing that
\begin{align}
     \operatorname{Hom}_K(V_n, W_{m}) \to \operatorname{Hom}_{\ul{K}}(\ul{V_{n}},\ul{W_m})
\end{align}
is a bijection, and this holds because $V_n$ and $W_m$ are Fr\'echet and thus metrizable hence compactly generated. 
\end{proof}



\DeclareRobustCommand{\VAN}[3]{#3}

\end{document}